\documentclass[a4paper,10pt]{amsart}
\usepackage{amsmath,amsfonts,amssymb, amsthm, xypic}
\usepackage{epsfig,psfrag}
\usepackage{bbm}
\usepackage[all]{xy}

\def\endexample{$\hfill \triangle$}
\def\vecdq{\overline{Q}}

\def\Z{\mathbb{Z}}

\def\H{\mathbf{H}}
\def\U{\mathbf{U}}
\def\qed{$\hfill \checkmark$}
\def\endexample{$\hfill \triangle$}
\def\N{\mathbb{N}}
\def\tto{\twoheadrightarrow}

\def\a{\alpha}
\def\b{\beta}

\def\qlb{\overline{\mathbb{Q}_l}}
\def\C{\mathbb{C}}

\def\fqb{\overline{\mathbb{F}_q}}
\def\fq{\mathbb{F}_q}
\def\A{\mathcal{A}}
\def\g{\mathfrak{g}}
\def\bo{\mathfrak{b}}
\def\ux{\underline{x}}
\def\FD{\boldsymbol{\Theta}}
\def\h{\mathfrak{h}}
\def\n{\mathfrak{n}}

\def\v{\nu}

\def\Ma{\underline{\mathcal{M}}^{\alpha}_{\vec{Q}}}

\def\QQ{\mathcal{Q}_{\vec{Q}}}
\def\PQ{\mathcal{P}_{\vec{Q}}}
\def\KQ{\mathcal{K}_{\vec{Q}}}

\newtheorem{theo}{\bf{Theorem}}[section]
\newtheorem{lem}[theo]{Lemma}
\newtheorem{cor}[theo]{Corollary}
\newtheorem{prop}[theo]{Proposition}

\newtheorem{conj}[theo]{Conjecture}

\numberwithin{equation}{section}

\title{Lectures on canonical and crystal bases of Hall algebras}
\author{Olivier Schiffmann}

\begin{document}
\maketitle
\tableofcontents

\newpage

%\tableofcontents
%\setlength{\unitlength}{10pt}

\centerline{\large{\textbf{Introduction}}}\addcontentsline{toc}{section}{\tocsection {}{}{Introduction}}

\vspace{.2in}

These notes form the support of a series of lectures given for the summer school ``Geometric methods in representation theory'' at the Institut Fourier  in Grenoble in June 2008. They represent the second half of the lecture series. The first half of the series was dedicated to the notion of the Hall algebra $\mathbf{H}_\A$ of an abelian (or derived) category $\A$, and the notes for that part are written in \cite{Trieste}. The present text is a companion to \cite{Trieste}; we will use the same notation as in \cite{Trieste} and sometimes refer to \cite{Trieste} for definitions. Nevertheless, we have tried to make this text as self-contained as (reasonably) possible.

\vspace{.1in}

In this part of the lecture series we explain how to translate the purely algebraic construction of Hall algebras given in \cite{Trieste} into a geometric one. This geometric lift amounts to replacing the ``naive'' space of functions on the set $\mathcal{M}_{\A}$ of objects of a category $\A$ by a suitable category $\mathcal{Q}_{\A}$ of constructible sheaves on the moduli space (or more precisely, moduli stack) $\underline{\mathcal{M}}_{\A}$ parametrizing the objects of $\A$. The operations in the Hall algebra (multiplication and comultiplication) then ought to give rise to functors
$$\underline{m}:  \mathcal{Q}_{\A} \times \mathcal{Q}_{\A} \to \mathcal{Q}_{\A},$$
$$\underline{\Delta}: \mathcal{Q}_{\A} \to \mathcal{Q}_{\A} \times \mathcal{Q}_{\A}.$$
 The Faisceaux-Fonctions correspondence of Grothendieck, which associates to a constructible sheaf $\mathbb{P} \in {\mathcal{Q}}_{\A}$ its trace--a (constructible) function on the ``naive'' moduli space $\mathcal{M}_{\A}$--  draws a bridge between the ``geometric'' Hall algebra (or rather, \textit{Hall category}) $\mathcal{Q}_{\A}$ and the ``algebraic'' Hall algebra $\mathbf{H}_{\A}$.  Such a geometric lift from $\mathbf{H}_{\A}$ to $\mathcal{Q}_{\A}$ may be thought of as ``categorification'' of the Hall algebra (and is, in fact, one of the early examples of ``categorification'').
 
 \vspace{.1in}
 
 Of course, for the above scheme to start making any sense, a certain amount of technology is required : for one thing, the moduli stack $\underline{\mathcal{M}}_{\A}$ has to be rigorously defined and the accompanying formalism of constructible or $l$-adic sheaves has to be developped. The relevant language for a general theory is likely to be \cite{Toen}.  Rather than embarking on the (probably risky) project of defining the Hall category $\mathcal{Q}_{\A}$ for an arbitrary abelian category $\A$ using that language we believe it will be more useful to focus in these lectures on several examples. 
 Another reason for this is that, as explained in \cite[Section~5]{Trieste}, the correct setting for the theory of Hall algebras (especially for categories of global dimension more than one) seems to be that of derived or triangulated categories. The necessary technology to deal with moduli stacks parametrizing objects in \textit{triangulated} (or \textit{dg}) categories is, as far as we know, still in the process of being fully worked out, see \cite{Toen}, \cite{Toen2}, (this, in any case, far exceeds the competence of the author).

 \vspace{.1in}

The main body of the \textit{existing} theory is the work of Lusztig when $\A=Rep_k\vec{Q}$ is the category of representations of a quiver $\vec{Q}$ over a finite field $k$ (see \cite{L1}, \cite{L2})\footnote{Actually, the theory really originates from Lusztig's theory of \textit{character sheaves} in the representation theory of finite groups of Lie type, see \cite{L3}. This, however, has little to do with Hall algebras.}, which we now succintly describe. In that case (see \cite{Trieste}), there is an embedding
$$\phi: \U^+_v(\g) \hookrightarrow \H_{\vec{Q}}$$
of the positive half of the quantum enveloping algebra of the Kac-Moody algebra $\g$ associated to $\vec{Q}$ into the Hall algebra. The image of this map is called the \textit{composition subalgebra} $\mathbf{C}_{\vec{Q}}$ of $\H_{\vec{Q}}$ and is generated by the constant functions $1_{\alpha}$ for $\alpha$ running among the classes of \textit{simple} objects in ${Rep}_k\vec{Q}$. These classes bijectively correspond to the positive simple roots of $\g$ and we will call them in this way. The moduli stack
$\underline{\mathcal{M}}_{\vec{Q}}$ parametrizing objects of $Rep_{\bar{k}}\vec{Q}$ splits into a disjoint union
$$\underline{\mathcal{M}}_{\vec{Q}} = \bigsqcup_{\alpha \in K_0(\vec{Q})} \underline{\mathcal{M}}_{\vec{Q}}^{\alpha}$$
according to the class in the Grothendieck group. Let $D^b(\underline{\mathcal{M}}^{\alpha}_{\vec{Q}})$ stand for the triangulated category of constructible complexes on $\underline{\mathcal{M}}^{\alpha}_{\vec{Q}}$ (see Lecture~1 for precise definitions). For $\alpha, \beta \in K_0(\vec{Q})$, let $\underline{\mathcal{E}}^{\alpha,\beta}$ be the stack parametrizing inclusions
$M \supset N$, where $M$ and $N$ are objects in $Rep_{\bar{k}}\vec{Q}$ of class $\alpha+\beta$ and $\beta$ respectively. There are natural maps $p_1$ and $p_2$ ~:
\begin{equation}\label{E:corresp}
\xymatrix{ &  \underline{\mathcal{E}}^{\alpha, \beta} \ar[ld]_-{p_1} \ar[rd]^-{p_2} &\\
\ \underline{\mathcal{M}}_{\vec{Q}}^{\alpha} \times  \underline{\mathcal{M}}_{\vec{Q}}^{\beta} & &
\underline{\mathcal{M}}_{\vec{Q}}^{\alpha+\beta}
}
\end{equation}
defined by $p_1: (M \supset N) \mapsto (M/N, N)$, $p_2: (M \supset N) \mapsto (M).$
The map $p_2$ is \textit{proper} (the fiber of $p_2$ over $M$ is the Grassmanian of subobjects $N$ of $M$ of class $\beta$, a projective scheme) while $p_1$ can be shown to be \textit{smooth}. One then considers the functors
\begin{equation*}
\begin{split}
\underline{m}~: D^b(\underline{\mathcal{M}}^{\alpha}_{\vec{Q}} \times \underline{\mathcal{M}}^{\beta}_{\vec{Q}}) &\to  D^b(\underline{\mathcal{M}}^{\alpha+\beta}_{\vec{Q}} )\\
\mathbb{P} &\mapsto p_{2!} p_1^* (\mathbb{P}),
\end{split}
\end{equation*}
and 
\begin{equation*}
\begin{split}
\underline{\Delta}~: D^b(\underline{\mathcal{M}}^{\alpha+\beta}_{\vec{Q}})  &\to  D^b(\underline{\mathcal{M}}^{\alpha}_{\vec{Q}} \times \underline{\mathcal{M}}^{\beta}_{\vec{Q}}) \\
\mathbb{P} &\mapsto p_{1!} p_2^* (\mathbb{P}).
\end{split}
\end{equation*}
These functors can be shown to be (co)associative. Because $p_1$ is smooth and $p_2$ is proper, the functor $\underline{m}$ preserves the subcategory $D^b(\underline{\mathcal{M}}_{\vec{Q}})^{ss}$ of $D^b(\underline{\mathcal{M}}_{\vec{Q}})$ consisting of \textit{semisimple complexes} of geometric origin (this is a consequence of the celebrated Decomposition Theorem of \cite{BBD}). The category $\mathcal{Q}_{\vec{Q}}$ is defined to be the smallest triangulated subcategory of $D^b(\underline{\mathcal{M}}_{\vec{Q}})^{ss}$  which is stable under $\underline{m}$ and taking direct summands, and which contains the
constant complexes $\mathbbm{1}_{\alpha}=\overline{\mathbb{Q}_l}_{\underline{\mathcal{M}}_{\vec{Q}}^{\alpha}}$ as $\alpha$ runs among the set of simple roots. In other words, if for any two constructible complexes $\mathbb{P}_1, \mathbb{P}_2$ we set $\mathbb{P}_1 \star \mathbb{P}_2 =\underline{m} ( \mathbb{P}_1 \boxtimes \mathbb{P}_2)$ then $\mathcal{Q}_{\vec{Q}}$ is by definition the additive subcategory of $D^b(\underline{\mathcal{M}}_{\vec{Q}})^{ss}$ generated by the set of all simple constructible complexes (i.e. simple perverse sheaves) appearing as a direct summand of some
semisimple complex
$$\mathbbm{1}_{\alpha_1} \star \cdots \star \mathbbm{1}_{\alpha_r}$$
for $\alpha_1, \ldots, \alpha_r$ some simple roots. These simple perverse sheaves generating $\mathcal{Q}_{\vec{Q}}$ are hard to determine for a general quiver. They are, however, known when the quiver is of finite or affine type (see Lecture~2). The category $\mathcal{Q}_{\vec{Q}}$ thus constructed is preserved under the functors $\underline{m}$ and $\underline{\Delta}$ (this is clear for $\underline{m}$, but not so obvious for $\underline{\Delta}$).

The Faisceaux-Fonctions correspondence associates to a constructible complex $\mathbb{P} \in D^b(\underline{\mathcal{M}}_{\vec{Q}})$ its trace\footnote{One has to be a little more precise here~: for this to be well-defined, $\mathbb{P}$ has to be endowed with a Weil, or Frobenius, structure. This is the case of the objects of $\mathcal{Q}_{\vec{Q}}$--see Lecture~3.}, defined as
$$Tr (\mathbb{P}): x \mapsto \sum_i (-1)^i Tr( Fr, H^i(\mathbb{P})_{|x}).$$
This is a constructible function on $\mathcal{M}_{\vec{Q}}$ and may thus be viewed as an element of the Hall algebra $\mathbf{H}_{\vec{Q}}$. General formalism ensures that the trace map carries the functors $\underline{m}$ and $\underline{\Delta}$ to the multiplication and comultiplication maps $m$ and $\Delta$ of the Hall algebra.

It turns out that the the image under the trace map from the Hall category $\mathcal{Q}_{\vec{Q}}$ to the Hall algebra $\H_{\vec{Q}}$ is precisely equal to $\mathbf{C}_{\vec{Q}}$, and that $\mathbf{C}_{\vec{Q}}$ gets in this way identified with the graded Grothendieck group of $\mathcal{Q}_{\vec{Q}}$ (this identification is the hardest part of the theory). Classes of the simple perverse sheaves in $\mathcal{Q}_{\vec{Q}}$ then give rise to a basis of the composition algebra $\mathbf{C}_{\vec{Q}}$, and thus also to a basis of the quantum group $ \U^+_v(\g)$--the so-called \textit{canonical basis} (see Lecture 3). This basis $\mathbf{B}$ has many remarkable properties (such as integrality and positivity of structure constants, compatibility with all highest weight integrable representations, etc.). It was later shown in \cite{GL} that $\mathbf{B}$ coincides with Kashiwara's \textit{global basis} of $ \U^+_v(\g)$ (see \cite{Kas}), which is defined in a purely algebraic way. The theory of canonical bases for quantum groups has attracted an enormous amount of research since its invention and has found applications in fields like algebraic and combinatorial representation theory, algebraic geometry and knot theory (see \cite{LLT}, \cite{Ariki}, \cite{VV}, \cite{Caldero}, \cite{FrenkelK}, \cite{FZ}, \cite{Lustot},...).

\vspace{.2in}

Kashiwara defined a certain colored graph structure, or \textit{crystal graph} structure, on the set of elements of the global (or canonical) basis $\mathbf{B}$ of $ \U^+_v(\g)$, encoded by the so-called \textit{Kashiwara operators} $\tilde{e}_{\alpha}, \tilde{f}_{\alpha}$ (for $\alpha$ a simple root). To say that the concept of crystal graphs has found useful applications in algebraic and 
combinatorial representation theory is a gross understatement\footnote{In fact, this sentence itself is a gross understatement.}.
As discovered by Kashiwara and Saito, the action of the operators $\tilde{e}_{\alpha}, \tilde{f}_{\alpha}$ on $\mathbf{B}$ are also beautifuly related to the geometry of the moduli spaces of quiver representations (see \cite{KS}). More precisely, the canonical basis $\mathbf{B}$ may be identified with the set of irreducible components of the \textit{Lusztig nilpotent variety} $\underline{\Lambda}_{\vec{Q}}=\bigsqcup_{\alpha} \underline{\Lambda}^{\alpha}_{\vec{Q}}$ which is a certain Lagrangian subvariety in the cotangent bundle $T^*\underline{\mathcal{M}}_{\vec{Q}}=\bigsqcup_{\alpha} T^* \underline{\mathcal{M}}^{\alpha}_{\vec{Q}}$. The Kashiwara operators naturally correspond to certain generic affine fibrations between irreducible components for different values of $\alpha$ (see Lecture~4). It is important to note the difference with the construction of the canonical basis $\mathbf{B}$ itself, which is given in terms of the geometry of $\underline{\mathcal{M}}_{\A}$ rather than that of $T^*\underline{\mathcal{M}}_{\A}$. The interaction between constructible sheaves (or perverse sheaves, or $D$-modules) on a space $X$ and Lagrangian subvarieties of $T^*X$ is a well-known and common phenomenon in topology and geometric representation theory (see e.g. \cite{KasShap}, \cite{Nadler}). 

\vspace{.2in}

Ideas analogous to the above, but for the categories of coherents sheaves on smooth projective curves rather than representations of quivers, have only more recently been developped, starting with the work of Laumon in \cite{Laumon} (see \cite{SInvent}, \cite{Scano}). Hence if $X$ is a smooth projective curve (which may have a finite number of points with orbifold structure), one defines a certain category 
$\mathcal{Q}_{X}$ of constructible sheaves on the moduli stack $\underline{Coh}_X$ parametrizing coherent sheaves on $X$. There is again a trace map $Tr : K_0(\mathcal{Q}_{X}) \to \mathbf{H}_X$, and it is conjectured (and proved in low genera) that its image coincides with the composition algebra $\mathbf{C}_X$ of the Hall algebra $\mathbf{H}_X$. Recall from \cite{Trieste} that (again, at least in low genera) these composition algebras are related to quantum \textit{loop algebras}.

It is interesting to note that, although the point of view advocated here --motivated by the theory of Hall algebras-- is new, the objects themselves (such as the moduli stack $\underline{Coh}_X$) are very classical~: the category $\mathcal{Q}_X$ is closely related to Laumon's theory of Eisenstein sheaves (for the group $GL(n)$); similarly, the analogue of Lusztig's nilpotent variety $\underline{\Lambda}_{\vec{Q}}$ in this context is the so-called (Hitchin) \textit{global nilpotent cone} $\underline{\Lambda}_X$.

Even though the theory of canonical bases is still very far from being in its definite form in the case of coherent sheaves on curves, we cannot resist mentioning some aspects of it, if only very briefly, in the last part of this text (Lecture~5).

\vspace{.2in}

The plan for these lectures is as follows~:

\begin{itemize}
\item Lecture~1. Lusztig's category $\mathcal{Q}_{\A}$ for $\mathcal{A}=Rep(\vec{Q})$,

\vspace{.03in}

\item Lecture~2. Examples  (finite type quivers, cyclic quivers and affine quivers),

\vspace{.03in}

\item Lecture~3.  The canonical basis $\mathbf{B}$ and the trace map,

\vspace{.03in}

\item Lecture~4. Kashiwara's crystal graph $\mathcal{B}(\infty)$ and Lusztig's nilpotent variety $\underline{\Lambda}_{\vec{Q}}$,

\vspace{.03in}

\item Lecture~5.  Hall categories for Curves.
\end{itemize}

\vspace{.1in}

As \cite{Trieste}, these notes are written in an informal style and stroll around rather than speed through. They are mostly intended for people interested in quantum groups or representation theory of quivers and finite-dimensional algebras but only marginal familiarity with these subjects is required (and whatever is provided by the appendices in \cite{Trieste} is enough). We have assumed some knowledge of the theory of construcible or perverse sheaves ~: very good introductory texts abound, such as \cite{Massey}, \cite{Achar}, \cite{Rietsch} or \cite{BorelDmod}. Of course, the ultimate reference is \cite{BBD}. The short Chap.~8 of \cite{Lusbook} contains a list of most results which we will need.

\vspace{.1in}

Despite the fact that the geometric approach to canonical bases has attracted a lot of research over the years we have chosen by lack of time to focus here on the more ``classical'' aspects of the theory
(put aside the inevitable personal bias of Lecture~5). We apologize to all those whose work we don't mention in these notes.
Experts would be hard pressed to find a single new, original statement in these notes, except for some of the conjectures presented in Section~5.5.

\vspace{.2in}

\newpage

\centerline{\large{\textbf{Lecture~1.}}}

\addcontentsline{toc}{section}{\tocsection {}{}{Lecture~1.}}

\setcounter{section}{1}

\vspace{.2in}

The aim of this first Lecture is to explain the construction, due to Lusztig, of a category $\mathcal{Q}_{\vec{Q}}$ of semisimple perverse sheaves on the moduli space $\underline{\mathcal{M}}_{\vec{Q}}$ of representations of any quiver $\vec{Q}$.
The category $\mathcal{Q}_{\vec{Q}}$ will be endowed with two exact functors
$\underline{m}:  \mathcal{Q}_{\vec{Q}} \times \mathcal{Q}_{\vec{Q}} \to \mathcal{Q}_{\vec{Q}}$ (the \textit{induction} functor) and
$\underline{\Delta}: \mathcal{Q}_{\vec{Q}} \to \mathcal{Q}_{\vec{Q}} \times \mathcal{Q}_{\vec{Q}}$ (the \textit{restriction functor}). These turn its Grothendieck group $\mathcal{K}_{\vec{Q}}$ into an algebra and a coalgebra (and in fact, as we will see later, into a bialgebra). The precise relation with the quantum group associated to $\vec{Q}$, as well as the relation with the Hall algebra of $\vec{Q}$ over a finite field, will be discussed in Lecture~3.

There is one point worth making here. By ``moduli space $\underline{\mathcal{M}}_{\vec{Q}}$ '', we mean a space which parametrizes \textit{all} representations of $\vec{Q}$ (rather than the kind of moduli spaces provided by, say, Geometric Invariant Theory, which account only for semistable representations). Such a ``moduli space'' cannot be given the structure of an honest algebraic variety and one has to consider it as a \textit{stack}. This causes several technical difficulties. Luckily, the stacks in question are of a simple nature (they are global quotients) and can be dealt with concretely (using equivariant sheaves, etc...). Our leitmotiv will be to state as many things as possible heuristically using the stacks language--hoping that this will make the ideas more transparent-- and then to translate it back to a more concrete level.

We begin by constructing these moduli spaces $\underline{\mathcal{M}}_{\vec{Q}}$ and the categories of constructible sheaves over them. Then we introduce the induction/restriction functors and use these to define Lusztig's category $\mathcal{Q}_{\vec{Q}}$. All the results of this Lecture are due to Lusztig (see \cite{Lusbook} and the references therein).

\vspace{.1in}

\textit{We fix once and for all an algebraically closed field $k=\overline{\mathbb{F}_q}$ of positive characteristic\footnote{see Remark~3.27 for a comment on this assumption.}.}

\vspace{.3in}

\centerline{\textbf{1.1. Recollections on quivers.}}
\addcontentsline{toc}{subsection}{\tocsubsection {}{}{\;\;1.1. Recollections on quivers.}}

\vspace{.15in}

We will usually stick to the notation of \cite[Lecture~3]{Trieste} which we briefly recall for the reader's convenience. Let $\vec{Q}=(I,H)$ be a quiver. Here $I$ is the set of vertices, $H$ is the set of edges, and the source and target maps are denoted by $s,t: H \to I$. We assume that there are no edge loops\footnote{In fact, quite a bit of the theory may  be generalized to the situation in which loops are allowed. The relevant algebraic objects are not Kac-Moody algebras, but the so-called \textit{generalized Kac-Moody algebras} (see \cite{SK}, \cite{SKK}, \cite{Lustight}, \cite{LZ}). Similarly, one may consider quivers equipped with a cyclic group of automorphisms-- these will correspond to Cartan matrix of non simply laced types (see \cite{Lusbook}). }, i.e. that $s(h) \neq t(h)$ for any $h \in H$.
 The Cartan matrix $A=(a_{i,j})_{i,j \in I}$ of $\vec{Q}$ is defined to be
$$a_{i,j}=\begin{cases} 2 & \qquad \text{if}\;i=j,\\
-\#\{h: i \to j\} -\#\{h: j \to i\} & \qquad \text{otherwise.}\end{cases}
$$
It is a symmetric integral matrix.

\vspace{.1in}

Recall that a representation of $\vec{Q}$ over $k$ is a pair $(V,\underline{x})$ where $V= \bigoplus_{i \in I} V_i$ is a finite-dimensional $I$-graded $k$-vector space and $\underline{x}=(x_h)_{h \in H}$ is a collection of linear maps $x_h: V_{s(h)} \to V_{t(h)}$. A morphism between two representations $(V,\underline{x})$ and $(V',\underline{x'})$ is an $I$-graded linear map $f: V \to V'$ satisfying $f \circ x_h=x'_h \circ f$ for all edges $h$.
The category of representations of $\vec{Q}$ over $k$ form an abelian category denoted $Rep_k\vec{Q}$. A representation $(V, \underline{x})$ is nilpotent if there exists $N \gg 0$ such that any composition of maps $x_{h_n} \cdots x_{h_1}$ of length $n >N$ vanishes. Nilpotent representations form a Serre subcategory $Rep^{nil}_k\vec{Q}$. This means that $Rep^{nil}_k\vec{Q}$ is stable under subobjects, quotients and extensions. This Serre subcatgeory will actually be more important to us than $Rep_k\vec{Q}$ itself. Of course, if $\vec{Q}$ has no oriented cycles then $Rep_k\vec{Q}$ and $Rep^{nil}_k\vec{Q}$ coincide.

\vspace{.15in}

The dimension vector of a representation $(V,\underline{x})$ is by definition the element $\underline{dim}\;V=({dim}\;V_i)_{i \in I} \in \N^I$. The dimension vector is clearly additive under short exact sequences. This gives rise to a linear form $K_0(Rep_{{k}}\vec{Q}) \to \Z^I$ on the Grothendieck group, which restricts to an isomorphism $K_0(Rep^{nil}_{k}\vec{Q}) \simeq \Z^I$ (see \cite[Cor. 3.2.]{Trieste} ). We will set $K_0(\vec{Q})=K_0(Rep^{nil}_{{k}}\vec{Q})$ and identify it with $\Z^I$ as above.

\vspace{.15in}

The category $Rep_k\vec{Q}$ is hereditary. The additive Euler form
\begin{align*}
\langle\;,\; \rangle_a~: K_0(Rep_k\vec{Q}) \otimes K_0(Rep_k\vec{Q})&\to \Z\\
\langle M, N \rangle_a &=dim\;Hom(M,N) - dim\; Ext^1(M,N)
\end{align*}
factors through $\Z^I$. It is given by the formula
\begin{equation}\label{E:10}
\langle \alpha, \beta \rangle_a =\sum_{i} \alpha_i \beta_i -\sum_{h \in H} \alpha_{s(h)}\beta_{t(h)}.
\end{equation}
Observe that the matrix of the symmetrized Euler form defined by 
$$(M,N)_a=\langle M,N\rangle_a + \langle N,M\rangle_a$$
is equal to the Cartan matrix $A$. The same result holds for $Rep^{nil}_k\vec{Q}$. We will often drop the index $a$ in the notation for $\langle\;,\;\rangle$ and $(\;,\;)$.

\vspace{.15in}

Let $(\epsilon_i)_{i \in I}$ stand for the standard basis of $\Z^I$. Since $\vec{Q}$ is assumed to have no edge loops, there is a unique representation $S_i$ of dimension vector $\epsilon_i$ (up to isomorphism). The representation $S_i$ is simple and every simple nilpotent representation is of this form.

\vspace{.3in}

\centerline{\textbf{1.2. Moduli spaces $\underline{\mathcal{M}}^{\alpha}_{\vec{Q}}$ of representations of quivers.}}
\addcontentsline{toc}{subsection}{\tocsubsection {}{}{\;\;1.2. Moduli spaces of representations of quivers.}}

\vspace{.15in}

Our aim is now to construct a ``moduli space'' of representations of $\vec{Q}$ over our field $k$. As mentioned in the introduction,  we need a moduli space which will account for \textit{all} representations. Fix $\alpha \in \N^I$ and let us try to construct a moduli space $\underline{\mathcal{M}}^{\alpha}_{\vec{Q}}$ parametrizing representations of dimension $\alpha$. This is actually quite simple~: we consider the space
$$E_{\alpha}=\bigoplus_{h \in H} {Hom}(k^{\a_{s(h)}},k^{\a_{t(h)}})$$
of all representations in the \textit{fixed} vector space $V_{\alpha}:=\bigoplus_{i} k^{\a_i}$. The group
$$G_{\alpha}=\prod_i GL(\alpha_i,k)$$
acts on $E_{\alpha}$ by conjugation $g \cdot \underline{y}=g \underline{y} g^{-1}$. More precisely, if
$\underline{y}=(y_h)_{h \in H}$ then $g \cdot \underline{y}=(g_{t(h)} y_h g_{s(h)}^{-1})_{h \in H}$.

\vspace{.1in}

\begin{prop} There is a canonical bijection

$$\begin{Bmatrix} isoclasses\;of\\ representations\;of\\ \vec{Q}\;of\;dimension\;\alpha \end{Bmatrix}\quad
\longleftrightarrow \quad\begin{Bmatrix} \\G_{\alpha}-orbits\;in\;E_{\alpha}\\ \\ \end{Bmatrix}.
$$
\end{prop}

\begin{proof} By definition two representations $\underline{y}, \underline{y}' \in E_{\alpha}$ are isomorphic if and only if they belong to the same $G_{\alpha}$-orbit. On the other hand,  any representation $(V,\underline{x})$ of $\vec{Q}$ of dimension $\alpha$ is isomorphic to a representation $(V_{\alpha},\underline{y})$ for some $\underline{y} \in E_{\alpha}$--it suffices to choose an $I$-graded isomorphism of vector spaces $\phi: V_{\alpha} \stackrel{\sim}{\to} V$ and set $\underline{y}=\phi^* \underline{x}$. The Proposition follows.\end{proof}

\vspace{.1in}

In view of this it is natural to set
\begin{equation}\label{E:11}
\underline{\mathcal{M}}^{\alpha}_{\vec{Q}}=E_{\alpha}/G_{\alpha}.
\end{equation}

The quotient $E_{\alpha}/G_{\alpha}$ parametrizes isomorphism classes of representations of $\vec{Q}$ of dimension $\a$. It is not an algebraic variety but can be given sense as a \textit{stack} (see \cite{LMB}).  Note that $E_{\alpha}$ is a smooth algebraic variety (a vector space) and $G_{\alpha}$ is a connected reductive algebraic group. We will (more or less) easily translate all operations on $\underline{\mathcal{M}}^{\alpha}_{\vec{Q}}$ in terms of $E_{\alpha}$ and $G_{\alpha}$. The reader may, if he wishes, think of (\ref{E:11}) as a convenient notation, and later as a heuristic guide, rather than as a precise mathematical definition.

\vspace{.1in}

Let $E_{\alpha}^{nil} \subset E_{\alpha}$ stand for the closed subset consisting of nilpotent representations. We put
\begin{equation}\label{E:12}
\underline{\mathcal{M}}^{\alpha,nil}_{\vec{Q}}=E^{nil}_{\alpha}/G_{\alpha}.
\end{equation}
It is a closed (in general singular) substack of $\underline{\mathcal{M}}^{\alpha}_{\vec{Q}}$. 

\vspace{.1in}

The geometry of the spaces $\underline{\mathcal{M}}^{\alpha}_{\vec{Q}}$ and $\underline{\mathcal{M}}^{\alpha,nil}_{\vec{Q}}$ (i.e, the orbit geometry of the spaces $E_{\alpha}, E_{\alpha}^{nil}$) has been intensively studied by many authors (see e.g. \cite{Bongartz}, \cite{Kraft-Erdmann}, \cite{Zwara} ...). One may consult \cite{CrawleyNotes} for a good introduction to this topic. We will give several examples  in Lecture~2. At this point, we will content ourselves with

\vspace{.1in}

\begin{lem}\label{L:dim} We have $dim\;E_{\alpha}=-\langle\alpha,\alpha \rangle + dim\;G_{\alpha}.$
\end{lem}
\begin{proof} This can be derived from
$$dim\;E_{\alpha}=\sum_{h \in H} \alpha_{s(h)} \alpha_{t(h)}, \qquad
\langle \alpha, \alpha \rangle =\sum_i \alpha_i^2-\sum_{h \in H} \alpha_{s(h)} \alpha_{t(h)}.$$
\end{proof}

\vspace{.1in}

As a consequence, we have

\vspace{.1in}

\begin{cor} The (stacky) dimension of the moduli space $\Ma$ is given by 
\begin{equation}\label{E:101}
dim\;\Ma=-\langle \alpha, \alpha \rangle.
\end{equation}
\end{cor}

\vspace{.1in}

We state one useful generalization of Lemma~\ref{L:dim}. For dimension vectors $\a_1, \ldots, \a_n$ we define
$$E_{\a_1, \ldots, \a_n}=\big\{ (\underline{y},\, W_n \subset W_{n-1} \subset \cdots \subset W_1=V_{ \a_1+ \cdots + \a_n})\big\}$$
where $\underline{y} \in E_{\a_1 + \cdots + \a_n}$, $W_n, \ldots, W_1$ are $\underline{y}$-stable and $\underline{dim}\; W_{i}/W_{i+1}=\a_i$ for all $i$.
The group $G_{\sum \a_i}$ naturally acts on $E_{\a_1, \ldots, \a_n}$ and the quotient may be thought of as a moduli space parametrizing (isomorphism classes of) flags of 
representations of $\vec{Q}$ with successive factors of dimension $\a_1, \ldots, \a_n$ respectively.

\vspace{.1in}

\begin{lem} We have $dim\;E_{\a_1, \ldots, \a_n}=dim\;G_{\a}-\sum_{i \leq j} \langle \a_i,\a_j\rangle$ where $\a=\sum_i \a_i$.
\end{lem}
\begin{proof}
Consider the flag variety $Gr(\a_n, \a_n+\a_{n-1}, \ldots, \a)$ parametrizing chains of $I$-graded subspaces 
$W_n \subset \cdots \subset W_1=V_{\a}$ satisfying $dim\;W_{i}/W_{i+1}=\a_i$. This is a (smooth) projective variety of dimension 
$$ dim\; Gr(\a_n, \a_n+\a_{n-1}, \ldots, \a)=dim\;G_{\a}-\sum_i \sum_{k \leq l} \a_k^i\a_l^i$$
where $\a_t=(\a_t^i)_{i \in I}$.  It is easy to see that the natural projection $\pi: E_{\a_1, \ldots, \a_n} \to Gr(\a_n, \ldots, \a)$ is a vector bundle, whose rank is given by the following formula
$$rank(\pi)=\sum_{h\in H} \sum_{k \leq l} \a_k^{s(h)}\a_l^{t(h)}.$$
The Lemma follows by (\ref{E:10}).
\end{proof}

\vspace{.1in}

 Let $D^b(\Ma):=D^b_{G_{\alpha}}(E_{\alpha})$ be the $G_{\a}$-equivariant derived category of $\qlb$-constructible complexes over $E_{\a}$, as in \cite{Bernstein-Lunts}. We also let $D^b(\Ma)^{ss}:=D^b_{G_{\alpha}}(E_{\alpha})^{ss}$ stand for the full subcategory of $D^b_{G_{\alpha}}(E_{\alpha})$ consisting of semisimple complexes. We will think of these as the triangulated categories of constructible complexes and semisimple constructible complexes over $\Ma$ respectively. In particular, if $U$ is a smooth locally closed and $G_{\a}$-invariant subset of $E_{\a}$, and if $\mathcal{L}$ is a $G_{\a}$-equivariant local system over $U$ then we denote by $IC(U,\mathcal{L})$ the associated simple $G_\a$-equivariant perverse sheaf~: this is the (nonequivariant) simple perverse sheaf associated to $(U,\mathcal{L})$, equipped with its unique $G_{\a}$-equivariant structure. We refer to \cite{FK} and \cite{Bernstein-Lunts} for some of the properties of these categories. 

\vspace{.3in}

\centerline{\textbf{1.3. The induction and restriction functors.}}
\addcontentsline{toc}{subsection}{\tocsubsection {}{}{\;\;1.3. The induction and restriction functors.}}

\vspace{.15in}

The basic idea behind the induction and restriction functors is to consider a correspondence
\begin{equation}\label{E:13}
\xymatrix{ &  \underline{\mathcal{E}}^{\alpha, \beta} \ar[ld]_-{p_1} \ar[rd]^-{p_2} &\\
\ \underline{\mathcal{M}}_{\vec{Q}}^{\alpha} \times  \underline{\mathcal{M}}_{\vec{Q}}^{\beta} & &
\underline{\mathcal{M}}_{\vec{Q}}^{\alpha+\beta}
}
\end{equation}
where $\underline{\mathcal{E}}^{\alpha, \beta}$ is the moduli space parametrizing pairs $(N \subset M)$ satisfying $\underline{dim}(M)=\beta, \;\underline{dim}(N)=\a+\beta$, and where 
$$p_1: (M \supset N) \mapsto (M/N, N), \qquad p_2: (M \supset N) \mapsto M.$$
The map $p_1$ is smooth while $p_2$ is proper. One may view (\ref{E:13}) as encoding the set of extensions of a representation of dimension $\beta$ by one of dimension $\alpha$.

We then define functors
\begin{equation}\label{E:14}
\begin{split}
\underline{m}~: D^b(\underline{\mathcal{M}}^{\alpha}_{\vec{Q}} \times \underline{\mathcal{M}}^{\beta}_{\vec{Q}}) &\to  D^b(\underline{\mathcal{M}}^{\alpha+\beta}_{\vec{Q}} ) \\
\mathbb{P} &\mapsto p_{2!} \,p_1^* (\mathbb{P})[dim\;p_1],
\end{split}
\end{equation}
(induction) and 
\begin{equation}\label{E:145}
\begin{split}
\underline{\Delta}~: D^b(\underline{\mathcal{M}}^{\alpha+\beta}_{\vec{Q}})  &\to  D^b(\underline{\mathcal{M}}^{\alpha}_{\vec{Q}} \times \underline{\mathcal{M}}^{\beta}_{\vec{Q}}) \\
\mathbb{P} &\mapsto p_{1!} \,p_2^* (\mathbb{P})[dim\;p_1]
\end{split}
\end{equation}
(restriction).

\vspace{.1in}

The above definitions actually do make sense if one uses the appropriate language of stacks. However, since we want to be independent of such a language, we rephrase them in more concrete terms. This will necessarily look more complicated than (\ref{E:13}).

\vspace{.2in}

\noindent \textbf{ Induction functor.} First note that $\underline{\mathcal{E}}^{\alpha, \beta}$ may be presented as a quotient stack~:
let $E_{\alpha,\beta}$ be the variety of tuples $(\underline{y},W)$ where $\underline{y} \in E_{\alpha+\beta}$ and $W \subset V_{\alpha+\beta}$ is an $I$-graded subspace of dimension $\beta$ and stable under $\underline{y}$. The group $G_{\alpha+\beta}$ naturally acts on $E_{\alpha,\beta}$ by
$g \cdot (\underline{y},W)=(g \underline{y} g^{-1}, gW)$
and we have $\underline{\mathcal{E}}^{\alpha, \beta}=E_{\alpha,\beta}/G_{\alpha+\beta}$.

\vspace{.1in}

Let $Gr(\beta,\alpha+\beta)$ stand for the Grassmanian of $\beta$-dimensional subspaces in $V_{\alpha+\beta}$. There are two obvious projections
\begin{equation}\label{E:15}
\xymatrix{ &  E_{\alpha, \beta} \ar[ld]_-{q'} \ar[rd]^-{q} &\\
 Gr(\beta, \alpha+\beta)& &
E_{\alpha+\beta}
}
\end{equation}
given respectively by $q'(\underline{y},W)=W$ and $q(\underline{y},W)=\underline{y}$. The map $q'$ is easily seen to be a vector bundle; the fiber over a point
$W$ is equal to $\bigoplus_{h \in H} Hom((V_{\a+\beta})_{s(h)}, W_{t(h)})$. In particular, $E_{\alpha,\beta}$ is smooth. On the other hand, the fiber of $q$ over a point $\underline{y}$ is the set of $W \subset V_{\alpha+\beta}$ which are of dimension $\beta$ and which are $\underline{y}$-stable. These form a closed subset of the projective variety $Gr(\beta,\alpha+\beta)$ and therefore $q$ is proper. 

The map $q$ is clearly $G_{\alpha+\beta}$-equivariant. Moreover, the induced map on the quotient spaces $$\underline{\mathcal{E}}^{\alpha, \beta} = E_{\alpha,\beta}/G_{\alpha+\beta} \to E_{\alpha+\beta}/G_{\alpha+\beta}=\underline{\mathcal{M}}^{\alpha+\beta}_{\vec{Q}}$$
coincides with $p_2$. 

\vspace{.1in}

We cannot  obtain a similar lift of $p_1$ directly since there is no natural map $E_{\alpha,\beta} \to E_{\alpha} \times E_{\beta}$. So we introduce another presentation of
$\underline{\mathcal{E}}^{\alpha, \beta}$ as a quotient stack~: let ${E}^{(1)}_{\alpha,\beta}$ be the variety parametrizing tuples $(\underline{y},W,\rho_{\alpha},\rho_{\beta})$ where $(\underline{y},W)$ belongs to $E_{\alpha,\beta}$ and where $\rho_{\alpha}: V_{\a+\beta}/W \stackrel{\sim}{\to} V_{\alpha}, \; \rho_{\beta}: W \stackrel{\sim}{\to} V_{\beta}$ are linear isomorphisms.
The group $G_{\alpha} \times G_{\beta}$ and $G_{\alpha+\beta}$ act on ${E}^{(1)}_{\alpha,\beta}$ by
$$(g_{\alpha}, g_{\beta}) \cdot (\underline{y}, W, \rho_{\alpha}, \rho_{\beta})=(\underline{y}, W, g_{\alpha} \rho_{\alpha}, g_{\beta}\rho_{\beta}),$$
$$g \cdot (\underline{y}, W, \rho_{\alpha}, \rho_{\beta})=(g \underline{y} g^{-1}, g W, \rho_{\alpha} g^{-1}, \rho_{\beta} g^{-1})$$
respectively. The projection $r:{E}^{(1)}_{\alpha,\beta} \to E_{\alpha, \beta}$ is a principal $G_{\alpha} \times G_{\beta}$-bundle and we have $\underline{\mathcal{E}}^{\alpha, \beta}={E}^{(1)}_{\alpha,\beta}/(G_{\alpha+\beta} \times G_{\alpha} \times G_{\beta})$. 

There is now a canonical map $p:{E}^{(1)}_{\alpha,\beta} \to E_{\alpha} \times E_{\beta}$ given by 
$$p(\underline{y},W, \rho_{\alpha}, \rho_{\beta})=\big(\rho_{\alpha,*} (\underline{y}_{|V_{\alpha+\beta}/W}), \rho_{\beta,*}(\underline{y}_{|W})\big).$$
The map $p$ is smooth and $G_{\alpha+\beta} \times G_{\alpha} \times G_{\beta}$-equivariant (we equip $E_{\alpha} \times E_{\beta}$ with the trivial $G_{\alpha+\beta}$-action). To sum up, we have the following analogue of diagram (\ref{E:13})~:
\begin{equation}\label{E:16}
\xymatrix{ &  {E}^{(1)}_{\alpha,\beta} \ar[r]^{r} \ar[ld]_-{p} & E_{\alpha,\beta} \ar[rd]^-{q} &\\
E_{\alpha} \times  E_{\beta} & & &
E_{\alpha+\beta}
}
\end{equation}
We are now ready to translate (\ref{E:14}). The pull-back functor induces a canonical equivalence of categories
$$r^*~: D^b_{G_{\a+\b}}({E}_{\alpha,\beta}) \stackrel{\sim}{\to} D^b_{G_{\a+\b}\times G_{\a} \times G_{\beta}}({E}^{(1)}_{\alpha,\beta})$$
which has an inverse
$$r_{\flat}~:D^b_{G_{\a+\b}\times G_{\a} \times G_{\beta}}({E}^{(1)}_{\alpha,\beta}) \stackrel{\sim}{\to} D^b_{G_{\a+\b}}({E}_{\alpha,\beta}).$$
We set $r_{\#}=r_{\flat}[-dim\;r]$, so that $r_{\#}r^* (\mathbb{P})=\mathbb{P}[-dim\;r]$ for any $\mathbb{P} \in D^b_{G_{\a+\beta}}(E_{\a,\beta})$. One advantage of $r_{\#}$ over $r_{\flat}$ is that it preserves the subcategories of perverse sheaves. Note that $D^b_{G_{\a+\b}\times G_{\a} \times G_{\beta}}({E}^{(1)}_{\alpha,\beta})$ and $D^b_{G_{\a+\b}}({E}_{\alpha,\beta})$ may both be interpreted as $D^b(\underline{\mathcal{E}}^{\a,\beta})$. We define the induction functor as
\begin{equation} \label{E:165}
\begin{split}
\underline{m}~: D^b_{G_{\a} \times G_{\beta}}(E_{\alpha} \times E_{\beta}) &\to D^b_{G_{\a+\beta}}(E_{\alpha+\beta})\\
\mathbb{P} &\mapsto q_{!}\,r_{\#}\,p^* (\mathbb{P})[dim\;p]
\end{split}
\end{equation}

To finish, observe that

\begin{lem} We have $dim\;p= dim\;G_{\a+\beta}-\langle \alpha, \beta \rangle$. 
\end{lem}
\begin{proof} By definition, 
\begin{equation*}
\begin{split}
dim\;p&=dim\;{E}^{(1)}_{\alpha,\beta}-dim\;E_{\alpha}-dim\;E_{\beta}\\
&=dim\;Gr(\beta,\alpha+\beta)+ dim\;G_{\a} + dim\;G_{\beta} + \sum_{h \in H} \alpha_{s(h)}\beta_{t(h)}\\
&=\sum_i \alpha_i \beta_i + dim\;G_{\a} + dim\;G_{\beta} + \sum_{h \in H} \alpha_{s(h)}\beta_{t(h)}\\
&=dim\;G_{\a+\beta}-\langle \alpha, \beta \rangle
\end{split}
\end{equation*}
as wanted
\end{proof}

For simplicity we often omit the indices $\alpha,\beta$ from the notation $\underline{m}$, hoping that this will not cause any confusion. We will also often use the notation $\mathbb{P} \star \mathbb{Q}=\underline{m}(\mathbb{P} \boxtimes \mathbb{Q})$.

\vspace{.2in}

\paragraph{\textbf{Restriction functor.}} We now turn to the functor $\underline{\Delta}$.  We use yet another presentation of $\underline{\mathcal{E}}^{\alpha, \beta}$ as a quotient stack. Let us fix a $\beta$-dimensional subspace $W_0 \subset V_{\a+\beta}$ as well as a pair of isomorphisms $\rho_{\alpha,0}: V_{\a+\beta}/W_0 \stackrel{\sim}{\to} V_{\alpha}, \;\rho_{\beta,0}: W_0 \stackrel{\sim}{\to} V_{\beta}$. Let $F_{\alpha,\beta}$ be the closed subset of $E_{\alpha+\beta}$ consisting of representations $\underline{y}$ such that $\underline{y}(W_0) \subset W_0$. Let $P_{\alpha,\beta} \subset G_{\a+\beta}$ be the parabolic subgroup associated to $W_0$. Then $\underline{\mathcal{E}}^{\alpha, \beta} = F_{\alpha,\beta}/P_{\a,\beta}$. Note that
\begin{equation}\label{E:FGP}
F_{\a,\beta} \underset{P_{\a,\beta}}{\times} G_{\a+\beta} = E_{\a,\beta}, \qquad F_{\a,\beta} \underset{U_{\a,\beta}}{\times} G_{\a+\beta} = E^{(1)}_{\a,\beta}
\end{equation}
where $U_{\a,\beta} \subset P_{\a,\beta}$ is the unipotent radcal .

We consider the diagram
\begin{equation}\label{E:17}
\xymatrix{ &  F_{\alpha, \beta} \ar[ld]_-{\kappa} \ar[rd]^-{\iota} &\\
E_{\alpha} \times E_{\beta} & &
E_{\alpha+\beta}
}
\end{equation}
where $\kappa (\underline{y})=\big(\rho_{\alpha,*} (\underline{y}_{|V_{\alpha+\beta}/W}), \rho_{\beta,*}(\underline{y}_{|W})\big)$ and where $\iota$ is the embedding. Note that $\kappa$ is a vector bundle of rank
\begin{equation}\label{E:rankkappa}
rank\;\kappa=\sum_h \alpha_{s(h)} \beta_{t(h)}=\sum_i \a_i \beta_i - \langle \a, \beta \rangle.
\end{equation}
Moreover, $\kappa$ is $P_{\alpha,\beta}$-equivariant, where $P_{\alpha,\beta}$ acts on $E_{\alpha} \times E_{\beta}$ via the projection $P_{\alpha,\beta} \to G_{\alpha} \times G_{\beta}$ induced by the pair $\rho_{\alpha,0}, \rho_{\beta,0}$.

\vspace{.1in}

We define the restriction functor as 
\begin{equation} \label{E:18}
\begin{split}
\underline{\Delta}~: D^b_{G_{\a+\beta}}(E_{\alpha+\beta}) &\to D^b_{G_{\a} \times G_{\beta}}(E_{\alpha} \times E_{\beta})\\
\mathbb{P} &\mapsto \kappa_{!}\,\iota^* (\mathbb{P})[-\langle \alpha,\beta\rangle].
\end{split}
\end{equation}
We will at times write $\underline{\Delta}_{\alpha,\beta}$ when we want to specify the dimension vectors.

\vspace{.2in}

\addtocounter{theo}{1}
\noindent \textbf{Remarks \thetheo .} The reader might wonder why we use distinct diagrams for the induction and restriction functors (i.e., why (\ref{E:16}) and (\ref{E:17}) are different). In fact, this is not so essential~: different choices of presentations of the quotient stacks involved in (\ref{E:13}) will lead to the same functors $\underline{m}$ and $\underline{\Delta}$ \textit{up to normalization}. We have given those which are the most convenient (and which coincide with (\ref{E:14}) and (\ref{E:145})).

\vspace{.15in}

To finish this section, we give a few elementary properties of $\underline{m}$ and $\Delta$. We begin with the following useful observation~:

\vspace{.1in}

\begin{lem}\label{L:Verdierind} The functor $\underline{m}$ commutes with Verdier duality.\end{lem}
\begin{proof} By definition, and using the same notations as above, $\underline{m}=q_!r_{\#}p^*[dim\;p]$. Since $p$ is smooth we have $p^*D=D p^*[2dim\;p]$ and hence $D p^*[dim\;p]=p^*D[-dim\;p]=p^*[dim\;p]D$. Similarly, $D r_{\#}=r_{\#}D$ and because $q$ is proper we have $Dq_!=q_*D=q_!D$. The result follows.
\end{proof}

\vspace{.1in}

The functor $\underline{\Delta}$, however, does not commute with Verdier duality in general. This is illustrated (for instance) by the following important result. For $\alpha \in \N^I$ we denote by $\mathbbm{1}_{\alpha}$ the constant sheaf $\qlb_{|E_{\alpha}}[dim\;E_{\alpha}]$. Note that since $E_{\alpha}$ is smooth $D \mathbbm{1}_{\alpha}=\mathbbm{1}_{\alpha}$ and $\mathbbm{1}_{\alpha}$ is perverse.

\vspace{.1in}

\begin{lem} For any $\alpha, \beta , \gamma \in \N^I$ with $\gamma=\alpha+\beta$ it holds
$$\underline{\Delta}_{\alpha,\beta}(\mathbbm{1}_{\gamma})=\mathbbm{1}_{\alpha} \boxtimes \mathbbm{1}_{\beta}[-\langle \beta, \alpha\rangle].$$
\end{lem}
\begin{proof} This is a straightforward computation. From the definitions, 
\begin{equation*}
\begin{split}
\underline{\Delta}_{\alpha,\beta}(\mathbbm{1}_{\gamma})&= \kappa_{!} \iota^*(\mathbbm{1}_{\gamma})[-\langle \alpha, \beta\rangle]\\ 
&=\kappa_{!}(\qlb_{|F_{\alpha,\beta}})[dim\;E_{\gamma} -\langle \alpha,\beta\rangle]\\
&=\qlb_{|E_{\alpha}}\boxtimes \qlb_{|E_{\beta}}[dim\;E_{\gamma}-\langle \alpha, \beta \rangle -2 rank\;\kappa]\\
&=\mathbbm{1}_{\alpha} \boxtimes \mathbbm{1}_{\beta}[-\langle \beta, \alpha \rangle] 
\end{split}
\end{equation*}
since $\kappa$ is a vector bundle. We have used the equalities
$$dim\;E_{\gamma}=\sum (\alpha_{s(h)}+\beta_{s(h)})(\alpha_{t(h)}+\beta_{t(h)}),$$
$$dim\;E_{\alpha}=\sum \alpha_{s(h)}\alpha_{t(h)},$$
$$dim\;E_{\beta}=\sum \beta_{s(h)}\beta_{t(h)}.$$
\end{proof}

\vspace{.1in}

The next proposition shows that $\underline{m}$ and $\underline{\Delta}$ are associative and coassociative functors.

\vspace{.1in}

\begin{prop}\label{P:assoc} For any $\alpha, \beta, \gamma \in \N^I$ there are canonical natural transformations
\begin{equation}\label{E:P1}
\underline{m}_{\alpha, \beta+\gamma} \circ (Id \times \underline{m}_{\beta, \gamma}) \simeq 
\underline{m}_{\alpha+ \beta,\gamma} \circ ( \underline{m}_{\alpha,\beta} \times Id),
\end{equation}
\begin{equation}\label{E:P2}
(\underline{\Delta}_{\alpha, \beta} \times Id) \circ \underline{\Delta}_{\alpha+\beta, \gamma} \simeq
(Id \times \underline{\Delta}_{\beta, \gamma}) \circ \underline{\Delta}_{\alpha,\beta+ \gamma}.
\end{equation}
\end{prop}
\begin{proof} We begin with a heuristic argument. The induction functor $\underline{\Delta}_{\alpha,\beta}$ is defined by a correspondence
\begin{equation}\label{E:P3}
\xymatrix{  \underline{\mathcal{E}}^{\alpha, \beta} \ar[d]_-{p_1} \ar[r]^-{p_2} &\underline{\mathcal{M}}_{\vec{Q}}^{\alpha+\beta}\\
 \underline{\mathcal{M}}_{\vec{Q}}^{\alpha} \times  \underline{\mathcal{M}}_{\vec{Q}}^{\beta} & 
}
\end{equation}
and the induction functor $\underline{m}_{\alpha+\beta, \gamma}$ by the correspondence
\begin{equation}\label{E:P4}
\xymatrix{ \underline{\mathcal{E}}^{\alpha+ \beta,\gamma} \ar[d]_-{p'_1} \ar[r]^-{p'_2} &\underline{\mathcal{M}}_{\vec{Q}}^{\alpha+\beta+\gamma}
\\
 \underline{\mathcal{M}}_{\vec{Q}}^{\alpha+\beta} \times  \underline{\mathcal{M}}_{\vec{Q}}^{\gamma} &
}
\end{equation}
Hence the r.h.s. of (\ref{E:P1}) is obtained by running through a diagram
\begin{equation}\label{E:P5}
\xymatrix{&  \underline{\mathcal{E}}^{\alpha+ \beta,\gamma} \ar[d]_-{p'_1} \ar[r]^-{p'_2} &\underline{\mathcal{M}}_{\vec{Q}}^{\alpha+\beta+\gamma}\\
 \underline{\mathcal{E}}^{\alpha, \beta}\times \underline{\mathcal{M}}^{\gamma}_{\vec{Q}} \ar[d]_-{p_1} \ar[r]^-{p_2} & \underline{\mathcal{M}}_{\vec{Q}}^{\alpha+\beta} \times  \underline{\mathcal{M}}_{\vec{Q}}^{\gamma} & \\
\underline{\mathcal{M}}_{\vec{Q}}^{\alpha} \times  \underline{\mathcal{M}}_{\vec{Q}}^{\beta}\times  \underline{\mathcal{M}}_{\vec{Q}}^{\gamma} & &
}
\end{equation}
We may complete (\ref{E:P5}) by adding a cartesian square
\begin{equation}\label{E:P6}
\xymatrix{
\underline{\mathcal{E}}^{(\alpha, \beta),\gamma} \ar[d]_-{p''_1} \ar[r]^-{p''_2} &  \underline{\mathcal{E}}^{\alpha+ \beta,\gamma} \ar[d]_-{p'_1} \ar[r]^-{p'_2} &\underline{\mathcal{M}}_{\vec{Q}}^{\alpha+\beta+\gamma}\\
\underline{\mathcal{E}}^{\alpha, \beta}\times \underline{\mathcal{M}}^{\gamma}_{\vec{Q}} \ar[d]_-{p_1} \ar[r]^-{p_2} & \underline{\mathcal{M}}_{\vec{Q}}^{\alpha+\beta} \times  \underline{\mathcal{M}}_{\vec{Q}}^{\gamma} & \\
\underline{\mathcal{M}}_{\vec{Q}}^{\alpha} \times  \underline{\mathcal{M}}_{\vec{Q}}^{\beta}\times  \underline{\mathcal{M}}_{\vec{Q}}^{\gamma} & &
}
\end{equation}
By base change, we have
\begin{equation*}
\begin{split}
\underline{m}_{\alpha+\beta,\gamma} \circ (\underline{m}_{\alpha,\beta} \times Id)&= p'_{2!}(p'_1)^*p_{2!}p_1^*[dim\;p_1 + dim\;p'_1]\\
&=p'_{2!}p''_{2!}(p''_1)^*p_1^*[dim\;p_1 + dim\;p''_1]\\
&=(p_2p''_2)_!(p_1p''_1)^*[dim\;p_1p''_1].
\end{split}
\end{equation*}
The stack 
$$\underline{\mathcal{E}}^{(\alpha, \beta),\gamma}:=(\underline{\mathcal{E}}^{\alpha, \beta}\times \underline{\mathcal{M}}^{\gamma}_{\vec{Q}}) \underset{ \underline{\mathcal{M}}_{\vec{Q}}^{\alpha+\beta} \times  \underline{\mathcal{M}}_{\vec{Q}}^{\gamma}}{\times}(\underline{\mathcal{E}}^{\alpha+ \beta,\gamma})$$ parametrizes pairs of inclusions $(\overline{N} \subset R/M, M \subset R)$ where $M,\overline{N},R$ are representations of $\vec{Q}$ of respective dimensions $\gamma, \beta$ and $\alpha+\beta+\gamma$.

Similarly, the l.h.s. of (\ref{E:P1}) corresponds to a diagram
\begin{equation}\label{E:P7}
\xymatrix{
&\underline{\mathcal{E}}^{\alpha, \beta+\gamma} \ar[d]_-{p'_1} \ar[r]^-{p'_2} &\underline{\mathcal{M}}_{\vec{Q}}^{\alpha+\beta+\gamma}\\
 \Ma \times \underline{\mathcal{E}}^{\beta,\gamma} \ar[d]_-{p_1} \ar[r]^-{p_2} &
 \Ma \times \underline{\mathcal{M}}_{\vec{Q}}^{\beta+\gamma} &\\
 \underline{\mathcal{M}}_{\vec{Q}}^{\alpha} \times  \underline{\mathcal{M}}_{\vec{Q}}^{\beta}\times  \underline{\mathcal{M}}_{\vec{Q}}^{\gamma} & &
}
\end{equation}
which may be completed to
\begin{equation}\label{E:P8}
\xymatrix{ \underline{\mathcal{E}}^{\alpha, (\beta,\gamma)} \ar[d]_-{p''_1} \ar[r]^-{p''_2} 
&\underline{\mathcal{E}}^{\alpha, \beta+\gamma} \ar[d]_-{p'_1} \ar[r]^-{p'_2} &\underline{\mathcal{M}}_{\vec{Q}}^{\alpha+\beta+\gamma}\\
 \Ma \times \underline{\mathcal{E}}^{\beta,\gamma} \ar[d]_-{p_1} \ar[r]^-{p_2} &
 \Ma \times \underline{\mathcal{M}}_{\vec{Q}}^{\beta+\gamma} &\\
 \underline{\mathcal{M}}_{\vec{Q}}^{\alpha} \times  \underline{\mathcal{M}}_{\vec{Q}}^{\beta}\times  \underline{\mathcal{M}}_{\vec{Q}}^{\gamma} & &
}
\end{equation}
where now $\underline{\mathcal{E}}^{\alpha, (\beta,\gamma)}$ parametrizes pairs of inclusions
$(M \subset N, N \subset R)$ where $M,N,R$ are of respective dimensions $\gamma, \beta+\gamma, \alpha+\beta+\gamma$. But it is clear that both $\underline{\mathcal{E}}^{\alpha, (\beta,\gamma)}$ and $\underline{\mathcal{E}}^{(\alpha, \beta),\gamma)}$ are (canonically) isomorphic to the stack $\underline{\mathcal{E}}^{\alpha, \beta,\gamma}$ which parametrizes chains of inclusions $(M \subset N \subset R)$. These isomorphisms gives rise to the natural transformation (\ref{E:P1}). The proof of (\ref{E:P2}) (for the restriction functor $\underline{\Delta}$) uses the same diagrams, but run through in the other direction.

\vspace{.1in}

Converting the above arguments into concrete proofs requires some application, but little imagination. We do it for the induction functor and leave the (simpler) case of the restriction functor to the reader. The r.h.s. of (\ref{E:P1}) is given by the composition $q'_! r'_{\#} (p')^*q_!r_{\#}p^*[dim\;p + dim\;p']$ in the diagram
\begin{equation}\label{E:P21}
\xymatrix{
 &  & E_{\a+\beta, \gamma} \ar[r]^-{q'} & E_{\a+\beta+\gamma}\\
 & & \widetilde{E}_{\a+\beta,\gamma}\ar[u]^-{r'} \ar[d]_-{p'}  &\\
\widetilde{E}_{\a,\beta} \times E_{\gamma} \ar[r]^-{r} \ar[d]_-{p}& E_{\a,\beta} \times E_{\gamma} \ar[r]^-{q} & E_{\a+\beta} \times E_{\gamma}  &\\
E_{\a} \times E_{\beta} \times E_{\gamma} & & &
}
\end{equation}
We complete (\ref{E:P21}) by adding cartesian diagrams and get the following monster~:
\begin{equation}\label{E:P22}
\xymatrix{
 & E_{\a, \beta, \gamma} \ar[r]^-{q''} & E_{\a+\beta, \gamma} \ar[r]^-{q'} & E_{\a+\beta+\gamma}\\
E^{(2)}_{(\a,\beta),\gamma} \ar[d]_-{p''} \ar[r]^-{s} & E^{(1)}_{(\a,\beta),\gamma} \ar[u]^-{t} \ar[r]^-{k} \ar[d]_-{h}& \widetilde{E}_{\a+\beta,\gamma}\ar[u]^-{r'} \ar[d]_-{p'}  &\\
\widetilde{E}_{\a,\beta} \times E_{\gamma} \ar[r]^-{r} \ar[d]_-{p}& E_{\a,\beta} \times E_{\gamma} \ar[r]^-{q} & E_{\a+\beta} \times E_{\gamma}  &\\
E_{\a} \times E_{\beta} \times E_{\gamma} & & &
}
\end{equation}
where
\begin{align*}
E_{\alpha, \beta, \gamma}&=\big\{(\underline{y},\, U \subset W \subset V_{\alpha+\beta+\gamma})\big\},\\
E^{(1)}_{\alpha, \beta, \gamma}&=\big\{(\underline{y},\, U \subset W \subset V_{\alpha+\beta+\gamma},\; \rho_{\a+\beta}, \rho_{\gamma})\big\},\\
E^{(2)}_{\alpha, \beta, \gamma}&=\big\{(\underline{y},\, U \subset W \subset V_{\alpha+\beta+\gamma},\; \rho_{\a+\beta}, \rho_{\gamma}, \rho_{\alpha}, \rho_{\beta})\big\}
\end{align*}
where in the above definitions we have $\underline{y} \in E_{\alpha+\beta+\gamma}$, the subspaces $ U,W$ are $\underline{y}$-stable, $\underline{dim}\;U=\gamma,\; \underline{dim}\;W/U=\beta$ and
\begin{alignat*}{2}
\rho_{\gamma}:&\; U \stackrel{\sim}{\to} V_{\gamma}, \qquad  & \rho_{\alpha+\beta}: \;V_{\a+\beta+\gamma}/U \stackrel{\sim}{\to} V_{\a+\beta},\\
\rho_{\beta}:&\; W/U \stackrel{\sim}{\to} V_{\beta}, \qquad & \rho_{\alpha}:\; V_{\a+\beta+\gamma}/W \stackrel{\sim}{\to} V_{\alpha}.
\end{alignat*}

Using base change and Lemma~\ref{L:AppU} i) below, we have
$(p')^*q_!=k_!h^*,$
$h^*r_{\#}=s_{\#}(p'')^*,$ and
$r'_{\#}k_{!} =q''t_{\#}.$
Hence 
$$q'_! r'_{\#} (p')^*q_!r_{\#}p^*= q'_{!} q''_! t_{\#}s_{\#}(p'')^*p^*=(q'q'')_!(ts)_{\#}(pp'')^*.$$
In other words, we have simplified (\ref{E:P21}) to a diagram
\begin{equation}\label{E:P23}
\xymatrix{
E^{(2)}_{(\a,\beta),\gamma} \ar[d]_-{v} \ar[r]^-{u} & E_{\a,\beta,\gamma} \ar[r]^-{w} & E_{\a+\beta+\gamma}\\
E_{\a} \times E_{\beta} \times E_{\gamma} & &}
\end{equation}
where $u=ts, w=q'q''$ and $v=pp''$. We may reduce this further~: consider the variety
$$E^{(2)}_{\a,\beta,\gamma}=\big\{(\underline{y}, \,U \subset W \subset V_{\alpha+\beta+\gamma},\; \rho_{\a+\beta}, \rho_{\gamma}, \rho_{\alpha}, \rho_{\beta})\big\}$$
where we use the same notations as before. Then $u$ factors as a composition of principal bundles
$$\xymatrix{
E^{(2)}_{(\a,\beta),\gamma} \ar[r]^-{u''} & E^{(2)}_{\a,\beta,\gamma} \ar[r]^-{u'} & E_{\a,\beta,\gamma}}$$
and by Lemma~\ref{L:AppU} ii) below we finally obtain a canonical equivalence
\begin{equation}
\underline{m}_{\alpha+ \beta,\gamma} \circ ( \underline{m}_{\alpha,\beta} \times Id) \simeq w_! u'_{\#} (v')^*[dim\;v']
\end{equation}
where $v': E^{(2)}_{\a,\beta,\gamma} \to E_{\a} \times E_{\beta} \times E_{\gamma}$ is the natural map.
Starting from the l.h.s. of (\ref{E:P1}) and arguing as above we reach exactly the same expression. This proves the equivalence (\ref{E:P1}).
\end{proof}

\vspace{.1in}

We have used the following two easy facts~:

\vspace{.05in}

\begin{lem}\label{L:AppU} i) Let
$$\xymatrix{ W \ar[d]_-{p} \ar[r]^-{r'} & Z \ar[d]_-{p'}\\ X \ar[r]^-{r} & Y}$$
be a cartesian square, where $r, r'$ are principal $G$-bundles for some (connected) group $G$, and where $p, p'$ are $G$-equivariant. Then
\begin{enumerate}
\item[a)] $(p')^* r_{\#} = r'_{\#} p^*,$
\item[b)] $p'_! r'_{\#} = r_{\#} p_!$.
\end{enumerate}
ii) Let $\xymatrix{ X \ar[r]^-r & Y \ar[r]^-s & Z}$ be a sequence of principal bundles. Then $s_{\#}r_{\#}=(sr)_\#$.
\end{lem}
\begin{proof} In the first statement , a) comes from the relation $(r')^*(p')^*=p^*r^*$ while b) comes from the base change formula $p_!(r')^*=r^*p'_!$. The second statement
is a consequence of $r^*s^*=(sr)^*$.
\end{proof}

\vspace{.15in}

We may define an iterated induction functor
$$ (\mathbb{P}_1,  \ldots , \mathbb{P}_n) \mapsto ( \cdots (\mathbb{P}^1 \star ( \cdots \star \mathbb{P}^n) \cdots )$$
for each choice of parantheses in the expression $\mathbb{P}^1 \star \cdots \star \mathbb{P}^1$. Proposition~\ref{P:assoc} provides us with
isomorphisms between all these iterated induction functors, and we denote the resulting complex by $\mathbb{P}_1 \star \cdots \star \mathbb{P}_n$.
A priori, $\mathbb{P}_1 \star \cdots \star \mathbb{P}_n$ is only defined up to isomorphism (and this would suffice for all our purposes). However it can be
shown that the various associativity isomorphisms $(\mathbb{P}_i \star (\mathbb{P}_j \star \mathbb{P}_k)) \stackrel{\sim}{\to} ((\mathbb{P}_i \star \mathbb{P}_j) \star \mathbb{P}_k)$ provided by Proposition~\ref{P:assoc} are all compatible\footnote{i.e. satisfy the pentagon axiom of a tensor category, see \cite{Maclane}} and hence that
$\mathbb{P}_1 \star \cdots \star \mathbb{P}_n$ is actually defined up to a \textit{canonical} isomorphism.

A similar result holds for the restriction functor; we denote by $\underline{\Delta}^{(n)}(\mathbb{P})$ the iterated restriction functor applied to a complex $\mathbb{P}$. This is, again,  defined up to canonical isomorphism.

\vspace{.3in}

\centerline{\textbf{1.4. The Lusztig sheaves and the category $\mathcal{Q}_{\vec{Q}}$.}}
\addcontentsline{toc}{subsection}{\tocsubsection {}{}{\;\;1.4. The Lusztig sheaves and the Hall category.}}

\vspace{.15in}

Recall that we have denoted by $\mathbbm{1}_{\alpha}=\qlb_{|E_{\a}}[dim\;E_{\a}] \in D^b(\Ma) = D^b_{G_{\a}} (E_{\a})$ the constant perverse sheaf. For any tuple $(\a, \ldots, \a_n) \in (\N^I)^n$ we set
\begin{equation}\label{E:41}
L_{\a_1, \ldots,\a_n}=\mathbbm{1}_{\a_1} \star \cdots \star \mathbbm{1}_{\a_n} \in D^b(\underline{\mathcal{M}}_{ \a_1+ \cdots + \a_n})
\end{equation}
and call these elements the \textit{Lusztig sheaves}. Consider as in Section~1.2. the variety
$$E_{\a_1, \ldots, \a_n}:=\big\{ (\underline{y},\, W_n \subset W_{n-1} \subset \cdots \subset W_1=V_{ \a_1+ \cdots + \a_n})\big\}$$
where $\underline{y} \in E_{\a_1 + \cdots + \a_n}$, $W_n, \ldots, W_1$ are $\underline{y}$-stable and $\underline{dim}\; W_{i}/W_{i+1}=\a_i$ for all $i$. The projection $q: E_{\a_1, \ldots, \a_n} \to E_{\a_1 + \cdots + \a_n}$ is proper. It follows from the definitions and from the proof of Proposition~\ref{P:assoc} that 
$$L_{\a_1, \ldots, \a_n}=q_{!} (\qlb_{E_{\a_1, \ldots, \a_n}}[dim\;E_{\a_1, \ldots, \a_n}+\sum_i dim\;E_{\a_i}])$$
(recall that $E_{\a_1, \ldots, \a_n}$ is smooth). By the Decomposition Theorem of \cite{BBD} (and its equivariant version \cite{Bernstein-Lunts}), $L_{\a_1, \ldots, \a_n}$ is a semisimple $G_{\a_1 + \cdots +\a_n}$-equivariant complex.  Furthermore, by Lemma~\ref{L:Verdierind}, $L_{\a_1, \ldots, \a_n}$ is Verdier self-dual~:
\begin{equation}\label{E:Verdierdual}
D L_{\a_1, \ldots, \a_n} \simeq D(\mathbbm{1}_{\a_1}) \star \cdots \star D(\mathbbm{1}_{\a_n}) \simeq \mathbbm{1}_{\a_1} 
\star \cdots \mathbbm{1}_{\a_n} = L_{\a_1, \ldots, \a_n}
\end{equation}
since $D(\mathbbm{1}_{\a_l})=\mathbbm{1}_{\a_l}$ (recall that $E_{\a_l}$ is always smooth).

\vspace{.1in}

The following particular case is important for us. Call a dimension vector $\a$ \textit{simple} if it is the class of a simple object $S_i$, i.e. if $\a \in \{\epsilon_i\}_{i \in I}$ (in plainer terms, if $\a_j=1$ for one value of $j \in I$ and $\a_j=0$ for all other values). Because we have required our quivers not to have any edge loops, $E_{\a}=\{0\}$ for any simple $\a$. 

\vspace{.1in}

 Let us fix $\gamma \in \N^I$ and denote by $\mathcal{P}^{\gamma}$ the collection of all \textit{simple} $G_{\gamma}$-equivariant perverse sheaves on $E_{\gamma}$ arising, up to shift, as a direct summand of some Lusztig sheaf $L_{\a_1, \ldots, \a_n}$ with $\a_1 + \cdots + \a_n=\gamma$ and
 $\a_i$ \textit{simple} for all $i$. The set $\mathcal{P}^{\gamma}$ is invariant under Verdier duality because of (\ref{E:Verdierdual})\footnote{actually, we will prove later that all elements of $\mathcal{P}^{\gamma}$ are \textit{self}-dual.}.

\vspace{.1in}

We define $\mathcal{Q}^\gamma$ as the full subcategory of $D^b(\underline{\mathcal{M}}_{\gamma})^{ss} = D^b_{G_{\gamma}}(E_{\gamma})^{ss}$ additively generated by the elements of $\mathcal{P}^{\gamma}$. Thus any $\mathbb{P} \in \mathcal{Q}^{\gamma}$ is isomorphic to a direct sum $\mathbb{P}_1[n_1] \oplus \cdots \oplus \mathbb{P}[n_k]$ where $n_i \in \Z$ and each $\mathbb{P}_i$ belongs to $\mathcal{P}^{\gamma}$. We also set
$$\mathcal{P}_{\vec{Q}}=\bigsqcup_{\gamma} \mathcal{P}^{\gamma}, \qquad \mathcal{Q}_{\vec{Q}}=\bigsqcup_{\gamma} \mathcal{Q}^{\gamma}.$$
We will call the category $\QQ$ the \textit{Hall category}.
Note that $\QQ$ is preserved by Verdier duality. In addition, since any successive extension of simple representations of $\vec{Q}$ is nilpotent, all the objects of $\QQ$ are actually supported on the closed substacks $\underline{\mathcal{M}}^{\a,nil}_{\vec{Q}}$. The category $\mathcal{Q}_{\vec{Q}}$ and its set of simple objects $\mathcal{P}_{\vec{Q}}$ are our main objects of interest.

\vspace{.1in}

\begin{prop}\label{P:closed} The category $\mathcal{Q}_{\vec{Q}}$ is preserved by the functors $\underline{m}$ and $\underline{\Delta}$.
\end{prop}
\begin{proof} To prove that the induction product of two objects of $\QQ$ still belongs to $\QQ$ it suffices to consider the case of two simple objects, and then (because $\QQ$ is obviously closed under direct summands) the case of two Lusztig sheaves $L_{\a_1, \ldots, \a_n}, L_{\beta_1, \ldots, \beta_m}$. This is clear since by construction we have
$$L_{\alpha_1, \ldots, \alpha_n} \star L_{\beta_1, \ldots, \beta_m}=L_{\a_1, \ldots, \a_n, \beta_1, \ldots, \beta_m}.$$

\vspace{.1in}

The proof of the fact that $\QQ$ is closed under the restriction functor is much more delicate. Again, it suffices to show that for any simple $\a_1, \ldots, \a_n$ and $\beta, \gamma$ such that $\beta+\gamma=\sum \a_i$ we have $\underline{\Delta}_{\beta, \gamma}(L_{\a_1, \ldots, \a_n}) \in \QQ \times \QQ.$ We will actually prove a more precise result~:

\vspace{.1in}

\begin{lem}\label{L:coprodun} We have
\begin{equation}\label{E:42}
\underline{\Delta}_{\beta, \gamma}(L_{\a_1, \ldots, \a_n})=\bigoplus_{(\underline{\beta}, \underline{\gamma})} L_{\underline{\beta}} \boxtimes L_{\underline{\gamma}}[d_{\underline{\beta}, \underline{\gamma}}]
\end{equation}
where $(\underline{\beta}, \underline{\gamma})$ runs through all tuples
$$\underline{\beta}=(\beta_1, \ldots, \beta_n), \qquad \underline{\gamma}=(\gamma_1, \ldots, \gamma_n)$$
satisfying $\sum \beta_i=\beta, \sum \gamma_i=\gamma$ and $\beta_i+\gamma_i=\a_i$ for all $i$; and where
$$d_{\underline{\beta}, \underline{\gamma}}=-\sum_k \langle \gamma_k, \beta_k \rangle -\sum_{k <l} (\gamma_k, \beta_l).$$
\end{lem}
\begin{proof} Set $\a=\sum \a_i$. As in the definition of the restriction functor, let us fix a subspace $W \subset V_{\a}$ as well as identifications $W \stackrel{\sim}{\to} V_{\gamma}, \;V_{\a}/W \stackrel{\sim}{\to} V_{\beta}$. Recall that $F_{\beta,\gamma} \subset E_{\a}$ is the set of representations $\underline{y}$ which preserve $W$. Consider the following diagram
\begin{equation}\label{E:P51}
\xymatrix{E_{\a_1, \ldots, \a_n} \ar[d]_-{q} & F_{\a_1, \ldots, \a_n} \ar[l]_-{\iota'} \ar[d]_-{q'} & \\
E_{\a} & F_{\beta,\gamma} \ar[l]_-{\iota} \ar[r]^-{\kappa} & E_{\beta} \times E_{\gamma}}
\end{equation}
where 
\begin{equation*}
\begin{split}
&F_{\a_1, \ldots, \a_n}\\
& \qquad =\big\{ (\underline{y}, W_n \subset \cdots \subset W_1=V_{\a})\;|\; \underline{y} \in F_{\beta, \gamma};\, \underline{y}(W_i) \subset W_i ,\; \underline{dim}\;W_i/W_{i+1}=\a_i \;\forall\;i \big\}
\end{split}
\end{equation*}
Note that the square in (\ref{E:P51}) is cartesian. Thus by base change, we have
\begin{equation*}
\begin{split}
\underline{\Delta}(L_{\a_1, \ldots, \a_n})&=\underline{\Delta}(q_!(\qlb_{E_{\a_1, \ldots, \a_n}}))[dim\;E_{\a_1, \ldots, \a_n} -\langle \beta, \gamma \rangle]\\
&=(\kappa q')_!(\qlb_{|F_{\a_1, \ldots, \a_n}})[dim\;E_{\a_1, \ldots, \a_n} -\langle \beta, \gamma \rangle].
\end{split}
\end{equation*}
Of course, the composition $\kappa q' $ is not a vector bundle anymore, but we may decompose $F_{\a_1, \ldots, \a_n}$ in strata along which it behaves well. For this, let $\underline{\beta}, \underline{\gamma}$ be as in the statement of the Lemma, and consider the subvariety $F_{\underline{\a}}^{\underline{\beta},\underline{\gamma}}$ of $F_{\underline{\a}}$ consisting of pairs
$(\underline{y}, W_n \subset \cdots \subset W_1)$ for which the induced filtrations
$$ ( W_n \cap W) \subset (W_{n-1} \cap W) \subset \cdots \subset W_1 \cap W=W, $$
$$ (W_n / W_n \cap W) \subset (W_{n-1} / W_{n-1} \cap W) \subset \cdots \subset (W_1 / W_1 \cap W)=(W_1/W)$$
are of type $\underline{\gamma}$ and $\underline{\beta}$ respectively. This is equivalent to requiring (for instance) that
$\underline{dim}\; (W_i \cap W / W_{i+1} \cap W) = \gamma_i$
for all $i$.  There is a commutative diagram
\begin{equation*}
\xymatrix{ F_{\underline{\a}}^{\underline{\beta}, \underline{\gamma}} \ar[d]_-{q'} \ar[r]^-{\kappa^{\underline{\beta},\underline{\gamma}}} & E_{\underline{\beta}} \times E_{\underline{\gamma}} \ar[d]^-{q^{\underline{\beta}} \times q^{\underline{\gamma}}}\\ F_{\beta, \gamma} \ar[r]^-{\kappa} & E_{\beta} \times E_{\gamma}}
\end{equation*}
where as before
$$E_{\underline{\beta}}=\big\{ (\underline{y}^{\beta}, W^{\beta}_1 \subset \cdots \subset W^{\beta}_n = V_{\beta}) \big\},$$
$$E_{\underline{\gamma}}=\big\{ (\underline{y}^{\gamma}, W^{\gamma}_1 \subset \cdots \subset W^{\beta}_n = V_{\gamma}) \big\},$$
with $\underline{y}^{\beta} \in E_{\beta}, \underline{y}^{\gamma} \in E_{\gamma}$ and $\underline{dim}\; W^{\beta}_i / W^{\beta}_{i+1}=\beta_i$, $\underline{dim}\; W^{\gamma}_i / W^{\gamma}_{i+1}=\gamma_i$. It is easy to see that $\kappa^{\underline{\beta}, \underline{\gamma}}$ is a vector bundle, whose rank equals
\begin{equation*}
\begin{split}
rank\; \kappa^{\underline{\beta},\underline{\gamma}} &= \sum_i \sum_{k <l} (\beta_l)_i (\gamma_k)_i+ 
\sum_{h \in H} \;\sum_{k \leq l} \;(\beta_k)_{s(h)} (\gamma_l)_{t(h)}\\
&=\sum_i \beta_i \gamma_i -\sum_{k \leq l} \langle \beta_k, \gamma_l \rangle.
\end{split}
\end{equation*}
Thus we have decomposed the restriction of $\kappa q'$ to $F_{\underline{\a}}^{\underline{\beta}, \underline{\gamma}}$ as a composition of a vector bundle $\kappa^{\underline{\beta},\underline{\gamma}}$ and \textit{then} a proper map $q^{\underline{\beta}} \times q^{\underline{\gamma}}$. The $F_{\underline{\a}}^{\underline{\beta}, \underline{\gamma}}$ form a finite stratification of $F_{\a_1, \ldots, \a_n}$ into smooth locally closed pieces as $(\underline{\beta}, \underline{\gamma})$ varies. We now make use of the following general result due to Lusztig (see \cite[Section~8.1.6.]{Lusbook} )~:

\begin{lem}\label{L:816} Let $X,Y$ be algebraic varieties and let  $f: X \to Y$ be a morphism. Suppose that there exists a finite stratification $X=\bigsqcup_\sigma X_\sigma$ by locally closed subsets and a collection of maps
$\xymatrix{ X_\sigma \ar[r]^-{f_{\sigma}} & Z_{\sigma} \ar[r]^-{f'_{\sigma}} & Y}$ such that $f'_{\sigma}$ is proper and $f_{\sigma}$ is a vector bundle, $Z_{\sigma}$ is smooth and such that $f_{|X_{\sigma}}=f'_{\sigma} f_{\sigma}$. Then the complex $f_!(\qlb_{|X})$ is semisimple and moreover
$$f_!(\qlb_{|X})\simeq \bigoplus_{\sigma} (f_{|X_{\sigma}})_!(\qlb_{|X_{\sigma}}).$$
\end{lem}

It follows that $(\kappa q')_{!} (\qlb_{|F_{\a_1, \ldots, \a_n}})$ is a semisimple complex, and that 
\begin{equation*}
\begin{split}
(\kappa q')_! (\qlb_{|F_{\a_1, \ldots, \a_n}}) &\simeq \bigoplus_{(\underline{\beta}, \underline{\gamma})} (\kappa q')_! (\qlb_{|F_{\underline{a}}^{\underline{\beta}, \underline{\gamma}}})\\
&= \bigoplus_{(\underline{\beta}, \underline{\gamma})} (q^{\underline{\beta}} \times q^{\underline{\gamma}})_! \kappa^{\underline{\beta}, \underline{\gamma}}_! (\qlb_{|F_{\underline{a}}^{\underline{\beta}, \underline{\gamma}}})\\
&= \bigoplus_{(\underline{\beta}, \underline{\gamma})} (q^{\underline{\beta}} \times q^{\underline{\gamma}})_! (\qlb_{|E_{\underline{\beta}} \times E_{\underline{\gamma}}})[-2 \,rank\; \kappa^{\underline{\beta}, \underline{\gamma}}]\\
&= \bigoplus_{(\underline{\beta}, \underline{\gamma})}L_{\underline{\beta}} \boxtimes L_{\underline{\gamma}} [-2 \,rank\; \kappa^{\underline{\beta}, \underline{\gamma}}].
\end{split}
\end{equation*}
Therefore, 
\begin{equation}\label{E:43}
\underline{\Delta}_{\beta, \gamma}(L_{\a_1, \ldots, \a_n})=\bigoplus_{(\underline{\beta}, \underline{\gamma})} L_{\underline{\beta}} \boxtimes L_{\underline{\gamma}}[e_{\underline{\beta}, \underline{\gamma}}]
\end{equation}
where
\begin{equation*}
e_{\underline{\beta}, \underline{\gamma}}=
dim\;E_{\underline{\a}} - dim\;E_{\underline{\beta}} - dim\;E_{\underline{\gamma}}-\langle \beta, \gamma \rangle -2 \,rank\; \kappa^{\underline{\beta}, \underline{\gamma}}.
\end{equation*}
A simple calculation shows that $e_{\underline{\beta}, \underline{\gamma}}=d_{\underline{\beta}, \underline{\gamma}}$. We are done.
\end{proof}
To conclude the proof of Proposition~\ref{P:closed} it remains to observe that if $\a_1, \ldots, \a_n$ are simple and if $(\underline{\beta}, \underline{\gamma})$ is as in Lemma~\ref{L:coprodun} then all $\beta_i$ and $\gamma_i$ are simple as well. 
\end{proof}

\vspace{.3in}

\centerline{\textbf{1.5. The geometric pairing in $\QQ$.}}
\addcontentsline{toc}{subsection}{\tocsubsection {}{}{\;\;1.5. The geometric pairing on the Hall category.}}

\vspace{.15in}

To finish this Lecture, we introduce one last piece of structure on $\QQ$~: a scalar product on objects, which takes values in the ring of Laurent series $\mathbf{K}:=\N((v))$. The natural idea which comes to mind is to define a scalar product by integrating over the whole moduli space, i.e. by setting $\{ \mathbb{P}, \mathbb{Q}\}=dim \;H^*(\mathbb{P} \otimes \mathbb{Q})=dim\; \pi_!(\mathbb{P} \otimes \mathbb{Q})$, where $\pi: \underline{\mathcal{M}}^{\gamma}_{\vec{Q}} \to \{pt\}$ if $\mathbb{P}, \mathbb{Q}$ belong to $\mathcal{Q}^{\gamma}$. 

In order to do this properly one uses equivariant cohomology (see \cite{Lcusp} and references therein). To any algebraic $G$-variety $X$ and any $G$-equivariant complex $\mathbb{P} \in D^b_G(X)$ are associated $\qlb$-vector spaces $H^j_G(\mathbb{P},X)$ for $j \in \Z$, which one might think of as $H^j(\mathbb{P}, X/G)$. Let us briefly recall the construction. Let $\Gamma$ be a smooth, irreducible and sufficiently acyclic free $G$-space. For any $G$-space $X$ we set $X_\Gamma=(X \times \Gamma)/G$ and if $\mathbb{T} \in D^b_G(X)^{ss}$ then we denote by $\mathbb{T}_{\Gamma}$ the semisimple complex $X_{\Gamma}$ such that $s^*(\mathbb{T}_{\Gamma}) \simeq \pi^*(\mathbb{T})$, where $s, \pi$ are the canonical maps
$$\xymatrix{ X & X \times \Gamma \ar[l]_-{\pi} \ar[r]^-{s} & (X  \times \Gamma)/G = X_{\Gamma}}.$$
Then by definition $H^j_G(\mathbb{T},X)=H^{ (2\,dim\; \Gamma/G-j)}(\mathbb{T}_{\Gamma})$. 

\vspace{.1in}

Now let $\gamma \in \N^I$ and let $\mathbb{P}, \mathbb{Q}$ be two objects of $\mathcal{Q}^{\gamma}$. We set
\begin{equation}\label{E:defscalar}
\{\mathbb{P}, \mathbb{Q}\}=\sum_j (dim\;H^j_{G_{\gamma}}(\mathbb{P} \otimes \mathbb{Q}, E_{\gamma}))v^{j}.
\end{equation}
It is obvious that for any $\mathbb{P}, \mathbb{P}', \mathbb{Q}$ and any integer $n \in \Z$ we have
\begin{equation}\label{E:Obvious}
\{\mathbb{P}, \mathbb{Q}\}=\{\mathbb{Q},\mathbb{P}\}, \qquad
\{\mathbb{P} \oplus \mathbb{P}', \mathbb{Q}\}=\{\mathbb{P}, \mathbb{Q}\} + \{\mathbb{P}, \mathbb{Q}\}, \qquad
\{\mathbb{P}[n], \mathbb{Q}\}=v^n \{\mathbb{P}, \mathbb{Q}\}.
\end{equation}

\vspace{.1in}

\begin{prop}\label{C:scalarprod} The following hold~:
\begin{enumerate}
\item[i)] If $\mathbb{P}, \mathbb{P}'$ are two simple perverse sheaves then $\{ \mathbb{P},\mathbb{P}'\} \in 1+v\N[[v]]$ if $\mathbb{P} \simeq D \mathbb{P}'$ and $\{\mathbb{P}, \mathbb{P}'\} \in v \N[[v]]$ otherwise.
\item[ii)] For any complexes $\mathbb{P}, \mathbb{Q}$ in $\mathcal{Q}^{\gamma}$ we have $\{\mathbb{P}, \mathbb{Q}\} \in \mathbf{K}$.
\end{enumerate}
\end{prop}
\begin{proof} See \cite{GL}. Note that since any object of $\QQ$ is semisimple, we have $H^j_{G_{\gamma}}(\mathbb{P} \otimes \mathbb{Q}) =0$ for $j \ll 0$ by (\ref{E:Obvious}) and i) above, so that ii) follows from i). \end{proof}
\vspace{.1in}

The next important property of $\{\,,\,\}$ states that $\underline{m}$ and $\underline{\Delta}$ are in a certain sense adjoint functors~:

\vspace{.1in}

\begin{prop}\label{P:Hopfscalar} For any objects $\mathbb{P}, \mathbb{Q}, \mathbb{R}$ of $\QQ$ we have
\begin{equation}\label{E:P41}
\{ \underline{m}(\mathbb{P} \boxtimes \mathbb{Q}), \mathbb{R} \}=\{\mathbb{P} \boxtimes \mathbb{Q}, \underline{\Delta}(\mathbb{R})\}.
\end{equation}
\end{prop}
\begin{proof} We begin with some heuristic argument as usual. Recall that the induction and restriction functors are ``defined'' by  correspondences
\begin{equation*}
\xymatrix{ &  \underline{\mathcal{E}}^{\alpha, \beta} \ar[ld]_-{p_1} \ar[rd]^-{p_2} &\\
\ \underline{\mathcal{M}}_{\vec{Q}}^{\alpha} \times  \underline{\mathcal{M}}_{\vec{Q}}^{\beta} & &
\underline{\mathcal{M}}_{\vec{Q}}^{\alpha+\beta}
}
\end{equation*}
as $\underline{m}_{\a,\beta}=p_{2!}p_1^*[dim\;p_1], \underline{\Delta}_{\a,\beta}=p_{1!}p_2^*[dim\;p_1]$. And the scalar product is, morally, $\{ \mathbb{P}, \mathbb{Q}\}=dim\; H^*(\mathbb{P} \otimes \mathbb{Q})$. By the projection formula,
\begin{equation*}
\begin{split}
\{ \underline{m}( \mathbb{P} \boxtimes \mathbb{Q}), \mathbb{R}\}&=dim\;H^*(p_{2!}p_{1}^*(\mathbb{P} \boxtimes \mathbb{Q}) \otimes \mathbb{R})[dim\;p_1]\\
&=dim\;H^*\big(p_{2!} (p_{1}^*(\mathbb{P} \boxtimes \mathbb{Q}) \otimes p_{2}^* \mathbb{R})\big)[dim\;p_1]\\
&=dim\;H^* \big( p_1^*(\mathbb{P} \boxtimes \mathbb{Q}) \otimes p_{2}^* \mathbb{R}\big)[dim\;p_1]\\
&=dim\;H^* (\mathbb{P} \boxtimes \mathbb{Q} \otimes p_{1!}p_2^* \mathbb{R})[dim\;p_1]\\
&=dim\;H^*(\mathbb{P} \boxtimes \mathbb{Q} \otimes \underline{\Delta}( \mathbb{R}))
\end{split}
\end{equation*}
The actual proof of Proposition~\ref{P:Hopfscalar} is also essentially just the projection formula ( ``essentially'' means that we need to draw many (large) commutative diagrams before we may actually use the projection formula).

\vspace{.1in}

Let $\mathbb{P}, \mathbb{Q}, \mathbb{R}$ be as in the Proposition, of respective dimensions $\a, \beta$ and $\gamma$. Consider the induction diagram (\ref{E:16})
\begin{equation*}
\xymatrix{ E_{\alpha} \times  E_{\beta}&  {E}^{(1)}_{\alpha,\beta} \ar[r]^{r} \ar[l]_-{p} & E_{\alpha,\beta} \ar[r]^-{q} &
E_{\alpha+\beta}.
}
\end{equation*}
We have $\underline{m}(\mathbb{P} \boxtimes \mathbb{Q})= q_! r_{\#} p^*(\mathbb{P} \boxtimes \mathbb{Q})[dim\;p-\langle \alpha, \beta \rangle]$. We set $\mathbb{T}=r_{\#} p^*(\mathbb{P} \boxtimes \mathbb{Q})$ for simplicity. The restriction diagram is 
\begin{equation*}
\xymatrix{ 
E_{\alpha} \times E_{\beta} &F_{\alpha, \beta} \ar[l]_-{\kappa} \ar[r]^-{\iota} &
E_{\alpha+\beta}
}
\end{equation*}
and $\underline{\Delta}(\mathbb{R})=\kappa_!\iota^*(\mathbb{R})[-\langle \alpha, \beta\rangle]$. Let $P \subset G_{\gamma}$ be the stabilizier of $V_{\beta} \subset V_{\gamma}$ and let $U$ be its unipotent radical. Thus $P/U \simeq G_{\a} \times G_{\beta}$ and $\overline{\Gamma}:=\Gamma/U$ is a (sufficiently acyclic, ...) free $G_{\a} \times G_{\beta}$-space. Note that
\begin{equation}\label{E:431}
dim\;p + 2\;dim\;\Gamma/G = 2 \;dim\;\overline{\Gamma}/(G_{\a} \times G_{\beta}).
\end{equation}
Using (\ref{E:431}) we see that (\ref{E:P41}) is equivalent to the collection of equalities
\begin{equation}\label{E:P42}
dim\;H^i\big( (\mathbb{P} \boxtimes \mathbb{Q})_{\overline{\Gamma}} \otimes (\kappa_{!} \iota^* (\mathbb{R}))_{\overline{\Gamma}}\big)=dim\;H^i\big( (q_! \mathbb{T})_{\Gamma} \otimes \mathbb{R}_{\Gamma} \big)
\end{equation}
for all $i \in \Z$.

\vspace{.1in}

To prove (\ref{E:P42}) we consider the diagram
$$\xymatrix{
(E_{\gamma})_{\Gamma} & (E_{\alpha,\beta})_{\Gamma} \ar[l]_-{\widetilde{q}} & (E^{(1)}_{\alpha, \beta})_{\Gamma} \ar[l]_-{\widetilde{r}} \ar[r]^-{\widetilde{p}} & (E_{\a} \times E_{\beta})_{\overline{\Gamma}}\\
E_{\gamma} \times \Gamma \ar[u]^-{s} \ar[d]_-{\pi} & E_{\a, \beta} \times \Gamma \ar[u]^-{s} \ar[d]_-{\pi} \ar[l]_-{q'} & E^{(1)}_{\alpha,\beta} \times \Gamma \ar[l]_-{r'} \ar[u]^-{{s}} \ar[d]_-{{\pi}} \ar[r]^-{p'} & E_{\a} \times E_{\beta} \times \overline{\Gamma} \ar[u]^-{\overline{s}} \ar[d]_-{\overline{\pi}} \\
E_{\gamma} & E_{\a, \beta} \ar[l]_-{q} & E^{(1)}_{\a, \beta} \ar[l]_-{r} \ar[r]^-{p} & E_{\a} \times E_{\beta}
}
$$
The two leftmost columns consist of cartesian squares. Define a vector bundle map $\phi$ via the indentifications
$$\xymatrix{
(E_{\a,\beta})_{\Gamma} \ar@{=}[d] \ar[r]^-{\phi} & (E_{\a} \times E_{\beta})_{\overline{\Gamma}} \ar@{=}[d] \\
(F_{\a,\beta} \times \Gamma)/P \ar[r]^-{\kappa} & (E_{\a} \times E_\beta \times \overline{\Gamma})_{G_{\a} \times G_{\beta}}
}
$$
We have $\phi \widetilde{r}=\widetilde{p}$. By standard base change arguments, we obtain
\begin{equation}
\widetilde{q}_! \mathbb{T}_{\Gamma}=(q_! \mathbb{T})_{\Gamma}, \qquad \mathbb{T}_{\Gamma}=\phi^*\big((\mathbb{P} \boxtimes \mathbb{Q})_{\overline{\Gamma}}\big).
\end{equation}
Hence, by the projection formula
\begin{equation*}
\begin{split}
H^i\big( (q_! \mathbb{T})_{\Gamma} \otimes \mathbb{R}_{\Gamma} \big) &= H^i\big( (\widetilde{q}_{!} \mathbb{T}_{\Gamma}) \otimes \mathbb{R}_{\Gamma} \big)\\
&= H^i\big( \widetilde{q}_! ( \mathbb{T}_{\Gamma} \otimes \widetilde{q}^* \mathbb{R}_{\Gamma})\big)\\
&=H^i \big( \mathbb{T}_{\Gamma} \otimes \widetilde{q}^* \mathbb{R}_{\Gamma}\big).
\end{split}
\end{equation*}
The projection formula again gives
\begin{equation*}
\begin{split}
H^i\big(  \mathbb{T}_{\Gamma} \otimes \widetilde{q}^*\mathbb{R}_{\Gamma} \big) &= H^i\big( \phi^*( (\mathbb{P} \boxtimes \mathbb{Q})_{\overline{\Gamma}}) \otimes \widetilde{q}^* \mathbb{R}_{\Gamma} \big)\\
&=H^i \big( \phi_! ( \phi^*( (\mathbb{P} \boxtimes \mathbb{Q})_{\overline{\Gamma}}) \otimes \widetilde{q}^* \mathbb{R}_{\Gamma}) \big)\\
&=H^i \big( (\mathbb{P} \boxtimes \mathbb{Q})_{\overline{\Gamma}} \otimes \phi_! \widetilde{q}^* \mathbb{R}_{\Gamma} \big).
\end{split}
\end{equation*}
We are hence reduced to proving that 
\begin{equation}\label{E:brak}
\phi_! \widetilde{q}^* \mathbb{R}_{\Gamma}\simeq (\kappa_! \iota^* \mathbb{R})_{\overline{\Gamma}}.
\end{equation}
To this aim, we consider the following diagram
$$\xymatrix{
(E_{\a} \times E_{\beta} \times \Gamma)/P & (F_{\a, \beta} \times \Gamma)/P \ar[r]^-{\widetilde{\iota}} \ar[l]_-{\widetilde{\kappa}=\phi} & (E_\gamma \times \Gamma)/P \ar[r]^-{h} & (E_{\gamma})_{\Gamma} \\
E_{\a} \times E_{\beta} \times \Gamma \ar[u]^-s \ar[d]_-{\pi} & F_{\a, \beta} \times \Gamma \ar[u]^-{s} \ar[d]_-{\pi} 
\ar[r]^-{\iota'} \ar[l]_-{\kappa'} & E_\gamma \times \Gamma \ar[u]^-{s} \ar[d]_-{\pi} \ar[ur]^-{s}\\
E_{\a} \times E_{\beta}& F_{\a, \beta} \ar[l]_-{\kappa} \ar[r]^-{\iota} & E_{\gamma}
}
$$
(note that $(E_{\a} \times E_{\beta} \times \Gamma)/P = (E_{\a} \times E_{\beta})_{\overline{\Gamma}}$).
Let $\widetilde{\mathbb{R}} \in D^b((E_{\gamma} \times \Gamma)/P)$ be the unique semisimple complex satisfying $s^*(\widetilde{\mathbb{R}}) \simeq \pi^* \mathbb{R}$. It is easy to see that $\widetilde{\mathbb{R}}=h^* \mathbb{R}_{\Gamma}$. From the factorisation $\widetilde{q}= h \widetilde{\iota}$ we thus deduce that $\phi_! \widetilde{q}^* \mathbb{R}_{\Gamma}= \widetilde{\kappa}_! \widetilde{\iota}^* \widetilde{\mathbb{R}}$. By an argument similar to 
the proof of Proposition~\ref{P:closed} one shows that $\mathbb{U}:=   \widetilde{\kappa}_! \widetilde{\iota}^* \widetilde{\mathbb{R}}$ is semisimple. A standard diagram chase gives that $s^* \mathbb{U} \simeq \pi^* (\kappa_! \iota^* \mathbb{R})$. It follows that $\mathbb{U} \simeq (\kappa_! \iota^* \mathbb{R})_{\overline{\Gamma}}$. Equation (\ref{E:brak}) is proved, and so is Proposition~\ref{P:Hopfscalar}.
\end{proof}

\vspace{.1in}

Examples of computations of this scalar product will be given in the next Lecture. 

\newpage

\centerline{\large{\textbf{Lecture~2.}}}
\addcontentsline{toc}{section}{\tocsection {}{}{Lecture~2.}}

\setcounter{section}{2}
\setcounter{theo}{0}
\setcounter{equation}{0}

\vspace{.2in}

In this second Lecture, we will provide a few sample computations with the functors $\underline{m}$ and $\underline{\Delta}$ and with the scalar product $\{\;,\;\}$, and describe the category $\QQ$ with its collection of simple objects $\mathcal{P}_{\vec{Q}}$ in several important cases.  

\vspace{.1in}

We will begin with some computations for quivers with one or two vertices (the so-called \textit{fundamental relations}). These will be crucial in making the link with quantum groups in Lecture~3.
As for the category $\QQ$, there are only very few cases in which it is well understood (by ``well understood'', we mean that the set of simple perverse sheaves $\mathcal{P}_{\vec{Q}}$ is determined). The simplest case is that of finite type quivers~: these have zero-dimensional moduli spaces $\Ma$ (actually, in the sense of (\ref{E:101}), their dimension is even negative !), and it is not surprising that $\mathcal{P}^{\gamma}$ is easy to describe. The next class of quivers for which the moduli spaces
$\Ma$ are known is that of affine (or tame) quivers. Although of positive dimension, these may be stratified in nice pieces looking very much like an adjoint quotient $\mathfrak{gl}_k/GL_k$ (see Section~2.5.). We will state the classification (due to Lusztig and Li-Lin) of the elements of $\mathcal{P}_{\vec{Q}}$ for tame quivers, and give a sketch of the proof. We pay special attention in Section~2.4. to the class of cyclic quivers (including the Jordan quiver); as usual, these are somewhat intermediate between finite type and tame type. 

\vspace{.3in}

\centerline{\textbf{2.1. The simplest of all quivers.}}
\addcontentsline{toc}{subsection}{\tocsubsection {}{}{\; 2.1. The simplest of all quivers.}}

\vspace{.15in}

The following notation will be useful. For $n \in \N$ we denote as usual the $v$-integer and $v$-factorials
$$[n]=\frac{v^n-v^{-n}}{v-v^{-1}}, \qquad [n]!=[2] \cdots [n].$$
There should be no risk of confusion with the shift operation $\mathbb{P} \mapsto \mathbb{P}[n]$ on complexes.

\vspace{.1in}

If $\mathbb{P}$ is any object of a triangulated category with shift functor $X \to X[1]$ and if $R(v)=\sum_i r_i v^i \in \N[v,v^{-1}]$ is a Laurent polynomial then we set
$$\mathbb{P}^{\oplus R(v)}=\bigoplus_i \mathbb{P}^{\oplus r_i}[i].$$

We will often make use of the classical (see \cite{Hiller})

\begin{lem}\label{L:21}Let $d_1, \ldots, d_r$ be positive integers and put $d=d_1 + \cdots + d_r$. Let $P_{d_1, \ldots, d_r} \subset GL(d,k)$ be the associated parabolic subgroup and 
$$\mathcal{B}_{d_1, \ldots, d_r}=GL(d,k)/P_{d_1, \ldots, d_r}$$ 
the corresponding partial flag variety. Then $dim\; \mathcal{B}_{d_1, \ldots, d_r}=\sum_{i <j} d_i d_j$ and the graded dimension of the total cohomology $H^*(\mathcal{B}_{d_1, \ldots, d_r})[dim\;\mathcal{B}_{d_1, \ldots,d_r}]$ is equal to
$$\sum_i dim\; H^i(\mathcal{B}_{d_1, \ldots, d_r}[dim\;\mathcal{B}_{d_1, \ldots,d_r}]) v^{-i}=\frac{[d]!}{[d_1]! \cdots [d_r]!}$$ 
\end{lem}

\vspace{.2in}

The simplest of all quivers alluded to in the title of this section has one vertex and no arrow, i.e. $I=\{1\}, H=\emptyset$. Then $E_n=\{0\}$, $G_n=GL(n,k)$. There is a unique simple perverse sheaf $\mathbbm{1}_{n}=\qlb_{|E_n}$ on $E_n$ for any $n$, and this perverse sheaf is $G_n$-equivariant.

\vspace{.1in}

Let us determine the set of simple perverse sheaves $\mathcal{P}^n$. By definition, elements of $\mathcal{P}^n$ are simple summands of the semisimple complex $L_{1,1, \ldots, 1}$ (there is no choice for a splitting of $n$ as a sum of simple dimension vectors). We have $L_{1, \ldots, 1}=q_! (\qlb_{E_{1, \ldots, 1}})[dim\;E_{1, \ldots, 1}]$
where 
$$E_{1, \ldots, 1}=\big\{ ( \underline{y}, W_n \subset \cdots \subset W_1=V_n)\big\}=\big\{(0,W_n \subset \cdots \subset W_1=V_n)\big\}$$
is simply the flag variety $\mathcal{B}=\mathcal{B}_{1, \ldots, 1}$ of $GL(n,k)$ and where $q: E_{1, \ldots, 1} \to E_{n}=\{0\}$ is the projection to a point. Therefore 
\begin{equation}\label{E:n1}
\begin{split}
L_{1, \ldots, 1}&=
q_!(\qlb_{\mathcal{B}})[dim\;\mathcal{B}]\\
&=\bigoplus_k H^k(\mathcal{B}, \qlb)[n(n-1)/2-k]\\
&= (\mathbbm{1}_n)^{\oplus[n]!}.
\end{split}
\end{equation}
Thus $\mathbbm{1}_n$ appears in $L_{1, \ldots, 1}$ and  belongs to $\mathcal{P}^n$ (of course, here $\mathbbm{1}_n$ is the \textit{only} simple perverse sheaf on $E_n$). In conclusion, $$\mathcal{P}_{\vec{Q}}=\{\mathbbm{1}_n\}_{n \in \N}, \qquad \QQ=\bigsqcup_n D^b_{G_{\gamma}}(E_{\gamma})^{ss}.$$

\vspace{.1in}

To compute the action of the restriction functor $\underline{\Delta}$ it is enough here to apply Lemma~\ref{L:coprodun}~:
\begin{equation}
\underline{\Delta}(\mathbbm{1}_n)=\sum_{l=0}^n \mathbbm{1}_l \boxtimes \mathbbm{1}_{n-l}[-l(n-l)].
\end{equation}

The scalar product is also easy to determine~: we have 
\begin{equation}\label{E:22}
\{\mathbbm{1}_1, \mathbbm{1}_1\}=\sum_j dim\; H^j_{GL(1)}(pt)=\sum_{j \geq 0} v^{2j}=\frac{1}{1-v^{2}},
\end{equation}
and more generally
\begin{equation}\label{E:23}
\{\mathbbm{1}_n, \mathbbm{1}_n\}=\sum_j dim\;H^j_{GL(n)}(pt)=\prod_{k=1}^n\frac{1}{1-v^{2k}}
\end{equation}
for any $n$. Finally, note that all $\mathbb{P} \in \mathcal{P}_{\vec{Q}}$ are self-dual.

\vspace{.2in}

\addtocounter{theo}{1}
\noindent \textbf{Remark \thetheo .} By equation~(\ref{E:n1}) the Lusztig sheaves $L_{1, \ldots, 1}$ and $L_{n}$ are proportional. Of course, this means that for any quiver $\vec{Q}$ and vertex $i \in I$ the Lusztig sheaves $L_{\epsilon_i, \ldots, \epsilon_i}$ and $L_{n \epsilon_i}$ are proportional. As a consequence, we may slightly relax the conditions in the definition of the set $\mathcal{P}^{\gamma}$ (see Section~1.4) by allowing Lusztig sheaves $L_{\a_1, \ldots, \a_n}$ where each $\a_k$ is a \textit{multiple} of a simple dimension vector.

\vspace{.3in}

\centerline{\textbf{2.2. The fundamental relations.}}
\addcontentsline{toc}{subsection}{\tocsubsection {}{}{\; 2.2. The fundamental relations.}}

\vspace{.15in}

We move on to the next simplest class of quivers~: we assume that $\vec{Q}$ has two vertices.

\vspace{.2in}

\addtocounter{theo}{1}
\noindent \textbf{Example \thetheo .} Let us now assume that $\vec{Q}$ has vertices $\{1, 2\}$ and one edge $h: 1 \to 2$. 

\vspace{.15in}

\centerline{
\begin{picture}(120, 10)
\put(40,0){\circle*{5}}
\put(90,0){\circle*{5}}
\put(65,0){\vector(1,0){5}}
\put(37,-10){$1$}
\put(87,-10){$2$}
\put(40,0){\line(1,0){50}}
\put(0,-2){$\vec{Q}=$}
\end{picture}}

\vspace{.3in}

We have $K_0(Rep\;\vec{Q}) \simeq \Z^2$ and we denote as usual by $\epsilon_1, \epsilon_2$ the dimension vectors of the simple representations $S_1, S_2$. Besides the two simple representations, there is another indecomposable representation $I_{12}$ (of dimension $\epsilon_1+\epsilon_2$). By definition if $\a=d_1\epsilon_1+d_2\epsilon_2$ then
$$E_{\a}=Hom(k^{d_1}, k^{d_2}),$$
$$G_{\a}=GL(d_1, k) \times GL(d_2,k).$$
The $G_{\a}$-orbits in $E_{\a}$ are formed by the set of linear maps of a fixed rank $r \leq inf(d_1, d_2)$. In terms of representations, these are the representations isomorphic to $I^{\oplus r} \oplus S_1^{\oplus d_1-r} \oplus S_2^{\oplus d_2-r}$. 
We will denote these orbits by $\mathcal{O}^{\a}_r$, or simply $\mathcal{O}_r$, so that $E_{\a}=\bigsqcup_r \mathcal{O}_r$. It is easy to see that $\overline{\mathcal{O}_{r}}=\bigsqcup_{s \leq r} \mathcal{O}_s$. 

The stabilizer of $\mathcal{O}_r$ is isomorphic to $Aut(I^{\oplus r} \oplus S_1^{\oplus d_1-r} \oplus S_2^{\oplus d_2-r})$, and it can be checked that this group is connected. As a consequence, any $G_{\a}$-equivariant local system on $\mathcal{O}_r$ is trivial. It follows that any simple $G_{\a}$-equivariant perverse sheaf on $E_{\a}$ is of the form $IC(\mathcal{O}_r)$ for some $r$. 

\vspace{.1in}

For appetizers, let us compute $L_{\epsilon_1, \epsilon_2}$ and $L_{\epsilon_2, \epsilon_1}$. We have $E_{\epsilon_1+\epsilon_2}=Hom(k,k) \simeq k$. The two orbits are $\mathcal{O}_0=\{0\}$ and $\mathcal{O}_1=k \backslash\{0\}$. By definition,
$$E_{\epsilon_1, \epsilon_2}=\big\{ (\underline{y}, W \subset V_{\epsilon_1+\epsilon_2})\;|\; \underline{dim}\;W=\epsilon_2; \;\underline{y}(W) \subset W\big\} \simeq E_{\epsilon_1+\epsilon_2}$$
since there is a unique subspace $W$ of $V_{\epsilon_1+\epsilon_2}$ of dimension $\epsilon_2$ and it is stable under the action of any $\underline{y} \in E_{\epsilon_1+\epsilon_2}$. Therefore the map $q_{\epsilon_1, \epsilon_2}$ is trivial and 
$$L_{\epsilon_1, \epsilon_2}=\qlb_{|E_{\epsilon_1+\epsilon_2}}[1]=\mathbbm{1}_{\epsilon_1+\epsilon_2}=IC(\mathcal{O}_1).$$
On the other hand,
$$E_{\epsilon_2, \epsilon_1}=\big\{ (\underline{y}, W \subset V_{\epsilon_1+\epsilon_2})\;|\; \underline{dim}\;W=\epsilon_1;\; \underline{y}(W) \subset W\big\} \simeq \{pt\}$$
since there is again a unique subspace $W$ of $V_{\epsilon_1+\epsilon_2}$ of dimension $\epsilon_1$, but this time only the trivial representation $\underline{y}=0$ preserves it. It follows that
$$L_{\epsilon_2, \epsilon_1}=\qlb_{\{0\}}=IC(\mathcal{O}_0).$$

\vspace{.1in}

We now consider the dimension vector $\a=\epsilon_1+2\epsilon_2$. Here
$$E_{\a}=Hom(k,k^2) \simeq k^2,$$
$$\mathcal{O}_0=\{0\}, \qquad \mathcal{O}_{1}=k^2 \backslash \{0\}.$$
There are three possible simple induction products of dimension $\a$~: $L_{\epsilon_1, \epsilon_2, \epsilon_1}, L_{\epsilon_1, \epsilon_1, \epsilon_2}$ and $L_{\epsilon_2, \epsilon_1, \epsilon_1}$. The respective incidence varieties are
$$E_{\epsilon_1, \epsilon_2, \epsilon_1}=\big\{ (\underline{y}, W_2 \subset W_1 \subset W_{\a})\;|\;\underline{dim}\;W_2=\epsilon_1, \underline{dim}\; W_1/W_2=\epsilon_2; \underline{y}(W_i) \subset W_i\big\},$$
$$E_{\epsilon_1, \epsilon_1, \epsilon_2}=\big\{ (\underline{y}, W_2 \subset W_1 \subset W_{\a})\;|\;\underline{dim}\;W_2=\epsilon_2, \underline{dim}\; W_1/W_2=\epsilon_1; \underline{y}(W_i) \subset W_i\big\},$$
$$E_{\epsilon_2, \epsilon_1, \epsilon_1}=\big\{ (\underline{y}, W_2 \subset W_1 \subset W_{\a})\;|\;\underline{dim}\;W_2=\epsilon_1, \underline{dim}\; W_1/W_2=\epsilon_1; \underline{y}(W_i) \subset W_i\big\}.$$
We describe in a table the types of fibers of the proper maps $q_{\epsilon_1, \epsilon_2, \epsilon_1},...$ over the two orbits $\mathcal{O}_0$ and $\mathcal{O}_1$~:

\begin{equation}\label{E:table}
\begin{tabular}{c|c|c}
 & $\;\;\mathcal{O}_0\;\;$ & $\;\;\mathcal{O}_1\;\;$\\
\hline
$q_{\epsilon_1, \epsilon_2,\epsilon_1}$ & $\mathbb{P}^1$ &  $\{pt\}$\\
\hline
$q_{\epsilon_1, \epsilon_1,\epsilon_2}$ & $\mathbb{P}^1$ &  $\mathbb{P}^1$\\
\hline
$q_{\epsilon_2, \epsilon_1,\epsilon_1}$ & $\mathbb{P}^1$ &  $\emptyset$
\end{tabular} 
\end{equation}
Using this we can give the dimensions of the cohomology spaces of the Lusztig sheaves $L_{\epsilon_1, \epsilon_2, \epsilon_1}, ...$ (we only write those which are nonzero)~:
\begin{equation}\label{E:table2}
\begin{tabular}{c|c|c}
 & $\;\;\mathcal{O}_0\;\;$ & $\;\;\mathcal{O}_1\;\;$\\
\hline
$H^{-2}(L_{\epsilon_1, \epsilon_2,\epsilon_1})$ & $1$ &  $1$\\
\hline
$H^{0}(L_{\epsilon_1, \epsilon_2,\epsilon_1})$ & $1$ &  $0$
\end{tabular} 
\end{equation}
\begin{equation}\label{E:table3}
\begin{tabular}{c|c|c}
 & $\;\;\mathcal{O}_0\;\;$ & $\;\;\mathcal{O}_1\;\;$\\
\hline
$H^{-3}(L_{\epsilon_1, \epsilon_1,\epsilon_2})$ & $1$ &  $1$\\
\hline
$H^{-1}(L_{\epsilon_1, \epsilon_1,\epsilon_2})$ & $1$ &  $1$
\end{tabular} 
\end{equation}
\begin{equation}\label{E:table4}
\begin{tabular}{c|c|c}
 & $\;\;\mathcal{O}_0\;\;$ & $\;\;\mathcal{O}_1\;\;$\\
\hline
$H^{-1}(L_{\epsilon_2, \epsilon_1,\epsilon_1})$ & $1$ &  $0$\\
\hline
$H^{1}(L_{\epsilon_2, \epsilon_1,\epsilon_1})$ & $1$ &  $0$
\end{tabular} 
\end{equation}
Comparing dimensions of stalks of $\mathcal{O}_0, \mathcal{O}_1$ we deduce from the above tables (\ref{E:table2}), (\ref{E:table3}), (\ref{E:table4}) that
\begin{equation}
\begin{split}
L_{\epsilon_1, \epsilon_2, \epsilon_1} &\simeq IC(\mathcal{O}_0) \oplus IC(\mathcal{O}_1), \\
L_{\epsilon_1, \epsilon_1, \epsilon_2} &\simeq IC(\mathcal{O}_1)[-1] \oplus IC(\mathcal{O}_1)[1], \\
L_{\epsilon_2, \epsilon_1, \epsilon_1} &\simeq IC(\mathcal{O}_0)[-1] \oplus IC(\mathcal{O}_1)[1].
\end{split}
\end{equation} 
As a corollary we arrive at the following identity
\begin{equation}\label{E:fund1}
L_{\epsilon_1, \epsilon_2, \epsilon_1}[1] \oplus L_{\epsilon_1, \epsilon_2, \epsilon_1}[-1] \simeq L_{\epsilon_1, \epsilon_1, \epsilon_2} \oplus L_{\epsilon_2, \epsilon_1, \epsilon_1}.
\end{equation}
Using the equation $L_{\epsilon_1, \epsilon_1} \simeq \mathbbm{1}_{2\epsilon_1}[1] \oplus  \mathbbm{1}_{2\epsilon_1}[-1] = L_{2\epsilon_1}[1] \oplus  L_{2\epsilon_1}[-1]$ (see (\ref{E:n1})), and the associativity $L_{\beta} \star L_{\gamma}=L_{\beta,\gamma}$ we may rewrite (\ref{E:fund1}) simply as
\begin{equation}\label{E:dunf1}
L_{\epsilon_1, \epsilon_2, \epsilon_1} \simeq L_{2\epsilon_1, \epsilon_2} \oplus L_{\epsilon_2, 2\epsilon_1}.
\end{equation}

A very similar computation when $\a=2\epsilon_1 +\epsilon_2$ shows that
\begin{equation}\label{E:fund2}
L_{\epsilon_2, \epsilon_1, \epsilon_2}[1] \oplus L_{\epsilon_2, \epsilon_1, \epsilon_2}[-1] \simeq L_{\epsilon_2, \epsilon_2, \epsilon_1} \oplus L_{\epsilon_1, \epsilon_2, \epsilon_2}.
\end{equation}
\begin{equation}\label{E:dunf2}
L_{\epsilon_2, \epsilon_1, \epsilon_2} \simeq L_{2\epsilon_2, \epsilon_1} \oplus L_{\epsilon_1, 2\epsilon_2}.
\end{equation}
\endexample

\vspace{.1in}

Relations (\ref{E:fund1}) and (\ref{E:fund2}) (or equivalently (\ref{E:dunf1}) and (\ref{E:dunf2})) are called the \textit{fundamental relations}. There are analogues of these relations for any quiver with two vertices~:

\vspace{.1in}

\begin{lem} Let $\vec{Q}$ be a quiver with vertices $\{1, 2\}$ and $r$ arrows $h_1, \ldots, h_r$ linking $1$ and $2$ (in any direction)~: 

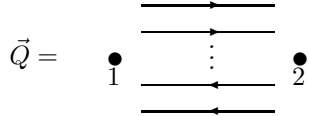
\centerline{
\begin{picture}(130, 50)
\put(30,20){\circle*{5}}
\put(100,20){\circle*{5}}
\put(27,10){$1$}
\put(97,10){$2$}
\put(65,16){$\vdots$}
\put(65,30){\vector(1,0){5}}
\put(40,30){\line(1,0){50}}
\put(65,40){\vector(1,0){5}}
\put(40,40){\line(1,0){50}}
\put(70,0){\vector(-1,0){5}}
\put(40,0){\line(1,0){50}}
\put(70,10){\vector(-1,0){5}}
\put(40,10){\line(1,0){50}}
\put(-10,18){$\vec{Q}=$}
\end{picture}}

\vspace{.15in}

Then the following identity holds~:
$$ 
\bigoplus_{\substack{l =0\\ l\;even}}^r L_{l \epsilon_1, \epsilon_2, (r-l)\epsilon_1} \simeq 
\bigoplus_{\substack{l =0\\ l\;odd}}^r L_{l \epsilon_1, \epsilon_2, (r-l)\epsilon_1}.
$$
\end{lem}
\begin{proof} It is very similar to the above special case $r=1$ (see e.g. \cite[Section~3]{SK} ). Details are left to the reader.
\end{proof}

\vspace{.1in}

Let us return to the case of the quiver $\vec{Q}$ with vertices $\{1,2\}$ and one arrow $h: 1\to 2$, and let us determine the category $\QQ$ and the simple perverse sheaves $\mathcal{P}_{\vec{Q}}$. Let us fix a dimension vector $\gamma=(d_1, d_2)$ and set $d=inf(d_1,d_2)$. There are $d+1$ simple perverse sheaves $IC(\mathcal{O}_0)=\qlb_{\{0\}}, IC(\mathcal{O}_1), \ldots, IC(\mathcal{O}_{d})=\mathbbm{1}_{E_{\gamma}}$ in $D^b(\underline{\mathcal{M}}^{\gamma}_{\vec{Q}})$. We claim that all of these do belong to $\mathcal{P}^{\gamma}$. To see this, it suffices to observe that the Lusztig sheaf $L_{(d_1-r)\epsilon_2, d_1\epsilon_1, r\epsilon_2}$ satisfies
$$supp \;L_{(d_1-r)\epsilon_2, d_1\epsilon_1, r\epsilon_2} =\overline{\mathcal{O}_r}$$
(it is easy to see that $Im (q_{(d_1-r)\epsilon_2, d_1\epsilon_1, r\epsilon_2}) \subset \overline{\mathcal{O}_r}$ and that the fiber over $\mathcal{O}_r$ is nonempty).
It follows that $IC(\overline{\mathcal{O}_r})$ appears in $L_{(d_1-r)\epsilon_2, d_1\epsilon_1, r\epsilon_2}$. 

In conclusion we have, as in Section~2.1. above,
$$\PQ=\bigsqcup_{\gamma} \big\{ IC(\mathcal{O}^{\gamma}_r)\;|\; r \leq inf(\gamma_1,\gamma_2)\big\}, \qquad \QQ= \bigsqcup_{\gamma}D^b_{G_{\gamma}}(E_{\gamma})^{ss}.$$

\vspace{.1in}

To wrap it up with quivers having two vertices, we provide sample computations of the restriction functor and of the scalar product.
We start with 
\begin{equation*}
\begin{split}
\underline{\Delta}_{\epsilon_1, \epsilon_1+\epsilon_2}(IC(\mathcal{O}_0))&=\underline{\Delta}_{\epsilon_1, \epsilon_1+\epsilon_2}(\qlb_{\{0\}})\\
&=\kappa_{!}\iota^* (\qlb_{\{0\}})=\kappa_!(\qlb_{\{0\}})=\qlb_{\{0\}} \boxtimes \qlb_{\{0\}}
\end{split}
\end{equation*}
where we have used the notations of Section~1.3. For $\underline{\Delta}_{\epsilon_1, \epsilon_1+\epsilon_2}(IC(\mathcal{O}_1))$ we may use Lemma~\ref{L:coprodun}~:
\begin{equation*}
\Delta_{\epsilon_1, \epsilon_1+\epsilon_2}(IC(\mathcal{O}_1))=\Delta_{\epsilon_1, \epsilon_1+\epsilon_2}(\mathbbm{1}_{2\epsilon_1+\epsilon_2})
=\mathbbm{1}_{\epsilon_1} \boxtimes \mathbbm{1}_{\epsilon_1 + \epsilon_2}[-1].
\end{equation*}
We now assume that $\gamma=\epsilon_1+\epsilon_2$. Then
\begin{equation*}
\begin{split}
\{IC(\mathcal{O}_0), IC(\mathcal{O}_0\}&=\{\qlb_{\{0\}}, \qlb_{\{0\}}\}\\
&=\sum_j dim\; H^j_{GL(1) \times GL(1)}(\qlb_{\{0\}}, E_{\gamma})v^j\\
&=\sum_j dim\;H^j_{GL(1) \times GL(1)}(pt)v^j=\frac{1}{(1-v^2)^2},
\end{split}
\end{equation*}
while
\begin{equation*}
\begin{split}
\{IC(\mathcal{O}_1), IC(\mathcal{O}_1\}&=\{\qlb_{E_{\gamma}}[1], \qlb_{E_{\gamma}}[1]\}\\
&=\sum_j dim\;H^j_{GL(1) \times GL(1)}(\qlb_{E_{\gamma}}, E_{\gamma})v^{j+2}=\frac{1}{(1-v^2)^2}
\end{split}
\end{equation*}
and 
\begin{equation*}
\begin{split}
\{IC(\mathcal{O}_0), IC(\mathcal{O}_1\}&=\{\qlb_{\{0\}}, \qlb_{E_{\gamma}}[1]\}\\
&=\sum_j dim\; H^j_{GL(1) \times GL(1)}(\qlb_{\{0\}}[1], E_{\gamma})v^j\\
&=\sum_j dim\;H^j_{GL(1) \times GL(1)}(pt)v^{j+1}=\frac{v}{(1-v^2)^2}.
\end{split}
\end{equation*}
This also gives an example of Corollary~\ref{C:scalarprod}. Here again, all the perverse sheaves $\mathbb{P} \in \mathcal{P}^{\gamma}$ are self dual.

\vspace{.2in}
We close this section with the following simple observation~:

\vspace{.1in}

\addtocounter{theo}{1}
\noindent \textbf{Example \thetheo .} Let $\vec{Q}$ be any quiver which has no oriented cycles. Then for all dimension vectors $\gamma$ the perverse sheaves $\qlb_{\{0\}}$ and $\mathbbm{1}_{\gamma}=\qlb_{E_{\gamma}}[dim\;E_{\gamma}]$ belong to $\mathcal{P}^{\gamma}$. To see this, let us relabel the vertices of $I$ as $\{1, 2, \ldots, n\}$ in such a way that $i<j$ if there exists an arrow going from $i$ to $j$ in $\vec{Q}$. Write $\gamma=\sum_i \gamma_i \epsilon_i$. We claim that 
\begin{equation}
L_{\gamma_1\epsilon_1, \ldots, \gamma_n \epsilon_n} = \mathbbm{1}_{\gamma},
\end{equation}
\begin{equation}
L_{\gamma_n\epsilon_n, \ldots, \gamma_1 \epsilon_1} = \qlb_{\{0\}}.
\end{equation}
Indeed, 
$$q_{\gamma_1\epsilon_1, \ldots, \gamma_n \epsilon_n}: E_{\gamma_1\epsilon_1, \ldots, \gamma_n\epsilon_n} \stackrel{\sim}{\to} E_{\gamma}$$
is an isomorphism, while
$$q_{\gamma_n\epsilon_n, \ldots, \gamma_1 \epsilon_1}: E_{\gamma_1\epsilon_1, \ldots, \gamma_n\epsilon_n} {\to} E_{\gamma}$$
is the closed embedding of $\{0\}$.
\endexample

\vspace{.1in}

As a corollary of this, we can state~:

\vspace{.1in}

\begin{cor} Assume that $\vec{Q}$ has no oriented cycles. Then for \textit{any} dimension vectors $\a_1, \ldots, \a_n$ the Lusztig sheaf $L_{\a_1, \ldots, \a_n} $ belongs to $\QQ$.
\end{cor}

\vspace{.3in}

\centerline{\textbf{2.3. Finite type quivers.}}
\addcontentsline{toc}{subsection}{\tocsubsection {}{}{\; 2.3. Finite type quivers.}}

\vspace{.15in}

In this section we consider a quiver $\vec{Q}$ of finite type. Recall (see e.g., \cite[Lecture~3]{Trieste} ) that this is equivalent to either of the following statements~:
\begin{enumerate}
\item[a)] There are only finitely many (nonisomorphic) indecomposable representations of $\vec{Q}$,
\item[b)] For any $\a$, there are finitely many $G_{\alpha}$-orbits in $E_{\alpha}$,
\item[c)] For all $\alpha$, we have $dim\;\Ma =-\langle \a, \a \rangle <0$ .
\end{enumerate}

\vspace{.1in}

Our aim is to determine the Hall category $\QQ$ and the set $\mathcal{P}_{\vec{Q}}$. Let us first describe all the objects in $D^b_{G_{\a}}(E_{\a})^{ss}$ for $\a \in I^\N$. We begin with a general result~:

\vspace{.1in}

\begin{lem}\label{L:syst} Let $\vec{Q}$ be any quiver and let $\alpha$ be a dimension vector. Let $\mathcal{O} \subset E_{\a}$ be a $G_{\a}$-orbit. Then any $G_{\a}$-equivariant local system on $\mathcal{O}$ is trivial.
\end{lem}
\begin{proof} We have to show that the stabilizer of any point $\underline{x} \in \mathcal{O}$ is connected. We have $Stab_{G_{\a}}\ux \simeq Aut(M_{\ux})$ where $M_{\ux}$ is the representation of $\vec{Q}$ corresponding to $\ux$. There is a chain of inclusions $k^* \subset Aut(M_{\ux}) \subset End(M_{\ux})$. Each connected component of $Aut(M_{\ux})$ is stable under multiplication by $k^*$ and thus contains the point $0 \in End(M_{\ux})$ in its closure. This implies that $Aut(M_{\ux})$ is connected since it is an algebraic group.
\end{proof}

\vspace{.1in}

It follows from Lemma~\ref{L:syst} that (for $\vec{Q}$ a finite type quiver) the only simple $G_{\a}$-equivariant perverse sheaves on $E_{\a}$ are of the form $IC(\mathcal{O})$ for some $G_{\a}$-orbit $\mathcal{O}$. As we will see, and as in the examples of Sections~2.1 and 2.2, \textit{all} of these belong to the Hall category $\QQ$. Let us denote by $|E_{\gamma}/G_{\gamma}|$ the set of $G_{\gamma}$-orbits in $E_{\gamma}$.

\vspace{.1in}

\begin{theo}\label{T:finiteype} Let $\vec{Q}$ be a finite type quiver. Then 
$$\mathcal{P}_{\vec{Q}} =\bigsqcup_{\gamma} \big\{ IC(\mathcal{O})\;|\; \mathcal{O} \in | E_{\gamma}/G_{\gamma}|\big\},$$
$$\QQ = \bigsqcup_{\gamma} D^b_{G_{\gamma}}(E_{\gamma})^{ss}.$$
In particular, any $\mathbb{P} \in \mathcal{P}_{\vec{Q}}$ is self-dual.
\end{theo}
\begin{proof} We will prove this by explicitly constructing, for any orbit $\mathcal{O}$, a Lusztig sheaf $L_{\a_1, \ldots, \a_n}$ satisfying
\begin{equation}\label{E:26}
Supp\; L_{\a_1, \ldots, \a_n}=\overline{\mathcal{O}}.
\end{equation}
This Lusztig sheaf will also incidentally satisfy
\begin{equation}\label{E:27}
(L_{\a_1, \ldots, \a_n})_{|\mathcal{O}} \simeq \qlb_{\mathcal{O}}[dim \;\mathcal{O}].
\end{equation}
The construction of the above Lusztig sheaf in turn proceeds from the existence, for any $\gamma$ and any orbit $\mathcal{O} \subset E_{\gamma}$ of a sequence 
$\nu=(l_1\epsilon_{i_1}, \ldots, l_s \epsilon_{i_s})$ for which the map $q_{\underline{\nu}}~:E_{\underline{\nu}} \to E_{\gamma}$ is a desingularization of $\overline{\mathcal{O}}$. Indeed, for such $\underline{\nu}$ we have
$$L_{\underline{\nu}} =(q_{\underline{\nu}})_! \qlb_{E_{\underline{\nu}}}[dim\;E_{\underline{\nu}}] = IC(\mathcal{O}) \oplus \mathbb{P}$$ where $\mathbb{P}$ is a semisimple complex supported on $ \overline{\mathcal{O}}\backslash \mathcal{O}$. It follows that $IC(\mathcal{O}) \in \mathcal{P}^{\gamma}$.

\vspace{.1in}

We now give, following \cite{Reinekedesing}, the proof of the existence of the above desingularization ~: 

\begin{prop}[Reineke]\label{P:Reineke} Let $\vec{Q}$ be a finite type quiver and let $\gamma$ be a dimension vector. For any orbit $\mathcal{O} \subset E_{\gamma}$ there exists a sequence $\underline{\nu} =(l_1\epsilon_{i_1}, \ldots, l_s \epsilon_{i_s})$ of weight $\gamma$ such that the map
$$\xymatrix{ E_{\underline{\nu}} \ar[r]^-{q_{\underline{\nu}}}& E_{\gamma}}$$
is a desingularization of $\overline{\mathcal{O}}$.
\end{prop}

\noindent
\textit{Proof.} Recall that there exists a total ordering $\prec$ on the set of indecomposable representations of $\vec{Q}$, such that
$$N \prec N' \Rightarrow Hom(N',N) = Ext^1(N,N')=0$$
and that moreover $Ext^1(N,N)=0$ for any indecomposable $N$ (see \cite[Lemma~3.19.]{Trieste} ). Let $M$ be the representation associated to $\mathcal{O}$ and let us write
\begin{equation}
M \simeq N_1^{\oplus l_1} \oplus \cdots \oplus N_r^{\oplus l_r}
\end{equation}
where $N_1, \ldots, N_r$ are all the indecomposables, ordered in such a way that $N_k \prec N_l$ if $k <l$. We also relabel the vertices as $\{1, 2, \ldots, m\}$ in such a manner that there are arrows $i \to j$ only when $i<j$. This is possible since a finite type quiver does not have any oriented cycle.

Let us write $\mathcal{O}_R$ for the $G_{\underline{dim}\;R}$-orbit of $E_{\underline{dim}\;R}$ associated to a representation $R$. We will say that $R$ degenerates to $S$ if $\underline{dim}\;R=\underline{dim}\;S$ and $\mathcal{O}_S \subset \overline{\mathcal{O}_R}$. We cite the following classical result (valid for any quiver, see e.g. \cite{Brion})~:
\begin{lem}\label{L:codim} Let $M$ be a representation of $\vec{Q}$ of dimension $\gamma$. Then 
$$codim_{E_{\gamma}}\;\mathcal{O}_M = dim\;Ext^1(M,M).$$
\end{lem}
Put $\a_i=l_i \underline{dim}\;N_i$. It follows from the above lemma that $\mathcal{O}_{N_i^{\oplus l_i}}$ is the dense orbit of $E_{\a_i}$ for all $i$, and thus degenerates to any other representation of dimension $\a_i$. Consider the sequence
$$\underline{\nu}=\big( (\a_1)_1, \ldots, (\a_1)_m, (\a_2)_1, \ldots, (\a_2)_m, \ldots, (\a_r)_1, \ldots, (\a_r)_m\big).$$
We claim that $\underline{\nu}$ satisfies the requirements of the Proposition. First of all, the maps 
$$q_{(\a_i)_1, \ldots, (\a_i)_m}: E_{(\a_i)_1, \ldots, (\a_i)_m} \to E_{\a_i}$$
are all isomorphisms (see Example~2.5.), and therefore $E_{\underline{\nu}}$ and $q_{\underline{\nu}}$ may be identified with $E_{\a_1, \ldots, \a_r}$ and $q_{\a_1, \ldots, \a_r}$. Hence we have to show that $q_{\a_1, \ldots, \a_r}$ is a desingularization of $\overline{O}=\overline{\mathcal{O}_M}$.

\vspace{.1in}

The fact that $Im(q_{\a_1, \ldots, \a_r}) \subset \overline{\mathcal{O}}$ is implied by the following general lemma~:
\begin{lem}\label{L:213} Let $X_1, \ldots, X_r$ be representations of a quiver and assume that $Ext^1(X_k,X_l)=0$ if $k<l$. If $Y$ is a representation possessing a filtration
$$Y_r \subset Y_{r-1} \subset \cdots \subset Y_1=Y$$
such that $Y_k/Y_{k+1}$ is a degeneration of $X_k$ then $Y$ is a degeneration of $X_1 \oplus \cdots \oplus X_r$.
\end{lem}
\begin{proof} We only need to treat the case $r=2$ and then argue by induction on $r$. Let $\mathcal{O}_{X_1} \subset E_{d_1}, \mathcal{O}_{X_2} \subset E_{d_2}$ be the corresponding orbits, where $d_i=\underline{dim}\;X_i$ for $i=1,2$. Consider the induction diagram
$$\xymatrix{ E_{d_1} \times E_{d_2} & E^{(1)}_{d_1,d_2} \ar[l]_-{p} \ar[r]^-{r} & E_{d_1,d_2} \ar[r]^-{q} & E_{d_1+d_2}.
}$$
Since $p$ is smooth and $r$ is a principal bundle, $r p^{-1} (\overline{\mathcal{O}_{X_1} \times \mathcal{O}_{X_2}}) = \overline{rp^{-1}(\mathcal{O}_{X_1} \times \mathcal{O}_{X_2})}$. Hence
$$qr p^{-1} (\overline{\mathcal{O}_{X_1} \times \mathcal{O}_{X_2}}) = q\overline{rp^{-1}(\mathcal{O}_{X_1} \times \mathcal{O}_{X_2})}
\subset \overline{qrp^{-1}(\mathcal{O}_{X_1} \times \mathcal{O}_{X_2})}.$$
But because $Ext^1(X_1,X_2)=0$ we have $qrp^{-1}(\mathcal{O}_{X_1} \times \mathcal{O}_{X_2}) = \mathcal{O}_{X_1 \oplus X_2}$. The Lemma follows.\end{proof}

\vspace{.1in}

Since obviously $\mathcal{O} \subset Im(q_{\a_1, \ldots, \a_r})$ we get $\overline{Im(q_{\a_1, \ldots, \a_r})}=\overline{\mathcal{O}}$. 
It remains to prove that $q_{\a_1, \ldots , \a_r}$ is an isomorphism over $\mathcal{O}$. We now use~:

\begin{lem}\label{L:22} Let $X_1, \ldots, X_r$ and $Y_r \subset \cdots \subset Y_1=Y$ be as in Lemma~\ref{L:213}. Assume in addtion that $Hom(X_k,X_l)=0$ if $k>l$, and that $Y \simeq X_1 \oplus \cdots \oplus X_r$. Then $Y_k/Y_{k+1} \simeq X_k$  and $Y_k \simeq X_k \oplus \cdots \oplus X_r$ for all $k$.
\end{lem}
\begin{proof} It suffices again to deal with the case $r=2$. We have a chain of inclusions
$$Hom(X_2,Y_2) \subseteq Hom(X_2,Y) =Hom(X_2,X_2)$$
since $Y \simeq X_1 \oplus X_2$ and $Hom(X_2, X_1)=0$. But the function $N \mapsto dim\;Hom(X_2,N)$ is upper semicontinuous hence $dim\;Hom(X_2,Y_2) \geq dim\;Hom(X_2,X_2)$. It follows that $Hom(X_2,Y_2)=Hom(X_2,X_2)=Hom(X_2,Y)$. But then the image of the canonical map $X_2 \otimes Hom(X_2,Y) \to Y$ lies in $Y_2$ and is equal to $X_2$. Therefore $Y_2 \simeq X_2$ as desired.
\end{proof}

We are in position to conclude the argument. By definition, the fiber of $q_{\a_1, \dots, \a_r}$ over a point of $\mathcal{O}$ is the variety of filtrations
\begin{equation}\label{E:P221}
M_r \subset \cdots \subset M_1=M \simeq N_1^{\oplus l_1} \oplus \cdots \oplus N_r^{\oplus l_r}
\end{equation}
where $\underline{dim}\;M_k/M_{k+1}=\a_k$. By Lemma~\ref{L:22} we have $M_r \subset N_r^{\oplus l_r}$. But because $Hom(N_r,N_k)=0$ for all $k<r$ there is a unique submodule of $M$ isomorphic to $N_r^{\oplus l_r}$ and hence $M_r$ is fixed. The same argument applied to to $M'=M/M_r$ shows that $M_{r-1}$ is uniquely determined, and so on. Thus there is indeed a unique filtration of the form (\ref{E:P221}), and $q_{\a_1, \ldots, \a_r}$ is a desingularization of $\overline{\mathcal{O}}$. Proposition~\ref{P:Reineke} and Theorem~\ref{T:finiteype} are proved.
\end{proof}

\vspace{.2in}

\addtocounter{theo}{1}
\noindent \textbf{Remark \thetheo .} The arguments used in the proof of Theorem~\ref{T:finiteype} admit a straightforward but useful generalization to the following situation. Suppose that we are given $G_{\a_i}$-stable locally closed subsets $U_{\a_i} \subset E_{\a_i}$ for $i=1, \ldots, m$ such that for any collection of points $(\ux_i)_i \in \prod_i U_{\a_i}$ we have
$$Ext^1(M_{\ux_i}, M_{\ux_j})=Hom(M_{\ux_j},M_{\ux_i})=0 \quad \text{if}\;i<j.$$
Set $\a_1 + \cdots + \a_m=\a$ and define a subset $Z$ of $E_{\a}$ as
$$Z=\big\{ \ux \in E_{\a}\;|\; M_{\ux} \simeq M_{\ux_1} \oplus \cdots \oplus M_{\ux_m} \; \text{for \;some\;} (\ux_i)_i \in \prod_i U_{\a_i}\big\}.$$
Consider the (iterated) induction diagram (see Section~1.3 and the proof of Proposition~\ref{P:assoc})~:
\begin{equation}\label{E:assoc}
\xymatrix{
E^{(m-1)}_{\a_1, \ldots, \a_m} \ar[d]_-{p} \ar[r]^-{r} & E_{\a_1, \ldots, \a_m} \ar[r]^-{q} & E_\a\\
E_{\a_1} \times \cdots \times E_{\a_m} & &}
\end{equation}
Restricting (\ref{E:assoc}) to $U_{\a_1} \times \cdots \times U_{\a_m} \subset E_{\a_1} \times \cdots \times E_{\a_m}$ we obtain
\begin{equation}\label{E:assoc2}
\xymatrix{
 p^{-1}(U_{\a_1} \times \cdots \times U_{\a_m}) \ar[d]_-{p} \ar[r]^-{r'} & rp^{-1}(U_{\a_1} \times \cdots \times U_{\a_m}) \ar[r]^-{q'} & E_\a\\
 U_{\a_1} \times \cdots \times U_{\a_m} & &}
\end{equation}
From Lemmas~\ref{L:213} and \ref{L:22} we deduce that $r'$ is still a principal bundle and that $q'$ is an isomorphism onto $Z$.  For similar reasons, the restriction diagram
\begin{equation}
\xymatrix{ E_{\a_1} \times \cdots \times E_{\a_m} & F_{\a_1, \ldots, \a_m} \ar[l]_-{\kappa} \ar[r]^-{\iota} & E_{\a}}
\end{equation}
restricts (sic) to
\begin{equation}
\xymatrix{ U_{\a_1} \times \cdots \times U_{\a_m} & \kappa^{-1}(U_{\a_1} \times \cdots \times U_{\a_m}) \ar[l]_-{\kappa} \ar[r]^-{\iota} & Z.}
\end{equation}
We have (see Section~1.3, (\ref{E:FGP}))
\begin{equation}\label{E:FGP2}
Z \simeq rp^{-1}(U_{\a_1} \times \cdots \times U_{\a_m}) = \kappa^{-1}(U_{\a_1} \times \cdots \times  U_{\a_m}) \underset{P_{\a_1, \ldots, \a_m}}{\times} G_{\a}.
\end{equation}
In particular, by (\ref{E:FGP2})

\vspace{.05in}

\hspace{.1in} i) \textit{$Z$ is locally closed}

\vspace{.05in}

\noindent and from (\ref{E:assoc2}) we deduce

\vspace{.05in}

\hspace{.1in} ii) \textit{if $\mathbb{P}_i \in \mathcal{Q}^{\a_i}$ satisfy $supp\;\mathbb{P}_i \subset \overline{U_{\a_i}}$ for all $i$ then $supp\;\mathbb{P}_1 \star \cdots \star \mathbb{P}_m \subset \overline{Z}$ and}
$$\big(\mathbb{P}_1 \star \cdots \star \mathbb{P}_m\big)_{|Z} = r_{\#}p^*\big( (\mathbb{P}_1 \boxtimes \cdots \boxtimes \mathbb{P}_m)_{|U_{\a_1} \times \cdots \times U_{\a_m}}\big).$$

\vspace{.1in}

In order to get a good behavior with respect to the restriction functor we need to make an extra assumption~: suppose that for any $\ux \in Z$ there exists a \textit{unique} filtration $W_r \subset \cdots \subset W_1=V_{\a}$ compatible with $\ux$ such that $\underline{dim}\;W_i/W_{i+1}=\a_i$. Then

\vspace{.05in}

\hspace{.1in} iii) \textit{if $\mathbb{Q} \in \mathcal{Q}^{\a}$ satisfies $supp\;\mathbb{Q} \subset \overline{Z}$ then
$supp\; \underline{\Delta}_{\a_1, \ldots, \a_m}(\mathbb{P}) \subset \overline{U}_{\a_1} \times \cdots \times \overline{U}_{\a_m}$ and, up to a global shift,
$$p^*\big( \underline{\Delta}_{\a_1, \ldots, \a_m}(\mathbb{P})_{|U_{\a_1} \times \cdots \times U_{\a_m}} \big) \simeq r^*(\mathbb{P}_{|Z}).$$}

\vspace{-.1in}

We justify iii) briefly~: by our extra assumption it holds $F_Z := F_{\a_1, \ldots, \a_m} \cap Z =\kappa^{-1}(U_{\a_1} \times \cdots \times U_{\a_m})$. Thus $Z \simeq   F_Z \times_{P_{\a_1, \ldots,\a_m}} G_{\a} $. We deduce that $F \cap \overline{Z} \subset \overline{F_Z}$ and $\kappa(F \cap \overline{Z}) \subset \kappa(\overline{F_Z}) \subset \overline{\kappa(F_Z)}=\overline{U_{\a_1}} \times \cdots \times \overline{U_{\a_m}}$. The conclusion on the support in iii) follows. The last statement in iii) follows from the definitions.

\vspace{.3in}

\centerline{\textbf{2.4. The Jordan quiver and the cyclic quivers.}}
\addcontentsline{toc}{subsection}{\tocsubsection {}{}{\; 2.4. The Jordan quiver and the cyclic quivers.}}

\vspace{.15in}

Let us begin with the Jordan quiver

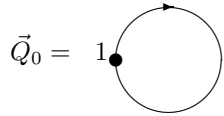
\vspace{.35in}

\centerline{
\begin{picture}(260, 10)
\put(80,0){$\vec{Q}_0=$}
\put(120,0){\circle*{5}}
\put(112,0){$1$}
\put(140,0){\circle{40}}
\put(140,20){\vector(1,0){2}}
\end{picture}}

\vspace{.4in}

Strictly speaking this is not a quiver of the type we are considering in these lectures, since it possesses an edge loop. However, it may be seen as a degenerate example of cyclic quivers and the theory works out well in this situation also. In fact, we are here in a very classical setting (indeed, the Hall algebra of the Jordan quiver is the \textit{classical Hall algebra} --see \cite[ Lecture~2]{Trieste}) and Lusztig's constructions for this case go back to \cite{LusztigGreen}. For a dimension vector $\gamma \in \N^I \simeq \N$ we have
$$E_{\gamma}=\mathfrak{gl}_{\gamma}, \qquad E_{\gamma}^{nil}=\mathcal{N}_{\gamma}, \qquad G_{\gamma}=GL_{\gamma}$$
where $\mathcal{N}_{\gamma} \subset \mathfrak{gl}_\gamma$ is the nilpotent cone. Of course, the nilpotent orbits in $\mathcal{N}_{\gamma}$ are parametrized by partitions $\lambda \dashv \gamma$.

\vspace{.1in}

\begin{prop} We have
$$\mathcal{P}_{\vec{Q}_0} =\{ IC(\mathcal{O}_{\lambda})\;|\; \lambda \in \Pi\big\}, \qquad \mathcal{Q}_{\vec{Q}_0}=\bigsqcup_{\gamma} D^b_{GL_{\gamma}}(E_{\gamma}^{0})^{ss}.$$
\end{prop}
\begin{proof} Lusztig's map coincides with the Springer resolution (see e.g. \cite[Chap. 3]{ChrissGinzburg} )
$$\xymatrix{
E_{1, \ldots, 1} \ar[d]_-{q_{1, \ldots, 1}} \ar@{=}[r]  &\widetilde{\mathcal{N}}_{\gamma} \ar[d]_-{\pi}\\
E^0_{\gamma} \ar@{=}[r] & \mathcal{N}_{\gamma}
}
$$
where 
$$\widetilde{\mathcal{N}}_{\gamma}=\big\{ (x, \mathfrak{b}) \in \mathcal{N}_{\gamma}\times \mathcal{B}_{\gamma}\;|\; x \in \mathfrak{b} \big\}$$
and $\mathcal{B}_{\gamma}$ is the variety of Borel subalgebras in $\mathfrak{gl}_k$ (the flag variety of $GL_{k}$). It is well-known that $\pi$ (and hence $q_{1, \ldots, 1}$) is a semismall map, and that all the strata $\{\mathcal{O}_{\lambda}\}$ are relevant. Thus
$$(q_{1, \ldots, 1})_! \big( \qlb_{E_{1, \ldots, 1}}[dim\;E_{1, \ldots, 1}]\big) = \bigoplus_{\lambda \dashv \gamma} IC(\mathcal{O}_{\lambda}) \otimes \mathbb{V}_{\lambda}$$
where $\mathbb{V}_{\lambda}$ is a nonzero complex of $\qlb$-vector spaces. Therefore all $IC(\mathcal{O}_{\lambda})$ belong to $\mathcal{P}_{\vec{Q}_0}$ as we wanted to show.
\end{proof}

\vspace{.1in}

The partition $(1^n)$ corresponds to the closed orbit  $\{0\}_n$ and we have $IC(\mathcal{O}_{(1^n)})=\qlb_{\{0\}_n}$. By a variant of the Springer desingularization, we have

\vspace{.1in}

\begin{lem}\label{L:springervariant} Let $\mu=(\mu_1 \geq \mu_2 \geq \cdots)$ be a partition and let $\mu'=(\mu'_1 \geq \mu'_2 \geq \cdots \geq \mu'_l)$ be the transpose partition. Then
$$\qlb_{\{0\}_{\mu'_1}} \star \qlb_{\{0\}_{\mu'_2}} \cdots  \star \qlb_{\{0\}_{\mu'_l}}=IC(\mathcal{O}_{\mu}) \oplus \mathbb{P} $$
for some complex $\mathbb{P}$ with $supp\;\mathbb{P}\subset \overline{\mathcal{O}_{\mu}} \backslash \mathcal{O}_{\mu}$.
\end{lem}
\begin{proof} Put 
$$E^{nil}_{\a_1, \ldots, \a_n}=\big\{ (\underline{x}, W_n \subset W_{n-1} \cdots \subset W_1=V_{\a_1 + \cdots + \a_n})\;|\; \underline{x} ( W_k) \subset W_{k+1} \big\}$$
and let $q^{nil}_{\a_1, \ldots, \a_n} : E^{nil}_{\a_1, \ldots, \a_n} \to E_{\a_1 + \cdots + \a_n}$ be the projection.
It is clear that $q^{nil}_{\a_1, \ldots, \a_n}$ is proper and that 
$$\qlb_{\{0\}_{\mu'_1}} \star \qlb_{\{0\}_{\mu'_2}} \cdots  \star \qlb_{\{0\}_{\mu'_l}}=(q^{nil}_{\mu'_1, \ldots, \mu'_l})_!(\qlb_{E^{nil}_{\mu'_1, \ldots, \mu'_l}}[dim\;E^{nil}_{\mu'_1, \ldots, \mu'_l}]).$$
The Lemma is now a consequence of the classical fact that $q^{nil}_{\mu'_1, \ldots, \mu'_l} $ is a desingularization of $\overline{\mathcal{O}_{\mu}}$.
\end{proof}

\vspace{.1in}

We can use the above Lemma to describe in part the induction product of two elements of $\mathcal{P}_{\vec{Q}_0}$. We begin with

\vspace{.1in}

\begin{lem}\label{L:commut} For any $\mathbb{P}, \mathbb{P}' \in \mathcal{Q}_{\vec{Q}_0}$ we have $\mathbb{P} \star \mathbb{P}'=\mathbb{P}' \star \mathbb{P}$.
\end{lem}
\begin{proof} The diagrams for $\underline{m}_{\a,\beta}$ and $\underline{m}_{\beta,\a}$ may be put together as
\begin{equation}\label{E:diagr}
\xymatrix{
& E_{\a,\beta}^{(1)} \ar[dl]_-{p}\ar[r]^-{r} & E_{\a,\beta} \ar[dr]^-{q} & \\
E_{\a} \times E_{\beta} & & & E_{\a+\beta}\\
& E_{\beta,\a}^{(1)} \ar[ul]^-{p'}\ar[r]^-{r'} & E_{\beta,\a} \ar[ur]_-{q'} & }
\end{equation}
Fixing an identification $V_{\gamma} \simeq V^*_{\gamma}$ for all $\gamma$ we get canonical isomorphisms
$i^{(1)}: E^{(1)}_{\a,\beta} \stackrel{\sim}{\to} E^{(1)}_{\beta,\alpha}, i: E_{\a,\beta} \stackrel{\sim}{\to} E_{\beta,\alpha}$ defined by
$$i^{(1)}(\underline{x}, W_{\beta} \subset V_{\a+\beta}, \rho_{\a}, \rho_{\beta}) = (\underline{x}^*, (V_{\a+\beta}/W_{\beta})^* \subset V^*_{\a+\beta} \simeq V_{\a_+\beta}, \rho_{\beta}^*, \rho_{\a}^*),$$
$$i(\underline{x}, W_{\beta} \subset V_{\a+\beta}) = (\underline{x}^*, (V_{\a+\beta}/W_{\beta})^* \subset V^*_{\a+\beta} \simeq V_{\a_+\beta})$$
(see Section~1.3 for the notations). The above identifications $i^{(1)}$ and $i$ turn (\ref{E:diagr}) into a commutative diagram. It follows that $\underline{m}_{\a,\beta} \simeq \underline{m}_{\beta,\alpha}$.
\end{proof}

\vspace{.1in}

If $\lambda=(\lambda_1 \geq \lambda_2 \geq \cdots)$ and $\mu=(\mu_1 \geq \mu_2 \geq \cdots )$ are partitions, we set $\lambda+\mu=(\lambda_1+ \mu_1 \geq \lambda_2+ \mu_2 \geq \cdots)$.

\vspace{.1in}

\begin{cor}\label{C:commut} For any partitions $\lambda,\mu$ it holds $IC(\mathcal{O}_{\lambda}) \star IC(\mathcal{O}_{\mu}) = IC(\mathcal{O}_{\lambda+\mu}) \oplus \mathbb{P}$, where $supp\;\mathbb{P} \subset \overline{\mathcal{O}_{\lambda+\mu}} \backslash \mathcal{O}_{\lambda+\mu}.$
\end{cor}
\begin{proof} We have, by Lemma~\ref{L:springervariant} 
$$IC(O_{\lambda}) \star IC(\mathcal{O}_{\mu}) \subset \qlb_{\{0\}_{\lambda'_1}} \star \qlb_{\{0\}_{\lambda'_2}} \cdots  \star \qlb_{\{0\}_{\mu'_1}} \star \qlb_{\{0\}_{\mu'_2}} \cdots  \star \qlb_{\{0\}_{\mu'_l}}.$$
Observe that $(\lambda+\mu)'$ is obtained by reordering the parts of $\lambda'$ and $\mu'$ together. Using Lemma~\ref{L:commut} we see that $IC(O_{\lambda}) \star IC(\mathcal{O}_{\mu}) \subset IC(\mathcal{O}_{\lambda+\mu}) \oplus \mathbb{P}$ for some complex $\mathbb{P}$ as above. It remains to check that $IC(O_{\lambda}) \star IC(\mathcal{O}_{\mu})_{|\mathcal{O}_{\lambda+\mu}} \neq 0$. This is easy.
\end{proof}

\vspace{.25in}

Let us now briefly turn our attention to the closely related case of equioriented cyclic quivers of type $A_{n-1}^{(0)}$~:

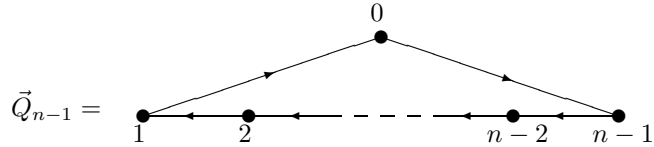
\vspace{.55in}

\centerline{
\begin{picture}(280, 10)
\put(20,0){$\vec{Q}_{n-1}=$}
\put(160,30){\circle*{5}}
\put(70,0){\line(3,1){90}}
\put(160,30){\line(3,-1){90}}
\put(70,0){\circle*{5}}
\put(110,0){\circle*{5}}
\put(210,0){\circle*{5}}
\put(250,0){\circle*{5}}
\put(70,0){\line(1,0){40}}
\put(210,0){\line(1,0){40}}
\put(110,0){\line(1,0){35}}
\put(180,0){\line(1,0){40}}
\put(150,0){\line(1,0){5}}
\put(160,0){\line(1,0){5}}
\put(170,0){\line(1,0){5}}
\put(90,0){\vector(-1,0){5}}
\put(130,0){\vector(-1,0){5}}
\put(195,0){\vector(-1,0){5}}
\put(230,0){\vector(-1,0){5}}
\put(115,15){\vector(3,1){5}}
\put(205,15){\vector(3,-1){5}}
\put(156,36){$0$}
\put(66,-10){$1$}
\put(106,-10){$2$}
\put(200,-10){$n-2$}
\put(240,-10){$n-1$}
\end{picture}}

\vspace{.3in}

These are distinct from all previous examples (and from other tame quivers) in the sense that not every representation is nilpotent, i.e. the spaces $\underline{\mathcal{M}}^{\a,nil}_{\vec{Q}}$ and $\Ma$ are different in general. In particular, $E_{\a}$ has infinitely many $G_{\a}$-orbits (at least for large enough dimension vector $\a$) while $E^{nil}_{\a}$ always has finitely many orbits. In addition, $E^{nil}_{\alpha}$ (hence also $\underline{\mathcal{M}}^{\a,nil}_{\vec{Q}}$) is in general not irreducible while $E_{\a}$ is always a vector space (see Example~2.19).

\vspace{.1in}

The nilpotent orbits of $\vec{Q}_{n-1}$ may be parametrized as follows. It is convenient to identify the set of vertices $I$ with $\Z/n\Z$ (notice the orientation of arrows, going from $i$ to $i-1$ for all $i \in I$). Set $\delta=(1, 1, \ldots, 1) \in \N^I$. There exists a unique nilpotent indecomposable representation $I_{[i;l]}$ of $\vec{Q}_{n-1}$ of length $l$ and socle $\epsilon_i$. It has dimension $\epsilon_{i} + \epsilon_{i-1} + \cdots + \epsilon_{i+1-l}$. These form a complete set of nilpotent indecomposables, as $i \in I$ and $l \in \N$ vary (see \cite[Prop.~3.24]{Trieste}). The set of isoclasses of nilpotent representations of $\vec{Q}_{n-1}$ is identified in this way with the collection of \textit{$n$-multipartitions}
$$\Pi^n=\big\{ (\underline{\lambda}_1, \ldots, \underline{\lambda}_n\; | \; \underline{\lambda}_i = (\lambda_i^1 \geq \lambda_i^2 \geq \cdots ) \;\text{is\;a\;partition}\big\}$$
via
$$(\underline{\lambda}_1, \ldots, \underline{\lambda}_n) \mapsto M_{\underline{\lambda}_1, \ldots, \underline{\lambda}_n} := \bigoplus_{i \in I} \bigoplus_{j} I_{[i;\lambda_i^j]}.$$
We will denote by $\mathcal{O}_{\underline{\lambda}_1, \ldots, \underline{\lambda}_n}$ the orbit of $M_{\underline{\lambda}_1, \ldots, \underline{\lambda}_n}$. 

\vspace{.2in}

 Let us call an $n$-multipartition $(\underline{\lambda}_1, \ldots, \underline{\lambda}_n)$ \textit{aperiodic} if the partitions 
$\underline{\lambda}_1, \ldots, \underline{\lambda}_n$ do \textit{not} share a common part (that is, if for any integer $\lambda$ there exists an $i \in I$ for which $\lambda \not\in \underline{\lambda}_i$). Lusztig proved the following result (see \cite{Lusaff}, \cite{Lusaper})~:

\vspace{.1in}

\begin{theo}[Lusztig]\label{P:Luscyc} If $n >1$ we have
$$\mathcal{P}_{\vec{Q}_{n-1}}=\big\{ IC(\mathcal{O}_{\underline{\lambda}_1, \ldots, \underline{\lambda}_n})\;|\; (\underline{\lambda}_1, \ldots, \underline{\lambda}_n)\;\text{is\;aperiodic}\big\}.$$
\end{theo}

The proof of this theorem, which we won't give, uses the technology of singular supports, and is based on Theorem~\ref{T:SSLU} (see Section~4.5.). 

\vspace{.2in}

\addtocounter{theo}{1}
\noindent \textbf{Example \thetheo .} Set $\delta=\sum_i \epsilon_i$. Let us show directly that $IC(\mathcal{O}_{(1), \ldots, (1)})=\qlb_{\{0\}} \not\in \mathcal{P}^{\delta}$.
For $(i_1, \ldots, i_n)$ a permutation of $(1,2, \ldots, n)$, the fibers of the map 
$$q_{\epsilon_{i_1}, \ldots, \epsilon_{i_n}}: E_{\epsilon_{i_1}, \ldots, \epsilon_{i_n}} \to E_{\delta}$$
are either empty or reduced to a point; in fact $q_{\epsilon_{i_1}, \ldots, \epsilon_{i_n}}$ is the embedding of a smooth subvariety $X_{\epsilon_{i_1}, \ldots, \epsilon_{i_n}}$ of $E_{\delta}$. It is easy to see that $X_{\epsilon_{i_1}, \ldots, \epsilon_{i_n}}$ is closed (it is given by the vanishing of certain arrows). Therefore $L_{\epsilon_{i_1}, \ldots, \epsilon_{i_n}}=IC(X_{\epsilon_{i_1}, \ldots, \epsilon_{i_n}})$. Now observe that because there is no total ordering on the vertices compatible with the arrows, $X_{\epsilon_{i_1}, \ldots, \epsilon_{i_n}} \neq \{0\}$ for all choices of ${\epsilon_{i_1}, \ldots, \epsilon_{i_n}}$. Thus $\qlb_{\{0\}}$ does not belong to $\mathcal{P}^\delta$.

\vspace{.1in}

As an explicit example, take $n=2$. Then $E_{\delta} =k^{\oplus 2}$ while $E^{nil}_{\delta}=\{(x,y)\in k^{\oplus 2}\;|\; xy=0\}$ is the union of the two axes $T_{{x}}=k \times \{0\},  T_y=\{0\} \times k$ in the plane~:

\vspace{.45in}

\centerline{
\begin{picture}(260, 10)
\put(60,0){$E^0_{\delta}=$}
\put(120,0){\line(1,0){60}}
\put(115,5){$T_x$}
\put(155,25){$T_y$}
\put(150,-30){\line(0,1){60}}
\end{picture}}

\vspace{.5in}

The three orbits are $\mathcal{O}_x=T_x\backslash \{0\}, \mathcal{O}_y=T_y\backslash \{0\}$ and $\{0\}$. Thus there are three possible simple $G_{\delta}$-equivariant perverse sheaves on $E^{nil}_{\delta}$ namely $IC(\mathcal{O}_x)= \qlb_{T_x}[1], IC(\mathcal{O}_y)= \qlb_{T_y}[1]$ and $IC(\{0\})=\qlb_{\{0\}}$. But since the Lusztig sheaves are $L_{\epsilon_1, \epsilon_2}=\qlb_{T_x}[1]$ and $L_{\epsilon_2, \epsilon_1}=\qlb_{T_y}[1]$ only these last two perverse sheaves belong to $\mathcal{P}^{\delta}$.
\endexample

\vspace{.2in}

\addtocounter{theo}{1}
\noindent \textbf{Remark \thetheo .} The cases of the Jordan quiver and more generally higher rank cyclic quivers are especially important due to the role they play in representation theory of quantum groups and Hecke algebras of type $A$. We refer the reader who wants to learn more in this direction to \cite{GV}, \cite{V}, \cite{VV}, \cite{Quiversurvey}.

\vspace{.3in}

\centerline{\textbf{2.5. Affine quivers.}}
\addcontentsline{toc}{subsection}{\tocsubsection {}{}{\; 2.5. Affine quivers.}}

\vspace{.15in}

Lusztig obtained in \cite{Lusaff} a precise description of all the elements of $\mathcal{P}_{\vec{Q}}$ for affine quivers equipped with the so-called McKay orientation. This classification was later extended to affine quivers with arbitrary orientation by Y. Li and Z. Lin \cite{LZ}. As in Proposition~\ref{P:Luscyc}, the elements in $\mathcal{P}_{\vec{Q}}$ form a small subset of the set of all possible simple perverse sheaves on the moduli spaces $\Ma$. Before explaining this classification, we need to recall briefly the structure of the category of finite-dimensional representations of an affine quiver (see e.g. \cite[Section~3.6.]{Trieste} ). 

\vspace{.1in}

So let us fix an affine quiver $\vec{Q}$ (which we assume, for simplicity, not to be a cyclic quiver). The Grothendieck group $K_0(\vec{Q})$ equipped with the Cartan form is isomorphic to an affine root lattice. We let 
$\delta$ stand for the indivisible positive imaginary root (see e.g. \cite[App. A] {Trieste}). It may be characterized as the smallest dimension vector lying in the radical of the symmetrized Euler form $(\;,\;)$.

\vspace{.1in}

The main tool to decompose the category $Rep_{k} \vec{Q}$ is given in the following fundamental Theorem~:

\vspace{.1in}

\begin{theo}[Auslander-Reiten]\label{T:AR} There exists a unique pair of adjoint functors $\tau,\tau^-: Rep_k\vec{Q} \to Rep_k\vec{Q}$ equipped with natural isomorphisms
\begin{align*}
{Ext}^1(M,N)^* &\simeq {Hom}(N,\tau M)\\
{Ext}^1(M,N)^* &\simeq {Hom}(\tau^-N, M).
\end{align*}
\end{theo}

\vspace{.1in}

An indecomposable representation $M$ of $\vec{Q}$ is called \textit{preprojective} if $\tau^i M =0$ for $i \gg 0$,
\textit{preinjective} if $\tau^{-i} M =0$ for $i \gg 0$ and \textit{regular} if $\tau^i M \neq 0$ for $i \in \Z$. 
Observe that $\tau M=0$ if and only $M$ is projective and that $\tau^{-} N=0$ if and only if $N$ is injective.
We let $\texttt{P}, \texttt{R}, \texttt{I}$ stand for the set of preprojective, resp. regular, resp. preinjective indecomposable representations. More generally, we will say that a (decomposable) representation $N$ is preprojective, regular or preinjective if all of its indecomposable summands are, and we denote by $\mathbb{P}, \mathbb{R}, \mathbb{I}$ the full subcategories of $Rep_k\vec{Q}$ whose objects are preprojective, regular or preinjective.

\vspace{.1in}

\begin{prop}\label{P:extiszero} The categories $\mathbb{P},\mathbb{I}$ are exact and stable under extensions. The category $\mathbb{R}$ is abelian and stable under extensions. In addition, if $M \in \mathbb{P}, N \in \mathbb{I}$ and $L \in \mathbb{R}$ then
\begin{align}\label{E:preprojeqs}
{Hom}(N,M) &={Hom}(N,L)={Hom}(L,M)=\{0\}\\
{Ext}^1(M,N)&={Ext}^1(L,N)={Ext}^1(M,L)=\{0\}
\end{align}
\end{prop}

\vspace{.1in}
 
This allows us to picture the indecomposables of $Rep_k\vec{Q}$ as follows~:

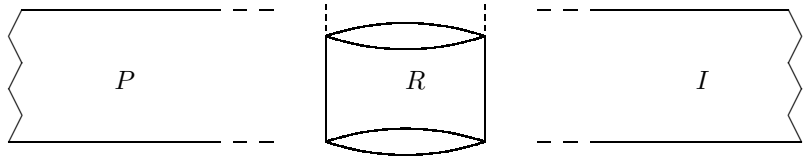
\vspace{.2in}

\centerline{
\begin{picture}(300,60)
\put(0,0){\line(1,0){80}}
\put(220,0){\line(1,0){80}}
\put(5,50){\line(1,0){75}}
\put(220,50){\line(1,0){75}}
\put(85,0){\line(1,0){5}}
\put(95,0){\line(1,0){5}}
\put(85,50){\line(1,0){5}}
\put(95,50){\line(1,0){5}}
\put(200,0){\line(1,0){5}}
\put(210,0){\line(1,0){5}}
\put(200,50){\line(1,0){5}}
\put(210,50){\line(1,0){5}}
\put(120,0){\line(0,1){40}}
\put(180,0){\line(0,1){40}}
\put(120,42){\line(0,1){2}}
\put(120,46){\line(0,1){2}}
\put(120,50){\line(0,1){2}}
\put(180,42){\line(0,1){2}}
\put(180,46){\line(0,1){2}}
\put(180,50){\line(0,1){2}}
\put(0,0){\line(1,2){5}}
\put(5,10){\line(-1,2){5}}
\put(0,20){\line(1,2){5}}
\put(5,30){\line(-1,2){5}}
\put(0,40){\line(1,2){5}}
\put(300,0){\line(-1,2){5}}
\put(295,10){\line(1,2){5}}
\put(300,20){\line(-1,2){5}}
\put(295,30){\line(1,2){5}}
\put(300,40){\line(-1,2){5}}
\put(40,20){${P}$}
\put(150,20){${R}$}
\put(260,20){${I}$}
\qbezier(120,0)(150,-10)(180,0)
\qbezier(120,40)(150,30)(180,40)
\qbezier(120,0)(150,10)(180,0)
\qbezier(120,40)(150,50)(180,40)
\end{picture}}

\vspace{.3in}

\noindent
where we have put the projectives on the extreme left, followed by $\{\tau^- P\;|\; P \;projective\}$, etc..; the injectives are on the extreme right, then $\{\tau I\;|\; I\;injective\}$, etc..; the regular modules are put in the middle. Proposition~\ref{P:extiszero} says that morphisms go from left to right, whereas extensions go from right to left.

\vspace{.1in}

Let $M$ be any representation of $\vec{Q}$. It splits as a direct sum 
\begin{equation}\label{E:deco}
M=M_P \oplus M_{R} \oplus M_I
\end{equation}
where $M_P \in \mathbb{P}, M_I \in \mathbb{I}$, $M_{R} \in \mathbb{R}$. The isomorphism classes of $M_P, M_{R}$ and $M_I$ are of course uniquely determined, but the decomposition (\ref{E:deco}) is not canonical. However, as can easily be deduced from Proposition~\ref{P:extiszero}, the induced filtration 
\begin{equation}\label{E:filtr}
M_I \subset M_{R} \oplus M_{I} \subset M_{P} \oplus M_{R'} \oplus M_I=  M
\end{equation}
\textit{is} unique. 

\vspace{.1in}

It remains for us to describe more precisely the structure of the category of regular modules. We will say that a regular module is simple if it is simple as an object of $\mathbb{R}$. Recall that $\vec{Q}_{p-1}$ denotes the equioriented cyclic quiver of rank $p$. We cite the following two results due to Ringel (see \cite{Ringelbook}) ~:

\vspace{.1in}

\begin{theo}[Ringel] Let $R$ be a regular module. Then
\begin{enumerate}
\item[i)] There exists $p \geq 1$ such that $\tau^p R \simeq R$,
\item[ii)] If $R$ is a simple regular module of $\tau$-order one (i.e. if $\tau R \simeq R$) then the Serre subcategory generated by $R$ is equivalent to the $Rep_k \vec{Q}_0$,
\item[iii)] If $R$ is a simple regular module of $\tau$-order $p>1$ (i.e. if $\tau^p R \simeq R$ but $\tau^{q} R \not\simeq R$ for any $0<q<p$ ) then the Serre subcategory generated by the objects $R, \tau R, \ldots, \tau^{p-1}R$ is equivalent to $Rep_k \vec{Q}_{p-1}$,
\item[iv)] If $R$ is simple of $\tau$-order $p$ then $\underline{dim}\;(R \oplus \tau R \oplus \cdots \tau^{p-1}R) =\delta$.
\end{enumerate}
\end{theo}

\vspace{.1in}

In the above, the Serre category is taken \textit{in $\mathbb{R}$}~: it is the smallest full subcategory of $\mathbb{R}$ containing the given objects which is stable under extensions and subquotients. The Serre subcategories generated by the $\tau$-orbits $\{R, \tau R, \cdots , \tau^{p-1}R\}$ of simple regular modules are called \textit{tubes}. It is customary to call a simple regular module $R$ \textit{homogeneous} if it is of $\tau$-order one; the corresponding tube is called \textit{homogeneous} as well. We will denote by $\mathcal{C}_R$ or $\mathcal{C}_{\mathcal{O}}$ the tube generated by the $\tau$-orbit $\mathcal{O}=\{R, \tau R, \cdots \tau^{p-1}R\}$ of a regular simple module $R$.

\vspace{.1in}

\begin{theo}[Ringel]\label{T:Ringelstructure} Let $d$ and $p_1, \ldots, p_d$ be attached to $\vec{Q}$ as in the table (\ref{E:table7}) below. Then
\begin{enumerate}
\item[i)] There is a natural bijection $R_z \leftrightarrow z$ between the set of homogeneous regular simple modules and points of $\mathbb{P}^1 \backslash D$ where $D$ consists of $d$ points.
\item[ii)] There are exactly $d$ $\tau$-orbits $\mathcal{O}_1, \ldots, \mathcal{O}_d$ of non-homogeneous regular simple modules, and they are of size $p_1, \ldots, p_d$ respectively.
\item[iii)] The whole category $\mathbb{R}$ decomposes as a direct sum of orthogonal blocks
$$\mathbb{R} = \prod_{z \in \mathbb{P}^1\backslash D} \hspace{-.08in}\mathcal{C}_{M_z} \times\; \prod_{l=1, \ldots, d}\hspace{-.08in} \mathcal{C}_{\mathcal{O}_l}.$$
\end{enumerate}
\end{theo}

\vspace{.1in}

We define two subcategories $\mathbb{R}'=\prod_{z} \mathcal{C}_{M_z}$ and $\mathbb{R}''=\prod_{l} \mathcal{C}_{\mathcal{O}_l}$. The numbers $d, p_1, \ldots, p_d$ may be read off from the following table~:

\vspace{.15in}

\begin{equation}\label{E:table7}
\begin{tabular}{|c|c|c|}
\hline
type of\;$\vec{Q}$ & $d$ & $p_1, \ldots, p_d$\\
\hline
$A_1^{(1)}$ & 0 & \\
\hline
$A_n^{(1)}, \;n>1$ & 2 & $p_1=\# \mathrm{arrows\;going\;clockwise}$\\
& & $p_2=\# \mathrm{arrows\;going\;counterclockwise}$\\
\hline
$D_n^{(1)}$ & 3 & $2,2,n-2$\\
\hline
$E_n^{(1)},\; n=6,7,8$ & 3 & $2,3,n-3$\\
\hline
\end{tabular} 
\end{equation}

\vspace{.2in}

We will soon use the above description of $Rep_k\vec{Q}$ to define certain strata in the moduli spaces $\Ma$.  Before that, we state some general results about the geometry of the orbits in $E_{\a}$.

\vspace{.1in}

\begin{lem}\label{L:231} Let $M$ be an indecomposable module which is either preprojective or preinjective. Then $\mathcal{O}_M$ is an open (dense) orbit in $E_{\underline{dim}\;M}$.
\end{lem}
\begin{proof} By Lemma~\ref{L:codim} it is equivalent to show that $Ext^1(M,M)=0$.
By Theorem~\ref{T:AR} we have 
$$Ext^1(\tau M, \tau M) \simeq Ext^1(M,M) \simeq Ext^1(\tau^-M, \tau^-M).$$
If $M$ is preprojective then $M\simeq (\tau^-)^k P$ for some projective $P$ and $Ext^1(M,M)=Ext^1(P,P)=0$, while if $M$ is preinjective then $M \simeq \tau^k I$ for some injective $I$ and $Ext^1(M,M) =Ext^1(I,I)=0$. The lemma is proved.\end{proof}

\vspace{.1in}

It follows from Lemma~\ref{L:231} that there are only finitely many orbits of preprojective (resp. preinjective) modules of any given dimension $\a$. 

\vspace{.1in}

\begin{lem} For any $l \geq 1$ the union of the $G_{l\delta}$-orbits of all regular modules is an open (dense) subset of $E_{l\delta}$.
\end{lem}
\begin{proof} By Proposition~\ref{P:extiszero} a module $M$ of dimension $l\delta$ is regular if and only if for any indecomposable preprojective $P$ and for any indecomposable preinjective $I$, both of dimension at most $l\delta$ we have
$$Hom(M,P)=Hom(I,M)=0.$$
But the functions $M \mapsto dim\;Hom(M,P)$ and $M \mapsto dim\;Hom(I,M)$ are upper semicontinuous, which implies that 
$$\{\ux \in E_{l\delta}\;|\; dim\;Hom(M_{\ux},P) + dim\;Hom(I,M_{\ux}) >0\}$$
is closed. By Lemma~\ref{L:231} there is at most one preprojective or preinjective indecomposable in each dimension. Hence the number of choices for $P$ or $I$ above is finite. We are done (notice that Ringel's Theorem~\ref{T:Ringelstructure} states that the set of regular simples is nonempty). \end{proof}

\vspace{.1in}

We let $E_{l\delta}^{\mathbb{R}'}$ be the open subset of $E_{l\delta}$ consisting of orbits of regular \textit{homogeneous} modules. For any $l \geq 1$, we define an open (dense) stratum 
$$U_{l \delta}=\big\{ \ux \in E_{l\delta}\;|\; M_{\ux} \simeq R_{z_1}\oplus \cdots \oplus R_{z_l},\quad \;z_i \neq z_j\;if\;i \neq j\big\}$$
of $E_{l\delta}$ (and $E_{l\delta}^{\mathbb{R}'}$). Here the points $z$ are taken in $\mathbb{P}^1\backslash D$ as in Ringel's Theorem~\ref{T:Ringelstructure} and thus all the $R_z$ are simple homogeneous (as well as mutually orthogonal).

\vspace{.1in}

We can now give the desired stratification of $E_{\a}$.  By Proposition~\ref{P:extiszero} and Theorem~\ref{T:Ringelstructure} we are in the situation of Remark~2.13.

Let us call \textit{stratum data} a tuple $\underline{A}=(P, l, N_1, \ldots, N_d, I)$ where ~: $P$ is a preprojective module; $l \in \N$; $N_1, \ldots, N_d$ are modules in $\mathcal{C}_1, \ldots, \mathcal{C}_d$; $I$ is a preinjective module. In the above, we allow any of the modules to be zero. We also set $l({\underline{A}})=l$.  The collection of all stratum data will be denoted by $\mathcal{S}$. We define the dimension of $\underline{A}$ by 
$$\underline{dim}\; \underline{A}=\underline{dim}\;P + l\delta + \sum_k \underline{dim}\;N_k + \underline{dim}\;I.$$
To a stratum data $\underline{A}$ of dimension $\a$ corresponds a subset of $E_{\a}$
$$S_{\underline{A}}=\big\{ \underline{x} \in E_{\a}\;|\;M_{\ux} \simeq P \oplus R' \oplus N_1 \oplus \cdots \oplus N_d \oplus I; \; R' \in \mathbb{R}', \underline{dim}\;R' =l\delta\big\}$$
and by construction
$$E_{\a} =\bigsqcup_{\underline{A}} S_{\underline{A}}.$$
We also introduce some open piece of the strata $S_{\underline{A}}$ 
$$S^{\circ}_{\underline{A}}=\big\{ \underline{x}\;|\;M_{\ux} \simeq P \oplus R' \oplus N_1 \oplus \cdots \oplus N_d \oplus I; \; R' \in U_{l\delta} \big\}.$$
As in Remark~2.13 (see \ref{E:FGP2}) we have
$$S_{\underline{A}} \simeq F_{\underline{A}} \underset{P}{\times} G_{\a}, \qquad S^{\circ}_{\underline{A}} \simeq F^{\circ}_{\underline{A}} \underset{P}{\times} G_{\a}$$
for some vector bundles
\begin{equation}\label{E:sasa'}
\begin{split}
\kappa:  F_{\underline{A}} &\to \mathcal{O}_P \times E_{l\delta}^{\mathbb{R}'} \times \mathcal{O}_{N_1 \oplus \cdots \oplus N_d} \times \mathcal{O}_I, \\
\kappa^{\circ}:  F^{\circ}_{\underline{A}} &\to \mathcal{O}_P \times U_{l\delta} \times \mathcal{O}_{N_1 \oplus \cdots \oplus N_d} \times \mathcal{O}_I.
\end{split}
\end{equation}
In particular, $S_{\underline{A}}, S^{\circ}_{\underline{A}}$ are smooth and locally closed.

\vspace{.1in}

Now that we have the strata, the last step is to define some suitable local systems over them.
As soon as $l >1$ the variety $U_{l\delta}$ carry nontrivial local systems constructed as follows. Define a covering $\widetilde{U}_{l\delta}$ of $U_{l\delta}$ as
$$\widetilde{U}_{l\delta}=\big\{ (\ux, \omega)\;|\; \ux \in U_{l\delta}, \,M_{\ux} \simeq R_{z_1} \oplus \cdots \oplus R_{z_l}; \;\omega: \{ R_{z_1}, \ldots, R_{z_n}\} \stackrel{\sim}{\to} \{1, \ldots, l\}\big\}.$$
It is clear that the projection $\pi: \widetilde{U}_{l\delta} \to U_{l\delta}$ is an $\mathfrak{S}_l$ Galois cover. The local system $\pi_!(\qlb_{\widetilde{U}_{l\delta}})$ carries a fiberwise action of $\mathfrak{S}_l$ and breaks into isotypical components
\begin{equation}\label{E:2loc}
\pi_!(\qlb_{\widetilde{U}_{l\delta}}) = \bigoplus_{\chi \in Irr\;\mathfrak{S}_l} \mathcal{L}_{\chi}^{\oplus dim\;\chi}.
\end{equation}
Heuristically, we may identify the stack $\underline{U}_{l\delta}=U_{l\delta}/G_{l\delta}$ with the configuration space $S^l(\mathbb{P}^1 \backslash D)\backslash \Delta$ of $l$ \textit{distinct} points in $\mathbb{P}^1 \backslash D$. The local systems $\mathcal{L}_{\chi}$ are obtained as the pullbacks of the local systems $\mathcal{L}'_{\chi}$ over $S^l(\mathbb{P}^1 \backslash D)\backslash \Delta$ constructed from the cover
$(\mathbb{P}^1 \backslash D)^l\backslash \Delta \to S^l(\mathbb{P}^1 \backslash D)\backslash \Delta$.

We will also denote by $\mathcal{L}_{\chi}$ the $G_{\a}$-equivariant local system induced on $S^{\circ}_{\underline{A}}$ via (\ref{E:sasa'}).

\vspace{.1in}

We are finally ready to describe the classification of the simple objects in $\mathcal{P}_{\vec{Q}}$ for an affine quiver $\vec{Q}$. Call a strata datum $\underline{A}=(P,l,N_1, \ldots, N_d, I)$ \textit{aperiodic} if all the orbits of $N_i$ are aperiodic in the sense of Section~2.4. We denote the set of aperiodic stratum data by $\mathcal{S}_{aper}$.

\vspace{.1in}

\begin{theo}[Lusztig, Li-Lin] We have
$$\mathcal{P}_{\vec{Q}}=\big\{IC(S^{\circ}_{\underline{A}},\mathcal{L}_{\chi})\;|\; \underline{A} \in \mathcal{S}_{aper},\; \chi \in Irr\;\mathfrak{S}_{l(\underline{A})}\big\}.$$
\end{theo}

\vspace{.1in}

Observe again that all the elements of $\mathcal{P}_{\vec{Q}}$ are self-dual (because any representation of a symmetric group is self dual). We refer the reader to the original papers for the full proof. We will just sketch the strategy here. One begins by showing that $\mathcal{P}_{\vec{Q}}$ contains all simple perverse sheaves of the form $IC(\mathcal{O}_P)$, $IC(\mathcal{O}_I)$, $IC(U_{l\delta}, \mathcal{L}_{\chi})$ and $IC(\mathcal{O}_{N_i})$ (for $\mathcal{O}_{N_i}$ aperiodic). For the first two cases, the argument is the same as in the finite type case (using Lemma~\ref{L:231}). For $IC(U_{l\delta},\mathcal{L}_{\chi})$ this results from a direct computation--see Example~2.28. The last case is more delicate. Next, for an aperiodic stratum data $\underline{A}=(P,l,N_1, \ldots, N_d, I)$ and $\chi \in Irr\;\mathfrak{S}_{l(\underline{A})}$ we use Remark~2.13, ii) to get
$$IC(\mathcal{O}_P) \star IC(U_{l\delta},\mathcal{L}_{\chi}) \star IC(\mathcal{O}_{N_1}) \star \cdots \star IC(\mathcal{O}_{N_d}) \star IC(\mathcal{O}_I) =IC(S^{\circ}_{\underline{A}},\mathcal{L}_{\chi}) \oplus \mathbb{T}$$
where $supp\;\mathbb{T} \in \overline{S_{\underline{A}}} \backslash S^{\circ}_{\underline{A}}$. This proves the inclusion ''$\supset$'' of the Theorem. To prove the opposite inclusion one first shows that any $\mathbb{P} \in \mathcal{P}_{\vec{Q}}$ is supported on $ \overline{S_{\underline{A}}}$ for some strata $\underline{A}$ and use the restriction functor as in Remark~2.13 iii).

\vspace{.2in}

\addtocounter{theo}{1}
\noindent \textbf{Example \thetheo .} Here is the smallest example for which a  nontrivial local system appears. Let $\vec{Q}$ be the Kronecker quiver (see \cite[Example~3.34]{Trieste} ) with two vertices $1,2$ and two arrows $x,y: 1 \to 2$~:

\centerline{
\begin{picture}(130, 50)
\put(30,20){\circle*{5}}
\put(100,20){\circle*{5}}
\put(27,10){$1$}
\put(97,10){$2$}
\put(55,28){$x$}
\put(55,8){$y$}
\put(65,25){\vector(1,0){5}}
\put(40,25){\line(1,0){50}}
\put(65,15){\vector(1,0){5}}
\put(40,15){\line(1,0){50}}
\put(-10,18){$\vec{Q}=$}
\end{picture}}

\vspace{.1in}

Let us first have a look at the dimension vector $\delta=\epsilon_1+\epsilon_2$ (the indivisible imaginary root). Then $E_{\delta}=Hom(k,k)^{\oplus 2} = k^2$. The group is $G_{\delta}=k^* \times k^*$. There is a unique closed orbit $\mathcal{O}_{(0,0)}$ (which corresponds to a direct sum of a preprojective and a preinjective module) and a $\mathbb{P}^1$-family of orbits $\mathcal{O}_{R_z}=G_{\delta} \cdot R_z$ corresponding to the family of regular modules 
$$R_z=(\lambda,\mu), \; \;\;\; z=(\lambda,\mu) \in \mathbb{P}^1(k).$$
Hence there are two strata here, $S_1=\{(0,0)\}$ and $S_2=\bigcup_{z} \mathcal{O}_{R_z}$. One can check directly that 
$$\mathcal{P}^{\delta}=\big\{ IC(S_1)=\qlb_{\{0,0)\}}, IC(S_2)=\qlb_{E_{\delta}}[2]\big\}.$$

\vspace{.1in}

We now consider the dimension vector $2\delta=2\epsilon_1+2\epsilon_2$. We have 
$$E_{2\delta}=Hom(k^2, k^2)^{ \oplus 2} =k^8$$
and the group is $G_{2\delta}=GL(2) \times GL(2)$. Let us briefly give the classification of representations; there are (unique) preprojective modules $P_2, P_{1,2^2}$ of dimension $\epsilon_2$ and $\epsilon_2+2\epsilon_2$ respectively; there are (unique) preinjective modules $I_1, I_{1^2,2}$ of respective dimensions $\epsilon_1, 2\epsilon_1+\epsilon_2$; besides the regular modules $R_{z}$ for $z \in \mathbb{P}^1$, which are of dimension $\delta$, there are regular modules $R_z^{(2)}$ for $z \in \mathbb{P}^1$ which are self-extensions of $R_z$ and hence of dimension $2\delta$; the representations of dimension $2\delta$ are obtained by combining the above indecomposables, namely
\begin{equation}\label{E:classind}
\big\{ I_1^{\oplus 2} \oplus P_2^{\oplus 2},\; I_{1^2,1} \oplus P_2, \;I_{1} \oplus P_{1,2^2}, \;I_1 \oplus P_2 \oplus R_z, \;R_{z_1} \oplus R_{z_2}, \;R_z^{(2)}\big\}.
\end{equation}
Only the last two families of representations are regular. 

\vspace{.1in}

Let us now list the various strata. It is easy to check that the regular locus is the open subset of $E_{2\delta}$ given by the conditions
$$S_{0,2,0}=E_{2\delta}^{\mathbb{R}}=\big\{ (x,y) \in Hom(k^2,k^2)^{\oplus 2}\;|\; Ker\;x \cap Ker\;y=\{0\};\; Im\;x + Im\;y =k^2\big\}.$$
Inside this regular locus the open set $U_{2\delta}$ is the complement of the divisor
$$\bigcup_{z \in \mathbb{P}^1}  \big( G_{2\delta} \cdot R_z^{(2)} \cup G_{2\delta}\cdot R_z^{\oplus 2} \big).$$
The other strata are (using the notation of Section~2.5)~:
$$S_{P_2, 1, I_{ 1}}=\bigcup_{z \in \mathbb{P}^1} G_{2\delta} \cdot (I_1 \oplus R_z \oplus P_2),$$
$$S_{P_{1,2^2}, 0, I_{1}}=G_{2\delta} \cdot (I_{1} \oplus P_{1,2^2}), \qquad S_{P_{2}, 0, I_{1^2,2}}=G_{2\delta} \cdot (I_{1^2,2} \oplus P_{2}),$$ 
and the closed strata
$$S_{P_2^{\oplus 2}, I_1^{\oplus 2}}=G_{2\delta} \cdot (I_1^{\oplus 2} \oplus P_2^{\oplus 2}).$$
These strata are of respective dimensions~$4,5,5$ and $0$.
 
\vspace{.1in}
 
The open strata $U_{2\delta}=S_{0,2,0}^{\circ}$ may be explicitly described. There is a natural map $U_{\delta} \to \mathbb{P}^1$ (essentially the quotient map by $G_{\delta}$) and hence also a map $\psi: U_{\delta} \times U_{\delta} \to \mathbb{P}^1 \times \mathbb{P}^1$. This map is $G_{\delta} \times G_{\delta}$-equivariant, where $G_{\delta} \times G_{\delta}$ acts trivially on $\mathbb{P}^1 \times \mathbb{P}^1$. Put $(U_{\delta} \times U_{\delta})^0=\psi^{-1} (\mathbb{P}^1 \times \mathbb{P}^1 \backslash \Delta)$ where $\Delta =\{ (z,z)\in \mathbb{P}^1 \times \mathbb{P}^1\}$ is the diagonal. Let us also fix a splitting $V_{2\delta} =V_{\delta} \oplus V_{\delta}$. This induces an embedding $G_{\delta} \times G_{\delta}\subset G_{2\delta}$. Let $N \subset G_{2\delta}$ be the subgroup generated by $G_{\delta} \times G_{\delta}$ and the permutation $\sigma \in Aut(V_{\delta} \oplus V_{\delta}),\; (a,b) \mapsto (b,a)$. The group $N$ normalizes $G_{\delta} \times G_{\delta}$, contains it with index $2$ and acts on $U_{\delta} \times U_{\delta}$ via $\sigma (u,v)=(v,u)$. We now have an $N$-equivariant map $\psi': (U_{\delta} \times U_{\delta}^0 \to S^2 \mathbb{P}^1 \backslash \Delta$ and an identification
$$\psi'': U_{2\delta} \simeq G_{2\delta}\underset{N}{\times} (U_{\delta} \times U_\delta)^0 \to S^2 \mathbb{P}^1\backslash \Delta.$$
The morphism $\psi''$ is $G_{2\delta}$-equivariant and may be thought of as a quotient map.

\vspace{.1in}

To exhibit the desired local system, we consider the restriction of the Lusztig sheaf $L_{\delta,\delta}$ to $U_{2\delta}$. By definition, $(L_{\delta,\delta})_{| R_{z_1} \oplus R_{z_2}} =H^*(Gr_{\delta}(R_{z_1} \oplus R_{z_2}),\qlb)$ with
$$Gr_{\delta}(R_{z_1} \oplus R_{z_2})=\big\{ M \subset R_{z_1} \oplus R_{z_2}\;|\; \underline{dim}\;M=\delta \big\}.$$
Since $z_1 \neq z_2$, $Gr_{\delta}(R_{z_1} \oplus R_{z_2})$ consists of two points and $L_{\delta,\delta}$ restricts to a rank two local system over $U_{2\delta}$. It is clear that $(L_{\delta,\delta})_{U_{2\delta}}=(\psi'')^*(\mathcal{L}^{\mathbb{P}^1}_{reg})$, where $\mathcal{L}^{\mathbb{P}^1}_{reg}=\pi_!(\qlb)$ with $\pi: \mathbb{P}^1 \times \mathbb{P}^1 \backslash \Delta \to S^2\mathbb{P}^1 \backslash \Delta$ being the canonical projection. One has
$\pi_!(\qlb)=\mathcal{L}^{\mathbb{P}^1}_{triv} \oplus \mathcal{L}^{\mathbb{P}^1}_{sign}$ where $\mathcal{L}^{\mathbb{P}^1}_{triv} , \mathcal{L}^{\mathbb{P}^1}_{sign}$are the rank one local systems corresponding to the trivial and sign representations of the braid group $B_2=\pi_1(S^2\mathbb{P}^1 \backslash \Delta)$. Note that these representations in fact come from the symmetric group $\mathfrak{S}_2$ via the quotient $B_2 \tto \mathfrak{S}_2$. In conclusion, we get
$$(L_{\delta,\delta})_{|U_{2\delta}} =\mathcal{L}_{triv} \oplus \mathcal{L}_{sign}.$$
In particular, $IC(U_{2\delta}, \mathcal{L}_{triv}), IC(U_{2\delta},\mathcal{L}_{sign})$ belong to $\mathcal{P}_{2\delta}$. The other elements are easier to determine and one obtains in the end~:
\begin{equation}
\begin{split}
\mathcal{P}_{2\delta}=\big\{ &IC(S_{P_2^{\oplus 2}, I_1^{\oplus 2}}), \; IC(S_{P_{2}, 0, I_{1^2,2}}),\; IC(S_{P_{1,2^2}, 0, I_{1}}),\; IC(S_{P_2, 1, I_{ 1}}),\\
&\;IC(U_{2\delta}, \mathcal{L}_{triv}), \;IC(U_{2\delta},\mathcal{L}_{sign})
\big\}.
\end{split}
\end{equation}
\endexample

\vspace{.2in}

\newpage

\centerline{\large{\textbf{Lecture~3.}}}
\addcontentsline{toc}{section}{\tocsection {}{}{Lecture~3.}}

\setcounter{section}{3}
\setcounter{theo}{0}
\setcounter{equation}{0}

\vspace{.2in}

In this Lecture, we explain the relationship between the Hall category $\QQ$ of an arbitrary quiver $\vec{Q}$ and the quantum enveloping algebra of the Kac-Moody Lie algebra naturally associated to $\vec{Q}$. Namely we show that the Grothendieck group $\mathcal{K}_{\vec{Q}}$ of $\QQ$ is naturally isomorphic to $\U^{\Z}_{\v}(\mathfrak{n}_+)$. The quantum variable $\v$ comes from the grading on $\mathcal{K}_{\vec{Q}}$ corresponding to the shift $\mathbb{P} \to \mathbb{P}[1]$ (i.e. from the cohomological grading). This result was proved by Lusztig in the early 90's. Its most important corollary is the existence and construction of the \textit{canonical basis} $\mathbf{B}$ of $\mathbf{U}^{\Z}_\v(\mathfrak{n}_+)$, which is formed by the classes of the simple perverse sheaves in $\mathcal{P}_{\vec{Q}}$.

\vspace{.1in}

We will begin by giving an elementary proof of Lusztig's theorem for finite type quivers. This proof involves the explicit classification of Section~2.3 (there is an analogous ``hands on'' approach in the case of affine (noncyclic) quivers but we won't detail it, see \cite{LZ}). The proof in the general case requires more work, and hinges in a crucial way on the use of a Fourier-Deligne transform (which, roughly speaking, allows to change the orientation
of the quiver at any time). This ingredient is also necessary to show that the canonical basis $\mathbf{B}$ of $\U^{\Z}_\v(\n_+)$ is independent of the particular choice of the quiver to construct it. Our treatment of Sections~3.4 and 3.5 borrows heavily from \cite{Lusbook}.

\vspace{.1in}

In the last part of this Lecture we introduce and study the (Frobenius)  trace map, which associates to a given complex $\mathbb{P} \in \QQ$ a constructible function on the moduli space $\underline{\mathcal{M}}_{\vec{Q}}(\mathbb{F}_q)$ of representations of the quiver \textit{over the finite field} $\mathbb{F}_q$. This provides a direct link with the \textit{Hall algebra} $\mathbf{H}_{\vec{Q}}$ as it is defined in \cite{Trieste}. There are various subtle issues to be dealt with here (such as the existence of a Weil structure on the objects of $\QQ$, questions of purity, etc...). We collect these --mostly without proof-- in the last section. The reference for that Section is \cite{LusHall}.

\vspace{.3in}

\centerline{\textbf{3.1. The graded Grothendieck group of the Hall category.}}
\addcontentsline{toc}{subsection}{\tocsubsection {}{}{\; 3.1. The graded Grothendieck group of the Hall category.}}

\vspace{.15in}

Let $\vec{Q}$ be any quiver as in Section~1.1 and let $\QQ=\bigsqcup_{\gamma} \mathcal{Q}^{\gamma}$ and $\mathcal{P}_{\vec{Q}}=\bigsqcup_{\gamma} \mathcal{P}^{\gamma}$ be the associated Hall category and set of simple objects. We denote by $\mathcal{K}_{\vec{Q}}=\bigoplus_{\gamma} \mathcal{K}^{\gamma}$ the Grothendieck group of $\QQ$. The class of an object $\mathbb{P}$ of $\QQ$ will be denoted by $\mathbf{b}_{\mathbb{P}}$. We equip $\mathcal{K}_{\vec{Q}}$ with the structure of a $\Z[v,v^{-1}]$-module by $v^n \mathbf{b}_{\mathbb{P}}=\mathbf{b}_{\mathbb{P}[n]}$. By construction, 
$$\mathcal{K}^{\gamma}=\bigoplus_{\substack{\mathbb{P} \in \mathcal{P}^{\gamma}\\ n \in \Z}} \Z \mathbf{b}_{\mathbb{P}[n]}=\bigoplus_{\mathbb{P} \in \mathcal{P}^{\gamma}} \Z[v,v^{-1}] \mathbf{b}_{\mathbb{P}}.$$
In particular, $\mathcal{K}_{\vec{Q}}$ is a free $\Z[v,v^{-1}]$-module. 

\vspace{.1in}

The induction and restriction functors
\begin{align*}
\underline{m}_{\a,\beta}:~& \mathcal{Q}^{\a} \times \mathcal{Q}^{\beta} \to \mathcal{Q}^{\a+\beta},\\
\underline{\Delta}_{\a,\beta}:~& \mathcal{Q}^{\a+\beta} \to \mathcal{Q}^{\a} \times \mathcal{Q}^{\beta}
\end{align*}
are biadditive so they give rise to bilinear maps
\begin{align*}
{m}_{\a,\beta}:~& \mathcal{K}^{\a} \times \mathcal{K}^{\beta} \to \mathcal{K}^{\a+\beta},\\
{\Delta}_{\a,\beta}:~& \mathcal{K}^{\a+\beta} \to \mathcal{K}^{\a} \times \mathcal{K}^{\beta}
\end{align*}
and hence also to graded bilinear operations $m: \mathcal{K}_{\vec{Q}} \otimes \mathcal{K}_{\vec{Q}} \to \mathcal{K}_{\vec{Q}}$ and $\Delta: \mathcal{K}_{\vec{Q}} \to \mathcal{K}_{\vec{Q}} \otimes \mathcal{K}_{\vec{Q}}$.

\vspace{.1in}

\begin{prop}\label{P:assoc2} The maps $m$ and $\Delta$ endow $\mathcal{K}_{\vec{Q}}$ with the structure of an associative algebra and a coassociative coalgebra respectively.
\end{prop}
\begin{proof} The follows directly from Proposition~\ref{P:assoc}.\end{proof}

\vspace{.1in}

We will sometimes wite $a \cdot b$ instead of $m(a,b)$. We will later prove that $m$ and $\Delta$ satisfy some compatibility relation turning $\mathcal{K}_{\vec{Q}}$ into a \textit{twisted} bialgebra, see Corollary~\ref{C:bialg}. For the moment we have to be content with Proposition~\ref{P:assoc2}. 

\vspace{.1in}

Let us continue to translate the results of Lecture~1 into algebraic properties of $\mathcal{K}_{\vec{Q}}$. Let $u \mapsto \overline{u}$ be the semilinear endomorphism of $\mathcal{K}_{\vec{Q}}$ defined by $\overline{v}=v^{-1}$ and
$\overline{\mathbf{b}_{\mathbb{P}}}=\mathbf{b}_{D\mathbb{P}}.$

\vspace{.1in}

\begin{prop}\label{P:Verdier} The map $u \mapsto \overline{u}$ is a ring involution, i.e. $\overline{a} \cdot \overline{b}=\overline{a \cdot b}$.
\end{prop}
\begin{proof} This is Lemma~\ref{L:Verdierind}.\end{proof}

\vspace{.1in}

We will give a proof of the following result in Section~3.5~:

\vspace{.1in}

\begin{prop}[Lusztig]\label{P:LU} All the simple perverse sheaves in $\PQ$ are self-dual.
\end{prop}

As a consequence, we have 
\begin{equation}\label{E:selfdual}
\overline{\mathbf{b}_{\mathbb{P}}}=\mathbf{b}_{\mathbb{P}}
\end{equation}
for all $\mathbb{P} \in \PQ$.

\vspace{.1in}

Let $\{\,,\,\}: \mathcal{K}_{\vec{Q}} \otimes \mathcal{K}_{\vec{Q}} \to \Z((v))$ be the pairing defined by $\{\mathbf{b}_{\mathbb{P}}, \mathbf{b}_{\mathbb{Q}}\}=\{\mathbb{P}, \mathbb{Q}\}$ (see Section~1.5.). By (\ref{E:Obvious}), $\{\,,\,\}$ is a well-defined $\Z[v,v^{-1}]$-linear pairing and it is nondegenerate by Corollary~\ref{C:scalarprod}. It satisfies (see Example 2.1.)
\begin{equation}\label{E:22n}
\{ \mathbf{b}_{\mathbbm{1}_{\epsilon_i}},  \mathbf{b}_{\mathbbm{1}_{\epsilon_j}}\}=\frac{\delta_{i,j}}{1-v^2}.
\end{equation}

Moreover,

\vspace{.1in}

\begin{prop}\label{P:Hopfpairing} The pairing $\{\,,\,\}$ is a Hopf pairing, namely $\{a\cdot b, c\}=\{a \otimes b, c\}$ for all $a,b,c \in \mathcal{K}_{\vec{Q}}$.\end{prop}
\begin{proof} See Proposition~\ref{P:Hopfscalar}.\end{proof}

\vspace{.1in}

The $\Z[v,v^{-1}]$-basis $\{\mathbf{b}_{\mathbb{P}}\;|\; \mathbb{P} \in \mathcal{P}_{\vec{Q}}\}$ of $\mathcal{K}_{\vec{Q}}$ enjoys, by its very definition, a number of positivity properties. For any $\mathbb{P}, \mathbb{P}'$ in $\mathcal{P}_{\vec{Q}}$ we have
\begin{equation}\label{E:Posmult}
\mathbf{b}_{\mathbb{P}} \cdot \mathbf{b}_{\mathbb{P}'} \in \bigoplus_{\mathbb{Q} \in \mathcal{P}_{\vec{Q}}} \N[v,v^{-1}] \mathbf{b}_{\mathbb{Q}},
\end{equation}
\begin{equation}\label{E:Poscop}
\Delta(\mathbf{b}_{\mathbb{P}} ) \in \bigoplus_{\mathbb{Q}, \mathbb{Q}' \in \mathcal{P}_{\vec{Q}}} \N[v,v^{-1}] \mathbf{b}_{\mathbb{Q}} \otimes \mathbf{b}_{\mathbb{Q}'},
\end{equation}
\begin{equation}\label{E:Posscal}
\{\mathbf{b}_{\mathbb{P}},\mathbf{b}_{\mathbb{P}'}\} \in \N((v)).
\end{equation}

\vspace{.15in}

In order to make the link with quantum groups more transparent in the next Section, we slightly extend the algebra $\mathcal{K}_{\vec{Q}}$ by adding an extra ``Cartan'' piece. Let $\mathbf{K}=\Z[v,v^{-1}] [K_0(\vec{Q})]=\bigoplus_{\a \in \Z^I} \Z[v,v^{-1}] \mathbf{k}_{\a}$ and set
$$\widetilde{\mathcal{K}}_{\vec{Q}} =\mathcal{K}_{\vec{Q}} \otimes \mathbf{K}.$$
We extend the algebra structure on $\mathcal{K}_{\vec{Q}}$ to $\widetilde{\mathcal{K}}_{\vec{Q}}$ by imposing
\begin{equation}\label{E:drtyu}
\begin{split}
\mathbf{k}_{\a} \mathbf{k}_{\beta}&=\mathbf{k}_{\a+\beta},\\
\mathbf{k}_{\a} u_{\gamma} \mathbf{k}^{-1}_{\a}=&v^{-(\a,\gamma)} u_{\gamma}, \qquad \forall\; u_{\gamma} \in \mathcal{K}^{\gamma},
\end{split}
\end{equation}
and we define a coalgebra structure $\widetilde{\Delta}: \widetilde{\mathcal{K}}_{\vec{Q}} \to \widetilde{\mathcal{K}}_{\vec{Q}} \otimes \widetilde{\mathcal{K}}_{\vec{Q}}$ via the formulas
\begin{equation}\label{E:DeltaK}
\begin{split}
\widetilde{\Delta}(\mathbf{k}_{\a})&=\mathbf{k}_{\a} \otimes \mathbf{k}_{\a},\\
\widetilde{\Delta} (u_{\gamma})=\sum_{\a+\beta=\gamma} &\Delta_{\a,\beta}(u_{\gamma}) \cdot (\mathbf{k}_{\beta} \otimes 1),\\
\widetilde{\Delta}(u_{\gamma} \mathbf{k}_{\beta})=&\widetilde{\Delta}(u_{\gamma}) \cdot \widetilde{\Delta}(\mathbf{k}_{\beta}).
\end{split}
\end{equation}
In the above, the product on $\widetilde{\mathcal{K}}_{\vec{Q}} \otimes \widetilde{\mathcal{K}}_{\vec{Q}}$ is simply $(a \otimes b) \cdot (c \otimes d)=a \cdot c \otimes b \cdot d$. One easily checks that $\widetilde{\mathcal{K}}_{\vec{Q}}$ is still a (co)associative (co)algebra.

We also extend the scalar product $\{\;,\;\}$ to $\widetilde{\mathcal{K}}_{\vec{Q}}$ by setting
\begin{equation}\label{E:22new}
\{\mathbf{b}_{\mathbb{P}}\mathbf{k}_{\a},\mathbf{b}_{\mathbb{Q}}\mathbf{k}_{\beta}\}=\{\mathbf{b}_{\mathbb{P}},\mathbf{b}_{\mathbb{Q}}\} v^{-(\a,\beta)}.
\end{equation}
It is still a Hopf pairing. Finally, we define an extension to $\widetilde{\mathcal{K}}_{\vec{Q}}$ of the bar involution $u \mapsto \bar{u}$ by imposing
$\overline{\mathbf{k}_\a}=\mathbf{k}_{-\a}$.

\vspace{.2in}

\addtocounter{theo}{1}
\noindent \textbf{Example \thetheo .} Let $\vec{Q}$ be any quiver and $i \in I$ any vertex of $\vec{Q}$. Then, by Example~2.1, (\ref{E:n1}) we have in $\mathcal{K}_{\vec{Q}}$~:
\begin{equation}\label{E:n12}
(\mathbf{b}_{\mathbbm{1}_{\epsilon_i}})^n=[n]! \mathbf{b}_{\mathbbm{1}_{n\epsilon_i}}
\end{equation}
where $[n]!$ is the $v$-factorial number. Also
\begin{equation}\label{E:Deltageom}
\begin{split}
\Delta(\mathbf{b}_{\mathbbm{1}_{\epsilon_i}})&=\mathbf{b}_{\mathbbm{1}_{\epsilon_i}} \otimes 1 + 1 \otimes \mathbf{b}_{\mathbbm{1}_{\epsilon_i}},\\
\widetilde{\Delta}(\mathbf{b}_{\mathbbm{1}_{\epsilon_i}})&=\mathbf{b}_{\mathbbm{1}_{\epsilon_i}} \otimes 1 + \mathbf{k}_{\epsilon_i} \otimes \mathbf{b}_{\mathbbm{1}_{\epsilon_i}}.
\end{split}
\end{equation}
Now let $\gamma \in \N^I$ be a dimension vector and assume that $\mathbbm{1}_{\a}$ belongs to $\mathcal{P}_{\vec{Q}}$ (if the quiver $\vec{Q}$ has no oriented cycles then this is automatic--see Example~2.5). Then by Lemma~\ref{L:coprodun}
\begin{equation}\label{E:geomcorpod}
\begin{split}
\Delta(\mathbf{b}_{\mathbbm{1}_{\gamma}})&=\sum_{\a+\beta=\gamma} v^{\langle \a,\beta \rangle} \mathbf{b}_{\mathbbm{1}_{\a}} \otimes  \mathbf{b}_{\mathbbm{1}_{\beta}},\\
\widetilde{\Delta}(\mathbf{b}_{\mathbbm{1}_{\gamma}})&=\sum_{\a+\beta=\gamma} v^{\langle \a,\beta \rangle} \mathbf{b}_{\mathbbm{1}_{\a}}\mathbf{k}_{\beta} \otimes  \mathbf{b}_{\mathbbm{1}_{\beta}}.
\end{split}
\end{equation}
To finish, let us translate the fundamental relations of Section~2.1. If $i,j \in I$ are two vertices of a quiver $\vec{Q}$ linked by $r$ arrows ($r \geq 0$) then
\begin{equation}\label{E:FundRel1}
\sum_{k=0}^{1+r} (-1)^k \begin{bmatrix} 1+r \\ k \end{bmatrix} \mathbf{b}_{\mathbbm{1}_{\epsilon_i}}^k \mathbf{b}_{\mathbbm{1}_{\epsilon_j}}\mathbf{b}_{\mathbbm{1}_{\epsilon_i}}^{r+1-k}=0
\end{equation}
which may be rewritten
\begin{equation}\label{E:FundRel2}
\sum_{k=0}^{1+r} (-1)^k  \mathbf{b}_{\mathbbm{1}_{k\epsilon_i}} \mathbf{b}_{\mathbbm{1}_{\epsilon_j}}\mathbf{b}_{\mathbbm{1}_{(r+1-k)\epsilon_i}}=0.
\end{equation}
\endexample

\vspace{.3in}

\centerline{\textbf{3.2. Relation to quantum groups.}}
\addcontentsline{toc}{subsection}{\tocsubsection {}{}{\; 3.2. Relation to quantum groups.}}

\vspace{.15in}

Let $\g'$ be the derived Kac-Moody algebra associated to the Dynkin diagram underlying $\vec{Q}$ (see e.g. \cite[App. A]{Trieste}  for definitions). Recall that we have denoted by $A=(a_{ij})_{i,j \in I}$ the Cartan matrix of $\vec{Q}$ (or $\g'$). We will identify the root lattice $Q_{\g'}$ of $\g'$ with $K_0(\vec{Q})$ by mapping the simple roots $\a_i$ to the dimension vectors $\epsilon_i$.

We briefly recall here the definitions of the relevant quantum groups. Let us fix a decomposition $\g'=\n_- \oplus \h \oplus \n_+$ and let $\bo'_+=\h \oplus \n_+$ be the positive Borel subalgebra. Let $\U_{\v}(\n_+), \U_{\v}(\bo'_+)$ denote the quantized enveloping algebras of $\n_+$ and $\bo'_+$ respectively. The algebra $\U_{\v}(\bo'_+)$ is generated by elements $K_i^{\pm 1}, E_i, i \in I$ satisfying the following relations

\begin{equation}\label{E:defquantumgroups}
\begin{split}
&K_i K_j=K_j K_i  \\
&K_iE_j K_i^{-1}=\v^{a_{ij}} E_j \\
&\sum_{l=0}^{1-a_{ij}} (-1)^l \begin{bmatrix} 1-a_{ij} \\ l \end{bmatrix}_{\nu} F_i^l F_j F_i^{1-a_{ij}-l}=0.
\end{split}
\end{equation}
It is graded by $K_0(\vec{Q})$, where $deg\;K_i=0, deg\;E_i=\epsilon_i$ for all $i$. 

\vspace{.1in}

There is a Hopf algebra structure on $\U_{\v}(\bo'_+)$ in which the coproduct is given by
\begin{equation}\label{E:coprquantum}
\Delta(K_i)=K_i \otimes K_i, \quad \Delta(E_i)=E_i \otimes 1 + K_i \otimes E_i.
\end{equation}

There is also a unique homogeneous Hopf pairing $(\,,\,)$ on $\U_{\v}(\bo'_+)$-- called \textit{Drinfeld's pairing}-- satisfying
\begin{equation}\label{E:Drinpairing}
(K_i,1)=1, \qquad (K_i,K_j)=\v^{a_{ij}}, \qquad (E_i,E_i)=\frac{1}{1-\v^{-2}}.
\end{equation}

\vspace{.1in}

The integral form $\U^{\Z}_{\v}(\bo'_+)$ is by definition the $\Z[\v,\v^{-1}]$-submodules of $\U_{\v}(\bo'_+)$ generated by $K_i^{\pm 1}, E_i^{(n)}:=E_i^n/[n]!$ for $i \in I$. It is stable under the coproduct map.

\vspace{.1in}

By definition, $\U_{\v}(\n_+)$ is the subalgebra of $\U_{\v}(\bo'_+)$ generated by the $E_i$ for $i \in I$. Its integral form $\U^{\Z}_{\v}(\n_+)$ is constructed in the same manner. It is known that the restriction of $(\,,\,)$ to $\U_{\v}(\n_+)$ is nondegenerate. Finally, let $\U_{\nu}(\mathfrak{h})$ denote the subalgebra generated by $K_i^{\pm 1}$. We have

\begin{equation}\label{E:trigdec}
\U_{\nu}(\bo'_+) = \U_{\nu}(\n_+) \otimes \U_{\nu}(\mathfrak{h}).
\end{equation}

\vspace{.15in}

The following is one of the main results in these notes~:

\vspace{.1in}

\begin{theo}[Lusztig]\label{T:LU} Set $\nu=v^{-1}$. The assignement 
\begin{equation}\label{E:Lustrel}
\begin{split}
E_i^{(n)} &\mapsto \mathbf{b}_{\mathbbm{1}_{n\epsilon_i}},\\
K_i &\mapsto \mathbf{k}_{\epsilon_i}
\end{split}
\end{equation}
extends to an isomorphism of (co)algebras $\Phi: \U_{\v}^{\Z}(\bo'_+) \stackrel{\sim}{\to} \widetilde{\mathcal{K}}_{\vec{Q}}$, which restricts to an isomorphism $\Phi: \U_{\v}^{\Z}(\n_+) \stackrel{\sim}{\to} {\mathcal{K}}_{\vec{Q}}$. Moreover, $\Phi$ maps $(\,,\,)$ to the geometric pairing $\{\,,\,\}$.
\end{theo}

\vspace{.1in}

The defining relations (\ref{E:defquantumgroups}) coincide with (\ref{E:drtyu}) and the fundamental relations (\ref{E:FundRel1}). Hence $\Phi$ extends to a morphism of algebras. It is also clear from (\ref{E:coprquantum}) and (\ref{E:DeltaK}), (\ref{E:Deltageom}) that $\Phi$ is compatible with the coproducts.
The difficult part of the Theorem is to show that $\Phi$ is an isomorphism. The proof for a general quiver will be given in Section~3.5. In the meantime, we draw some consequences~:

\vspace{.1in}

Put $\mathbf{B}=\big\{\Phi^{-1}(\mathbf{b}_{\mathbb{P}})\;|\; \mathbb{P} \in \mathcal{P}_{\vec{Q}}\big\}$. Elements of $\mathbf{B}$ form a $\Z[\v,\v^{-1}]$-basis of $\U^{\Z}_{\v}(\n_+)$ called the \textit{canonical basis}. The first illustration of the ``canonical'' nature of the basis $\mathbf{B}$ is provided by the following result, whose proof we will also postpone to Section~3.5~:

\vspace{.1in}

\begin{theo}[Lusztig]\label{T:LUU} The basis $\mathbf{B}$ is independent of the choice of orientation of $\vec{Q}$.
\end{theo}

\vspace{.1in}

Of course, whatever properties the basis $\big\{\mathbf{b}_{\mathbb{P}}\;|\; \mathbb{P} \in \mathcal{P}_{\vec{Q}}\big\}$ has translates into properties of $\mathbf{B}$. For instance, we have
\begin{equation}\label{E:Posmult2}
\mathbf{b} \cdot \mathbf{b}' \in \bigoplus_{\mathbf{b}'' \in \mathbf{B}} \N[\v,\v^{-1}] \mathbf{b}'',
\end{equation}
\begin{equation}\label{E:Poscop2}
\Delta(\mathbf{b} ) \in \bigoplus_{\mathbf{b}', \mathbf{b}'' \in \mathbf{B}} \N[\v,\v^{-1}] \mathbf{b}' \otimes \mathbf{b}'',
\end{equation}
\begin{equation}\label{E:Posscal2}
\begin{split}
\{\mathbf{b},\mathbf{b}\} &\in 1+\v^{-1}\N[[\v^{-1}]], \\
\{\mathbf{b},\mathbf{b}'\} &\in \v^{-1}\N[[\v^{-1}]]\;\text{if}\; \mathbf{b} \neq \mathbf{b}'
\end{split}
\end{equation}
for any elements $\mathbf{b},\mathbf{b}' \in \mathbf{B}$. By Proposition~\ref{P:Verdier} there exists a unique semilinear ring involution $x \mapsto \overline{x}$ on $\U_{\nu}^{\Z}(\n_+)$ satisfying $\overline{E_i^{(n)}}=E_i^{(n)}$ for all $i \in I$ and $n \in \N$, and we have
\begin{equation}\label{E:Invol}
\overline{\mathbf{b}}=\mathbf{b}
\end{equation}

\vspace{.1in}

The relevance of (\ref{E:Posscal2}) and (\ref{E:Invol}) to the theory of canonical bases is due to the following result (whose proof is an easy exercise, see \cite[Thm. 14.2.3.]{Lusbook})

\vspace{.1in}

\begin{theo}[Lusztig] Let $\mathcal{B}$ be the set of all elements $b \in \U_{\nu}^{\Z}(\n_+)$ satisfying $\overline{b}=b$ and $\{b,b\} \in 1 + \nu^{-1}\N[[\nu^{-1}]]$.
Then $\mathcal{B}=\mathbf{B} \cup -\mathbf{B}$.
\end{theo}

\vspace{.1in}

The discovery of the basis $\mathbf{B}$ was a tremendous breakthrough in algebraic representation theory. Applications have been found in numerous areas such as combinatorial representation theory, mathematical physics, algebraic geometry and knot theory. One of the important facts concerning $\mathbf{B}$ is the following~:

\vspace{.1in}

\begin{theo}[Lusztig]\label{T:Rep} Let $\lambda$ be an integral antidominant weight of $\g$ and let $V_{\lambda}$ be the corresponding integrable lowest weight representation of $\U^{\Z}_\nu(\g)$. Let $v_{\lambda} \in V_{\lambda}$ be the lowest weight vector. Then 
$$\mathbf{B}_{\lambda}:=\big\{ \mathbf{b} \cdot v_{\lambda}\;|\; \mathbf{b} \cdot v_{\lambda} \neq 0\big\}$$
forms a (weight) basis of $V_{\lambda}$.
\end{theo}

In other words, the canonical basis $\mathbf{B}$ projects to a basis in \textit{all} integrable lowest weight representations\footnote{We use lowest weight representations here rather than highest weight representations because we have have written Theorem~\ref{T:LU} using $\U^{\Z}_{\nu}(\n_+)$ rather than $\U^{\Z}_{\nu}(\n_-)$. Of course, one may exchange the roles of $+$ and $-$.}. For this reason, there has been a lot of activity in trying to parametrize and compute explicitly elements of the canonical basis (see e.g. \cite{Lusbook}, \cite{Lusparam}, \cite{Kashparam}, \cite{Marsh},...). By lack of space, we will not describe these results here, and give only a few (very simple) examples. We will give a proof of Theorem~\ref{T:Rep} in Section~3.5.

\vspace{.2in}

\addtocounter{theo}{1}
\noindent \textbf{Example \thetheo .} Let us come back to the Jordan quiver $\vec{Q}_0$. Again, this case is not covered by Theorem~\ref{T:LU}, but it was understood well before the advent of quantum groups. By Lemma~\ref{L:commut}, $\mathcal{K}_{\vec{Q}_0}$ is commutative, and it is easy to see from Corollary~\ref{C:commut} that it is freely generated over $\Z[v,v^{-1}]$ by the elements $\mathbf{b}_{IC(\mathcal{O}_{(1^n)})}$ for $n \geq 1$, i.e.
$$\mathcal{K}_{\vec{Q}_0} =\Z[v,v^{-1}] [ \mathbf{b}_{IC(\mathcal{O}_{(1)})}, \mathbf{b}_{IC(\mathcal{O}_{(1^2)})}, \ldots].$$
In fact there is a (canonical) identification
\begin{equation*}
\begin{split}
\Phi:\; \Lambda_{\Z[v,v^{-1}]} & \stackrel{\sim}{\to} \mathcal{K}_{\vec{Q}_0} \\
e_n &\mapsto \mathbf{b}_{IC(\mathcal{O}_{(1^n)})}
\end{split}
\end{equation*}
where $\Lambda_{\Z[v,v^{-1}]}$ is Macdonald's ring of symmetric functions and $e_n$ is the elementary
symmetric function (see \cite{Mac}). It is proved in \cite{LusztigGreen} that the basis $\big\{ \Phi^{-1}(\mathbf{b}_{IC(\mathcal{O}_{\lambda})})\;|\; \lambda \in \Pi\big\}$ consists of the Schur functions $\big\{s_{\lambda}\;|\; \lambda \in \Pi\big\}$.

\vspace{.1in}

The coefficients of Schur functions are special cases of (affine) \textit{Kazhdan-Lusztig polynomials} of type $A$. This corresponds to the fact that nilpotent orbit closures in $\mathfrak{gl}_n$ are locally isomorphic to Schubert varieties in affine Grasmannians. Similar results hold for cyclic quivers $\vec{Q}_{n-1}$ with $n>1$. We refer (again !) the interested reader to \cite[Lecture~2]{Trieste} or \cite{Quiversurvey} and its bibliography for more in this direction. 
\endexample

\vspace{.2in}

\addtocounter{theo}{1}
\noindent \textbf{Example \thetheo .} Let $\vec{Q}$ be a quiver without oriented cycles and let us label the vertices $I=\{1,2, \ldots, n\}$ in such a way that no arrow goes from $i$ to $j$ if $i >j$. By Example~2.5, we have that $\mathbbm{1}_{\a_1\epsilon_1} \star \cdots \star \mathbbm{1}_{\a_n\epsilon_n}=\mathbbm{1}_{\a_1\epsilon_1+ \cdots + \a_n\epsilon_n}$ belongs to $\PQ$ for any $(\a_i)_i \in \N^I$. Therefore
$E_1^{(\a_1)} \cdots E_n^{(\a_n)} \in \mathbf{B}$ for any $(\a_i)_i \in \N^I$. Combined with Theorem~\ref{T:LUU}, this shows that for any (symmetric) Kac-Moody algebra $\g'$, \textit{any} product $E_{i_1}^{(l_1)} \cdots E_{i_n}^{(l_n)}$ with $i_k \neq i_h$ for $k \neq h$ belongs to $\mathbf{B}$; indeed, it is always possible to orient the Dynkin diagram of $\g'$ in such a way that there are no arrows from $i_k$ to $i_h$ if $k >l$.
\endexample

\vspace{.3in}

\centerline{\textbf{3.3. Proof of Lusztig's theorem (finite type).}}
\addcontentsline{toc}{subsection}{\tocsubsection {}{}{\; 3.3. Proof of Lusztig's theorem (finite type).}}

\vspace{.15in}

In this section we assume that $\vec{Q}$ is a quiver of finite type, and we provide an elementary proof of Lusztig's Theorem~\ref{T:LU}. Recall that in this situation, there are finitely many $G_{\gamma}$-orbits in $E_{\gamma}$ for any dimension vector $\gamma$. We denote by $|E_{\gamma}/G_{\gamma}|$ the set of these orbits. Then, by Theorem~\ref{T:finiteype}, $\mathcal{P}^{\gamma}=\big\{IC(\mathcal{O})\;|\; \mathcal{O} \in |E_{\gamma}/G_{\gamma}|\big\}$. We want to prove that the map $\Phi: \U^{\Z}_{\nu}(\bo'_+) \to \widetilde{\KQ}$ is an isomorphism. By (\ref{E:trigdec}) this is clearly equivalent to showing that the restriction $\Phi: \U^{\Z}_{\nu}(\n_+) \to {\KQ}$ is an isomorphism. Both $\U^{\Z}_{\nu}(\n_+)$ and ${\KQ}$ are free $\Z[v,v^{-1}]$-modules, of finite (graded) rank. We begin by comparing these ranks. Let $\Delta_+ \subset Q_{\g'}=\Z^I$ be the set of positive roots of $\g'$. By the PBW theorem, 
\begin{equation}\label{E:PBW}
rank\;\U^{\Z}_{\nu}(\n_+)[\gamma]=\# \big\{ (n_{\a})_{\a} \in \N^{\Delta_+}\;|\; \sum n_\a  \a =\gamma \big\}.
\end{equation}
On the other hand, we have $rank\;\mathcal{K}_{\vec{Q}}[\gamma]= rank\;\mathcal{K}^{\gamma}=\# |E_{\gamma}/G_{\gamma}|$. Since $G_{\gamma}$-orbits bijectively correspond to isoclasses of representations of $\vec{Q}$ of dimension $\gamma$, and since any representation splits in an essentially unique way as a direct sum of indecomposables, we have
\begin{equation}\label{E:PBW2}
rank\;\mathcal{K}^\gamma=\# \big\{ (m_{M})_{M} \in \N^{Irr\;\vec{Q}}\;|\; \sum m_M  \underline{dim}\;M =\gamma \big\}.
\end{equation}
At this point, we invoke the fabled (see \cite[Lecture~3]{Trieste})~:

\vspace{.1in}

\begin{theo}[Gabriel] The map $Irr\;\vec{Q} \to \bigoplus_i \N \epsilon_i, \; M \mapsto \underline{dim}\;M$ sets up a bijection between $Irr\;\vec{Q}$ and $\Delta^+$. 
\end{theo}

\vspace{.1in}

It follows that (\ref{E:PBW}) and (\ref{E:PBW2}) are equal, and $rank\;\U^{\Z}_{\nu}(\n_+)[\gamma]=rank\;\mathcal{K}_{\vec{Q}}[\gamma]$.
It is thus enough to show that the map $\Phi$ is surjective. This will be a consequence of the existence of Reineke's desingularization see Proposition~\ref{P:Reineke}. We will show by induction on $dim\;\mathcal{O}$ that $\mathbf{b}_{IC(\mathcal{O})} \in Im\;\Phi$ for any orbit $\mathcal{O} \in |E_{\gamma}/G_{\gamma}|$. Assume first that $dim\;\mathcal{O}$ is minimal. Then $\mathcal{O}$ is closed\footnote{actually, $\mathcal{O}=\{0\}$.}.
Using Proposition~\ref{P:Reineke} we construct a Lusztig sheaf $L_{\a_1, \ldots, \a_n}$ satisfying
\begin{equation}\label{E:proofft}
\begin{split}
supp\; L_{\a_1, \ldots, \a_n} &=\overline{\mathcal{O}},\\
(L_{\a_1, \ldots, \a_n})_{|\mathcal{O}}&=\qlb_{\mathcal{O}} [dim\;\mathcal{O}].
\end{split}
\end{equation}
But then necessarily $L_{\a_1, \ldots, \a_n}=\qlb_{\mathcal{O}}[dim\;\mathcal{O}]=IC(\mathcal{O})$. Hence $\mathbf{b}_{IC(\mathcal{O})}=\Phi(E_{i_1}^{(l_1)}) \star \cdots \star \Phi(E_{i_n}^{(l_n)})=\Phi( E_{i_1}^{(l_1)} \cdots E_{i_n}^{(l_n)}) \in Im\;\Phi$. Now let $\mathcal{O}$ be arbitrary and let us assume that $\mathbf{b}_{IC(\mathcal{O}')} \in Im\;\Phi$ for any $\mathcal{O}'$ satisfying $dim\;\mathcal{O}' < dim\;\mathcal{O}$. Arguing as above, we construct a Lusztig sheaf $L_{\a_1, \ldots, \a_n}$ satisfying (\ref{E:proofft}). Then
\begin{equation}\label{E:proofft2}
L_{\a_1, \ldots, \a_n}=IC(\mathcal{O}) \oplus \bigoplus_{\mathcal{O}' \subset \overline{\mathcal{O}} \backslash \mathcal{O}} IC(\mathcal{O}') \otimes \mathbb{V}_{\mathcal{O}'}
\end{equation}
for some multiplicity complexes $\mathbb{V}_{\mathcal{O}'}$. Put $d_{\mathbb{V}_{\mathcal{O}'}}=\sum_j dim\;H^j(\mathbb{V}_{{\mathcal{O}'}}) v^{-j}$. Taking the class of (\ref{E:proofft2}) in the Grothendieck group $\KQ$ we get
$$\mathbf{b}_{IC(\mathcal{O})} + \sum_{\mathcal{O}' \subset \overline{\mathcal{O}} \backslash \mathcal{O}} d_{\mathbb{V}_{\mathcal{O}'}} \mathbf{b}_{IC(\mathcal{O}')}=\Phi\big(E_{i_1}^{(l_1)} \cdots E_{i_n}^{(l_n)}\big) \in Im\;\Phi$$
for some suitable $(i_j,l_j)_{j}$. By hypothesis, $\mathbf{b}_{IC(\mathcal{O}')}\in Im\;\Phi$ for all $\mathcal{O}' \subset \overline{\mathcal{O}} \backslash \mathcal{O}$, from which we conclude that $\mathbf{b}_{IC(\mathcal{O})} \in Im\;\Phi$ as wanted. We are done.\qed

\vspace{.3in}

\centerline{\textbf{3.4. Fourier-Deligne transform.}}
\addcontentsline{toc}{subsection}{\tocsubsection {}{}{\; 3.4. Fourier-Deligne transform.}}

\vspace{.15in}

This section contains some technical results pertaining to the Fourier-Deligne transform. A good reference for everything we will use is \cite{KW}. Let $\vec{Q}=(I,H)$ be our quiver, which may now be arbitrary. For any edge $h \in H$ we denote by $\bar{h}$ its reverse (i.e. the edge which goes in the opposite direction). Let $J \subset H$ be a subset of edges, and let $\vec{Q}'=(I,H') $ where $H'=\overline{J} \cup (H \backslash J) $ be the quiver obtained by reversing all the edges in $J$. We also denote by $\vec{Q}_0=(I,H \backslash J)$ the quiver obtained by removing all edges in $J$.  Here is an example, with $H=\{h_1,h_2,h_3,h_4,h_5\}$ and $J=\{h_1,h_2,h_4\}$~:

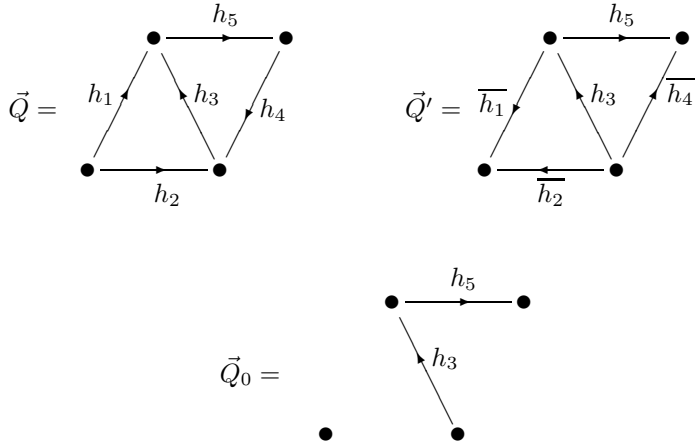
\centerline{
\begin{picture}(310, 210)
\put(120,40){\circle*{5}}
\put(170,40){\circle*{5}}
\put(145,90){\circle*{5}}
\put(195,90){\circle*{5}}
\put(148,84){\line(1,-2){20}}
\put(157,66){\vector(-1,2){2.5}}
\put(170,90){\vector(1,0){5}}
\put(150,90){\line(1,0){40}}
\put(160,66){$h_3$}
\put(167,96){$h_5$}
\put(80,60){$\vec{Q}_0=$}
\put(180,140){\circle*{5}}
\put(230,140){\circle*{5}}
\put(205,190){\circle*{5}}
\put(255,190){\circle*{5}}
\put(193,166){\vector(-1,-2){2.5}}
\put(182.5,145){\line(1,2){20}}
\put(208,184){\line(1,-2){20}}
\put(217,166){\vector(-1,2){2.5}}
\put(243,166){\vector(1,2){2.5}}
\put(232.5,145){\line(1,2){20}}
\put(230,190){\vector(1,0){5}}
\put(210,190){\line(1,0){40}}
\put(205,140){\vector(-1,0){5}}
\put(185,140){\line(1,0){40}}
\put(178,162){$\overline{h_1}$}
\put(220,166){$h_3$}
\put(249,166){$\overline{h_4}$}
\put(200,128){$\overline{h_2}$}
\put(227,196){$h_5$}
\put(150,160){$\vec{Q}'=$}
\put(30,140){\circle*{5}}
\put(80,140){\circle*{5}}
\put(55,190){\circle*{5}}
\put(105,190){\circle*{5}}
\put(43,166){\vector(1,2){2.5}}
\put(32.5,145){\line(1,2){20}}
\put(58,184){\line(1,-2){20}}
\put(67,166){\vector(-1,2){2.5}}
\put(92,164){\vector(-1,-2){2.5}}
\put(82.5,145){\line(1,2){20}}
\put(80,190){\vector(1,0){5}}
\put(60,190){\line(1,0){40}}
\put(55,140){\vector(1,0){5}}
\put(35,140){\line(1,0){40}}
\put(30,166){$h_1$}
\put(70,166){$h_3$}
\put(95,160){$h_4$}
\put(55,128){$h_2$}
\put(77,196){$h_5$}
\put(0,160){$\vec{Q}=$}
\end{picture}}

\vspace{-.2in}

Accordingly, we put
\begin{align*}
E_{\gamma}&=\bigoplus_{h \in H} Hom(k^{\gamma_{s(h)}},k^{\gamma_{t(h)}}),\\
E'_{\gamma}&=\bigoplus_{h \in H'} Hom(k^{\gamma_{s(h)}},k^{\gamma_{t(h)}}),\\
E_{\gamma,0}&=\bigoplus_{h \in H \backslash J} Hom(k^{\gamma_{s(h)}},k^{\gamma_{t(h)}}).
\end{align*}
The spaces $E_{\gamma}, E'_{\gamma}$ are vector bundles over $E_{\gamma,0}$ with fibers $\bigoplus_{h \in J} Hom(k^{\gamma_{s(h)}},k^{\gamma_{t(h)}})$ and $\bigoplus_{h \in J} Hom(k^{\gamma_{t(h)}},k^{\gamma_{s(h)}})$ respectively. The linear map
\begin{equation}\label{E:pairing}
\begin{split}
\bigoplus_{h \in J} Hom(k^{\gamma_{s(h)}},k^{\gamma_{t(h)}}) \otimes \bigoplus_{h \in J} Hom(k^{\gamma_{t(h)}},k^{\gamma_{s(h)}}) & \to k\\
(\bigoplus_h a_h) \otimes (\bigoplus {a}_{\bar{h}}) \hspace{.86in} & \mapsto \sum_h tr(a_h a_{\bar{h}})
\end{split}
\end{equation}
is a nondegenerate pairing between the vector bundles $E_{\gamma}$ and $E'_{\gamma}$ over $E_{\gamma,0}$. In that situation there is a Fourier-Deligne transform
$$\FD: D^b(E_\gamma) {\to} D^b(E'_{\gamma})$$
constructed as follows. Consider the projections
$$\xymatrix{ E_{\gamma} & E_{\gamma} \underset{E_{\gamma,0}}{\times} E'_{\gamma} \ar[l]_-{\pi_1} \ar[r]^-{\pi_2} & E'_{\gamma}}$$
and let $\chi: E_{\gamma} \times_{E_{\gamma,0}} E'_{\gamma} \to k$ be the pairing (\ref{E:pairing}).
Let us fix a nontrivial additive character $\rho: \mathbb{F}_q \to \qlb^*$. The Artin-Schreier map $x \mapsto x^q-x$ is a covering $k \to k$ with Galois group equal to $\mathbb{F}_q$. Using this we can define a nontrivial local system $\mathcal{L}$ of rank one over $k$. By definition, $\FD(\mathbb{P})=\pi_{2!}( \pi_1^*(\mathbb{P}) \otimes \chi^* (\mathcal{L}))[r]$, where $r$ is the rank of the bundle $E_{\gamma} \to E_{\gamma,0}$. There is of course a similarly defined Fourier transform
$\FD':  D^b(E'_\gamma) \stackrel{\sim}{\to} D^b(E_{\gamma})$. 

\vspace{.1in}

Some of the important properties of the Fourier transform are summarized in the following~

\vspace{.1in}

\begin{theo}\label{T:FDT} The Fourier-Deligne transform is an equivalence of triangulated categories
$$\FD: D^b(E_\gamma) \stackrel{\sim}{\to} D^b(E'_{\gamma}).$$
It restricts to an equivalence $D^b(E_{\gamma})^{ss} \stackrel{\sim}{\to} D^b(E'_{\gamma})^{ss}$ preserving perverse sheaves. Moreover, if $j$ denotes the
operation of multiplication by $-1$ along the fibers of $\pi_1$ then for any complex $\mathbb{P}$ in $D^b(E_{\gamma})$ we have $D \FD D (\mathbb{P})= \FD j^*(\mathbb{P})$ and 
$$\FD' \circ \FD (\mathbb{P}) \simeq j^*(\mathbb{P})$$
(Fourier inversion formula).
\end{theo}

\vspace{.1in}

\begin{prop}[Lusztig]\label{P:FD} The Fourier-Deligne transform commutes with the induction and restriction functors, i.e.
$$\FD(\mathbb{P} \star \mathbb{Q}) \simeq \FD(\mathbb{P}) \star \FD(\mathbb{Q}),$$
$$\FD \otimes \FD(\underline{\Delta} (\mathbb{P})) \simeq \underline{\Delta} (\FD(\mathbb{P}))$$
for any $\mathbb{P}, \mathbb{Q} \in \QQ $. Moreover, 
\begin{equation}\label{E:scalFD}
\{ \FD(\mathbb{P}), \mathbb{Q}'\} =\{ \mathbb{P}, \FD'(\mathbb{Q}')\}
\end{equation}
for any $\mathbb{P} \in \QQ$, $\mathbb{Q}' \in \mathcal{Q}_{\vec{Q}'}$.
\end{prop}
\begin{proof} We begin with the statement concerning the scalar product. Recall that if $X$ is a $G$-variety and if $\mathbb{T} \in D^b_{G}(X)^{ss}$ then $\mathbb{T}_{\Gamma}$ is the unique (up to isomorphism) semisimple complex on $(X \times \Gamma)/G$ such that $s^*(\mathbb{T}_{\Gamma}) = \pi^*(\mathbb{T})$, where $s, \pi$ are the canonical maps
$$\xymatrix{ X & X \times \Gamma \ar[l]_-{\pi} \ar[r]^-{s} & (X \times \Gamma)/G = X_{\Gamma}}.$$
Here $\Gamma$ is as usual a sufficiently acyclic free $G$-space. By definition, if $\mathbb{P} \in \QQ^{\gamma}, \mathbb{Q}' \in \mathcal{Q}_{\vec{Q}'}^{\gamma}$ then
$$\{ \FD(\mathbb{P}), \mathbb{Q}'\}=\sum_j dim\; H^{2 dim\;\Gamma/G_{\gamma}-j} \big( \pi_{2!}(\pi_1^*(\mathbb{P}) \otimes \chi^*(\mathcal{L}))_{\Gamma} \otimes \mathbb{Q}'_{\Gamma} \big).$$
There is an induced Fourier-Deligne diagram
$$\xymatrix{ (E_{\gamma})_{\Gamma} & (E_{\gamma})_{\Gamma} \underset{(E_{\gamma,0})_{\Gamma}}{\times} (E'_{\gamma})_{\Gamma} \ar[l]_-{\pi^{\Gamma}_1} \ar[r]^-{\pi^{\Gamma}_2} & (E'_{\gamma})_{\Gamma}}$$
and a pairing $\chi^{\Gamma}: (E_{\gamma})_{\Gamma} \times_{(E_{\gamma,0})_{\Gamma}} (E'_{\gamma})_{\Gamma} \to k$, and we have
$$\pi_{2!}(\pi_1^*(\mathbb{P}) \otimes \chi^*(\mathcal{L}))_{\Gamma} =\pi^{\Gamma}_{2!}(\pi^{\Gamma *}_1(\mathbb{P}_{\Gamma}) \otimes \chi^{\Gamma *}(\mathcal{L})).$$
Hence by the projection formula we obtain
$$\{ \FD(\mathbb{P}), \mathbb{Q}'\}=\sum_j dim\; H^{2 dim\;\Gamma/G_{\gamma}-j} \big( \pi_1^{\Gamma *}(\mathbb{P}_{\Gamma}) \otimes \chi^{\Gamma *}(\mathcal{L}) \otimes \pi^{\Gamma *}_2(\mathbb{Q}'_{\Gamma}) \big).$$
Starting from the r.h.s of (\ref{E:scalFD}) would yield the same expression. This proves (\ref{E:scalFD}).

\vspace{.1in}

Let us turn to the compatibility between $\FD$ and the restriction functor $\underline{\Delta}_{\a,\beta}$. For this, we consider the following commutative diagram~:
\begin{equation}\label{E:DiagFD}
\xymatrix{E_{\a} \times E_{\beta} & F \ar[l]_-{\kappa} \ar[rrr]^-{\iota} & & & E_{\gamma}\\
(E_{\a} \times E_{\beta}) \times (E'_{\a} \times E'_{\beta}) \ar[u]^-{\pi_1'} \ar[d]_{\pi_2'} & F \times (E'_{\a} \times E'_{\beta}) \ar[l]_-{\dot{\kappa}} \ar[u]^-{\dot{\pi}_1} & F \times F' \ar[l]_-{\phi} \ar[r]^-{\psi} & E_{\gamma} \times F' \ar[r]^-{\dot{\iota}}\ar[d]_-{\dot{\pi}_2} & E_{\gamma} \times E'_{\gamma} \ar[u]^-{\pi_1} \ar[d]_-{\pi_2} \\
E'_{\a} \times E'_{\beta} & & & F' \ar[lll]_-{\kappa'} \ar[r]^-{\iota'} & E'_{\gamma}}
\end{equation}
where all the products are understood to be over the base $E_{\gamma,0}$ or $E_{\a,0} \times E_{\beta,0}$, and all the maps are the obvious ones; $\kappa,\kappa',\dot{\kappa}$ and $\phi$, as well as $\pi_1, \pi'_1, \dot{\pi}_1, \pi_2, \pi'_2, \dot{\pi}_2$ are vector bundles; $\psi, \iota, \iota', \dot{\iota}$ are closed embeddings. We denote by $d_X$ the rank of the vector bundle $X$ (for $X$ one of the above). Note that the two square diagrams in (\ref{E:DiagFD}) are cartesian. The local system $\mathcal{L}$ on $k$ gives rise via pullbacks by $\chi_{\gamma}$ and $\chi_{\a} \times \chi_{\beta}$ to local systems $\mathcal{L}_{\gamma}$ and $\mathcal{L}_{\a,\beta}$ on $E_{\gamma} \times E'_{\gamma}$ and $(E_{\a} \times E_{\beta}) \times (E'_{\alpha} \times E'_{\beta})$ respectively.

\vspace{.1in}

Now let $\mathbb{P} \in D^b_{G_{\gamma}}(E_{\gamma})^{ss}$. We have by definition
\begin{equation}\label{E:FD11}
\begin{split}
\underline{\Delta}_{\a,\beta}(\FD(\mathbb{P}))&=\kappa'_! (\iota')^* \pi_{2!} \big( \pi_1^*(\mathbb{P}) \otimes \mathcal{L}_{\gamma}\big)[-\langle \a,\beta \rangle_{\vec{Q}'} + d_{\pi_1}]\\
&=\kappa'_! \dot{\pi_2}_{!} \dot{\iota}^*\big(\pi_1^*(\mathbb{P}) \otimes \mathcal{L}_{\gamma}\big)[-\langle \a,\beta \rangle_{\vec{Q}'} + d_{\pi_1}]\\
&=\kappa'_! \dot{\pi_2}_{!} \big(\dot{\iota}^*\pi_1^*(\mathbb{P}) \otimes \dot{\iota}^*(\mathcal{L}_{\gamma})\big)[-\langle \a,\beta \rangle_{\vec{Q}'} + d_{\pi_1}].
\end{split}
\end{equation}

We will now show that the relevant information concerning the local system $\dot{\iota}^*(\mathcal{L})$ on $E_{\gamma} \times F'$ is in a certain sense supported on $F \times F'$. Set
$$Z=\{0\} \times (\kappa')^{-1}(\{0\} \times \{0\}) \subset F \times F'. $$
Thus
$$Z=\big\{ (0,\underline{x}'), \; \underline{x}' =(x'_h)_{h \in \overline{J}}\;|\; \underline{x}' (V_{\beta})=0,\; \underline{x}'(V_{\gamma}) \subset V_{\beta} \big\}.$$
The map $\chi_{\gamma} \dot{\iota}: E_{\gamma} \times F' \to k$ restricts to an affine map on each of the affine subspaces $(\underline{x}, \underline{x}')+Z$. This restriction is constant if and only if $(\underline{x}, \underline{x}') \in F \times F'$. Let us denote by $\rho : E_{\gamma} \times F' \to E_{\gamma} \times F'/Z$ the quotient map and by $\rho^\circ$ the restriction of $\rho$ to the complement $(E_{\gamma} \times F')^{\circ}$ of $F \times F'$ in $E_{\gamma} \times F'$. By the above argument and Lemma~\ref{L:FDaffine} below we have 
\begin{equation}\label{E:FD12}
\rho^{\circ}_!(\dot{\iota}^*(\mathcal{L}_{\gamma}))=0.
\end{equation}
The closed embedding $\psi: F \times F' \hookrightarrow E_{\gamma} \times F'$ and the open embedding of the complement
 $\psi^{\circ}: (E_{\gamma} \times F')^{\circ} \hookrightarrow E_{\gamma} \times F'$ determine a distinguished triangle
$$\xymatrix{ \psi^{\circ}_! \psi^{\circ *}(\dot{\iota}^* (\mathcal{L}_{\gamma})) \ar[r] &\dot{\iota}^*(\mathcal{L}_{\gamma}) \ar[r]  &
\psi_!\psi^*(\dot{\iota}^*(\mathcal{L}_{\gamma})) \ar[r]^-{[1]} &}$$
Applying the functor $\rho_!$ to the above triangle and using (\ref{E:FD12}) yields $\rho_!\big( \dot{\iota}^*(\mathcal{L}_{\gamma})\big)=\rho_! \big( \psi_!\psi^* \dot{\iota}^*(\mathcal{L}_{\gamma})\big)$. Since $\kappa' \dot{\pi}_2$ factors through $\rho$ we obtain 
$$\kappa'_! \dot{\pi}_{2!}\big( \dot{\iota}^*(\mathcal{L}_{\gamma})\big)=\kappa'_!\dot{\pi}_{2!} \big( \psi_!\psi^* \dot{\iota}^*(\mathcal{L}_{\gamma})\big)$$
and finally
\begin{equation}\label{E:FD13}
\kappa'_! \dot{\pi}_{2!}\big(\dot{\iota}^* \pi_1^*(\mathbb{P}) \otimes \dot{\iota}^*(\mathcal{L}_{\gamma})\big)=\kappa'_!\dot{\pi}_{2!} \big( \dot{\iota}^* \pi_1^*(\mathbb{P}) \otimes \psi_!\psi^* \dot{\iota}^*(\mathcal{L}_{\gamma})\big).
\end{equation}
Using (\ref{E:FD13}) we get
\begin{equation}\label{E:FD14}
\begin{split}
\kappa'_! \dot{\pi}_{2!} \big( \dot{\iota}^* \pi_1^*(\mathbb{P}) \otimes \dot{\iota}^*(\mathcal{L})\big) &=\kappa'_!\dot{\pi}_{2!} \big( \dot{\iota}^* \pi_1^*(\mathbb{P}) \otimes \psi_!\psi^* \dot{\iota}^*(\mathcal{L}_{\gamma})\big)\\
&=\kappa'_!\dot{\pi}_{2!} \big( \dot{\iota}^* \pi_1^*(\mathbb{P}) \otimes \psi_!\phi^* \dot{\kappa}^* (\mathcal{L}_{\alpha,\beta})\big)\\
&=\kappa'_!\dot{\pi}_{2!}\psi_! \big( \psi^*\dot{\iota}^* \pi_1^*(\mathbb{P}) \otimes \phi^* \dot{\kappa}^* (\mathcal{L}_{\alpha,\beta})\big)\\
&=\pi'_{2!}\dot{\kappa}_{!}\phi_! \big( \phi^* \dot{\pi}_1^*\iota^*(\mathbb{P}) \otimes \phi^* \dot{\kappa}^* (\mathcal{L}_{\alpha,\beta})\big)\\
&=\pi'_{2!} \big( \dot{\kappa}_!\phi_!\phi^* \dot{\pi}_1^*\iota^*(\mathbb{P}) \otimes \otimes \mathcal{L}_{\alpha,\beta}\big)\\
&=\pi'_{2!} \big( \dot{\kappa}_! \dot{\pi}_1^*\iota^*(\mathbb{P}) \otimes \otimes \mathcal{L}_{\alpha,\beta}\big)[-2d_{\phi}]\\
&=\pi'_{2!} \big( (\pi'_1)^*\kappa_!\iota^*(\mathbb{P}) \otimes \otimes \mathcal{L}_{\alpha,\beta}\big)[-2d_{\phi}]\\
&=\pi'_{2!} \big( (\pi'_1)^*\underline{\Delta}_{\a,\beta}(\mathbb{P}) \otimes \otimes \mathcal{L}_{\alpha,\beta}\big)[-2d_{\phi}+\langle \a , \beta \rangle_{\vec{Q}}]\\
&=\FD \otimes \FD\big(\underline{\Delta}_{\a,\beta}(\mathbb{P})\big)[2d_{\phi}+\langle \a , \beta \rangle_{\vec{Q}}-d_{\pi'_1}].
\end{split}
\end{equation}
In the above we have used the projection formula and the fact that $\phi$ is a vector bundle, so that $\phi_!\phi^*=[-2d_{\phi}]$. The equality $\FD\otimes \FD(\underline{\Delta}_{\a,\beta}(\mathbb{P})) \simeq \underline{\Delta}_{\a,\beta}(\FD(\mathbb{P}))$ follows from
(\ref{E:FD11}), (\ref{E:FD14}) and the easily checked identity  $-2d_{\phi} + \langle \a, \beta \rangle_{\vec{Q}} - \langle \a, \beta \rangle_{\vec{Q}'}  +d_{\pi_1}- d_{\pi'_1}=0$.

\vspace{.1in}

It remains to show the compatibility of the Fourier-Deligne transform with the induction functor $\underline{m}_{\a,\beta}$. For $\mathbb{P} \in D^b_{G_{\a}}(E_{\a})^{ss}, \mathbb{Q} \in D^b_{G_{\beta}}(E_{\beta})^{ss}$ and any $\mathbb{R} \in D^b_{G_{\gamma}}(E'_{\gamma})^{ss}$ it holds
\begin{equation*}
\begin{split}
\{ \FD (\mathbb{P} \star \mathbb{Q}), \mathbb{R}\} &= \{ \mathbb{P} \star \mathbb{Q}, \FD'(\mathbb{R})\}\\
&=\{ \mathbb{P} \boxtimes \mathbb{Q}, \underline{\Delta}_{\a,\beta}\big(\FD'(\mathbb{R})\big)\}\\
&=\{ \mathbb{P} \boxtimes \mathbb{Q}, \FD' \otimes \FD' \big(\underline{\Delta}_{\a,\beta}(\mathbb{R})\big)\}\\
&=\{ \FD(\mathbb{P}) \boxtimes \FD(\mathbb{Q}), \underline{\Delta}_{\a,\beta}(\mathbb{R})\}\\
&=\{ \FD(\mathbb{P}) \star \FD(\mathbb{Q}), \mathbb{R}\}.
\end{split}
\end{equation*}
The result follows from the nondegeneracy of the pairing $\{\,,\,\}$ (see Proposition~\ref{C:scalarprod}). Proposition~\ref{P:FD} is proved.
\end{proof}

\vspace{.1in}

In the course of the proof, we have used the observation~:

\begin{lem}\label{L:FDaffine} Let $h: {k}^n \to {k}$ be a nonconstant affine map. Then $H^*_c({k}^n,h^*\mathcal{L})=0$.
\end{lem}
\begin{proof} Indeed, if $\pi: {k} \to \{pt\}$ denotes the projection to a point, we have
$$H^*_c({k}^n,h^*\mathcal{L})=\pi_!h_! h^*\mathcal{L} =\pi_!\mathcal{L}[-2(n-1)]=H^{*-2(n-1)}_c({k},\mathcal{L})=0.$$
\end{proof}

We sum up the important consequences of Proposition~\ref{P:FD} in the following

\vspace{.1in}

\begin{cor}\label{C:FD} The Fourier-Deligne transform restricts to an equivalence $\mathcal{Q}_{\vec{Q}} \stackrel{\sim}{\to} \mathcal{Q}_{\vec{Q}'}$, sets up a bijection $\mathcal{P}_{\vec{Q}} \leftrightarrow \mathcal{P}_{\vec{Q}'}$, and for any \textit{simple} dimension vectors $\a_1, \ldots, \a_n$ we have
\begin{equation}\label{E:FD15}
\FD( L^{\vec{Q}}_{\a_1, \ldots, \a_n})=L^{\vec{Q}'}_{\a_1, \ldots, \a_n}.
\end{equation}
Moreover, the map $\mathbf{b}_{\mathbb{P}} \mapsto \mathbf{b}_{\FD(\mathbb{P})}$ defines an isomorphism of (co)algebras 
$$\FD: \mathcal{K}_{\vec{Q}} \stackrel{\sim}{\to} \mathcal{K}_{\vec{Q}'}.$$
\end{cor}
\begin{proof} It is clear that $\FD(\mathbbm{1}_{\epsilon_i}) =\mathbbm{1}_{\epsilon_i}$ for all vertices $i \in I$. Equation (\ref{E:FD15}) follows by Proposition~\ref{P:FD}. Since $\FD$ is additive and maps perverse sheaves to perverse sheaves, it induces a bijection between the simple summands of the Lusztig sheaves for $\vec{Q}$ and $\vec{Q}'$. All the other assertions of the Corollary now follow from Proposition~\ref{P:FD}.\end{proof}

\vspace{.1in}

To finish, we observe that

\vspace{.1in}

\begin{lem} For any $\mathbb{P} \in \QQ$ we have $\FD' \circ \FD (\mathbb{P}) \simeq \mathbb{P}$.
\end{lem}
\begin{proof} The Fourier inversion formula gives $\FD' \circ \FD (\mathbb{P}) = j^*(\mathbb{P})$ where $j$ is the multiplication by $-1$ along the fibers of $\pi_1: E_{\gamma} \to E_{\gamma,0}$.  It is easy to check from the constructions that any Lusztig sheaf $L_{\a_1, \ldots, \a_n}$ is equivariant with respect to the action 
\begin{equation*}
\begin{split}
(k^*)^H \times E_{\gamma} & \to E_{\gamma}\\
(u_h)_h \times (x_h)_h&\mapsto (u_h x_h)_{h}
\end{split}
\end{equation*}
Since $(k^*)^H$ is connected, any simple direct summand $\mathbb{P}$ of $L_{\a_1, \ldots, \a_n}$ is also $(k^*)^H$-equivariant. In particular, $j^*(\mathbb{P}) \simeq \mathbb{P}$. This proves the Lemma.\end{proof}

\vspace{.2in}

\addtocounter{theo}{1}
\noindent \textbf{Example \thetheo .} Let $\vec{Q}$ be the Kronecker quiver of Example~2.28 and consider the dimension vector $\gamma=2\delta$. There are $6$ simple perverse sheaves in $\mathcal{P}^{\gamma}_{\vec{Q}}$, namely
$$\big\{\mathbbm{1}_{2\delta}=IC(U_{2\delta},\mathcal{L}_{triv}), \; IC(U_{2\delta},\mathcal{L}_{sign}),\; IC(S_{P_{1,2^2}\oplus I_1}),$$
$$IC(S_{P_2 \oplus I_{1^2,2}}),\; IC(S_{P_2,1,I_1}),\; IC(\{0\})=\qlb_{\{0\}}\big\}.$$
Let us use the Fourier-Deligne transform to reverse both arrows. Of course, we will get a quiver $\vec{Q}'$ isomorphic to the original one, with vertices exchanged. So we may write $\mathcal{P}^{\gamma}_{\vec{Q}'}=\big\{ IC(U_{2\delta},\mathcal{L}_{triv})', \ldots\big\}$.
We claim that the bijection $\FD: \mathcal{P}^{\gamma}_{\vec{Q}} \leftrightarrow \mathcal{P}_{\vec{Q}'}^{\gamma}$ is the following~:
\begin{equation}\label{E:FD16}
\begin{split}
IC(U_{2\delta},\mathcal{L}_{triv})=\mathbbm{1}_{2\delta} &\leftrightarrow IC(\{0\})'=\qlb_{\{0\}}'\\
IC(U_{2\delta},\mathcal{L}_{sign}) &\leftrightarrow IC(S_{P_1,1,I_2})'\\
IC(S_{P_{1,2^2} \oplus I_1}) &\leftrightarrow IC(S_{P_1 \oplus I_{1,2^2}})'\\
IC(S_{P_{2} \oplus I_{1^2,2}}) &\leftrightarrow IC(S_{P_{1^2,2} \oplus I_{2}})'\\
IC(S_{P_2,1,I_1}) &\leftrightarrow IC(U_{2\delta},\mathcal{L}_{sign})'\\
IC(\{0\})=\qlb_{\{0\}} &\leftrightarrow IC(U_{2\delta},\mathcal{L}_{triv})'=\mathbbm{1}_{2\delta}'.
\end{split}
\end{equation}
Indeed, we have (see Example~2.5.) 
$$\FD(\mathbbm{1}_{2\delta})=\FD\big(L^{\vec{Q}}_{2\epsilon_1,2\epsilon_2}\big)=L^{\vec{Q}'}_{2\epsilon_1,2\epsilon_2}=\qlb_{\{0\}}'.$$
This gives the first line of (\ref{E:FD16}). Next, some simple support considerations show that $L^{\vec{Q}}_{\epsilon_2,\epsilon_1,\epsilon_2,\epsilon_1} =IC(S_{P_2,1,I_1}) \oplus \mathbb{T}$ where $\mathbb{T} \simeq \qlb_{\{0\}} \otimes \mathbb{V}$ for some complex of $\qlb$-vector spaces $\mathbb{V}$. On the other hand, by Example~2.28,
$$L^{\vec{Q}'}_{\epsilon_2,\epsilon_1,\epsilon_2,\epsilon_1}=L^{\vec{Q}'}_{\delta,\delta}=IC(U_{2\delta},\mathcal{L}_{triv})' \oplus IC(U_{2\delta},\mathcal{L}_{sign})' \oplus \mathbb{S}'$$
for some complex $\mathbb{S}'$ supported on the irregular locus. It follows that $\FD(IC(S_{P_2,1,I_1})) = IC(U_{2\delta},\mathcal{L}_{sign})'$ (and also that $\mathbb{S}'=0, \mathbb{T}=\qlb_{\{0\}}$). This gives the second line of (\ref{E:FD16}). Finally, let us consider the Lusztig sheaf $L_{\epsilon_1,2\epsilon_2,\epsilon_1}$. Again for some reasons of support, we have
$$L^{\vec{Q}}_{\epsilon_1,2\epsilon_2,\epsilon_1}=IC(S_{P_{1,2^2} \oplus I_1}) \oplus (IC(S_{P_2,1,I_1}) \otimes \mathbb{V}_1) \oplus (IC(\{0\}) \otimes \mathbb{V}_2)$$
for some complexes $\mathbb{V}_1, \mathbb{V}_2$. Similarly,
$$L^{\vec{Q}'}_{\epsilon_1,2\epsilon_2,\epsilon_1}=IC(S_{P_{1} \oplus I_{1,2^2}})' \oplus (IC(S_{P_1,1,I_2})' \otimes \mathbb{V}'_1) \oplus (IC(\{0\})' \otimes \mathbb{V}'_2).$$
Applying $\FD$ and using the first two lines of (\ref{E:FD16}) we see that the only possibility is that $\mathbb{V}_1=\mathbb{V}'_1=\mathbb{V}_2=\mathbb{V}'_2=0$, and that $\FD\big(IC(S_{P_{1,2^2} \oplus I_1})\big)=IC(S_{P_1 \oplus I_{1,2^2}})'$. The last three lines of (\ref{E:FD16}) are obtained by symmetry.

\vspace{.1in}

What this example shows is that the Fourier-Deligne transform acts in a very nontrivial manner on the set of simple perverse sheaves : for instance, $IC(U_{2\delta},\mathcal{L}_{sign})$ --whose support is the whole space $E_{2\delta}$-- gets mapped to the perverse sheaf $IC(S_{P_1,1,I_2})'$ whose support is a proper subset of $E_{2\delta}$. This kind of simplification is crucial in the proof of Lusztig's Theorem~\ref{T:LU}.
\endexample

\vspace{.3in}

\centerline{\textbf{3.5. Proof of Lusztig's theorem.}}
\addcontentsline{toc}{subsection}{\tocsubsection {}{}{\; 3.5. Proof of Lusztig's theorem.}}

\vspace{.15in}

We are now ready to give the proofs of the various theorems announced in Sections~3.1 and 3.2. We will first state (and prove !) a key reduction lemma. For this, let us assume given a quiver $\vec{Q}$ with a \textit{sink} $i \in I$ (i.e., a vertex from which no oriented edge leaves). 

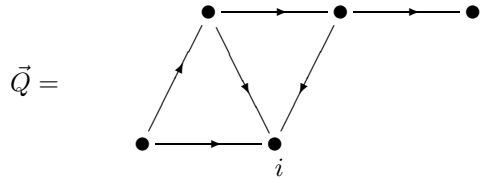
\centerline{
\begin{picture}(130, 110)
\put(0,40){\circle*{5}}
\put(50,40){\circle*{5}}
\put(25,90){\circle*{5}}
\put(75,90){\circle*{5}}
\put(125,90){\circle*{5}}
\put(50,28){$i$}
\put(13,66){\vector(1,2){2.5}}
\put(2.5,45){\line(1,2){20}}
\put(28,84){\line(1,-2){20}}
\put(38,64){\vector(1,-2){2.5}}
\put(62,64){\vector(-1,-2){2.5}}
\put(52.5,45){\line(1,2){20}}
\put(50,90){\vector(1,0){5}}
\put(30,90){\line(1,0){40}}
\put(100,90){\vector(1,0){5}}
\put(80,90){\line(1,0){40}}
\put(25,40){\vector(1,0){5}}
\put(5,40){\line(1,0){40}}
\put(-50,60){$\vec{Q}=$}
\end{picture}}

\vspace{-.2in}

We introduce some partition of the sets $\mathcal{P}^{\gamma}$ as follows. For $d \in \N$ and $\gamma \in \N^I$, let $E_{\gamma}^{i;d} \subset E_{\gamma}$ be the $G_{\gamma}$-invariant locally closed subset defined by the following condition~:
$$E_{\gamma}^{i;d}=\big\{ \underline{x} \in E_{\gamma}\;|\; codim_{(V_{\gamma})_i}\big( Im ( \bigoplus_{h, t(h)=i} x_h)\big)=d\big\}.$$
Each $E_{\gamma}^{i;d}$ is locally closed and $E_{\gamma}^{i;\geq d}=\bigsqcup_{d'\geq d} E_{\gamma}^{i;d'}$ is closed. Observe that $E^{i;0}_{\gamma}$ is open, and that more generally, $E_{\gamma}^{i,d}$ is open in $E_{\gamma}^{i, \geq d}$. For any $\mathbb{P} \in \mathcal{P}^{\gamma}$ there is a unique integer $d$ for which 
\begin{equation}\label{E:Proof1}
supp\;\mathbb{P} \in E_{\gamma}^{i;\geq d}, \qquad supp\;\mathbb{P} \not\in E_{\gamma}^{i;\geq d+1}.
\end{equation}
For a fixed $d$, we denote by $\mathcal{P}^{\gamma}_{i;d}$ the subset of $\mathcal{P}^{\gamma}$ consisting of elements satisfying (\ref{E:Proof1}). We also define $\mathcal{P}^{\gamma}_{i;\geq d}$ in the obvious way. We have
$$\mathcal{P}^{\gamma}=\bigsqcup_d \mathcal{P}^{\gamma}_{i;d}.$$
Let us split the set $Irr\;\vec{Q}$ of indecomposable representations of $\vec{Q}$ into $\{S_i\} \cup Irr'\;\vec{Q}$ where $Irr'\;\vec{Q}=Irr\;\vec{Q} \backslash \{S_i\}$. We also let $Rep^i\;\vec{Q}$ be the full subcategory of $Rep\;\vec{Q}$ whose objects are isomorphic to $S_i^{\oplus l}$ for some $l$; likewise, we define $Rep'\;\vec{Q}$ to be the full subcategory of $Rep\;\vec{Q}$ whose objects are direct sums of indecomposables in $Irr'\;\vec{Q}$. Note that, because $i$ is a sink, 
\begin{equation}\label{E:Proof2}
Hom(M,N)=Ext^1(N,M) =\{0\} \qquad \forall\; N \in Rep^i\;\vec{Q}, M \in Rep'\;\vec{Q}.
\end{equation}
Indeed, it is enough to check (\ref{E:Proof2}) for an indecomposable $M \neq S_i$ and $N=S_i$, in which case it is obvious.
The subsets $E_{\gamma}^{i;d}$ admit a clear interpretation in terms of $Rep^i\;\vec{Q}$ and $Rep'\;\vec{Q}$~:
$$E_{\gamma}^{i;d}=\big\{ \underline{x} \in E_{\gamma}\;|\; M_{\underline{x}} \simeq S_i^{\oplus d} \oplus M', \; M' \in Rep'\;\vec{Q}\big\}.$$
We are, by (\ref{E:Proof2}), in the situation of Remark~2.13. In particular, we have 
\begin{equation}\label{E:Proof3}
\begin{split}
E_{\gamma}^{i;d} &\simeq \kappa^{-1}(E^{i;d}_{d\epsilon_i} \times E^{i;0}_{\gamma-d\epsilon_i}) \underset{P_{d\epsilon_i, \gamma-d\epsilon_i}}{\times}G_{\gamma}\\
&=(\{pt\} \times E^{i;0}_{\gamma-d\epsilon_i}) \underset{P_{d\epsilon_i, \gamma-d\epsilon_i}}{\times}G_{\gamma} \\
&=E^{i;0}_{\gamma-d\epsilon_i} \underset{P_{d\epsilon_i, \gamma-d\epsilon_i}}{\times}G_{\gamma}
\end{split}
\end{equation}
(this can also be seen directly). It follows that each $E^{i;d}_{\gamma}$ is smooth. Let $j_{\gamma}^{i;d}: E_{\gamma}^{i;d} \to E_{\gamma}$ denote the inclusion. By definition, if $\mathbb{P} \in \mathcal{P}_{i;d}^{\gamma}$ then $(j_{\gamma}^{i;d})^*\mathbb{P}$ is a simple perverse sheaf on $E^{i;d}_{\gamma}$. Because of (\ref{E:Proof3}) there is an equivalence
$$r_{\#}: D^b_{G_{\gamma}}(E^{i;d}_{\gamma}) \stackrel{\sim}{\to} D^b_{G_{\gamma-d\epsilon_i}}(E^{i;0}_{\gamma-d\epsilon_i}).$$

\vspace{.1in}

\begin{lem}\label{L:Proofkey} Using the above notation, if $\mathbb{P} \in \mathcal{P}^{\gamma}_{i;d}$ and $\mathbb{R}=(j_{\gamma-d\epsilon_i}^{i;0})_{*!}(r_{\#}\mathbb{P})$ then
\begin{equation}\label{E:Proof4}
\underline{\Delta}_{d\epsilon_i, \gamma-d\epsilon_i}(\mathbb{P}) \simeq \big(\mathbbm{1}_{d\epsilon_i} \boxtimes \mathbb{R}\big) \oplus \mathbb{Q}
\end{equation}
where $supp\; \mathbb{Q} \subset E_{d\epsilon_i} \times E_{\gamma-d\epsilon_i}^{i, \geq 1}$; similarly,
\begin{equation}\label{E:Proof5}
\underline{m}_{d\epsilon_i, \gamma-d\epsilon_i}(\mathbbm{1}_{d\epsilon_i} \boxtimes  \mathbb{R}) \simeq \mathbb{P} \oplus \mathbb{T}
\end{equation}
where $supp\;\mathbb{T} \in E_{\gamma}^{i;\geq d+1}$.
\end{lem}
\begin{proof} This is a direct application of Remark~2.13 ii) and iii). Note that for any representation $M_{\underline{x}}$ with $\underline{x} \in E_{\gamma}^{i;d}$ there exists a unique submodule $M' \subset M_{\underline{x}}$ with $M' \in Rep'\;\vec{Q}$, namely
$$M'=\bigoplus_{j \neq i} (V_{\gamma})_j \oplus Im\big( \bigoplus_{h, t(h)=i} x_h \big).$$
Note also that since $\QQ$ is stable under $\underline{\Delta}$ and $\underline{m}$ we have $\mathbb{R}, \mathbb{Q}, \mathbb{T} \in \QQ$.
\end{proof}

\vspace{.1in}

\noindent
\textit{Proof of Theorem~\ref{T:LU}.} Let $\vec{Q}$ be an arbitrary quiver (as in Section~1.1.).  By Section~3.2 there is a morphism of algebras and coalgebras $\Phi: \U^{\Z}_{\nu}(\bo'_+) \to \widetilde{\mathcal{K}}_{\vec{Q}}$. 

We first show that $\Phi$ is injective. Let $\{\{\;,\;\}\}=\Phi^* \{\,,\,\}$ be the pairing on $\U^{\Z}_{\nu}(\bo'_+)$ pulled back of $\{\,,\,\}$ via $\Phi$. It is a Hopf pairing. Moreover, by (\ref{E:22n}), (\ref{E:22new}) and (\ref{E:Drinpairing}) we have
$$\{\{E_i,E_j\}\}=\frac{\delta_{i,j}}{1-\nu^{-2}}=(E_i,E_j), \qquad \{\{ K_i,K_j\}\}=\nu^{a_{ij}}=(K_i,K_j)$$
where $(\;,\;)$ stands for Drinfeld's pairing. A homogeneous Hopf scalar product on $\U^{\Z}_{\nu}(\bo'_+) $ is uniquely determined by its values on generators --this is a consequence of the Hopf property. This means that $(\;,\;)=\{\{\;,\;\}\}$. But then 
$$(Ker\;\Phi) \cap \U^{\Z}_{\nu}(\n_+) \subset (Ker\;\{\{\;,\;\}\}) \cap \U^{\Z}_{\nu}(\n_+)=(Ker\;(\;,\;)) \cap \U^{\Z}_{\nu}(\n_+)=0$$
i.e.the restriction of $\Phi$ to $\U^{\Z}_{\nu}(\n_+)$ is injective. This implies that $\Phi$ is injective since $\U^{\Z}_{\nu}(\bo'_+)\simeq \U^{\Z}_{\nu}(\n_+) \otimes \mathbf{K}$.

Let us now prove that $\Phi$ is surjective. Again, it is enough to restrict ourselves to $\Phi: \U^{\Z}_{\nu}(\n_+) \to \KQ$. Since the image of $\Phi$ is spanned by monomials $\Phi(E_{i_1}^{(n_1)} \cdots E_{i_l}^{(n_l)})=L_{n_1\epsilon_{i_1}, \ldots, n_l \epsilon_{i_l}}$, this amounts to showing that $\KQ$ is linearly spanned (over $\Z[v,v^{-1}]$) by classes of Lusztig sheaves. We will prove this by induction on the dimension vector $\gamma$. If $\gamma \in \{n\epsilon_i\}_{i \in I, n \in \N}$ is simple then $\mathcal{P}^{\gamma}=\{\mathbbm{1}_{n\epsilon_i}\}$ and $\mathbf{b}_{\mathbbm{1}_{n\epsilon_i}}=\Phi(E_i^{(n)})$. Fix a $\gamma \in \N^I$ not of the above form, and let us assume that $\mathcal{K}^{\a} \subset Im\;\Phi$ for any $\a < \gamma$. The main idea here is to consider not just $\vec{Q}$ but \textit{all} of the orientations of the underlying graph of $\vec{Q}$ simultaneously. Any two such orientations $\vec{Q}, \vec{Q}'$ are related by a Fourier-Deligne transform $\FD: \KQ \stackrel{\sim}{\to} \mathcal{K}_{\vec{Q}'}$, which is a ring homomorphism and which preserves the Lusztig sheaves (see corollary~\ref{C:FD}). In other words, there is a commutative diagram
$$\xymatrix{ \U^{\Z}_{\nu}(\n_+) \ar[r]^-{\Phi_{\vec{Q}}} \ar[rd]_-{\Phi_{\vec{Q}'}} & \KQ \ar[d]^-{\FD} \\ & \mathcal{K}_{\vec{Q}'}}$$
Then of course an element $\mathbf{b}_{\mathbb{P}}$ belongs to $Im\;\Phi_{\vec{Q}}$ if and only if $\mathbf{b}_{\FD(\mathbb{P})}$ belongs to $Im\;\Phi_{\vec{Q}'}$.
Fox a vertex $i \in I$ and let us now choose an orientation $\vec{Q}'$ for which $i$ is a sink. Let $\mathbb{P} \in \mathcal{P}^{\gamma}_{\vec{Q}',i;d}$ with $d\geq 1$. Lemma~\ref{L:Proofkey} furnishes a simple perverse sheaf $\mathbb{R}$ on $E_{\gamma-d\epsilon_i}$ such that 
\begin{equation}\label{E:prof1}
\underline{\Delta}_{d\epsilon_i, \gamma-d\epsilon_i}(\mathbb{P}) \simeq \big(\mathbbm{1}_{d\epsilon_i} \boxtimes \mathbb{R}\big) \oplus \mathbb{Q}
\end{equation}
where $supp\; \mathbb{Q} \subset E_{d\epsilon_i} \times E_{\gamma-d\epsilon_i}^{i, \geq 1}$ and
\begin{equation}\label{E:Prof2}
\mathbbm{1}_{d\epsilon_i} \star  \mathbb{R} \simeq \mathbb{P} \oplus \mathbb{T}
\end{equation}
where $supp\;\mathbb{T} \in E_{\gamma}^{i;\geq d+1}$. Since $\mathbb{R} \in \QQ$ we have by construction $\mathbb{R} \in \mathcal{P}_{\vec{Q}',i;0}^{\gamma-d\epsilon_i}$. In $\mathcal{K}_{\vec{Q}'}$, (\ref{E:Prof2}) may be written as
$$\mathbf{b}_{\mathbb{P}}=\mathbf{b}_{\mathbbm{1}_{d\epsilon_i}} \mathbf{b}_{\mathbb{R}} - \mathbf{b}_{\mathbb{T}}.$$
We have $\mathbf{b}_{\mathbbm{1}_{d\epsilon_i}}=\Phi(E_i^{(d)})$ and by our induction hypothesis $\mathbf{b}_{\mathbb{R}} \in Im\;\Phi_{\vec{Q}'}$ so that $\mathbf{b}_{\mathbbm{1}_{d\epsilon_i}} \mathbf{b}_{\mathbb{R}} \in Im\;\Phi_{\vec{Q}'}$. Arguing by descending induction on $d$, we may assume that $\mathbf{b}_{\mathbb{T}} \in Im\;\Phi_{\vec{Q}'}$ as well (here we use the obvious fact that $E_{\gamma}^{i; d}=\emptyset$ for $d > (\gamma)_i$ ). Thus $\mathbf{b}_{\mathbb{P}} \in Im\;\Phi_{\vec{Q}'}$. To sum up, we have shown that $\mathbf{b}_{\mathbb{P}}$ is in the image of $\Phi_{\vec{Q}'}$ for any reorientation $\vec{Q}'$ of $\vec{Q}$ with a sink at a vertex $i$, and any $\mathbb{P} \in \mathcal{P}_{\vec{Q}',i, \geq 1}^{\gamma}$.

We claim that, modulo the isomorphisms provided by the Fourier-Deligne transforms, this covers \textit{all} the cases. Indeed, let $\mathbb{P}$ be any element of $\mathcal{P}^{\gamma}_{\vec{Q}}$. By construction, $\mathbb{P} \subset L^{\vec{Q}}_{\epsilon_{i_1}, \ldots, \epsilon_{i_l}}$ for some sequence of vertices $(i_1, \ldots, i_l)$. Choose $\vec{Q}'$ in which $i_1$ is a sink. Then $\Theta(\mathbb{P}) \subset \Theta( L^{\vec{Q}}_{\epsilon_{i_1}, \ldots, \epsilon_{i_l}})= L^{\vec{Q}'}_{\epsilon_{i_1}, \ldots, \epsilon_{i_l}}$. It is easy to see that $supp\;  L^{\vec{Q}'}_{\epsilon_{i_1}, \ldots, \epsilon_{i_l}} \subset E_{\gamma}^{i_1, \geq 1}$ and hence $\Theta(\mathbb{P}) \in \mathcal{P}_{\vec{Q}',i_1, \geq 1}^{\gamma}$ as wanted. We have shown that $\mathcal{K}^{\gamma} \subset Im\;\Phi$ and hence that $\Phi$ is surjective. This finishes the proof of Theorem~\ref{T:LU}.\qed

\vspace{.15in}

\noindent
\textit{Proof of Theorem~\ref{T:LUU}.} This is a direct consequence of the existence of the Fourier-Deligne transform and of its properties (see Corollary~\ref{C:FD}). 
Let $\vec{Q}, \vec{Q}'$ be two orientations of the same graph, and let $\Phi~: \U^{\Z}_{\nu}(\n_+) \stackrel{\sim}{\to} \mathcal{K}_{\vec{Q}}, \Phi'~: \U^{\Z}_{\nu}(\n_+) \stackrel{\sim}{\to} \mathcal{K}_{\vec{Q}'}$ be the two isomorphisms provided by Theorem~\ref{T:LU}. Let also $\Theta~: \mathcal{K}_{\vec{Q}} \stackrel{\sim}{\to} \mathcal{K}_{\vec{Q}'}$ be the isomorphism coming from the Fourier-Deligne transform (see Corollary~\ref{C:FD}). We have, for any vertex $i$ and any $n \geq 1$
$$(\Phi')^{-1}\Theta \Phi(E_i^{(n)}) = (\Phi')^{-1}\Theta (\mathbf{b}_{\mathbbm{1}_{n\epsilon_i}})=(\Phi')^{-1} (\mathbf{b}_{\mathbbm{1}_{n\epsilon_i}})=E_i^{(n)}.$$
This implies that the following diagram of isomorphisms is commutative~:
$$\xymatrix{
\U^{\Z}_{\nu}(\n_+) \ar[d]_-{Id} \ar[r]^-{\Phi} & \mathcal{K}_{\vec{Q}} \ar[d]^-{\Theta} \\
\U^{\Z}_{\nu}(\n_+)  \ar[r]^-{\Phi'} & \mathcal{K}_{\vec{Q}'}  }
$$
In particular, $(\Phi')^{-1}(\{ \mathbf{b}_{\mathbb{P}}\; | \mathbb{P} \in \mathcal{P}_{\vec{Q}'}\})=\Phi^{-1} \Theta^{-1}(\{ \mathbf{b}_{\mathbb{P}}\; | \mathbb{P} \in \mathcal{P}_{\vec{Q}}\})$. Since by Corollary~\ref{C:FD} $\Theta$ induces a bijection between $\mathcal{P}_{\vec{Q}}$ and $\mathcal{P}_{\vec{Q}'}$ we have $\Theta^{-1}(\{ \mathbf{b}_{\mathbb{P}}\; | \mathbb{P} \in \mathcal{P}_{\vec{Q}'}\})=\{ \mathbf{b}_{\mathbb{P}}\; | \mathbb{P} \in \mathcal{P}_{\vec{Q}}\}$ from which we finally get
$$\mathbf{B}'=(\Phi')^{-1}(\{ \mathbf{b}_{\mathbb{P}}\; | \mathbb{P} \in \mathcal{P}_{\vec{Q}'}\})=(\Phi)^{-1}(\{ \mathbf{b}_{\mathbb{P}}\; | \mathbb{P} \in \mathcal{P}_{\vec{Q}}\})=\mathbf{B}$$
as wanted.
\qed

\vspace{.15in}

\noindent
\textit{Proof of Proposition~\ref{P:LU}.} The proof follows the same lines as that of Theorem~\ref{T:LU}. We argue by induction on the dimension vector $\gamma$, and use the Fourier-Deligne transform to reduce ourselves to some $\mathbb{P} \in \mathcal{P}^{\gamma}_{\vec{Q}',i;d}$ with $d \geq 1$. Then from
$$\mathbbm{1}_{d\epsilon_i} \star \mathbb{R} = \mathbb{P} \oplus \mathbb{T}$$
and $D \mathbbm{1}_{d\epsilon_i} =\mathbbm{1}_{d\epsilon_i}$, $D \mathbb{R}=\mathbb{R}$ we deduce that $D ( \mathbb{P} \oplus \mathbb{T})=\mathbb{P} \oplus \mathbb{T}$ and finally, using the support condition on $\mathbb{T}$, that $D \mathbb{P} = \mathbb{P}$. Note that Verdier duality almost commutes with Fourier-Deligne transforms (see Theorem~\ref{T:FDT}). \qed

\vspace{.15in}

\noindent
\textit{Proof of Theorem~\ref{T:Rep}.} Let $\lambda$ be integral antidominant. The annihilator of the lowest weight vector $v_{\lambda}$ in $\U_v^{\Z}(\n_+)$ is the left ideal generated by the elements $\{E_i^{((\lambda, \epsilon_i))}\;|\; i \in I\}$. It is therefore enough to prove that for any $k \in \N$ and any $i \in I$, there is a subset $\mathbf{B}_{i, \geq k}$ of $\mathbf{B}$ such that
$$\U_{\nu}^{\Z}(\n_+) E_i^{(k)} =\bigoplus_{\mathbf{b} \in \mathbf{B}_{i,\geq k}} \Z[\nu,\nu^{-1}] \mathbf{b}.$$
Indeed, setting $\mathbf{B}_{\lambda}=\mathbf{B} \backslash \bigcup_i \mathbf{B}_{i, \geq (\lambda,\epsilon_i)}$ we will have
$\mathbf{b} \cdot v_{\lambda}=0$ if $\mathbf{b} \not\in \mathbf{B}_{\lambda}$ while 
$$\big\{ \mathbf{b} \cdot v_{\lambda}\;|\; \mathbf{b} \in \mathbf{B}_{\lambda} \big\}$$ is a $\Z[\nu,\nu^{-1}]$-basis of $V_{\lambda}$. 

\vspace{.1in}

Now let $\vec{Q}$ be a quiver associated to $\g$, which we may choose to be a \textit{source} at the vertex $i$. We set
$$E_{\gamma}^{i;d}=\big\{ \underline{x} \in E_{\gamma}\;|\; dim \big( Ker ( \bigoplus_{h, s(h)=i} x_h)\big)=d\big\}.$$
Again, $E_{\gamma}^{i;d}$ is smooth, locally closed, and $E_{\gamma}^{i; \geq d}=\bigsqcup_{d'>d} E_{\gamma}^{i;d'}$ is closed in $E_{\gamma}$. In a way entirely similar to the case of a sink we may define a partition $\mathcal{P}^{\gamma}=\bigsqcup_{d} \mathcal{P}^{\gamma}_{i;d}$ (using (\ref{E:Proof1})), an equivalence $r_{\#}: D^b_{G_{\gamma}}(E^{i;d}_{\gamma}) \stackrel{\sim}{\to} D^b_{G_{\gamma-d\epsilon_i}}(E^{i;0}_{\gamma-d\epsilon_i})$ (using (\ref{E:Proof3})), and we have

\vspace{.1in}

\begin{lem}\label{L:Proofkey2} If $\mathbb{P} \in \mathcal{P}^{\gamma}_{i;d}$ and $\mathbb{R}=(j_{\gamma-d\epsilon_i}^{i;0})_{*!}(r_{\#}\mathbb{P})$ then
\begin{equation}\label{E:Proof4bis}
\underline{\Delta}_{ \gamma-d\epsilon_i,d\epsilon_i}(\mathbb{P}) \simeq \big( \mathbb{R}\boxtimes \mathbbm{1}_{d\epsilon_i} \big) \oplus \mathbb{Q}
\end{equation}
where $supp\; \mathbb{Q} \subset E_{\gamma-d\epsilon_i}^{i, \geq 1} \times  E_{d\epsilon_i} $; similarly,
\begin{equation}\label{E:Proof5bis}
\underline{m}_{ \gamma-d\epsilon_i,d\epsilon_i}( \mathbb{R}\boxtimes \mathbbm{1}_{d\epsilon_i}) \simeq \mathbb{P} \oplus \mathbb{T}
\end{equation}
where $supp\;\mathbb{T} \in E_{\gamma}^{i;\geq d+1}$.
\end{lem}

We claim that the set $\mathbf{B}_{i,\geq k}=\big\{\Phi^{-1}(\mathbf{b}_{\mathbb{P}})\;|\; \mathbb{P} \in  \bigsqcup_{\gamma} \mathcal{P}^{\gamma}_{i;\geq k}\big\}$ fits our needs. To see this, observe that by Lemma~\ref{L:Proofkey2} we have $supp (\mathbb{R} \star \mathbbm{1}_{k\epsilon_i}) \subset E_{\gamma}^{i;\geq k}$ for any $\mathbb{T} \in \mathcal{P}^{\gamma-k\epsilon_i}$, and hence 
$$\mathbf{b}_{\mathbb{R}} \mathbf{b}_{\mathbbm{1}_{k\epsilon_i}}\subset \bigoplus_{\mathbb{S} \in \mathcal{P}^{\gamma}_{i;\geq k}} \Z[v,v^{-1}] \mathbf{b}_{\mathbb{S}},$$ 
which implies that $\U_{\nu}^{\Z}(\n_+) E_i^{(k)} \subset \bigoplus_{\mathbf{b} \in \mathbf{B}_{i;\geq k}} \Z[\nu,\nu^{-1}] \mathbf{b}$. To get the reverse inclusion, we argue by induction. Fix a dimension vector $\gamma$, a perverse sheaf $\mathbb{P} \in \mathcal{P}^{\gamma}_{i ; d}$ with $d \geq k$, and assume that $\Psi^{-1}(\mathbf{b}_{\mathbb{T}}) \in \U_{\nu}^{\Z}(\n_+) E_i^{(k)}$ for any $\mathbb{T} \in \mathcal{P}^{\gamma}_{i; \geq d+1}$. By Lemma~\ref{L:Proofkey2} again, there exists a complex $\mathbb{R} \in \mathcal{Q}_{\gamma-d\epsilon_i}$ such that $\mathbb{R} \star \mathbbm{1}_{d\epsilon_i} \simeq \mathbb{P} \oplus \mathbb{T}$ where $supp(\mathbb{T}) \subset E_{\gamma}^{i;\geq d+1}$. By the induction hypothesis, $\Psi^{-1}(\mathbf{b}_{\mathbb{T}}) \in \U_{\nu}^{\Z}(\n_+) E_i^{(k)}$, from which we deduce that $\Psi^{-1}(\mathbf{b}_{\mathbb{R}}) \in \U_{\nu}^{\Z}(\n_+) E_i^{(k)}$ as well. We are done. \qed

\vspace{.2in}

To conclude this Section, we draw one final important consequence of Theorem~\ref{T:LU}.

\vspace{.1in}

\begin{cor}\label{C:bialg} The (co)algebra  $\widetilde{\KQ}$ is a bialgebra, i.e., for any $u,v \in \widetilde{\KQ}$ we have
$$\Delta(uv)=\Delta(u)\Delta(v).$$
\end{cor}

Equivalently, $\KQ$ is a {\rm{twisted}} bialgebra (see \cite[Lecture~1]{Trieste} ). Of course, Corollary~\ref{C:bialg} directly follows from the identification $\U_\nu^{\Z}(\bo'_+) \simeq \widetilde{\KQ}$. But it also follows from the fact (proved in this section) that $\KQ$ is linearly generated by the classes of Lusztig sheaves together with Lemma~\ref{L:coprodun}. 

Indeed, in the notations of (\ref{E:42}) we have, for any two collections $\underline{\a}'=(\a'_1, \ldots, \a'_n), \underline{\a}''=(\a''_1, \ldots, \a''_m)$ of simple dimension vectors,

\begin{equation*}
\begin{split}
\Delta(\mathbf{b}_{L_{\underline{\a}'}} \cdot \mathbf{b}_{L_{\underline{\a}''}}) 
&=\Delta(\mathbf{b}_{L_{\a'_1, \ldots, \a'_n,\a''_1, \ldots, \a''_m}})\\
&=\sum_{\underline{\beta},\underline{\gamma}} v^{d_{\underline{\beta},\underline{\gamma}}} \mathbf{b}_{L_{\underline{\beta}}} \mathbf{k}_{wt(\underline{\gamma})} \otimes \mathbf{b}_{L_{\underline{\gamma}}}\\
&=\big( \sum_{\underline{\beta}',\underline{\gamma}'} v^{d_{\underline{\beta}',\underline{\gamma}'}} \mathbf{b}_{L_{\underline{\beta}'}} \mathbf{k}_{wt(\underline{\gamma}')} \otimes \mathbf{b}_{L_{\underline{\gamma}'}}\big) \big(\sum_{\underline{\beta}'',\underline{\gamma}''} v^{d_{\underline{\beta}'',\underline{\gamma}''}} \mathbf{b}_{L_{\underline{\beta}''}} \mathbf{k}_{wt(\underline{\gamma}'')} \otimes \mathbf{b}_{L_{\underline{\gamma}''}}\big)\\
&=\Delta (\mathbf{b}_{L_{\underline{\a}'}}) \Delta(\mathbf{b}_{L_{\underline{\a}''}})
\end{split}
\end{equation*}
where $\underline{\beta}=(\beta_1, \ldots, \beta_{n+m}), \underline{\gamma}=(\gamma_1, \ldots, \gamma_{n+m})$ run among the set of tuples of simple dimension vectors satisfying
$$\beta_i+\gamma_i=\a'_i, \qquad \text{for}\; i=1, \ldots, n, $$
$$\beta_{n+i}+\gamma_{n+i}=\a''_i, \qquad \text{for}\; i=1, \ldots, m, $$
where $\beta'_i=\beta_i, \gamma'_i=\gamma_i$ for $i=1, \ldots, n$ and $\beta''_i =\beta_{i+n}, \gamma''_i=\gamma_{i+n}$ for $i=1, \ldots, m$, and where we have used the identity
$$d_{\underline{\beta},\underline{\gamma}}=d_{\underline{\beta}', \underline{\gamma}'} + d_{\underline{\beta}'', \underline{\gamma}''} -(wt(\underline{\gamma}'), wt(\underline{\beta}'')).$$

\vspace{.2in}

\centerline{\textbf{3.6. The Lusztig graph.}}
\addcontentsline{toc}{subsection}{\tocsubsection {}{}{\; 3.6. The Lusztig graph.}}

\vspace{.15in}

The reduction process used in the proof of Theorem~\ref{T:LU} may be encoded in a nice and compact fashion as a colored graph in the following way. Let $i \in I$ be a vertex. Let us choose a reorientation $\vec{Q}'$ of $\vec{Q}$ for which $i$ is a sink, and let $\FD=\mathcal{P}_{\vec{Q}}^{\gamma} \stackrel{\sim}{\to} \mathcal{P}_{\vec{Q}'}^{\gamma}$ be the corresponding Fourier-Deligne isomorphism. Lemma~\ref{L:Proofkey} sets up a bijection $r_{\gamma}^{i;d}: \mathcal{P}_{\vec{Q}',i;d}^{\gamma} \stackrel{\sim}{\to} \mathcal{P}_{\vec{Q}',i;0}^{\gamma-d\epsilon_i}$ for any $\gamma \in \N^I$ and any $d \geq 0$. We use these to define a map $f_i^{\#}: \mathcal{P}_{\vec{Q}'}^{\gamma} \to \mathcal{P}_{\vec{Q}'}^{\gamma-\epsilon_i} \cup \{0\}$ by
$$f_i^{\#} (\mathbb{Q})=
\begin{cases} \{0\} & if \;\mathbb{Q} \in \mathcal{P}_{\vec{Q}',i;0}^{\gamma},\\ (r_{\gamma-d\epsilon_i}^{i;d-1})^{-1} \circ r_{\gamma}^{i;d} (\mathbb{Q}) & if\; \mathbb{Q} \in \mathcal{P}_{\vec{Q}',i;d}^{\gamma}\;with\;d \geq 1.
\end{cases}
$$
and a map $e_i^{\#}: \mathcal{P}_{\vec{Q}'}^{\gamma} \to \mathcal{P}_{\vec{Q}'}^{\gamma+\epsilon_i} $ by
$$e_i^{\#} (\mathbb{Q})= (r_{\gamma-d\epsilon_i}^{i;d+1})^{-1} \circ r_{\gamma}^{i;d} (\mathbb{Q}) \qquad if\; \mathbb{Q} \in \mathcal{P}_{\vec{Q}',i;d}^{\gamma}
$$

Note that the map $e_i^{\#}$ is the inverse of $f_i^\#$; that is we have $e_i^\#(\mathbb{P})=\mathbb{P}'$ if $f_i^\#(\mathbb{P}')=\mathbb{P}$. 
We also denote by the same letter $f_i^{\#}: \mathcal{P}_{\vec{Q}}^{\gamma} \to \mathcal{P}_{\vec{Q}}^{\gamma-\epsilon_i} \cup \{0\}$ the map obtained by \textit{transport de structure} via $\FD$, $\FD'$, and likewise for $e_i^\#$. This notation is justified by the fact $f_i^{\#}, e_i^\#$ thus defined are independent of the choice of $\vec{Q}'$. Indeed, any two orientations $\vec{Q}', \vec{Q}''$ for which $i$ is a sink are related by a Fourier-Deligne transform which fixes any arrow adjacent to $i$. In particular, the subsets $E_{\gamma}^{i;d}$ are invariant under these Fourier-Deligne transforms, and the maps $r_{\gamma}^{i;d}$ are unambiguously determined by (\ref{E:Proof4}) and (\ref{E:Proof5}).

\vspace{.1in}

The collection of maps $e_i^\#, f_i^{\#}$ for all $i \in I$ define an $I$-colored graph $\mathcal{C}_{\vec{Q}}$ whose vertex set is $\mathcal{P}_{\vec{Q}}$. We (not very originally) call this graph the \textit{Lusztig graph} of $\vec{Q}$ or $\mathbf{B}$. It is by construction invariant under Fourier-Deligne transforms.

\vspace{.2in}

\addtocounter{theo}{1}
\noindent \textbf{Example \thetheo .} We continue with the Kronecker quiver with its two orientations $\vec{Q},\vec{Q}'$, and the dimension vector $\gamma=2\delta$ (see Example~3.18). The partitions of $\mathcal{P}_{\vec{Q}}^{\gamma}, \mathcal{P}_{\vec{Q}'}^{\gamma}$ are as follows~:
\begin{align*}
\mathcal{P}_{\vec{Q},2;2}^{\gamma}&=\big\{ IC(\{0\})\big\},\\
\mathcal{P}^{\gamma}_{\vec{Q},2;1}&=\big\{ IC(S_{P_2 \oplus I_{1^2,2}}), IC(S_{P_2,1,I_1})\big\}\\
\mathcal{P}^{\gamma}_{\vec{Q},2;0}&=\big\{IC(U_{2\delta},\mathcal{L}_{triv}), IC(U_{2\delta},\mathcal{L}_{sign}), IC(S_{P_{1,2^2} \oplus I_1})\big\},
\end{align*}
and
\begin{align*}
\mathcal{P}^{\gamma}_{\vec{Q}',1;2}&=\big\{ IC(\{0\})'\big\},\\
\mathcal{P}^{\gamma}_{\vec{Q}',1;1}&=\big\{ IC(S_{P_1 \oplus I_{1,2^2}})', IC(S_{P_1,1,I_2})'\big\}\\
\mathcal{P}^{\gamma}_{\vec{Q}',1;0}&=\big\{IC(U_{2\delta},\mathcal{L}_{triv})', IC(U_{2\delta},\mathcal{L}_{sign})', IC(S_{P_{1^2,2} \oplus I_2})'\big\}.
\end{align*}

\vspace{.1in}

Comparing with Example~3.18, we see that the Fourier-Deligne transforms exchange $\mathcal{P}^{\gamma}_{\vec{Q},2; \geq 1}$ with $\mathcal{P}^{\gamma}_{\vec{Q}',1; 0}$ and $\mathcal{P}^{\gamma}_{\vec{Q},2; 0}$ with $\mathcal{P}^{\gamma}_{\vec{Q}',1; \geq 1}$.   

\vspace{.1in}

The piece of the Lusztig graph $\mathcal{C}_{\vec{Q}}$ in dimensions up to $2\delta$ is as follows (we only draw the $e_i^\#$s since the $f_i^\#$s are just obtained by reversing the arrows)~:

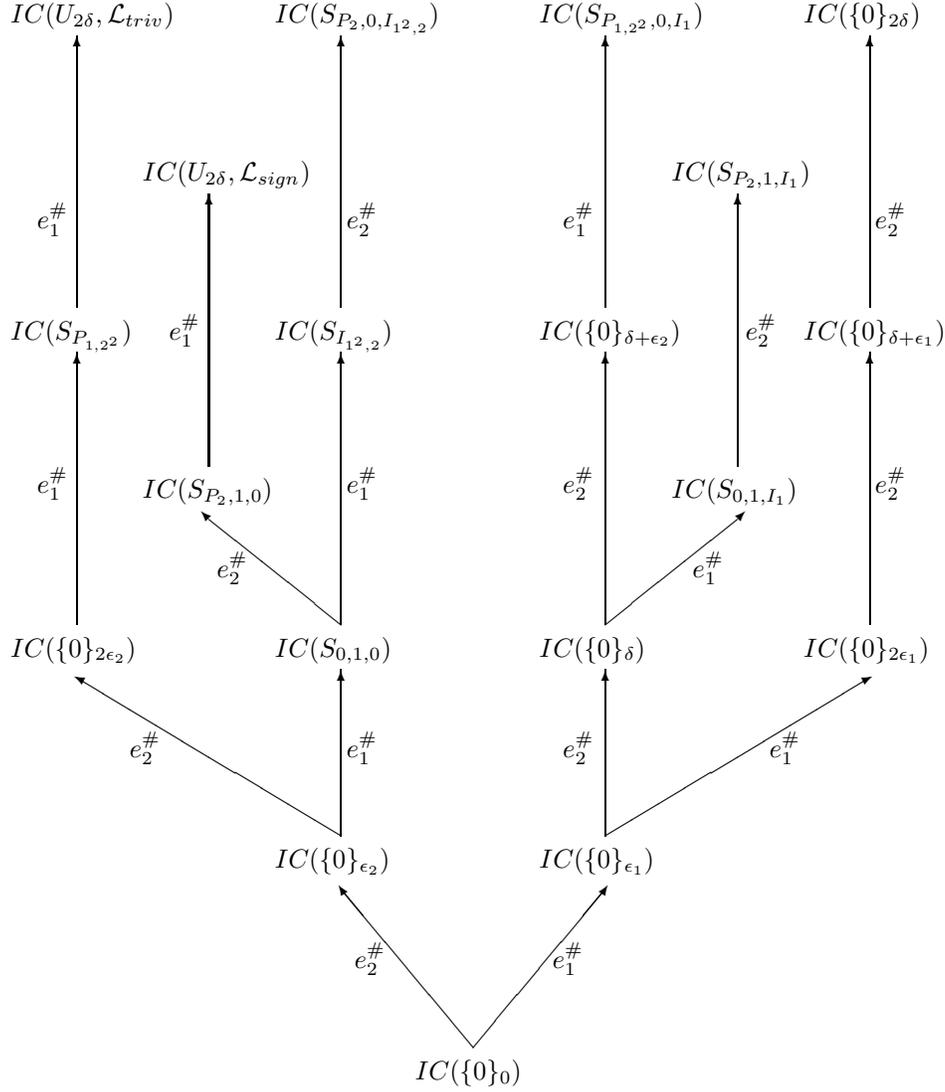
\centerline{
\begin{picture}(350, 450)
\put(175,12){\line(-5,6){50}}
\put(130,42){$e^{\#}_2$}
\put(205,42){$e^{\#}_1$}
\put(45,122){$e^{\#}_2$}
\put(127,122){$e^{\#}_1$}
\put(287,122){$e^{\#}_1$}
\put(209,122){$e^{\#}_2$}
\put(209,222){$e^{\#}_2$}
\put(127,222){$e^{\#}_1$}
\put(209,322){$e^{\#}_1$}
\put(127,322){$e^{\#}_2$}
\put(10,222){$e^{\#}_1$}
\put(327,222){$e^{\#}_2$}
\put(10,322){$e^{\#}_1$}
\put(327,322){$e^{\#}_2$}
\put(60,280){$e^{\#}_1$}
\put(278,280){$e^{\#}_2$}
\put(78,190){$e^{\#}_2$}
\put(258,190){$e^{\#}_1$}
\put(175,12){\line(5,6){50}}
\put(129,66){\vector(-2,3){5}}
\put(221,66){\vector(2,3){5}}
\put(125,92){\line(-5,3){100}}
\put(225,92){\line(5,3){100}}
\put(29,150){\vector(-2,1){5}}
\put(321,150){\vector(2,1){5}}
\put(125,92){\line(0,1){60}}
\put(225,92){\line(0,1){60}}
\put(125,270){\vector(0,1){5}}
\put(225,270){\vector(0,1){5}}
\put(125,172){\line(0,1){100}}
\put(225,172){\line(0,1){100}}
\put(125,390){\vector(0,1){5}}
\put(225,390){\vector(0,1){5}}
\put(125,292){\line(0,1){100}}
\put(225,292){\line(0,1){100}}
\put(125,150){\vector(0,1){5}}
\put(225,150){\vector(0,1){5}}
\put(25,172){\line(0,1){100}}
\put(325,172){\line(0,1){100}}
\put(25,390){\vector(0,1){5}}
\put(325,390){\vector(0,1){5}}
\put(25,292){\line(0,1){100}}
\put(325,292){\line(0,1){100}}
\put(25,270){\vector(0,1){5}}
\put(325,270){\vector(0,1){5}}
\put(75,232){\line(0,1){100}}
\put(275,232){\line(0,1){100}}
\put(75,330){\vector(0,1){5}}
\put(275,330){\vector(0,1){5}}
\put(125,172){\line(-5,4){50}}
\put(225,172){\line(5,4){50}}
\put(77,210){\vector(-1,1){5}}
\put(273,210){\vector(1,1){5}}
\put(153,0){$IC(\{0\}_0)$}
\put(100,80){$IC(\{0\}_{\epsilon_2})$}
\put(200,80){$IC(\{0\}_{\epsilon_1})$}
\put(0,160){$IC(\{0\}_{2\epsilon_2})$}
\put(100,160){$IC(S_{0,1,0})$}
\put(200,160){$IC(\{0\}_\delta)$}
\put(300,160){$IC(\{0\}_{2\epsilon_1})$}
\put(50,220){$IC(S_{P_2,1,0})$}
\put(250,220){$IC(S_{0,1,I_1})$}
\put(50,340){$IC(U_{2\delta},\mathcal{L}_{sign})$}
\put(250,340){$IC(S_{P_2,1,I_1})$}
\put(0,280){$IC(S_{P_{1,2^2}})$}
\put(100,280){$IC(S_{I_{1^2,2}})$}
\put(200,280){$IC(\{0\}_{\delta+\epsilon_2})$}
\put(300,280){$IC(\{0\}_{\delta+\epsilon_1})$}
\put(0,400){$IC(U_{2\delta},\mathcal{L}_{triv})$}
\put(100,400){$IC(S_{P_2,0,I_{1^2,2}})$}
\put(200,400){$IC(S_{P_{1,2^2},0,I_1})$}
\put(300,400){$IC(\{0\}_{2\delta})$}
\end{picture}}

\vspace{.3in}

In this particular case, $\mathcal{C}_{\vec{Q}}$ is a tree. Of course, this is \textit{not} true in general !
\endexample

\vspace{.2in}

\centerline{\textbf{3.7. The trace map and purity.}}
\addcontentsline{toc}{subsection}{\tocsubsection {}{}{\; 3.7. The trace map and purity.}}

\vspace{.15in}

Let $\vec{Q}$ be a quiver as in Section~1.1. We will now take into account some finer structure of the perverse sheaves in $\mathcal{P}_{\vec{Q}}$. Recall that we are working over the algebraic closure $k=\overline{\mathbb{F}_q}$ of the finite field $\mathbb{F}_q$. However, all the spaces used in the construction of the category $\QQ$ and all the maps between these spaces are defined over $\mathbb{F}_q$. If $X$ is such a space then we denote by $X^0$ the corresponding $\mathbb{F}_q$-space, so that $X = X^0 \otimes k$. For instance $V^0_{\a}=\bigoplus_i \fq^{\a_i}$,
$$E^0_{\alpha}=\bigoplus_{h \in H} Hom\big(\fq^{\a_{s(h)}}, \fq^{\a_{t(h)}}\big),$$
$$E^0_{\a,\beta}=\big\{ (\underline{y}, W)\;|\; \underline{y} \in E^0_{\a+\beta}; W \subset V^0_{\a+\beta}, \underline{dim}\;W=\beta; \underline{y}(W) \subset W\big\}$$
and
$$G_{\gamma}^0=\prod_i GL(\a_i, \fq).$$ 
We will apply the same notational rule to maps between spaces $X^0, Y^0$ and the induced maps between spaces $X,Y$. There is a Galois action of $Gal(k/\fq)$ on $X$. Let $F \in Gal(k/\fq)$ be the geometric Frobenius. Then $X^F=X^0$.

\vspace{.1in}

Recall that a \textit{Weil structure} on a constructible complex $\mathbb{P}$ over $X$ is an isomorphism $j: \mathbb{P} \stackrel{\sim}{\to} F^* \mathbb{P}$. A \textit{Weil complex} is a pair $(\mathbb{P},j)$ as above. The set of Weil complexes form a triangulated category in which the usual functors of pullback, pushforward, tensor products, etc. exist. Again, \cite{KW} is a good source for these technical matters.

\vspace{.1in}

If $(\mathbb{P},j)$ is a Weil complex and if $x^0 \in X^0(\fq)$ is an $\fq$-rational point of $X^0$ then there is an action of $F$ on the stalk of $\mathbb{P}_{|x^0}$ of $\mathbb{P}$ at $x^0$. The trace of $(\mathbb{P},j)$ is defined to be the $\C$-valued function
\begin{equation}\label{E:Trace}
\begin{split}
Tr (\mathbb{P})~: X^0(\fq) &\to \C\\
x^0 &\mapsto \sum_{i} (-1)^i Tr(F, H^i(\mathbb{P})_{|x^0}).
\end{split}
\end{equation}
In the above, we have fixed once and for all an identification $\qlb \stackrel{\sim}{\to} \C$.
For $n \in \Z$ we let $\qlb(n/2)$ denote the \textit{Tate twist} for $n \in \Z$. 
 It is the constant complex (over the point) with the Weil structure $j: \qlb(n/2) \stackrel{\sim}{\to} F^* \qlb(n/2)$ so that $Tr(\qlb(n))=q^{-n/2}$.
We write $(n/2)$ instead of $\otimes \qlb(n/2)$. 
Thus $Tr(\mathbb{P}(n/2))=\sqrt{q}^{n}Tr(\mathbb{P})$.When taking into account Frobenius actions, we always add a Tate twist $(n/2)$ to a shift $[n]$ in the derived category of constructible sheaves over $X$. For instance, using the notation of Section~1.3., the induction and restriction functors are given by formulas
$$\underline{m}(\mathbb{P})= q_!r_{\#}p^*(\mathbb{P})[dim\;p](dim\;p/2)$$
$$\underline{\Delta}(\mathbb{P})=\kappa_! \iota^*(\mathbb{P})[-\langle \a,\beta \rangle](-\langle \a,\beta\rangle /2).$$
We refer to \cite{FK} for the notions of pure, mixed, and pointwise pure (Weil) complexes.

\vspace{.1in}

\begin{prop}[Lusztig]\label{P:Pure} All the simple perverse sheaves $\mathbb{P} \in \mathcal{P}_{\vec{Q}}$ posses a (canonical) Weil structure, making them pure of weight zero.
\end{prop}
\begin{proof} The Lusztig sheaves $L_{\underline{\a}}=q_{\underline{\a}!}(\qlb_{E_{\underline{\a}}})$ have an obvious Weil structure (coming from that of the constant sheaf $\qlb_{E_{\underline{\a}}}$) . Moreover, they are pure of weight zero by Deligne's Theorem since $q_{\underline{\a}}$ is proper. More generally, if two complexes $\mathbb{P}', \mathbb{P}''$ possess Weil structures, then these induce one on $\mathbb{P}' \star \mathbb{P}'' $; it is pure of weight zero if $\mathbb{P}', \mathbb{P}''$ are.

Let us fix a dimension vector $\gamma$ and a perverse sheaf $\mathbb{P} \in \mathcal{P}^{\gamma}$. As in the proof of Theorem~\ref{T:LU}, we may assume that $\mathbb{P} \in \mathcal{P}^{\gamma}_{i; d}$ with $d \geq 1$ for some $i \in I$. Arguing by induction, we may in addition assume that all $\mathbb{S} \in \mathcal{P}^{\gamma'}$ with $\gamma'<\gamma$, as well as all $\mathbb{S} \in \mathcal{P}^{\gamma}_{i; \geq d+1}$ posses fixed Weil structures, and are pure of weight zero. By Lemma~\ref{L:Proofkey}, there exists $\mathbb{R} \in \mathcal{Q}^{\gamma-d\epsilon_i}$ such that $\mathbbm{1}_{d\epsilon_i} \star \mathbb{R} = \mathbb{P} \oplus \mathbb{T}$, where $supp\;\mathbb{T} \subset E_{\gamma}^{i;\geq d+1}$. It is clear that $\mathbbm{1}_{d\epsilon_i}$ has a Weil structure and is pure of weight zero, and by our assumptions the same holds for $\mathbb{R}$ and $\mathbb{T}$. It follows that the isotypical component $\mathbb{P}$ of $\mathbbm{1}_{d\epsilon_i} \star \mathbb{R}$ also has such a Weil structure, and is pure of weight zero. \end{proof}

\vspace{.1in}

Let $\mathbb{P}',\mathbb{P}'' \in \PQ$ and let us write
$$ \mathbb{P}' \star \mathbb{P}''=\bigoplus_{\mathbb{P}} M_{\mathbb{P}', \mathbb{P}''}^{\mathbb{P}} \otimes \mathbb{P}$$
where $M_{\mathbb{P}', \mathbb{P}''}^{\mathbb{P}}= Hom(\mathbb{P} ,\mathbb{P}' \star \mathbb{P}'')$ is the multiplicity complex. By Proposition~\ref{P:Pure}, $M_{\mathbb{P}', \mathbb{P}''}^{\mathbb{P}}$ has a Weil structure for which it is pure of weight zero. This means that the Frobenius eigenvalues of $H^i(M_{\mathbb{P}', \mathbb{P}''}^{\mathbb{P}})$ are all algebraic numbers of absolute value $\sqrt{q}^{i}$.

In a similar fashion, we may write
$$\underline{\Delta}(\mathbb{P})=\bigoplus_{\mathbb{P}', \mathbb{P}''}  N^{\mathbb{P}', \mathbb{P}''}_{\mathbb{P}} \otimes (\mathbb{P}' \boxtimes \mathbb{P}'')$$
for some multiplicity complex $N^{\mathbb{P}', \mathbb{P}''}_{\mathbb{P}}$. Note that (because the definition of the restriction functor involves a pushforward by a non proper map), the complex $N^{\mathbb{P}', \mathbb{P}''}_{\mathbb{P}}$ is mixed, but not pure of weight zero in general\footnote{see e.g. Lemma~\ref{L:coprodun}.}. This means that the Frobenius eigenvalues of $H^i(N^{\mathbb{P}',\mathbb{P}''}_{\mathbb{P}})$ are algebraic numbers of absolute value belonging to $\sqrt{q}^{\Z}$.

\vspace{.1in}

We may now define an algebra and a coalgebra $\mathfrak{U}_{\vec{Q}}$ as follows~: as a vector space
$$\mathfrak{U}_{\vec{Q}}=\bigoplus_{\gamma} \mathfrak{U}^{\gamma}, $$
$$\mathfrak{U}_{\gamma}=\bigoplus_{\mathbb{P} \in \mathcal{P}^{\gamma}} \C \bo_{\mathbb{P}}\;\;;$$
the multiplication and comultiplication are given by
$$\bo_{\mathbb{P}'} \cdot \bo_{\mathbb{P}''}=\sum_{\mathbb{P}} Tr(M_{\mathbb{P}', \mathbb{P}''}^{\mathbb{P}}) \bo_{\mathbb{P}}$$
$$\Delta(\bo_{\mathbb{P}})=\sum_{\mathbb{P}',\mathbb{P}''} Tr(N^{\mathbb{P}', \mathbb{P}''}_{\mathbb{P}}) \bo_{\mathbb{P}'} \otimes \bo_{\mathbb{P}''}.$$
It is easy to check that these operations are (co)associative. The same arguments as in the proof of Theorem~\ref{T:LU} show that $\mathfrak{U}_{\vec{Q}}$ is generated by $\{\bo_{\mathbbm{1}_{\epsilon_i}}\;|\; i \in I\}$ (it is not necessary to consider the divided powers $\bo_{\mathbbm{1}_{d\epsilon_i}}$ here since $\mathfrak{U}_{\vec{Q}}$ is defined over $\C$).

\vspace{.1in}

Let $\mathbf{H}_{\vec{Q}}$ be the Hall algebra of the category $Rep_{\fq}\vec{Q}$ (see \cite{Trieste}). The definition of $\H_{\vec{Q}}$ requires a choice of a square root $\nu$ of $q$. It will be convenient\footnote{the $-$ sign is taken here to balance out the signs appearing in the definition of the trace map (\ref{E:Trace}).} to choose $\nu=-\sqrt{q}$. Recall that as a vector space,
$$\H_{\vec{Q}}=\bigoplus_{\gamma} \C_{G^0_{\gamma}}[E^0_{\gamma}(\fq)]$$
where $\C_{G^0_{\gamma}}[E^0_{\gamma}(\fq)]$ denotes the set of $G^0_{\gamma}$-invariant $\C$-valued functions
on $E^0_{\gamma}(\fq)$. The multiplication and comultiplication in $\H_{\vec{Q}}$ may be written as
\begin{equation}\label{E:Hall11}
m_{\a,\beta}(f\otimes g)=\nu^{\langle \a,\beta\rangle} q^0_{!} r^0_{\#}(p^0)^* (f \boxtimes g),
\end{equation}
\begin{equation}\label{E:Hall22}
\Delta_{\a,\beta}(h)=\nu^{\langle \a,\beta\rangle -2\sum_i \a_i \beta_i} \kappa^0_! (\iota^0)^*(h)
\end{equation}
where $q^0,r^0_{\#} p^0,\kappa^0,\iota^0$ are the $\fq$-versions of the maps defined in Section~1.3. Note that $q^0_!, r_\#^0,$ etc. stand for the standard pushforward or pullback operations on spaces of \textit{functions} on sets.
The coincidence of (\ref{E:Hall11}) with the Hall multiplication as it is defined in \cite[Lecture~1]{Trieste}  is obvious. It is slightly less so for (\ref{E:Hall22}). We leave the details to the reader. For any dimension vector $\a$ we denote by $\mathbf{1}_{\a} \in \mathbf{H}_{\vec{Q}}$ the constant function on $E^0_{\a}(\fq)$.
Let $\mathbf{C}_{\vec{Q}} \subset \H_{\vec{Q}}$ denote the composition subalgebra of $\H_{\vec{Q}}$, generated by the functions $\{\mathbf{1}_{\epsilon_i}\;|\;i \in I\}$.

\vspace{.1in}

Consider the $\C$-linear map 
\begin{equation}
\begin{split}
tr: \mathfrak{U}_{\vec{Q}} &\to \H_{\vec{Q}}\\
\mathfrak{U}^{\gamma} \ni \bo_{\mathbb{P}} &\mapsto \nu^{dim\;G_{\gamma}} Tr(\mathbb{P}).
\end{split}
\end{equation}

For example, if $\mathbbm{1}_{\gamma}$ belongs to $\mathcal{P}^{\gamma}$ then $tr(\bo_{\mathbbm{1}_{\gamma}})=\nu^{\langle \gamma, \gamma \rangle} \mathbf{1}_{\gamma}$. This comes from Lemma~\ref{L:dim}.

\vspace{.1in}

\begin{theo}[Lusztig]\label{T:Lu2} The map $tr$ is an isomorphism of algebra and coalgebras onto $\mathbf{C}_{\vec{Q}}$.
\end{theo}
\begin{proof} Let us first check the compatibility of $tr$ with the product. We have $Tr \circ p^* = (p^0)^* \circ Tr$ and $Tr \circ
r_{\#}=\nu^{dim\;r} r^0_{\#} \circ Tr$. Moreover, by Grothendieck's trace formula, we have $Tr \circ q_!=q^0_! \circ Tr$. It follows that for $\mathbb{P} \in \mathcal{P}^{\a}, \mathbb{Q} \in \mathcal{P}^{\beta}$,
\begin{equation*}
\begin{split}
tr(\bo_\mathbb{P})\cdot tr(\bo_\mathbb{Q})&= \nu^{dim(G_{\a} \times G_{\beta}) + \langle \a, \beta \rangle} q^0_! r^0_{\#} (p^0)^* \big( Tr(\mathbb{P} \boxtimes \mathbb{Q})\big)\\
&=\nu^{\langle \a,\beta\rangle} Tr \big(q_!r_{\#} p^* (\mathbb{P} \boxtimes \mathbb{Q})\big)\\
&=\nu^{dim\;G_{\a+\beta}} Tr\big(q_!r_{\#} p^* (\mathbb{P} \boxtimes \mathbb{Q})[dim\;p]\big)\\
&=\nu^{dim\;G_{\a+\beta}} Tr( \mathbb{P} \star \mathbb{Q})\\
&=tr(\bo_{\mathbb{P}} \cdot \bo_{\mathbb{Q}})
\end{split}
\end{equation*}
The computation concerning the comultiplication is similar. We have $Tr \circ \kappa_!=\kappa^0_! \circ Tr$ and $ Tr \circ \iota^*=(\iota^0)^* \circ Tr$. Therefore,
\begin{equation*}
\begin{split}
\Delta_{\a,\beta}\big(tr(\bo_{\mathbb{R}})\big)&= \nu^{\langle \a,\beta \rangle -2\sum_i \a_i \beta_i} \kappa^0_! (\iota^0)^* \big(tr(\mathbb{R})\big)\\
&= \nu^{\langle \a, \beta \rangle + dim( G_{\a} \times G_{\beta})-dim\;G_{\a+\beta}}  \kappa^0_! (\iota^0)^* \big(tr(\mathbb{R})\big)\\
&=\nu^{\langle \a, \beta \rangle + dim( G_{\a} \times G_{\beta})}  \kappa^0_! (\iota^0)^* \big(Tr(\mathbb{R})\big)\\
&=\nu^{dim(G_{\a} \times G_{\beta})} Tr \big( \kappa_! \iota^* (\mathbb{R})[-\langle \a, \beta \rangle ]\big)\\
&=tr \big({\Delta}_{\a,\beta}(\bo_{\mathbb{R}})\big).
\end{split}
\end{equation*}
We thus have a well-defined algebra morphism $tr:\mathfrak{U}_{\vec{Q}} \to \H_{\vec{Q}}$. Since $\mathfrak{U}_{\vec{Q}}$ is generated by $\bo_{\mathbbm{1}_{\epsilon_i}}$ for $i \in I$, the image of $tr$ is equal to $\mathbf{C}_{\vec{Q}}$. By Ringel's theorem, (see e.g. \cite[Lecture~3]{Trieste}) $\mathbf{C}_{\vec{Q}} \simeq \U_{\nu}(\n_+)$. A comparison of graded dimensions now ensures that $tr$ is an isomorphism.
\end{proof}

\vspace{.1in}

It is interesting to compare Theorems~\ref{T:LU} and \ref{T:Lu2} : both yield
a realization of the same quantum group in terms of an algebra built out
of the perverse sheaves in $\PQ$, but the algebras in question
$\mathfrak{U}_{\vec{Q}}$ and $\KQ$ are apparently very different. Namely,
$\KQ$ is an algebra over $\Z[v,v^{-1}]$ and the structure constants ignore
any considerations of Frobenius weights, while $\mathfrak{U}_{\vec{Q}}$ is
defined over $\C$ but takes into account the Frobenius action. Also,
$\mathfrak{U}_{\vec{Q}}$ is related to the Hall algebra $\H_{\vec{Q}}$ while
there are \textit{a priori} no reason to expect such a link for $\KQ$. The
following deep (and difficult) theorem of Lusztig explains everything (see
\cite{LusHall}) :

\vspace{.1in}

\begin{theo}[Lusztig]\label{T:LUweights} For any $\mathbb{P}, \mathbb{P}',
\mathbb{P}'' \in \PQ$ and any $i\in \Z$, the Frobenius eigenvalues in
$H^i( M^{\mathbb{P}}_{\mathbb{P}',\mathbb{P}''})$ are all equal to
$\sqrt{q}^i$.
\end{theo}

\vspace{.1in}

Let us pause to reflect a little on the above Theorem. It states that all
the multiplicity complexes, which typically encode the cohomology of the
(potentially very singular) fibers of the maps $q_{\a,\beta}$, have very special Frobenius eigenvalues
(similar, say, to cohomology of varieties admitting cell decompositions).
Of course, we have checked this in a few simple examples in Lecture 2, but
Lusztig's Theorem says that this holds for \textit{any} quiver and
\textit{any} dimension vector. As a corollary of Theorem~\ref{T:LUweights}, we may
identify $\mathfrak{U}_{\vec{Q}}$ with the specialization ${\KQ}_{|v=-\sqrt{q}^{-1}}$ via $\bo_{\mathbb{P}} \mapsto \mathbf{b}_{\mathbb{P}}$, and we obtain a commutative diagram of isomorphisms
\begin{equation}\label{E:coomutatyj}
\xymatrix{
{\KQ}_{|v=-\sqrt{q}^{-1}} \ar[r]  & \mathfrak{U}_{\vec{Q}} \ar[d]_-{tr} \\
\U_{\nu}^{\Z}(\n_+)_{|\nu=-\sqrt{q}} \ar[u]^-{\Psi} \ar[r]^-{\Phi} & \mathbf{C}_{\vec{Q}}
}
\end{equation}
The isomorphism $\Phi : \U_{\nu}^{\Z}(\n_+)_{|\nu=-\sqrt{q}} =\mathbf{U}_{-\sqrt{q}}(\n_+) \stackrel{\sim}{\to} \mathbf{C}_{\vec{Q}}$ which appears here is slightly renormalized from \cite[Theorem~3.16]{Trieste}~: it is defined by $\Phi(E_i)=\nu \mathbf{1}_{\epsilon_i} (= \nu [S_i]$ in the notations of \cite{Trieste}\footnote{This renormalization comes from the (silly) fact that, in the ''stacky'' sense, we have $dim\; \underline{\mathcal{M}}^{\epsilon_i}=dim \;\{pt\}/GL(1)=-1$.}) for $i \in I$. One important consequence of (\ref{E:coomutatyj}) is that the elements of the canonical basis $\mathbf{B}=\{\Psi^{-1}(\mathbf{b}_{\mathbb{P}})\;|\; \mathbb{P} \in \PQ\}$ are realized, in the Hall algebra $\H_{\vec{q}}$, as the traces of the simple perverse sheaves $\{\mathbb{P}\;|\; \mathbb{P} \in \PQ\}$ (again, up to a simple normalization).

\vspace{.1in}

As for the Frobenius eigenvalues of the simple perverse sheaves $\mathbb{P} \in \PQ$ themselves, we state the following result proved in \cite{LusHall}~:

\vspace{.1in}

\begin{theo}[Lusztig] Let $\vec{Q}$ be a finite type quiver. For any dimension vector $\gamma$, for any $\mathbb{P} \in \mathcal{P}^{\gamma}$, for any $x^0 \in E_{\gamma}^0(\fq)$ and for any $i \in \Z$, the Frobenius eigenvalues of $H^i(\mathbb{P}_{|x^0})$ are all equal to $\sqrt{q}^i$.
\end{theo}

\vspace{.1in}

The author strongly suspects that the same holds for affine quivers, and perhaps for all other quivers as well.

\vspace{.2in}

\addtocounter{theo}{1}
\noindent \textbf{Remark \thetheo .} In defining the Hall category $\QQ$ we have chosen our ground field to be $k=\fqb$ from the start and worked with $\qlb$-coefficients. This was motivated by the desire to have a natural trace morphism to the Hall algebra $\H_{\vec{Q}}$, by passing to the finite field $\fq$ and using Frobenius eigenvalues. However, by Theorem~\ref{T:LUweights}, these Frobenius eigenvalues are all trivial and we see, \textit{a posteriori}, that we would have lost no information by working over $\C$, with $\C$-coefficients and by defining $\QQ$ and $\KQ$ in the same way as was done here (the arguments based on the Fourier-Deligne transform may be replaced by similar arguments based on the Fourier-Sato transform (see e.g. \cite{KasShap})). General comparison theorems (see \cite{BBD}) ensure that we would in fact get equivalent Hall categories $\QQ$ by working over $\C$.

\newpage

\centerline{\large{\textbf{Lecture~4.}}}
\addcontentsline{toc}{section}{\tocsection {}{}{Lecture~4.}}

\setcounter{section}{4}
\setcounter{theo}{0}
\setcounter{equation}{0}

\vspace{.2in}

This Lecture is devoted to a different aspect of the relationship between moduli spaces of representations of quivers and (quantized) enveloping algebras. Namely, we now work over the field of complex numbers $\C$ and instead of considering the moduli spaces $\underline{\mathcal{M}}_{\vec{Q}}$ themselves, we consider the cotangent bundles\footnote{An important technical point~: the correct notion of the cotangent bundle to $\underline{\mathcal{M}}_{\vec{Q}}$ is not a stack, but a \textit{derived stack}; we will only consider in this Lecture the naive, ''underived'' (or $H^0$) notion, which will be enough for our purposes} $T^* \underline{\mathcal{M}}_{\vec{Q}}$; and rather than constructible (or perverse) sheaves on $\underline{\mathcal{M}}_{\vec{Q}}$, we consider involutive (or Lagrangian) subvarieties in $T^* \underline{\mathcal{M}}_{\vec{Q}}$. Presumably, any operation in the derived category $D^b(\underline{\mathcal{M}}_{\vec{Q}})$ can be transposed to the derived category $D^b(Coh(T^*\underline{\mathcal{M}}_{\vec{Q}}))$, (or better yet,  to the so-called \textit{Fukaya category} of $T^*\underline{\mathcal{M}}_{\vec{Q}}$) via the fancy process of microlocalization (or of ''branification'' , see \cite{Nadler}). Hence one can hope to give a construction of the quantum enveloping algebra $\U_{\nu}^{\Z}(\n_+)$ as the Grothendieck group of a suitable category of coherent sheaves on $T^*\underline{\mathcal{M}}_{\vec{Q}}$ and to get in this way another realization of the canonical basis $\mathbf{B}$. This would have (at least) one important advantage over the approach given in Lectures~1--3 in that the cotangent bundle $T^* \underline{\mathcal{M}}_{\vec{Q}}$ is a more canonical object than $\underline{\mathcal{M}}_{\vec{Q}}$; it does not depend on the choice of an orientation of the quiver $\vec{Q}$ (see Section~4.2.).

\vspace{.1in}

Unfortunately, such a realization has, to the author's knowledge, not yet been worked out in the litterature\footnote{one of the reasons is perhaps that this would require exploiting the \textit{full} structure of $T^*\underline{\mathcal{M}}_{\vec{Q}}$ as a derived stack.}. Instead, we will present an important first step in that direction, due to Kashiwara and Saito (see \cite{KS})~: the construction of the crystal graph structure on $\mathbf{B}$ by means of well-chosen correspondences on a certain Lagrangian subvariety of $\underline{\Lambda}_{\vec{Q}} \subset T^*\underline{\mathcal{M}}_{\vec{Q}}$.

\vspace{.1in}

The plan of the Lecture is as follows~: after giving a brief introduction to Kashiwara's theory of crystals, we describe the cotangent space (stack) $T^*\underline{\mathcal{M}}_{\vec{Q}}$ and, following Lusztig, we construct a certain Lagrangian subvariety $\underline{\Lambda}_{\vec{Q}}$ of it (see Section~4.2). The different connected components of $\underline{\Lambda}_{\vec{Q}}$ (parametrized by dimension vectors of $\vec{Q}$) are related by some type of \textit{Hecke correspondences}. The construction of the crystal graph $\mathcal{B}(\infty)$ of $\mathbf{B}$ is itself given in Section~4.4. The precise relationship with the Hall category $\QQ$ and the set of simple perverse sheaves $\PQ$ is explained in Sections~4.5. and 4.6.

\vspace{.3in}

\centerline{\textbf{4.1. Kashiwara crystals.}}
\addcontentsline{toc}{subsection}{\tocsubsection {}{}{\; 4.1. Kashiwara crystals.}}

\vspace{.15in}

We give a quick introduction to Kashiwara's beautiful theory of crystals, for which \cite{Jo}, \cite{Kang}, \cite{Kash} are good references. Let $A=(a_{ij})_{i,j\in I}$ be a generalized Cartan matrix, which we assume to be symmetric for simplicity. Let $\Pi=\{\a_1, \ldots, \a_r\}, \Pi^{\vee}=\{h_1, \ldots, h_r\}$ be a realization of $A$, and let $\g$ denote the associated Kac-Moody algebra. We let $\g=\n_- \oplus \h \oplus \n_+$ be its Cartan decomposition, $\Delta=\Delta_+ \cup \Delta_-$ its root system and we denote by $P$ the weight lattice of $\g$.

\vspace{.2in}

\addtocounter{theo}{1}
\noindent \textbf{Definition \thetheo .} A \textit{$\g$-crystal} consists of a set $\mathcal{B}$ together with maps 
$$wt: \mathcal{B} \to P,$$
$$ \tilde{e}_i, \tilde{f}_i: \mathcal{B} \to \mathcal{B} \cup \{0\},$$
$$\varepsilon_i, \phi_i: \mathcal{B} \to \Z \cup \{-\infty\},$$
 for $i \in I$, subject to the following conditions~:
\begin{enumerate}
\item[i)] $\phi_i(b)=\varepsilon_i(b)+ \langle h_i, wt(b)\rangle$,
\item[ii)] $wt(\tilde{e}_ib)=wt(b)+\alpha_i$ if $\tilde{e}_ib \in \mathcal{B}$,
\item[iii)] $wt(\tilde{f}_ib)=wt(b)+-\alpha_i$ if $\tilde{f}_ib \in \mathcal{B}$,
\item[iv)] $\varepsilon_i(\tilde{e}_ib)=\varepsilon_i(b) -1$, $\phi_i(\tilde{e}_ib)=\phi_i(b) +1$ if $\tilde{e}_ib \in \mathcal{B}$,
\item[v)] $\varepsilon_i(\tilde{f}_ib)=\varepsilon_i(b)+1$, $\phi_i(\tilde{f}_ib)=\phi_i(b) -1$ if $\tilde{f}_ib \in \mathcal{B}$,
\item[vi)] $\tilde{f}_ib=b'$ if and only if $b=\tilde{e}_i b'$ for $b,b' \in \mathcal{B}$,
\item[vii)] if $\phi_i(b)=-\infty$ (or equivalently $\varepsilon_i(b)=-\infty$) then $\tilde{e}_ib=\tilde{f}_ib=0$.
\end{enumerate}

\vspace{.1in}

The notion of $\g$-crystal might seem complex at first glance. But note that conditions ii), iii) and iv),v) go together. These are (usually) easy to check in practice. Also, by i), $\phi_i$ can be reconstructed from $\varepsilon_i$ and $wt$, and vice versa.

\vspace{.1in}

Crystals can be thought of as some kind of combinatorial skeleton of $\g$-modules~: elements of $\mathcal{B}$ correspond to a ``basis'' of the module,  $\tilde{e}_i, \tilde{f}_i$ provide the action of the Chevalley generators $e_i, f_i$ on this ``basis'', while $wt: \mathcal{B} \to P$ gives the weight of the ``basis'' elements. 
Of course, unless $\g=\mathfrak{sl}_2$, extremely few $\g$-modules possess an actual basis preserved by the Chevalley generators $e_i,f_i$. However, this is the case for integrable lowest (or highest) weight representations for the quantum group $\U_\nu(\g)$ \textit{in the limit $\nu \mapsto 0$}. To make sense out of this last sentence, we need to introduce one more concept~: 

\vspace{.1in}

Let $M=\bigoplus_{\mu \in P} M_{\mu}$ be a highest\footnote{of course, all that follows also works with ''highest'' replaced by ''lowest'', and the roles of $E_i$ and $F_i$ interchanged.} weight integrable $\U_\nu(\g)$-module, and let us fix a simple root $\a_i \in \Pi$. The elements $E_i, F_i, K_i$ generate a subalgebra $\U_i$ of $\U_{\nu}(\g)$ which is isomorphic to $\U_{\nu}(\mathfrak{sl}_2)$. By the representation theory of $\U_i$, any element $u \in M$ may be written (in a unique fashion) as
\begin{equation}\label{E:Kashops}
u=u_0 + F_i u_1 + \cdots + F_i^{(n)} u_n
\end{equation}
where $u_k \in Ker\;E_i$ for all $k$. Define the \textit{Kashiwara operators} $\tilde{e}_i, \tilde{f}_i \in End(M)$ by
$$\tilde{e}_i(u)=\sum_{k=1}^n F_i^{(k-1)}u_k, \qquad \tilde{f}_i(u)=\sum_{k=0}^n F_i^{(k+1)}u_k.$$
 The Kashiwara operators are subtle renormalizations of the Chevalley operators. Let $\mathcal{A}_0=\C[\nu]_0 =\{f(\nu)/g(\nu)\;|\; g(0) \neq 0\}\subset \C(\nu)$ be the localization at $0$ of $\C[\nu]$. A \textit{crystal lattice} of $M$ is a free $\mathcal{A}_0$-submodule $\mathcal{L}$ of $M$ such that $\mathcal{L} \otimes_{\mathcal{A}_0} \C(\nu) = M$, $\mathcal{L}=\bigoplus_\mu \mathcal{L}_{\mu}$ where $\mathcal{L}_{\mu}=\mathcal{L} \cap M_{\mu}$, and such that $\tilde{e}_i \mathcal{L}, \tilde{f}_i \mathcal{L} \subset \mathcal{L}$ for all $i \in I$. Finally, a \textit{crystal basis} of $M$ consists of a pair $(\mathcal{L}, \mathcal{B})$ satisfying
\begin{enumerate}
\item[i)] $\mathcal{L}$ is a crystal lattice of $M$,
\item[ii)] $\mathcal{B}$ is a $\C$-basis of $\mathcal{L}/\nu \mathcal{L} \simeq \mathcal{L} \otimes_{\mathcal{A}_0} \C$,
\item[iii)] $\mathcal{B}$ is compatible with the weight decomposition $\mathcal{L}/\nu \mathcal{L}=\bigoplus_{\mu} \mathcal{L}_{\mu}/\nu \mathcal{L}_{\mu}$, 
\item[iv)] $\tilde{e}_i \mathcal{B} \subset \mathcal{B} \cup \{0\}$, $\tilde{f}_i \mathcal{B} \subset \mathcal{B} \cup \{0\}$ for all $i \in I$,
\item[v)] for any $b,b' \in \mathcal{B}$ and any $i \in I$ we have $\tilde{e}_ib=b'$ if and only if $\tilde{f}_ib'=b$.
\end{enumerate}
If $(\mathcal{L}, \mathcal{B})$ is a crystal basis of $M$ then one can indeed think of $\mathcal{B}$ as a basis of $M$ in the limit $\nu \mapsto 0$.
 
 \vspace{.1in}
 
 From a crystal basis it is easy to construct an actual crystal~: for any $i \in I$ and $b \in \mathcal{B}$ set
 $$\varepsilon_i(b)=max\{k \geq 0\;|\; \tilde{e}_i^kb \neq 0\}, \qquad \phi_i(b)=max\{k \geq 0\;|\; \tilde{f}_ib \neq 0\}.$$
 It is easy to check that $\mathcal{B}$, equipped with the obvious weight map $wt: \mathcal{B} \to  P$, $\tilde{e}_i, \tilde{f}_i$ and $\varepsilon_i, \phi_i$, is a $\g$-crystal.
 
 \vspace{.1in}
 
 Let us say that two crystal bases $(\mathcal{L}, \mathcal{B})$ and $(\mathcal{L}', \mathcal{B}')$ of a module $M$ are isomorphic if there exists an $\mathcal{A}_0$-linear isomorphism $\psi:\mathcal{L} \stackrel{\sim}{\to} \mathcal{L}'$ commuting with $\tilde{e}_i, \tilde{f}_i$ for all $i \in I$ and mapping $\mathcal{B} \to \mathcal{B}'$. It is clear that two isomorphic crystal bases produce the same crystal.
 
 \vspace{.1in}
 
The notion of crystal basis is sufficiently soft to afford many examples, and yet sufficiently rigid to have some good uniqueness properties. This is illustrated by the following fundamental result~:

\vspace{.1in}

\begin{theo}[Kashiwara] Let $M$ be an integrable $\U_{\nu}(\g)$-module belonging to category $\mathcal{O}$. Then there exists, up to isomorphism, a unique crystal basis $(\mathcal{L},\mathcal{B})$ of $M$. Moreover, if $M=V_\lambda$ is irreducible then
$$\mathcal{L}=\sum_{i_1, i_2, \ldots} \mathcal{A}_{0} \tilde{f}_{i_1} \tilde{f}_{i_2} \cdots \tilde{f}_{i_l} v_{\lambda}$$
$$\mathcal{B}=\big\{ \tilde{f}_{i_1} \tilde{f}_{i_2} \cdots \tilde{f}_{i_l} v_{\lambda}\;|\; i_1, i_2 \ldots \in I\big\} \backslash \{0\}$$
is a crystal basis of $M$.
\end{theo}

\vspace{.1in}

In particular, there is, for any integral dominant weight $\lambda \in P^+$, a well-defined $\g$-crystal $\mathcal{B}({\lambda})$. The relationship between crystal and canonical bases for highest weight integrable $\U_{\nu}(\g)$-modules, is given by

\vspace{.1in}

\begin{theo}[Kashiwara, Grojnowski-Lusztig] Let $\mathbf{B}_{\lambda} \subset V_{\lambda}$ be the canonical basis. Set
$$\mathcal{L}=\bigoplus_{\mathbf{b} \in \mathbf{B}_{\lambda}} \mathcal{A}_0 \mathbf{b}, \qquad \mathcal{B}=\{ \mathbf{b}\;mod\;\nu \mathcal{L} \;|\; \mathbf{b} \in \mathbf{B}_{\lambda}\}.$$
Then $(\mathcal{L}, \mathcal{B})$ is a crystal basis of $V_{\lambda}$.
\end{theo}

Kashiwara proved the above theorem for the global basis\footnote{which we didn't define, see e.g. \cite{Kang}.}, and Grojnowski-Lusztig proved the coincidence of the global and canonical bases. The reader will find in \cite{Kang} the explicit description of many crystals $\mathcal{B}(\lambda)$.

\vspace{.1in}

As we have said, $\mathcal{B}(\lambda)$ is some sort of combinatorial shadow of $V_{\lambda}$. For one thing, the character of $V_{\lambda}$ is simply obtained from $\mathcal{B}(\lambda)$ by
$$ch(V_{\lambda})=\sum_{\mu} \# \{ b \in \mathcal{B}(\lambda)\;|\; wt(b)=\mu\} e^{\mu}.$$
But $\mathcal{B}(\lambda)$ contains much more information. For instance, it allows one to recover tensor product multiplcities~:

\vspace{.2in}

\addtocounter{theo}{1}
\noindent \textbf{Definition \thetheo .} Let $\mathcal{B}_1, \mathcal{B}_2$ be two $\g$-crystals. The \textit{tensor product} $\mathcal{B}_1 \otimes \mathcal{B}_2$ is the set $\mathcal{B}_1 \times \mathcal{B}_2$ equipped with the following maps~:
$$ wt(b_1 \otimes b_2)=wt(b_1)+wt(b_2),$$
$$\varepsilon_i(b_1 \otimes b_2)=max\big( \varepsilon_i(b_1), \epsilon_i(b_2)-\langle h_i, wt(b_1)\rangle \big),$$
$$\phi_i(b_1 \otimes b_2)=max\big( \phi_i(b_2), \phi_i(b_2)+\langle h_i, wt(b_2)\rangle \big),$$
$$\tilde{e}_i(b_1 \otimes b_2)=\begin{cases} \tilde{e}_ib_1 \otimes b_2 & \text{if}\; \phi_i(b_1) \geq \varepsilon_i(b_2),\\ 
b_1 \otimes \tilde{e}_i b_2 & \text{if}\; \phi_i(b_1) < \varepsilon_i(b_2)
\end{cases} $$
$$\tilde{f}_i(b_1 \otimes b_2)=\begin{cases} \tilde{f}_ib_1 \otimes b_2 & \text{if}\; \phi_i(b_1) > \varepsilon_i(b_2),\\ 
b_1 \otimes \tilde{f}_i b_2 & \text{if}\; \phi_i(b_1) \leq \varepsilon_i(b_2).
\end{cases} $$
It is easy to check that $\mathcal{B}_1 \otimes \mathcal{B}_2$ is a $\g$-crystal. In other words, the category of $\g$-crystals is equipped with a monoidal structure. This monoidal structure is \textit{not} symmetric.

\vspace{.1in}

\begin{theo}[Kashiwara] For $\lambda,\mu \in P^+$ we have
$$\mathcal{B}(\lambda) \otimes \mathcal{B}(\mu)=\bigsqcup_{\sigma \in P^+} \mathcal{B}(\sigma)^{\sqcup m^{\lambda,\mu}_{\sigma}}$$
where the multiplicities $m^{\lambda,\mu}_{\sigma}$ are determined by
$$V_\lambda \otimes V_\mu=\bigoplus_{\sigma \in P^+} V_{\sigma}^{\oplus m^{\lambda,\mu}_{\sigma}}.$$
\end{theo}

\vspace{.1in}

In particular, there exists a unique embedding of crystals $\mathcal{B}(\lambda +\mu) \subset \mathcal{B}(\lambda) \otimes \mathcal{B}(\mu)$. Indeed, all the weights appearing in the tensor product $\mathcal{B}(\lambda) \otimes \mathcal{B}(\mu)$ are smaller than $\lambda + \mu$ except for the tensor product of the two highest weight elements $b_{\lambda} \otimes b_{\mu}$, hence there can be at most one embedding of crystals $\mathcal{B}(\lambda +\mu) \subset \mathcal{B}(\lambda) \otimes \mathcal{B}(\mu)$.
There is an obvious notion of a highest weight $\g$-crystal. One might wonder if there exists a characterization of $\mathcal{B}(\lambda)$ as the unique $\g$-crystal of highest weight $\lambda$ satisfying certain integrability conditions. Unfortunately, there seems to be many more $\g$-crystals than $\g$-modules and no such characterization is known. However, there does exist a very useful characterization of the \textit{collection} of all $\g$-crystals $\{\mathcal{B}(\lambda)\;|\; \lambda \in P^+\}$ due to Joseph (see \cite[Prop.~6.4.21]{Joseph}). We say that a $\g$-crystal $\mathcal{B}$ is \textit{normal} (also called \textit{semiregular}) if for any $b \in \mathcal{B}$ and any $i \in I$ we have
$$\varepsilon_i(b)=max\{k \geq 0\;|\; \tilde{e}_i^kb \neq 0\}, \qquad \phi_i(b)=max\{k \geq 0\;|\; \tilde{f}_ib \neq 0\}.$$
Note that any $\g$-crystal coming from an integrable highest weight $\g$-module is normal.

\vspace{.1in}

\begin{theo}[Joseph] There exists a unique collection of crystals $\{\mathcal{B}^{\lambda}\;|\; \lambda \in P^+\}$ satisfying
\begin{enumerate}
\item[i)] $\mathcal{B}(\lambda)$ is a normal highest weight crystal of highest weight $\lambda$,
\item[ii)] For any $\lambda,\mu \in P^+$ there exists a (unique) crystal embedding $\mathcal{B}^{\lambda+\mu} \subset \mathcal{B}^{\lambda} \otimes \mathcal{B}^{\mu}$.
\end{enumerate}
\end{theo}

\vspace{.1in}

Here is an example of a $\g$-crystal which clearly does not come from a $\g$-module. For any $i \in I$ put $\mathcal{B}_i=\{b_i(n)\;|\; n \in \Z\}$ and define maps $wt, \varepsilon_j, \phi_j, \tilde{e}_j, \tilde{f}_j$ by
\begin{align*}
&wt (b_i(n))=n\a_i,\\
&\phi_i(b_i(n))=n, \qquad \varepsilon_i(b_i(n))=-n,\\
&\phi_j(b_i(n))= \varepsilon_j(b_i(n))=-\infty \qquad (i \neq j),\\
&\tilde{e}_i(b_i(n))=b_i(n+1), \qquad \tilde{f}_i(b_i(n))=b_i(n-1),\\
&\tilde{e}_j(b_i(n))=\tilde{f}_j(b_i(n))=0 \qquad (i \neq j).
\end{align*}

Note that $\mathcal{B}_i$ is neither highest weight nor lowest weight.

\vspace{.1in}

In these notes we are not as much interested in crystals associated to $\g$-modules as to a particular crystal associated to the enveloping algebra $\U(\n_-)$ itself. In order to define it, we introduce Kashiwara operators in $\U_{\nu}(\n_-)$ as follows. One shows that for any $P \in \U_{\nu}(\n_-)$ there exists unique elements $R,S \in \U_{\nu}(\n_-)$ such that
$$[E_i,P]=\frac{K_i S-K_i^{-1}R}{\nu-\nu^{-1}}.$$
The assignement $P \mapsto R$ is a linear endomorphism of $\U_{\nu}(\n_-)$ which we denote by $e'_i$. It is a suitable substitute for the adjoint action by $E_i$. Next we decompose any $u \in \U_{\nu}(\n_-)$ in a unique fashion as
\begin{equation}\label{E:Kashops2}
u=u_0 + F_i u_1 + \cdots + F_i^{(n)} u_n
\end{equation}
where $u_k \in Ker\;e'_i$ for all $k$. We may now again define the \textit{Kashiwara operators} $\tilde{e}_i, \tilde{f}_i \in End(\U_{\nu}(\n_-))$ by
$$\tilde{e}_i(u)=\sum_{k=1}^n F_i^{(k-1)}u_k, \qquad \tilde{f}_i(u)=\sum_{k=0}^n F_i^{(k+1)}u_k.$$
The notion of a crystal basis for $\U_{\nu}(\n_-)$ is easily adapted from the module situation and we have

\vspace{.1in}

\begin{theo}[Kashiwara, Grojnowski-Lusztig]\label{T:Kassw} i) Any two crystal bases of $\U_{\nu}(\n_-)$ are isomorphic.\\
ii) Set
$$\mathcal{L}'=\sum_{i_1, i_2, \ldots} \mathcal{A}_{0} \tilde{f}_{i_1} \tilde{f}_{i_2} \cdots \tilde{f}_{i_l} 1$$
$$\mathcal{B}'=\big\{ \tilde{f}_{i_1} \tilde{f}_{i_2} \cdots \tilde{f}_{i_l} v_{\lambda}\;|\; i_1, i_2 \ldots \in I\big\} \backslash \{0\}.$$
Then $(\mathcal{L}', \mathcal{B}')$ is a crystal basis of $\U_{\nu}(\n_-)$.\\
iii) Let $\mathbf{B}$ be the canonical basis of $\U_{\nu}(\n_-)$. Put
$$\mathcal{L}''=\bigoplus_{\mathbf{b} \in \mathbf{B}} \mathcal{A}_0 \mathbf{b}, \qquad \mathcal{B}''=\{ \mathbf{b}\;mod\;\mathcal{L}'' \;|\; \mathbf{b} \in \mathbf{B}\}.$$
Then $(\mathcal{L}'', \mathcal{B}'')$ is a crystal basis of $\U_{\nu}(\n_-)$.
\end{theo}

\vspace{.1in}

Again, Kashiwara proved the above theorem for the global basis, and it translates to the canonical basis by \cite{GL}.

\vspace{.1in}

If $(\mathcal{L}, \mathcal{B})$ is a crystal basis of $\U_{\nu}(\n_-)$ we define the associated $\g$-crystal $\mathcal{B}$ by putting
$$\varepsilon_i(b)=max\{k \geq 0\;|\; \tilde{e}_i^kb \neq 0\}, \qquad \phi_i(b)=\varepsilon_i(b) + \langle h_i wt(b)\rangle.$$
This crystal is usually denoted $\mathcal{B}(\infty)$. It may be viewed as a certain limit as $\lambda \to \infty$ of the crystal $\mathcal{B}(\lambda)$, in the following sense (see \cite[Section~3.13.]{Jo} )~: there exists unique maps $\phi_{\lambda}~: \mathcal{B}(\infty) \to \mathcal{B}(\lambda) \cup \{0\}$ shifting weight by $\lambda$ and commuting with $\tilde{e}_i, \tilde{f}_i$s; the intersection of all the kernels of these maps $\phi_{\lambda}$ as $\lambda \to \infty$ is empty. 

\vspace{.1in}

We finish with a very useful characterization of $\mathcal{B}(\infty)$, proved in \cite{KS}~:

\vspace{.1in}

\begin{theo}[Kashiwara-Saito]\label{T:KSBinfty} Let $\mathcal{B}$ be a highest weight crystal of highest weight $0$ and assume that
\begin{enumerate}
\item[i)] $\varepsilon_i(b) \in \Z$ for any $b \in \mathcal{B}$ and $i \in I$,
\item[ii)] for any $i \in I$ there exists an embedding $\Psi_i :\mathcal{B} \to \mathcal{B} \otimes \mathcal{B}_i$,
\item[iii)] we have $\Psi_i(\mathcal{B}) \subset \mathcal{B} \times \{ \tilde{f}_i^k b_i(0); k \geq 0\}$,
\item[iv)] for any $b \in \mathcal{B}$ of weight $wt(b) \neq 0$ there exists $i$ such that $\Psi_i(b)=b' \otimes \tilde{f}_i^kb_i(0)$ with $k >0$.
\end{enumerate}
Then $\mathcal{B}$ is isomorphic to $\mathcal{B}(\infty)$.
\end{theo}

\vspace{.1in}

Our aim in Lecture~4 is to present Kashiwara and Saito's geometric realization of $\mathcal{B}(\infty)$ in terms of quivers. Because of our conventions, we will actually get a construction of the universal \textit{lowest weight} crystal $\mathcal{B}^+(\infty)$, associated to $\U_{\nu}(\n_+)$. Of course, $\mathcal{B}^+(\infty)$ is simply obtained from $\mathcal{B}(\infty)$ by interchanging the roles of $\tilde{e}_i, \varepsilon_i$ and $\tilde{f}_i, \phi_i$ (and replacing $wt$ by $-wt$). For the reader's convenience, we rewrite Theorem~\ref{T:KSBinfty} for $\mathcal{B}^+(\infty)$. It plays a crucial role in the approach we will use in the next sections.

\vspace{.1in}

\begin{theo}[Kashiwara-Saito]\label{T:KSBinfty2} Let $\mathcal{B}$ be a lowest weight crystal of lowest weight $0$ and assume that
\begin{enumerate}
\item[i)] $\phi_i(b) \in \Z$ for any $b \in \mathcal{B}$ and $i \in I$,
\item[ii)] for any $i \in I$ there exists an embedding $\Psi_i :\mathcal{B} \to \mathcal{B}_i \otimes \mathcal{B}$,
\item[iii)] we have $\Psi_i(\mathcal{B}) \subset \{ \tilde{e}_i^k b_i(0); k \geq 0\} \times \mathcal{B}$,
\item[iv)] for any $b \in \mathcal{B}$ of weight $wt(b) \neq 0$ there exists $i$ such that $\Psi_i(b)= \tilde{e}_i^kb_i(0)\otimes b'$ with $k >0$.
\end{enumerate}
Then $\mathcal{B}$ is isomorphic to $\mathcal{B}^+(\infty)$.
\end{theo}

\vspace{.3in}

\centerline{\textbf{4.2. Lusztig's Lagrangian.}}
\addcontentsline{toc}{subsection}{\tocsubsection {}{}{\; 4.2. Lusztig's Lagrangian.}}

\vspace{.15in}

Let $\vec{Q}$ be a fixed quiver as in Section~1.1. We will work over the field of complex numbers $\C$. We begin by describing the cotangent bundles to the moduli spaces $\underline{\mathcal{M}}^{\a}$ for $\a \in K_0(\vec{Q})$. As in Lecture~1, we will use the language of stacks only as a heuristic guide. We also want to stress that what we call and think of as the ``cotangent stack'' is only the ''underived cotangent stack'' rather than the real (derived) cotangent stack. This will suffice for all our purposes, but the reader should keep in mind that this is \textit{not} the correct notion of cotangent stack.

\vspace{.1in}

Recall that we have
$$\underline{\mathcal{M}}^{\a}=E_{\a} /G_{\a}.$$
To define $T^* \underline{\mathcal{M}}^{\a}$ we perform a symplectic (also called \textit{Marsden-Weinstein}) quotient (see \cite{Marsden}). Let $\vecdq$ be the doubled quiver of $\vec{Q}$~: it has the same vertex set $I$ but $H_{\vecdq}=H_{\vec{Q}} \sqcup \overline{H_{\vec{Q}}}$, i.e we replace each edge $h \in H_{\vec{Q}}$ by a pair of edges going in opposite orientations. If $k \in H_{\vecdq}$ then we put $\epsilon(k)=1$ if $k \in H_{\vec{Q}}$ and $\epsilon(k)=-1$ if $k \in \overline{H_{\vec{Q}}}$. Note that $\vecdq$ is independent of the choice of orientation of $\vec{Q}$. Let 
$$\overline{E}_{\a}=\bigoplus_{k \in H_{\vecdq}} \text{Hom}(V_{\a_{s(k)}}, V_{\a_{t(k)}})$$
be the space of representations of $\vecdq$ of dimension $\a$. It is a symplectic vector space with symplectic form
\begin{equation}
\begin{split}
\omega: \Lambda^2 \overline{E}_{\a} &\to \C\\
(\underline{x}, \underline{y}) &=\sum_{k \in H_{\vecdq}} Tr(\epsilon(k) x_ky_{\overline{k}}).
\end{split}
\end{equation}
In this way, $T^*E_{\a}$ gets identified with $\overline{E}_{\a}$. The $G_{\a}$-action on $E_{\a}$ gives rise to an action on $T^*E_{\a}$ which is the obvious one on $\overline{E}_{\a}$. The moment map associated to this action can be written as
\begin{equation}
\begin{split}
\mu:  \overline{E}_{\a} &\to \g_{\a}^* \simeq \g_{\a}\\
(\underline{x}) &=\sum_{k \in H_{\vecdq}} \epsilon(k) x_{\overline{k}}x_{k}.
\end{split}
\end{equation}
In the above we have identified $\g_{\a}=\bigoplus_{i} \mathfrak{gl}(\a_i,\C)$ with $\g^*_{\a}$ by means of the trace pairing. 

\vspace{.1in}

The level set $\mu^{-1}(0)$ is given by a collection of quadratic equations, one for each $i \in I$~:
$$\sum_{k \in H_{\vecdq}, s(k)=i} \epsilon(k) x_{\overline{k}}x_{k}=0.$$
There is a natural projection $\pi:\mu^{-1}(0) \to E_{\a}$. By construction, the fiber of $\pi$ over a point $(\underline{x}) \in E_{\a}$ is the orthogonal to the tangent space of $G_{\a} \cdot \underline{x}$ at the point $\underline{x}$. Another way of saying this is that $\mu^{-1}(0)$ is the union of the conormal bundles to all $G_{\a}$-orbits in $E_{\a}$. Note that by \cite{Brion}, the conormal bundle to $G_{\a} \cdot \underline{x}$ at $\underline{x}$ is canonically equal to $Ext^1(M_{\underline{x}},M_{\underline{x}})^*$, where $M_{\underline{x}}$ is the representation of $\vec{Q}$ associated to $\underline{x}$. We may naively think of the quotient $\mu^{-1}(0)/G_{\a}$ as the cotangent space (stack) $T^* \underline{\mathcal{M}}^{\a}$. One should keep in mind that $\mu^{-1}(0)$ is in general singular and reducible. Moreover, although a rapid count would give
$$dim\;\mu^{-1}(0)=dim\;\overline{E}_{\a}-dim\;\g_{\a}=2dim\;E_{\a}-dim\;G_{\a}$$
and hence $dim\;T^*\underline{\mathcal{M}}^{\a}=2\;dim E_{\a}-2 dim\;G_{\a}=2 dim\;\underline{\mathcal{M}}^{\a}$ as expected, one should be careful that $\mu$ is not submersive at $0$ in general, and $\mu^{-1}(0)$ may be of higher dimension\footnote{this is precisely why the correct notion of cotangent stack involves some higher derived terms.}. 

\vspace{.2in}

\addtocounter{theo}{1}
\noindent \textbf{Example \thetheo .} Let $\vec{Q}$ be a quiver of finite type. Then for any dimension vector $\a$, $E_{\a}$ has a finite number of $G_{\a}$-orbits. Hence $\mu^{-1}(0)=\bigsqcup_{\mathcal{O}} T^*_{\mathcal{O}}E_{\a}$ is a finite union of subvarieties of dimension $E_{\a}$. Hence $dim\;\mu^{-1}(0)=dim\;E_{\a}$ and the irreducible components of $\mu^{-1}(0)$ are parametrized by the $G_{\a}$-orbits in $E_{\a}$. It is easy to see that $\mu^{-1}(0)$ is singular as soon as $dim\;E_{\a} >0$. Note that in this case
$$dim\;\mu^{-1}(0)=dim\;E_{\a} > 2 dim\;E_{\a}-dim\;G_{\a}=dim\;E_{\a} -\langle \a, \a \rangle.$$
 \endexample

\vspace{.2in}

\addtocounter{theo}{1}
\noindent \textbf{Example \thetheo .} Let $\vec{Q}$ be the Kronecker quiver and let us consider the dimension vector $\delta=\epsilon_1 + \epsilon_2$. There is one semisimple orbit $\mathcal{O}_0$ (of dimension zero) in $E_{\a}$ and a $\mathbb{P}^1$-family of orbits $\{\mathcal{O}_{\lambda}\;|\; \lambda \in \mathbb{P}^!\}$, each of dimension one. Hence
$$\mu^{-1}(0)=T^*_{\mathcal{O}_0}E_{\a} \sqcup \bigsqcup_{\lambda} T^*_{\mathcal{O}_{\lambda}}E_{\a}$$
is of dimension three. Denoting the edges in $H_{\vec{Q}}$ by $h_1, h_2$ the equations for 
$\mu^{-1}(0)$ read
\begin{align}
x_{h_1}x_{\overline{h_1}} + x_{h_2} x_{\overline{h}_2}&=0 \tag{$1$}\\
x_{\overline{h_1}}x_{{h_1}} + x_{\overline{h_2}} x_{{h}_2}&=0 \tag{$\overline{1}$}
\end{align}
In dimension $\delta$, equations $(1), (\overline{1})$ are equivalent. One can check that $\mu^{-1}(0)$ 
is smooth away from zero and irreducible.
\endexample

\vspace{.2in}

We now introduce the main geometric character of this Lecture, which is an algebraic subvariety of $\mu^{-1}(0)$ much better behaved than $\mu^{-1}(0)$ itself. Set
\begin{equation}\label{E:deflambda}
\Lambda^{\a}=\big\{ \underline{x} \in \mu^{-1}(0)\;|\; \underline{x}\;\text{is\;nilpotent}\big\}.
\end{equation}
By nilpotent we mean that there exists $N \gg 0$ such that, for any path $k_1 \cdots k_N$ in $H_{\vecdq}$ of length $N$ the composition $x_{k_1} \cdots x_{k_n}=0$. The variety $\Lambda^{\a}$ was introduced by Lusztig in \cite{Lusaff}, and is sometimes called the \textit{Lusztig nilpotent variety}.

\vspace{.1in}

\begin{theo}[Lusztig]\label{T:LUlambda} The subvariety $\Lambda^\a \in \overline{E}_{\a}$ is Lagrangian.
\end{theo}

\vspace{.1in}

\noindent
We will prove Theorem~\ref{T:LUlambda} in Section~4.3. Here ``Lagrangian'' means by definition that $\Lambda^{\a}$ is of pure dimension $dim\;\overline{E}_{\a}/2=dim\;E_{\a}$, and that the symplectic form $\omega$ vanishes on the open set of smooth points of $\Lambda^{\a}$. 

\vspace{.1in}

One may also consider the quotient stack $\underline{\Lambda}^{\a}=\Lambda^{\a}/G_{\a}$. We have $dim\;\underline{\Lambda}^{\a}=dim\;E_{\a}-dim\;G_{\a}=-\langle \a, \a \rangle$, which is half of the (expected) dimension of $T^* \underline{\mathcal{M}}^{\a}$. We set $\underline{\Lambda}_{\vec{Q}}=\bigsqcup_{\a} \underline{\Lambda}^{\a}$.

\vspace{.2in}

\addtocounter{theo}{1}
\noindent \textbf{Example \thetheo .} Assume that $\vec{Q}$ is of finite type. Then as we have seen, $\mu^{-1}(0)$, being a finite union of conormal bundles, is already Lagrangian. This is explained by the following fact~:

\begin{prop}\label{P:finiteLag} If $\vec{Q}$ is of finite type then any $\underline{x} \in \mu^{-1}(0)$ is nilpotent.
\end{prop}
\begin{proof} We will again rely on \cite[Lemma~3.19]{Trieste} which asserts that there exists a total ordering $\prec$ on the set of indecomposable representations of $\vec{Q}$ such that $N \prec N' \Rightarrow Hom(N',N)=Ext^1(N,N')=0$, and that $Ext^1(N,N)=0$ for all $N$. Let $\underline{x}=(x_k)_{k \in H_{\vecdq}}$ be a point in $\mu^{-1}(0)$, and let us denote as usual by $M_{\pi(\underline{x})}$ the associated representation of $\vec{Q}$. We may decompose 
\begin{equation}\label{E:juyg}
M_{\pi(\underline{x})}=\bigoplus_{N} N^{\oplus d_N},
\end{equation}
and let $N_0$ be the maximal indecomposable appearing in (\ref{E:juyg}) for the order $\prec$. The subspace $N_0^{\oplus d_0} \subset V_{\a}$ is canonical since $Hom(N_0,N')=0$ for all $N' \neq N_0$, and the restriction to $N_0^{\oplus d_0}$ of the maps $x_k$ for $k \in \overline{H}_{\vec{Q}}$ are zero since $Ext^1(N',N_0)^*=0$ for $N' \neq N_0$. Because $\vec{Q}$ is of finite type, it has no oriented cycles and hence $N_0$ is itself nilpotent. Thus we can find a nontrivial subspace $W \subset N_0^{\oplus d_0}$ on which all maps $x_k$ for $k \in H_{\vecdq}$ are zero. In other words, $Ker\;\underline{x} \neq 0$. Replacing $V_{\a}$ by $V_{\a}/W$ and arguing by induction we arrive at the conclusion that $\underline{x}$ is nilpotent as wanted.
\end{proof}

\vspace{.1in}

The coincidence of $\mu^{-1}(0)$ and $\Lambda^{\a}$ for finite type quivers may seem awkward -- but remember that $\mu^{-1}(0)$ only corresponds to the degree zero part of $T^*\underline{\mathcal{M}}^{\a}$ while $\Lambda^{\a}$ corresponds to $\underline{\Lambda}^{\a}$ which is morally a Lagrangian subvariety in the \textit{true} $T^*\underline{\mathcal{M}}^{\a}$. Note moreover that, in the stacky sense, $dim\;\underline{\mathcal{M}}^{\a} \leq 0$.
\endexample

\vspace{.2in}

\addtocounter{theo}{1}
\noindent \textbf{Example \thetheo .} Let $\vec{Q}$ be the Kronecker quiver of Example~4.11. Then
$$\Lambda^{\delta}=\big\{ (x_{h_1},x_{h_2}, x_{\overline{h}_1},x_{\overline{h}_2})\;|\; x_{h_1}x_{\overline{h}_1}= x_{h_2}x_{\overline{h}_2}=x_{h_1}x_{\overline{h}_2}=x_{h_2}x_{\overline{h}_1}=0\big\}.$$
We see that $\Lambda^{\delta}=\C^2 \times \{0\} \cup \{0\} \times \C^2$. It is indeed Lagrangian and has two irreducible components.
\endexample

\vspace{.2in}

\addtocounter{theo}{1}
\noindent \textbf{Example \thetheo .} As a last example, take the cyclic quiver

\centerline{
\begin{picture}(130, 50)
\put(30,20){\circle*{5}}
\put(100,20){\circle*{5}}
\put(27,10){$1$}
\put(97,10){$2$}
\put(55,28){$h_1$}
\put(55,8){$h_2$}
\put(65,25){\vector(1,0){5}}
\put(40,25){\line(1,0){50}}
\put(70,15){\vector(-1,0){5}}
\put(40,15){\line(1,0){50}}
\put(-10,18){$\vec{Q}=$}
\end{picture}}

and choose $\delta=\epsilon_1+\epsilon_2$ for the dimension vector again. This time,
$$\Lambda^{\delta}=\big\{ (x_{h_1},x_{h_2}, x_{\overline{h}_1},x_{\overline{h}_2})\;|\; x_{h_1}x_{\overline{h}_1}= x_{h_2}x_{\overline{h}_2}=x_{h_1}x_{{h}_2}=x_{\overline{h}_1}x_{\overline{h}_2}=0\big\}.$$
We see that $\Lambda^{\delta}=(\C \times \{0\}) \times (\C \times \{0\}) \cup (\{0\} \times \C) \times (\C \times \{0\})$. It is again Lagrangian and has two irreducible components.
\endexample

\vspace{.2in}

The last two examples suggest that, like $\overline{E}_{\a}=T^*E_{\a}$, $\Lambda^{\a}$ is independent of the choice of a particular orientation for $\vec{Q}$. This is indeed true and can be easily verified. If $H_1, H_2$ are two different orientations of the same graph, define a linear isomorphism
$\Phi: \overline{E}_{\a} \stackrel{\sim}{\to} \overline{E}_{\a}$ by
$$\Phi(x_h)=\begin{cases} x_h & \text{if}\; h \in \overline{H}_1 \cup H_2\\ -x_h & \text{if}\; h \in H_1, h \notin H_2.\end{cases}$$
Then $\Phi$ restricts to an isomorphism $\Lambda^{\a}_{H_1} \stackrel{\sim}{\to} \Lambda_{H_2}^{\a}$.

\vspace{.3in}

\centerline{\textbf{4.3. Hecke correspondences.}}
\addcontentsline{toc}{subsection}{\tocsubsection {}{}{\; 4.3. Hecke correspondences.}}

\vspace{.15in}

For $i \in I$, $d \geq 1$ and $\a \in K_0(\vec{Q})$. Let us consider the variety $\Lambda_{(1)}^{\a,\a+d\epsilon_i}$ of tuples $(\underline{x}, W, \rho_{\a}, \rho_{d\epsilon_i})$ where $\underline{x} \in \Lambda^{\a+d\epsilon_i}$, $W$ is an $\underline{x}$-stable subspace of $V_{\a+d\epsilon_i}$ of dimension $\a$, and $\rho_{\a}: W \stackrel{\sim}{\to} V_{\a},\; \rho_{d\epsilon_i}: V_{\a+d\epsilon_i}/W \stackrel{\sim}{\to} V_{d\epsilon_i}$. We also define a variety $\Lambda^{\a,\a+d\epsilon_i}$ of pairs $(\underline{x},W)$ as above.
There are some natural maps
\begin{equation}\label{E:Heckelambda}
\xymatrix{  & \Lambda^{\a,\a+d\epsilon_i}_{(1)} \ar[dl]_-{p} \ar[r]^-{r} & \Lambda^{\a,\a+d\epsilon_i} \ar[dr]^-{q} &\\
\Lambda^{\a} \simeq \Lambda^{\a} \times \Lambda^{d\epsilon_i} & & & \Lambda^{\a+d\epsilon_i}}
\end{equation}
given by $p(\underline{x}, W, \rho_{\a},\rho_{d\epsilon_i})=(\rho_{\a,*}(\underline{x}_{|W}), \rho_{d\epsilon_i,*}(\underline{x}_{|V_{\a+d\epsilon_i}/W}))$, $r(\underline{x}, W, \rho_{\a},\rho_{d\epsilon_i})=(\underline{x},W)$ and $q(\underline{x}, W)=\underline{x}$. The varieties $\Lambda_{(1)}^{\a,\a+d\epsilon_i}, \Lambda^{\a,\a+d\epsilon_i}$ are often called \textit{Hecke correspondences}. They are direct analogs of the correspondences used in the definition of the induction functor of the Hall category (see Section~1.3).

\vspace{.1in}

The map $r$ is a principal $G_{\a} \times G_{d\epsilon_i}$-bundle and $q$ is proper. Observe that neither $q$ nor $p$ are usually locally trivial. This is clear for $q$, and for $p$ it comes from the fact that we are considering representations of a quiver $\vecdq$ with relations. However there exists a simple stratification of the varieties $\Lambda^{\a}$ into pieces over which $p, q$ are indeed smooth fibrations. For $l \geq 0$ and $\gamma \in K_0(\vec{Q})$ set
$$\Lambda^{\a}_{i;l}=\bigg\{ \underline{x}\;|\; codim_{(V_{\gamma})_i} \bigg( Im\big( \bigoplus_{\substack{h \in H_{\vecdq}\\ t(h)=i}} x_h \big)\bigg)=l\bigg\}.$$
As in Section~3.5, each $\Lambda^{\a}_{i;l}$ is locally closed in $\Lambda^{\a}$, and $\Lambda^{\a}_{i;0}$ is open.  The notation $\Lambda^{\a}_{i,\geq l}$ requires no explanation. We will say that an irreducible components $C$ of $\Lambda$ generically belongs to $\Lambda^{\a}_{i;l}$ if $C \cap \Lambda^{\a}_{i;l}$ is open dense in $C$. Hence any $C$ \textit{generically belongs} to $\Lambda^{\a}_{i;l}$ for some $0 \leq l \leq \a_i$. It is easy to deduce from the nilpotency condition that
\begin{equation}\label{E:fwqa}
\Lambda^{\a}=\bigcup_{i \in I} \Lambda^{\a}_{i, \geq 1}.
\end{equation}
It follows that any irreducible component $C$ of $\Lambda^{\a}$ belongs to $\Lambda^{\a}_{j;l}$ for some $j \in I$ and some $l \geq 1$. Indeed, otherwise $C \cap \Lambda^{\a}_{j;0}$ is open dense in $C$ for all $j$, and hence so is $C \cap \bigcap_j \Lambda^{\a}_{j,0}.$ But this last set is empty by (\ref{E:fwqa}).

\vspace{.1in}

\begin{lem}\label{L:lambdadim} We have $qrp^{-1}(\Lambda^{\a}_{i;l}) = \Lambda^{\a}_{i;l+d}$. Moreover,
\begin{enumerate}
\item[i)] The restriction of $q$ to $q^{-1}(\Lambda^{\a+d\epsilon_i}_{i;l})$ for $l \geq d$ is a locally
trivial fibration with fiber $Gr(l-d,l)$,
\item[ii)] The restriction of $p$ to $p^{-1}(\Lambda^{\a}_{i;s})$ for $s \geq 0$ is a smooth map with fibers isomorphic to $G_{\a+d\epsilon_i}/U_{\a,\a+d\epsilon_i} \times \C^{-d(\a,\epsilon_i) +d\a_i +ds}$, where
$U_{\a,\a+d\epsilon_i} \simeq \C^{d\a_i}$ is the unipotent radical of the parabolic subgroup of type $(\a,d\epsilon_i)$.
\end{enumerate}
\end{lem}
\begin{proof} Statement i) is clear~: the fiber of $q$ at $\underline{x} \in \Lambda^{\a+d\epsilon_i}$ is the set of $\underline{x}$-stable subspaces $W \subset V_{\a+d\epsilon_i}$ of dimension $\a$; a subspace $W$ of dimension $\a$ is $\underline{x}$-stable if and only if it contains the characteristic subspace $Im\big( \bigoplus_{\substack{t(h)=i}} x_h \big)$.\\
To prove statement ii), let us fix some $\underline{y} \in \Lambda^{\a}_{i;s}$. The fiber of $p$ at $\underline{y}$ is the set of tuples $(\underline{x}, W, \rho_{\a}, \rho_{d\epsilon_i})$ where $W, \rho_{\a}, \rho_{d\epsilon_i}$ may be chosen arbitrarily, and where $\underline{x}$ is an extension of $\rho_{\a}^*(\underline{y})$ from $W$ to $V_{\a+d\epsilon_i}$. The choice of $(W,\rho_{\a}, \rho_{d\epsilon_i})$ is given by a point of $G_{\a+d\epsilon_i}/U_{\a,\a+d\epsilon_i}$.
As for the extension, it is given by a map
$z: V_{d\epsilon_i} \to \bigoplus_{ t(h)=i} (V_{\a})_{s(h)}$, and it belongs to $\Lambda^{\a+d\epsilon_i}$ if and only if the composition
$$\xymatrix{V_{d\epsilon_i} \ar[r]^-{z} & \bigoplus_{t(h)=i} (V_{\a})_{s(h)} \ar[r]^-{\bigoplus \epsilon(\overline{h})y_{h}} & (V_{\a})_i}$$ 
vanishes, i.e. if and only if $Im(z) \subset Ker \big( \bigoplus_{\substack{ t(h)=i}} \epsilon(\overline{h})y_h\big)$
(note that the nilpotency condition for $\underline{x}$ is automatically verified). By assumption
we have $rank\big( \bigoplus_{\substack{ t(h)=i}} \epsilon(\overline{h})y_h\big)=\a_i-s$, so we deduce that $dim\;Ker\big( \bigoplus_{\substack{ t(h)=i}} \epsilon(\overline{h})y_h\big)=\sum_{t(h)=i} \a_{s(h)}-\a_i+s =-(\a,\epsilon_i) + \a_i +s$. 
\end{proof}

\vspace{.1in}

Set $q'=pr$, and write $N^{\a,\a+d\epsilon_i}=p^{-1}(\Lambda^\a_{i;0})=(q')^{-1}(\Lambda^{\a+d\epsilon_i}_{i;d})$.

\vspace{.1in}

\begin{cor}\label{C:dimlambda} We have
\begin{enumerate}
\item[i)] $q': N^{\a,\a+d\epsilon_i} \to \Lambda^{\a+d\epsilon_i}_{i;d}$ is a principal $G_{\a} \times G_{d\epsilon_i}$-bundle,
\item[ii)] $p: N^{\a,\a+d\epsilon_i} \to \Lambda^{\a}_{i;0}$ is a smooth map whose fiber is connected of dimension $\sum_{j \neq i} \a_j^2 + (\a_i+d)^2-d(\a_i,\epsilon_i)$.
\end{enumerate}
\end{cor} 

\vspace{.1in}

An important consequence of the above Corollary is that we have a canonical bijection between sets of irreducible components 
\begin{equation}\label{E:bijcomp}
\kappa^{\a}_{i;d}~: Irr\;\Lambda^{\a}_{i;0} \stackrel{\sim}{\to} Irr\;\Lambda^{\a+d\epsilon_i}_{i;d}.
\end{equation}

We are now in position to provide the

\vspace{.1in}

\noindent
\textit{Proof of Theorem~\ref{T:LUlambda}.} Let us first show that $\Lambda^{\a}$ is an isotropic subvariety of $T^*E_{\a}$. We will deduce this from the following general fact (see \cite[Section~8.4]{KasShap} for a proof)~:

\begin{prop} Let $X,Y$ be complex algebraic varieties and let $Z \subset X \times Y$ be a smooth algebraic subvariety. We assume that $Y$ is projective. Then the image of the projection
$T^*_Z(X \times Y) \cap (T^*X \times Y) \to T^*X$ is isotropic.
\end{prop}
We apply the above result to the case $X=E_{\a}, Y= \mathcal{B}_{\a}$, where $\mathcal{B}_{\a}$ is the flag variety of $G_{\a}$ parametrizing (full) flags $V_{\a} =W_1\supset W_2 \supset \cdots$. For the subvariety $Z$ we take the variety of pairs $(\underline{x}, W_{\bullet})$ satisfying $\underline{x} (W_k) \subset W_{k+1}$ for all $k$. Then we have
$$T^* X\simeq \overline{E}_{\a}, $$
$$T^*Y \simeq \big\{ (W_\bullet, (z_i)_{i \in I})\;|\; z_i \in End((V_{\a})_i),\; z_i(W_k) \subset W_{k+1}\big\},$$
\begin{equation*}
\begin{split}
T^*_Z(X \times Y)= \big\{ &(\underline{x}, W_{\bullet}, (z_i))\in \overline{E}_{\a} \times \mathcal{B}_{\a} \times \prod_i End((V_{\a})_i)\;|\\
&\; \sum_{s(k)=i} \epsilon(k) x_{\bar{k}} x_{k}=z_i\;\text{for}\;i \in I; \; x_h(W_l)\subset W_{l+1}, x_{\bar{h}}(W_l) \subset W_l\;\text{for}\;h \in H_{\vec{Q}} \big\}
\end{split}
\end{equation*}
Hence
\begin{equation*}
\begin{split}
T^*_Z(X \times Y) \cap (T^*X \times Y)=\big\{ (\underline{x}, &W_{\bullet})\in \overline{E}_{\a} \times \mathcal{B}_{\a} \;|\; \mu(\underline{x})=0;\\
 &\; x_h(W_l)\subset W_{l+1}, x_{\bar{h}}(W_l) \subset W_l\;\text{for}\;h \in H_{\vec{Q}} \big\}.
 \end{split}
 \end{equation*}
It is easy to see using the nilpotency condition that $\Lambda^{\a}$ lies in the projection to $T^*E_{\a}$ of $T^*_Z(X \times Y) \cap (T^*X \times Y)$. Therefore $\Lambda^{\a}$ is isotropic as wanted.

\vspace{.1in}

It remains to prove that all the irreducible components of $\Lambda^{\a}$ are of the correct dimension $\frac{1}{2} dim\;T^*E_{\a}=dim\;E_{\a}=-\langle \a,\a \rangle + dim\;G_{\a}$. For this we argue by induction on $\a$. If $\a \in \{l\epsilon_i\;|\; i \in I, l \in \N\}$ then we have $\Lambda^{\a}=\{pt\}$ and the Lemma is verified. Now let us fix some $\a \in K_0(\vec{Q})$ and assume that $\Lambda^{\gamma}$ is Lagrangian for all $\gamma < \alpha$. Let $C \subset Irr\;\Lambda^{\a}$ be an irreducible component and choose $i \in I$ such that $C$ generically belongs to $\Lambda^{\a}_{i;l}$ with $l \geq 1$. By Corollary~\ref{C:dimlambda} there is a bijection $\kappa^{\a-l\epsilon_i}_{i;l} : Irr\;\Lambda^{\a-l\epsilon_i}_{i;0} \stackrel{\sim}{\to} Irr\;\Lambda^{\a}_{i;l}$, and 
\begin{equation}\label{E:hgb}
\begin{split}
dim\;C&=dim\;((\kappa^{\a-l\epsilon_i}_{i;l})^{-1}(C)) + dim\;G_{\a} -l(\a-l\epsilon_i,\epsilon_i)-dim\;(G_{\a-l\epsilon_i} \times G_{l\epsilon_i}) \\
&=dim\;((\kappa^{\a-l\epsilon_i}_{i;l})^{-1}(C)+2l\a_i -l(\a,\epsilon_i).
\end{split}
\end{equation}
Since $\Lambda^{\a-l\epsilon_i}_{i;0}$ is open in $\Lambda^{\a-l\epsilon_i}$ and since by our induction hypothesis $\Lambda^{\a-l\epsilon_i}$ is Lagrangian, the dimension of any irreducible component of $\Lambda^{\a-l\epsilon_i}_{i;0}$ is equal to $-\langle \a-l\epsilon_i, \a-l\epsilon_i\rangle + dim\;G_{\a-l\epsilon_i}$. We deduce from this and (\ref{E:hgb}) that $dim\;C=-\langle \a, \a \rangle + dim\;G_{\a}$ as wanted.\qed

\vspace{.2in}

The above proof actually gives a bit more : not only does it show that al irreducible components of $\Lambda^{\a}$ are of the correct dimension, but it also shows that all the irreducible components of $\Lambda^{\a}_{i;l}$ are of the correct dimension, for any $i, l$. In other words, each irreducible component of $\Lambda^{\a}_{i;l}$ is open in some (unique) irreducible component of $\Lambda^{\a}$; in particular, for any $i \in I$ there exists a canonical bijection
\begin{equation}\label{E:bijlambda}
Irr\;\Lambda^{\a} \simeq \bigsqcup_{l \geq 0} Irr\;\Lambda^{\a}_{i;l}.
\end{equation}

\vspace{.3in}

\centerline{\textbf{4.4. Geometric construction of the crystal.}}
\addcontentsline{toc}{subsection}{\tocsubsection {}{}{\; 4.4. Geometric construction of the crystal.}}

\vspace{.15in}

We now define, for any $i \in I$,  maps $\tilde{e}_i, \tilde{f}_i$ on the set $Irr\;\Lambda =\bigsqcup_{\a} Irr\;\Lambda^{\a}$ as follows. If $C \in Irr\;\Lambda^{\a}$ generically belongs to $\Lambda^{\a}_{i;l}$ then we set
\begin{equation}\label{E:tildedef}
\begin{split}
\tilde{e}_i C&=\kappa^{\a-l\epsilon_i}_{i;l+1} \circ (\kappa^{\a-l\epsilon_i}_{i;l})^{-1} (C),\\
\tilde{f}_i C&=\begin{cases} \kappa^{\a-l\epsilon_i}_{i;l-1} \circ (\kappa_{i;l}^{\a-l\epsilon_i})^{-1} (C) & \text{if}\;l \geq 1,\\ 0 & \text{if}\;l=0.
\end{cases}
\end{split}
\end{equation}
In the above definition, we have implicitly used the identifications (\ref{E:bijlambda}). Note that if $l=0$ then we simply have $\tilde{e}_iC=\kappa^{\a}_{i;1}(C)$.  We also set $wt(C)=\a$ and
\begin{equation}\label{E:tildedef2}
\phi_i(C)= l, \qquad \varepsilon_i(C)=\phi_i(C)-\langle h_i,\a \rangle.
\end{equation}

\vspace{.1in}

There is a useful characterization of the operators $\tilde{f}_i$ based on the notion of a \textit{generic point}. By definition, a point $\underline{x}$ of an irreducible component $C \in Irr\;\Lambda^{\a}$ is generic if it does not belong to any other irreducible component, and if 
$$C \cap S_j \;dense\;in\;C \Rightarrow \underline{x} \in C \cap S_j$$
for all the locally closed stratifications $\Lambda^{\a}=\bigsqcup_j S_j$ mentioned in these lectures\footnote{we leave it to the reader to check that these stratifications form a finite set.}.

\vspace{.1in}

\begin{lem}\label{L:KSproof}
Let $C \in Irr\;\Lambda^{\a}$ and let us assume that $\phi_i(C) \geq 1$. Let $\underline{x}$ be a generic point of $C$, and let us put $X= Im\;\big( \bigoplus_{t(h)=i} x_h \big) \subset (V_{\a})_i$. Then for a generic hyperplane $H \subset (V_{\a})_i$ containing $X$, the restriction $\underline{y}$ of $\underline{x}$ to $\bigoplus_{j \neq i} (V_{\a})_j \oplus H$ is a generic point of $\tilde{f}_iC$.
\end{lem}
\begin{proof} From the definitions it follows that the restriction $\underline{z}$ of $\underline{x}$ to $\bigoplus_{j \neq i} (V_{\a})_j \oplus X$ is a generic point of $\tilde{f}_i^{\phi_i(C)}C$. It is clear that $\underline{y}$ is a generic point of its irreducible component $C'$. Since the restriction of $\underline{y}$ to $\bigoplus_{j \neq i} (V_{\a})_j \oplus X$ is also equal to $\underline{z}$, we have $\tilde{f}_i^{\phi_i(C')}C'=\tilde{f}_i^{\phi_i(C)}C$, and thus
$\tilde{f}_iC=C'$.\end{proof}

\vspace{.1in}

That being said, we have

\vspace{.1in}

\begin{prop}\label{P:KSstep0} The set $Irr\;\Lambda$, equipped with the maps $\tilde{e}_i, \varepsilon_i, \tilde{f}_i, \phi_i$ for $i \in I$ defined by (\ref{E:tildedef}), (\ref{E:tildedef2}), and the map $wt$, is a lowest weight $\g$-crystal of lowest weight $0$.
\end{prop}
\begin{proof} Axioms i)-v) of a $\g$-crystal are easily checked. We show that vi) holds~: let $C \in Irr\;\Lambda^{\a}_{i;l}$ with $l \geq 1$, and put $\tilde{f}_iC=C'$; by definition $C' \in Irr\;\Lambda^{\a-\epsilon_i}_{i;l-1}$, and we have
$$(\kappa^{\a-l\epsilon_i}_{i;l})^{-1}(C)=(\kappa^{\a-l\epsilon_i}_{i;l-1})^{-1}(C');$$
but then $\tilde{e}_iC'=C$ by definition. The last axiom vii) is automatically verified here since $\phi_i(C), \varepsilon_i(C) \neq -\infty$ for all $C$, $i$. Thus $Irr\;\Lambda$ is indeed a $\g$-crystal. We claim that it is generated by $\{0\} = Irr\; \Lambda^{0}$ under the operators $\{\tilde{e}_i\;|\; i \in I\}$. Indeed for any $C \in Irr\;\Lambda^{\a}$ with $\a >0$ there exists some $j \in I$ and $l \geq 1$ for which $C \in Irr\;\Lambda^{\a}_{j;l}$. Hence 
$C=\tilde{e}_j^{l} \tilde{f}_j^l C$ and $C$ belongs to the image of $\tilde{e}_j$. Arguing in this fashion by induction, we obtain the desired result. \end{proof}

We will soon show that $Irr\;\Lambda \simeq \mathcal{B}^+(\infty)$ (see Theorem~\ref{T:KSMain}). In order to apply the criterion in Theorem~\ref{T:KSBinfty2} we need to construct certain embeddings $\Psi_i :Irr\;\Lambda \to \mathcal{B}_i \otimes Irr\;\Lambda $. And in order to do this, we introduce a new operation on the set $Irr\;\Lambda$, namely a ''duality'' involution $C \mapsto C^*$. Let us fix an $I$-graded identification $V_{\a} \simeq V_{\a}^*$ for all $\a$. Then $\underline{x} \mapsto \hspace{.01in}^t\underline{x}$ defines an automorphism of $\overline{E}_{\a}$, which preserves $\Lambda^{\a}$ and induces a permutation of $Irr\;\Lambda^{\a}$. Observe that because $\Lambda^{\a}$ is $G_{\a}$-invariant, and because $G_{\a}$ acts (simply) transitively on the set of identifications $V_{\a} \simeq V_{\a}^*$, this permutation $*: Irr\;\Lambda^{\a} \to Irr\;\Lambda^{\a}$ is canonical. For the same reason, it is involutive, i.e. $(C^*)^* =C$. Let us now put
$$\varepsilon_i^*(C)= \varepsilon_i(C^*), \qquad \phi^*_i(C)= \phi_i(C^*),$$
$$\tilde{e}_i^*C= (\tilde{e}_iC^*)^*, \qquad \tilde{f}_i^*C= (\tilde{f}_iC^*)^*.$$
For any $i \in I$ define  a map $\Psi_i: Irr\;\Lambda \to \mathcal{B}_i \otimes Irr\;\Lambda $ by
$$C \mapsto  \tilde{e}_i^n b_i(0) \otimes (\tilde{f}_i^*)^nC$$
where $n=\phi^*_i(C)$.

\vspace{.1in}

\begin{prop}\label{P:KSstep1} The map $\Psi_i$ is an embedding of $\g$-crystals.
\end{prop}
\begin{proof} It is obvious that $\Psi_i$ preserves the weight. We have to check that it commutes to the functions $\varepsilon_j, \phi_j$ and that $\Psi_i \circ \tilde{e}_j=\tilde{e}_j \circ \Psi_i, \Psi_i \circ \tilde{f}_j=\tilde{f}_j \circ \Psi_i$ for $j \in I$. We consider the stratification $\Lambda^{\a}=\bigsqcup_{l} \Lambda^{\a,i;l}$ of $\Lambda^{\a}$ which is dual to $\Lambda^{\a}=\bigsqcup_{l} \Lambda^{\a}_{i;l}$, namely
$$\Lambda^{\a,i;l}=\bigg\{ \underline{x}\;|\; dim_{(V_{\a})_i}\bigg( Ker ( \bigoplus_{\substack{h \in H_{\vecdq}\\ s(h)=i}} x_h) \bigg)=l\bigg\}.$$
It is easy to see that if $C \in Irr\;\Lambda^{\a}$ then $\phi^*_i(C)=l$ if and only if $C \cap \Lambda^{\a,i;l}$ is dense in $C$.
 Moreover, there is a Hecke correspondence 
\begin{equation}\label{E:Heckelambda2}
\xymatrix{  & \Lambda^{d\epsilon_i,\a+d\epsilon_i}_{(1)} \ar[dl]_-{p^*} \ar[r]^-{r^*} & \Lambda^{d\epsilon_i,\a+d\epsilon_i} \ar[dr]^-{q^*} &\\
\Lambda^{\a} \simeq \Lambda^{d\epsilon_i} \times \Lambda^{\a} & & & \Lambda^{\a+d\epsilon_i}}
\end{equation}
which induces an isomorphism $\kappa^{\a,i;d}~: Irr\; \Lambda^{\a,i;0} \stackrel{\sim}{\to} Irr\;\Lambda^{\a+d\epsilon_i,i;d}$, and the maps
$Irr\;\Lambda^{\a} \to Irr\;\Lambda^{\a\pm \epsilon_i}$ which one can define by means of formulas (\ref{E:tildedef}) coincide with $\tilde{e}_i^*, \tilde{f}_i^*$. Now let us fix an irreducible component $C \in Irr\;\Lambda^{\a}$ and set $\phi_i^*(C)=l, (\tilde{f}_i^*)^lC=C'$. Choose $\underline{x}' \in C'$ generic so that $\bigoplus_{s(h)=i} x'_{\overline{h}}$ is injective. Elements $\underline{x}$ in $q^*r^*(p^*)^{-1} (\underline{x}')$ may be described as follows. Fix a projection $\pi: V_{\a} \tto V_{\a-l\epsilon_i}$ with kernel $K$ and an identification $V_{\a}\simeq V_{\a-l\epsilon_i} \oplus K$. Then up to $G_{\a}$-conjugation, they are of the form $\underline{x}=\underline{x}' + u$ for some $u \in Hom( \bigoplus_{s(h)=i} V_{\a_{t(h)}}, K)$ which satisfies the condition $u \circ \big(\bigoplus_{s(h)=i} x'_{\overline{h}}\big)=0$. Let us set $W=Coker\;\big( \bigoplus_{s(h)=i} x'_{\overline{h}}\big)$. Then the previous condition on $u$ just means that $u$ is induced by a map $\overline{u} \in Hom(W, K)$. We may summarize this into a diagram
$$\xymatrix{
W \ar[r]^-a \ar[dr]_-{a'} & (V_{\a})_i \ar[d]_-{\pi}\\ & (V_{\a-l\epsilon_i})_i}$$
where
$$a'=\bigoplus_{s(h)=i} x'_{\overline{h}}, \qquad a=\bigoplus_{s(h)=i} x_{\overline{h}}= \big( \bigoplus_{s(h)=i} x'_{\overline{h}} \big) + \overline{u}.$$
We conclude that for generic $\overline{u}$ (i.e for generic $\underline{x}$ in $q^*r^*(p^*)^{-1} (\underline{x}')$) it holds
$$dim\;Im\;a=inf\big( l + dim\;Im\; a', dim\;W)$$
and hence
\begin{equation}\label{E:KSproof1}
codim_{(V_{\a})_i}\;Im\;a= sup \big( codim_{(V_{\a-l\epsilon_i})_i} \;Im\;a', \;\a_i-dim\;W).
\end{equation}
Observe that
\begin{equation}\label{E:KSproof2}
dim\;W =\sum_{s(h)=i} \a_{t(h)} -\a_i +l=\a_i-l - (\a-l\epsilon_i, \epsilon_i).
\end{equation}
Combining (\ref{E:KSproof1}) and (\ref{E:KSproof2}) for $\underline{x}, \underline{x}'$ generic we obtain
\begin{equation}\label{E:KSproof3}
\phi_i(C)= sup \big( \phi_i(C'), l+(\a-l\epsilon_i, \epsilon_i) \big)= \phi_i(\Psi_i(C)).
\end{equation}
The fact that $\varepsilon_i(C)=\varepsilon_i(\Psi_i(C))$ can now be deduce from axiom i) in the definition of $\g$-crystals. From the discussion above, it also trivially entails that $\phi_j(C)=\phi_j(\Psi_i(C))$ and $\varepsilon_j(C)=\varepsilon_j(\Psi_i(C))$ for $j \neq i$.

\vspace{.1in}

Let us turn to the compatibility between $\Psi_i$ and the Kashiwara operators $\tilde{e}_j, \tilde{f}_j$. This is essentially obvious if $j \neq i$ so we assume that $j=i$. By axiom vi) in the definition of crystals, it is enough to deal with $\tilde{f}_i$ only. Thus we have to prove that for $C \in Irr\;\Lambda^{\a}$ with $\phi^*_i(C)=l$ and $C'= (\tilde{f}_i^*)^lC$
\begin{equation}\label{E:KSSproof4}
\Psi_i(\tilde{f}_iC)=\begin{cases} \tilde{e}_i^l b_i(0) \otimes \tilde{f}_iC' & \text{if}\; \varepsilon_i(C') \geq \phi_i(\tilde{e}_i^l b_i(0))=l\\
 \tilde{e}_i^{l-1} b_i(0) \otimes C' & \text{if}\; \varepsilon_i(C') < \phi_i(\tilde{e}_i^l b_i(0))=l.
 \end{cases}
\end{equation}
We will use Lemma~\ref{L:KSproof}. Let $\underline{x}$ be a generic point of $C \in Irr\;\Lambda^{\a}$ and let us keep the same notation $\underline{x}', u, \overline{u}, W, K$ as above. Let $H \subset (V_{\a})_i$ be a generic hyperplane containing $Im\;(a : W \to (V_{\a})_i)$, and let $\underline{y}$ be the restriction of $\underline{x}$ to $\bigoplus_{j\neq i} (V_{\a})_j \oplus H$. This is a generic point of $D=\tilde{f}_iC$.

Assume first that $\varepsilon_i(C') \geq l$. This is equivalent to $dim\;Ker(a': W \to (V_{\a-l\epsilon_i})_i) \geq l$. In this case, $K \subset Im\;a$ (because $\overline{u} \in Hom(W,K)$ can be chosen arbitrarily) and thus $K \subset H$. It follows that $\phi_i^*(D)=\phi_i^*(C)=l$, and that $\pi_i(H)$ is a (generic) hyperplane in $(V_{\a-l\epsilon_i})_i$. From this we get $(\tilde{f}_i^*)^lD=\tilde{f}_iC'$, from which the first case of (\ref{E:KSSproof4}) follows.

Next if  $\varepsilon_i(C') <l$ then $dim\;Ker(a': W \to (V_{\a-l\epsilon_i})_i) < l$ and the generic hyperplane $H$ does not contain $K$. Thus $K \cap H$ is of dimension $l-1$ and $\phi_i^*(D)=l-1$. Moreover, the projection $\pi_i(H) \to (V_{\a-l\epsilon_i})_i$ is surjective and the induced quotient $\pi_*(\underline{y}) \in \Lambda^{\a-l\epsilon_i}$ coincides with $\underline{x}'$. This means that $(\tilde{e}_i^*)^{l-1}D =C'$. This yields the second case of (\ref{E:KSSproof4}) and concludes the proof of Proposition~\ref{P:KSstep1}.
\end{proof}

\vspace{.1in}

\begin{theo}[Kashiwara-Saito]\label{T:KSMain} There exists a (unique) isomorphism of $\g$-crystals $\mathcal{B}^+(\infty) \simeq Irr\;\Lambda$.
\end{theo}
\begin{proof} By Proposition~\ref{P:KSstep0} $Irr\;\Lambda$ is a lowest weight crystal of lowest weight $0$. By Proposition~\ref{P:KSstep1} there are embeddings $\Psi_i:Irr\;\Lambda \to Irr\;\Lambda \otimes \mathcal{B}_i$ for all $i$. These satisfy condition iii) of  the criterion of Kashiwara and Saito by construction, and they satisfy condition iv) because for any $C \in Irr\;\Lambda$ there exists $i \in I$ for which $\phi^*_i(C) >0$. We may thus apply Theorem~\ref{T:KSBinfty2} to conclude.
\end{proof}

\vspace{.3in}

\centerline{\textbf{4.5. Relationship to $\QQ$ and $\PQ$.}}
\addcontentsline{toc}{subsection}{\tocsubsection {}{}{\; 4.5. Relationship to the Hall category.}}

\vspace{.15in}

In this final section we tie up the geometric construction of $\mathcal{B}^+(\infty)$ described here with Lusztig's Hall category $\QQ$ and the set of simple perverse sheaves $\PQ$.  Let $\vec{Q}=(I,H)$ be a quiver as usual, and let $\mathfrak{g}$ be the associated Kac-Moody algebra. All reference to crystals will be to $\mathfrak{g}$-crystals. We still work over $\C$.

\vspace{.1in}

The first link between $\QQ, \PQ$ and $Irr\;\Lambda$ is purely combinatorial~: by construction, there is a canonical bijection $\PQ \sim \mathbf{B}$ between the set of simple perverse sheaves in $\PQ$ and the canonical basis of $\U_{\nu}(\n_+)$. Next, by Theorem~\ref{T:Kassw} there is a a bijection $\mathbf{B} \simeq \mathcal{B}^+(\infty)$, and finally by Theorem~\ref{T:KSMain} there is a bijection
$\mathcal{B}^+(\infty) \simeq Irr\;\Lambda$. Composing all these, we obtain a natural bijection
\begin{equation}\label{E:bij1}
i_{\vec{Q}}~:\PQ \simeq Irr\;\Lambda.
\end{equation}

For $H'$ a different orientation of the graph underlying $\vec{Q}$ and $\vec{Q}'=(I,H')$, the Fourier-Deligne transform provides us with an identification $\boldsymbol{\Theta} :  \mathcal{P}_{\vec{Q}'} \stackrel{\sim}{\to} \PQ$ (see Section~3.4). Thanks to Theorem~\ref{T:LUU}, this is compatible with (\ref{E:bij1}), i.e. $i_{\vec{Q}'}=i_{\vec{Q}} \circ \boldsymbol{\Theta}$. Here we have implicitly used the fact that $\overline{E}_{\a}=T^*E_{\a}$ is orientation independent (for all $\a \in \N^I$).

\vspace{.1in}

There is another, more direct, relationship between the objects of $\PQ$ and $\Lambda$, given in terms of singular supports (or characteristic varieties). We refer to \cite{KasShap} for the relevant theory. We will content ourselves here by saying that if $\mathbb{P}$ is a perverse sheaf (or a semisimple complex) on a smooth complex variety then its singular support $SS(\mathbb{P})$ is a possibly singular Lagrangian subvariety of $T^*X$, which measures in some sense the directions in which $\mathbb{P}$ ''propagates''. 

\begin{theo}[Lusztig]\label{T:SSLU} The singular support $SS(\mathbb{P})$ of any $\mathbb{P} \in \PQ$ lies in $\Lambda$.
\end{theo}
This result is actually not hard to prove and follows from some simple functorial properties of $SS$ with respect to proper maps (see \cite{Lusaff}). Note that there is no reason to expect the converse of Theorem~\ref{T:SSLU} to hold. It is known that the singular support is invariant under Fourier transform, in particular $SS(\boldsymbol{\Theta}(\mathbb{P})=SS(\mathbb{P})$ for any $\mathbb{P} \in \PQ$. In this sense,  $SS(\mathbb{P})$
is a more canonical object than $\mathbb{P}$ itself.

\vspace{.2in}

\addtocounter{theo}{1}
\noindent \textbf{Example \thetheo .} Let us consider the setting of Example~2.20. So $\vec{Q}$ is the cyclic quiver of order $2$, and $\a=\delta=\epsilon_1+\epsilon_2$. Then $E_{\delta}=\C^2$ while $E^{nil}_{\delta}=\{(x,y) \in \C\;|\; xy=0\}=T_x \cup T_y$ is the union of the two coordinate axes in $\C^2$. 
Hence $\overline{E}_{\delta}=T^*\C^2=\{(x,y, x^*,y^*)\;|\; x,y,x^*,y^* \in \C\}$, and the Lusztig nilpotent variety is
$$\Lambda^{\delta}=\big\{(x,y,x^*,y^*)\;|\;xx^*=yy^*=xy=x^*y^*=0 \big\}.$$
There are three $G_{\delta}$-equivariant simple perverse sheaves on $E_{\delta}^{nil}$, namely $\qlb_{\{0\}}, \qlb_{T_x}[1]$ and $\qlb_{T_y}[1]$. Since each of these is the constant sheaf on a smooth close subvariety, we have
$$SS(\qlb_{T_x}[1])=T^*_{T_x}E_{\delta}=\{(x,0,0,y^*)\;|\; x,y^* \in \C\},$$
$$SS(\qlb_{T_y}[1])=T^*_{T_y}E_{\delta}=\{(0,y,x^*,0)\;|\; x^*,y \in \C\},$$
$$SS(\qlb_{\{0\}})=T^*_{\{0\}}E_{\delta}=\{(0,0,x^*,y^*)\;|\; x^*,y^* \in \C\}.$$
Of these three supports, note that only the first two lie in $\Lambda^{\delta}$. This is in accordance with Theorem~\ref{T:SSLU} above; indeed, as we have seen in Example~2.19, $\mathcal{P}^{\delta}=\{ \qlb_{T_x}[1], \qlb_{T_y}[1]\}$. One can check that the bijection $i_{\vec{Q}}$ maps $\qlb_{T_x}[1]$ to $T^*_{T_x} E_{\delta}$ and 
$\qlb_{T_y}[1]$ to $T^*_{T_y} E_{\delta}$.

\vspace{.1in}

Although this example is rather trivial, it also shows how one may use the singular support as a way to discard some potential candidates for $\PQ$. This is
for instance useful in the case of multiples of the imaginary root $\delta$ of a cyclic quiver of order $n \geq 2$. The generalization of the argument above is straightforward~: the singular support of the perverse sheaf $\qlb_{\{0\}}$ on $E_{d\delta}$ is not contained in $\Lambda^{d\delta}$, and hence $\qlb_{\{0\}} \not\in \mathcal{P}^{d\delta}$. Note that $\qlb_{\{0\}}=IC(\mathcal{O}_{(1^d), (1^d), \ldots, (1^d)})$ corresponds to a non aperiodic multipartition in the terminology of Section~2.4. In fact, one may prove using singular supports that only aperiodic multipartitions can belong to $\PQ$. (see Theorem~2.18 and \cite[Section~10]{Lusaff}).
\endexample

\vspace{.2in}

It was conjectured for some time that $SS(\mathbb{P})$ is irreducible for $\mathbb{P} \in \PQ$, at least when $\vec{Q}$ is of finite type. This would have provided a direct, geometric link, between $\PQ$ and $Irr\;\Lambda$. In \cite{KS} Kashiwara and Saito found a counterexample in type $A_5$, for the dimension vector $\a=(2,4,4,4,2)$\footnote{This particular counterexample also turns up in other questions, related for instance to \textit{dual canonical bases}--see \cite{Leclerc}.}. So in general $SS(\mathbb{P})$ consists of several irreducible components. Ideally, one would like to be able to extract a ``leading term'' in $SS(\mathbb{P})$ and obtain the desired bijection. For notational convenience, let us write $\Lambda_{\mathbb{R}}=i_{\vec{Q}}(\mathbb{R})$.

\vspace{.1in}

\begin{theo}[Kashiwara-Saito]\label{T:KSSSup} For any $\mathbb{P} \in \mathbb{P} \in \PQ$, and any $i \in I$
\begin{equation}\label{E:Kimura}
\Lambda_{\mathbb{P}} \subset SS(\mathbb{P}) \subset \Lambda_{\mathbb{P}}\; \cup \hspace{-.1in}\bigcup_{\substack{\mathbb{R} \in \PQ\\ \phi_i(\mathbb{R}) \geq \phi_i(\mathbb{P})}} \Lambda_{\mathbb{R}}.
\end{equation}
\end{theo}

As a consequence, we see that the singular support of $\mathbb{P}$ can be reducible only if there exists a perverse sheaf $\mathbb{R} \in \PQ$ of the same dimension vector as $\mathbb{P}$ satisfying $\phi_i(\mathbb{R}) \geq \phi_i(\mathbb{P})$ for \textit{all} $i \in I$.

\vspace{.2in}

\addtocounter{theo}{1}
\noindent \textbf{Remark \thetheo .} Theorem~\ref{T:KSSSup} is stated in \cite{KS} with the strict inequality $\phi_i(\mathbb{R}) > \phi_i(\mathbb{P})$ in (\ref{E:Kimura}) instead. However this strict inequality sometimes fails, as the following example shows. Let $\vec{Q}$ be the Kronecker quiver (see example 2.28) and consider the dimension vector $2\delta$. The set $\mathcal{P}_{\vec{Q}}^{2\delta}$ contains six elements among which are $IC(U_{2\delta}, \mathcal{L}_{triv})=\mathbbm{1}_{2\delta}$ and $IC(U_{2\delta},\mathcal{L}_{sign})$. Let us denote by $\pi: T^*\underline{\mathcal{M}}_{\vec{Q}}^{2\delta} \to \underline{\mathcal{M}}_{\vec{Q}}^{2\delta}$ the projection. Let also $\underline{\mathcal{M}}^{\mathbb{R},2\delta}_{\vec{Q}}=E^{\mathbb{R},2\delta}_{\vec{Q}}/G_{2\delta}$ stand for the open subset of $\underline{\mathcal{M}}^{2\delta}_{\vec{Q}}$ parametrizing regular modules (see Section~2.5.). There is a ``support map'' $\rho: \underline{\mathcal{M}}^{\mathbb{R},2\delta}_{\vec{Q}} \to S^2\mathbb{P}^1$ and we have $U_{2\delta}=\rho^{-1}(S^2\mathbb{P}^1 \backslash \Delta_{\mathbb{P}^1})$. One checks that
$$i_{\vec{Q}}(IC(U_{2\delta},\mathcal{L}_{triv}))=\underline{\mathcal{M}}_{\vec{Q}}^{2\delta} $$
(the zero section of the cotangent bundle $\pi$) while 
$$i_{\vec{Q}}(IC(U_{2\delta},\mathcal{L}_{sign}))=\overline{T^*(\rho^{-1}(\Delta_{\mathbb{P}^1})) \cap \underline{\Lambda}^{2\delta}_{\vec{Q}}}.$$
From this one gets $\phi_2(IC(U_{2\delta},\mathcal{L}_{sign}))=\phi_2(IC(U_{2\delta},\mathcal{L}_{triv}))=0$. But on the other hand $i_{\vec{Q}}(IC(U_{2\delta},\mathcal{L}_{triv})) =\underline{\mathcal{M}}_{\vec{Q}}^{2\delta} \subset SS(IC(U_{2\delta},\mathcal{L}_{sign}))$. Hence indeed in this example, there is an equality rather than a strict inequality in (\ref{E:Kimura}).

Note that, because of the above equality phenomenon, it is not directly possible to extract a leading term from the singular support of a simple perverse sheaf $\mathbb{P} \in \PQ$, at least for an arbitrary quiver $\vec{Q}$. The case of finite and affine quivers is the subject of \cite{KimuraNak}.

We thank Y. Kimura for kindly explaining this point to us.

\vspace{.2in}

\addtocounter{theo}{1}
\noindent \textbf{Remark \thetheo .} Let $\vec{Q}$ be a quiver of finite type. In that situation, we have for any dimension vector $\alpha$, $\Lambda^{\alpha}=\mu^{-1}(0)=\bigsqcup_{\mathcal{O}} T^*_{\mathcal{O}}E_{\alpha}$ where $\mathcal{O}$ runs among the set of $G_{\alpha}$-orbits in $E_{\alpha}$ (see Examples 4.10, 4.13). This provides a canonical parametrization of the set of irreducible components of $\Lambda^{\alpha}_{\vec{Q}}$ by $G_{\alpha}$-orbits.
On the other hand, $\mathcal{P}^{\alpha}_{\vec{Q}}=\{IC(\mathcal{O})\;|\; \mathcal{O}\}$ (see Section~2.3). Thus we have a chain of canonical bijections 
$$\mathcal{P}^{\alpha}_{\vec{Q}} \leftrightarrow \{\mathcal{O} \;|\; \mathcal{O} \in E_{\alpha}/G_{\alpha}\} \leftrightarrow Irr\;\Lambda^{\alpha}_{\vec{Q}}.$$ 
It turns out that the composed bijection $\mathcal{P}^{\alpha}_{\vec{Q}} \leftrightarrow Irr\;\Lambda^{\alpha}_{\vec{Q}}$ coincides with the map $i_{\vec{Q}}$. We thank P. Baumann for this remark (see \cite{Baumann}).

\vspace{.3in}

\centerline{\textbf{4.6. Relationship to the Lusztig graph.}}
\addcontentsline{toc}{subsection}{\tocsubsection {}{}{\; 4.6. Relationship to the Lusztig graph.}}

\vspace{.15in}

The reader who still remembers the content of Section~3.6. will surely wonder if there is any relation between the crystal graph $\mathcal{B}^+(\infty)$ for a Kac-Moody Lie algebra $\g$ and the Lusztig graph $\mathcal{C}_{\vec{Q}}$ of any quiver $\vec{Q}$ associated to $\g$.  The main point of the current Lecture was to show how one may obtain $\mathcal{B}^+(\infty)$ from the cotangent geometry of the moduli spaces $\underline{\mathcal{M}}_{\vec{Q}}$ of representations of $\vec{Q}$. On the other hand, the Lusztig graph encodes some information related to perverse sheaves on the spaces $\underline{\mathcal{M}}_{\vec{Q}}$ themselves. Hence, according to the general philosophy of microlocalization alluded to in the introduction to this Lecture, it is quite natural to expect such a link. 

\vspace{.1in}

More precisely, recall from Sections~3.2. and 4.1. that there is a canonical weight-preserving bijection $\theta~: \mathcal{P}_{\vec{Q}} \simeq \mathbf{B} \simeq \mathcal{B}^+(\infty)$ given by
$$\mathbb{P} \mapsto \Psi^{-1}(\mathbf{b}_{\mathbb{P}}) \mapsto  \Psi^{-1}(\mathbf{b}_{\mathbb{P}}) \;mod\; \nu \mathcal{L}.$$
Here $\mathcal{L}=\bigoplus_{\mathbf{b} \in \mathbf{B}}  \mathcal{A}_0 \mathbf{b} \subset \U_{\nu}(\n_+)$ (see Theorem~\ref{T:Kassw}). In an effort to unburden the notation we will henceforth identify $\mathcal{K}_{\vec{Q}}$ and $\U_{\nu}(\n_+)$ via $\Psi$ and write $\mathbf{b}_{\mathbb{P}}$ in place of $\Psi^{-1}(\mathbf{b}_{\mathbb{P}})$. The set $\mathcal{P}_{\vec{Q}}$ is equipped with an $I$-colored graph structure given by the maps $e_i^\#, f_i^\#$ while $\mathcal{B}^+(\infty)$ carries the $I$-colored graph structure given by the Kashiwara operators $\tilde{e}_i, \tilde{f}_i$. It turns out, as one would expect, that the bijection $\theta$ intertwines these two graph structures\footnote{the author was not able to locate a precise reference for this fact in the litterature.}. We will give a purely algebraic proof for this, which does not rely on any cotangent geometry. 

\vspace{.1in}

\begin{prop}\label{P:LusKas} The bijection $\theta$ realizes an isomorphism of $I$-colored graphs.
\end{prop}
\begin{proof} Let us fix some $i \in I$. We will show that 
\begin{equation}\label{E:LuKa1}
\theta \circ e_i^\# = \tilde{e}_i \circ \theta.
\end{equation}
Let $\vec{Q}'$ be a reorientation of $\vec{Q}$ in which $i$ is a sink and let $\FD~: \mathcal{P}_{\vec{Q}} \simeq \mathcal{P}_{\vec{Q}'}$ be the
Fourier-Deligne transform of Section~3.4.  We may compose $\theta$ with $\FD$ to obtain a bijection $\theta'~: \mathcal{P}_{\vec{Q}'} \simeq \mathcal{B}^+(\infty)$ and (\ref{E:LuKa1}) is equivalent to
\begin{equation}\label{E:LuKa2}
\theta' \circ e_i^\# = \tilde{e}_i \circ \theta'
\end{equation}
since the Lusztig graph $\mathcal{C}_{\vec{Q}}$ is invariant under Fourier-Deligne transform (by construction). 

\vspace{.1in}

We will begin by showing that $\theta'$ restricts to a bijection
\begin{equation}\label{E:LuKa3}
\theta'~: \{ \mathbb{P} \in \mathcal{P}_{i;0}\} \simeq \{ {b} \in \mathcal{B}^+(\infty)\;|\; \tilde{f}_i({b})=0\}.
\end{equation}
Here and below the notation $\mathcal{P}_{i;0}=\bigsqcup_{\gamma} \mathcal{P}^{\gamma}_{i;0}$ refers to the quiver $\vec{Q}'$, see Section~3.5. 

\vspace{.1in}

We first claim that
\begin{equation}\label{E:LuKa4}
\{b \in \mathcal{B}^+(\infty)\;|\; \tilde{f}_i(b)=0\} =\{ \mathbf{b}\;mod\;\nu \mathcal{L}\;|\; \mathbf{b} \not\in E_i \U_{\nu}(\n_+)\}.
\end{equation}
Indeed, let $\mathbf{b} \in \mathbf{B}$ and let us write $\mathbf{b}=u_0 + E_i u_1 + \cdots + E_i^{(k)}u_k$ for some $u_0, \ldots, u_k \in Ker\;f_i'$. By definition we have $\tilde{f}_i(\mathbf{b})=u_1 + \cdots + E_i^{(k-1)}u_k$. Set $b=\mathbf{b}\;mod\;\nu \mathcal{L}$. If $\tilde{f}_i(b)=0$ then $u_1+ \cdots + E_i^{(k-1)}u_k \in \nu \mathcal{L}$. But then $u_0 \neq 0$ for otherwise $\mathbf{b} = \tilde{e}_i \tilde{f}_i(\mathbf{b}) \in \nu \mathcal{L}$ since $\tilde{e}_i\mathcal{L} \subset \mathcal{L}$, which is impossible. It follows that $\mathbf{b} \equiv u_0\not\equiv 0\;(mod\;E_i\U_{\nu}(\n_-))$. Conversely, if $\tilde{f}_i(b) \neq 0$ then $\tilde{e}_i \tilde{f}_i(b) \equiv b \;(\nu \mathcal{L})$ and thus
$u_0 = \mathbf{b}-\tilde{e}_i\tilde{f}_i(\mathbf{b}) \in \nu\mathcal{L}$. On the other hand, from $\overline{\mathbf{b}}=\mathbf{b}$ and from the fact that $\overline{Ker\;f'_i}=Ker\;f'_i$ we get $\overline{u_0}=u_0$. Here $x \mapsto \overline{x}$ is the involution of $\U_{\nu}(\n_+)$ defined in Section~3.1. But since clearly $\nu \mathcal{L} \cap \overline{\nu \mathcal{L}}=\{0\}$ we deduce $u_0=0$. Hence $\mathbf{b}$ is indeed divisible by $E_i$. This proves (\ref{E:LuKa4}).

\vspace{.1in}

Next we invoke the following result whose proof is identical to that of Theorem~\ref{T:Rep} (up to replacing `source' by `sink').

\begin{lem}\label{L:LuKa1} For any $i \in I$ and $n \in \N$  there exists a subset $\mathbf{B}_i^{\geq n} \subset \mathbf{B}$ such that $E_i^n\U_{\nu}(\n_+) = \bigoplus_{\mathbf{b} \in \mathbf{B}^{\geq n}_i} \C(\nu) \mathbf{b}$.
Moreover, if $\vec{Q}'$ is a quiver associated to $\g$ in which $i$ is a sink then $\mathbf{B}^{\geq n}_i=\{\mathbf{b}_{\mathbb{P}}\;|\; \mathbb{P} \in \mathcal{P}_{i; \geq n}\}$.
\end{lem}

\vspace{.1in}

By Lemma~\ref{L:LuKa1} above we have
$$\{\mathbf{b}\;|\; \mathbf{b} \not\in E_i\U_{\nu}(\n_+)\}=\{\mathbf{b}_{\mathbb{P}}\;|\; \mathbb{P} \in \mathcal{P}_{i;0}\}.$$
This together with (\ref{E:LuKa4}) proves (\ref{E:LuKa3}).

\vspace{.1in}

We make the following simple but useful observation~: when we only consider $i$-arrows, the graph $\mathcal{B}^+(\infty)$ (resp. $\mathcal{C}_{\vec{Q}'}$) decomposes into infinite strings $\{b, \tilde{e}_i(b), \tilde{e}_i^2(b), \ldots\}$ (resp. $\{\mathbb{P}, e_i^\#(\mathbb{P}), (e_i^\#)^2(\mathbb{P}), \ldots\}$) starting at the vertices $b$ satisfying $\tilde{f}_i(b)=0$ (resp. the vertices $\mathbb{P}$ satisfying $f_i^\#(\mathbb{P})=0$).
This yields a parametrization of all vertices of $\mathcal{B}^+(\infty)$ as $\{\tilde{e}_i^n(b)\;|\; b \in Ker\;\tilde{f}_i, n \in \N\}$.  Of course, something similar holds for $\mathcal{C}_{\vec{Q}'}$. In view of (\ref{E:LuKa3}) we see that in order to prove (\ref{E:LuKa2}) it is now enough to show that
\begin{equation}\label{E:LuKa5}
\theta'((e_i^\#)^n(\mathbb{P}))=\tilde{e}_i^n(\theta'(\mathbb{P}))
\end{equation}
for all $\mathbb{P} \in \mathcal{P}_{i;0}$ and all $n \in \N$.

\vspace{.1in}

The final argument is based on the next result~:

\begin{lem}\label{L:LuKa2} The following holds~:
\begin{enumerate}
\item[i)] Let $\mathbb{P} \in \mathcal{P}_{i;0}$. For any $n \geq 1$
$$\mathbf{b}_{(e_i^\#)^n(\mathbb{P})} \equiv E_i^{(n)} \mathbf{b}_{\mathbb{P}} \;(mod\; E_i^{n+1}\U_{\nu}(\n_+)).$$
\item[ii)] Let $b \in \mathcal{B}^+(\infty)$ satisfy $\tilde{f}_i(b)=0$ and $n \in \N$. Let $\mathbf{b}, \mathbf{b}' \in \mathbf{B}$ be such that
$b=\mathbf{b} \;mod\;\nu \mathcal{L}$ and $\tilde{e}_i^n(b)= \mathbf{b}' \;mod\;\nu\mathcal{L}$. Then
$$\mathbf{b}' \equiv E_i^{(n)} \mathbf{b}\;(mod\;E_i^{n+1}\U_{\nu}(\n_+)).$$
\end{enumerate}
\end{lem}
\begin{proof} Statement i) is a direct consequence of Lemma~\ref{L:Proofkey}. To prove ii) we write $\mathbf{b}=u_0 + E_iu_1 + \cdots + E_i^{(k)}u_k$ as usual. Since $\tilde{f}_i(b)=0$ we have, as above $u_0 \neq 0$ and thus 
$$\tilde{e}_i^n(\mathbf{b})=E_i^{(n)}u_0 + \cdots + E_i^{(n+k)}u_k \equiv E_i^{(n)}\mathbf{b}\;(mod\;E_i^{n+1}\U_{\nu}(\n_+)).$$
On the other hand, $\tilde{e}_i^n(b) \equiv \mathbf{b}'\;(mod\;\nu \mathcal{L})$ by definition. Since $E_i^{(n)}\mathbf{b} \not\in \nu \mathcal{L}$ and since $E_i^{n+1}\U_{\nu}(\n_+)$ is compatible with the canonical basis $\mathbf{B}$ we deduce that in fact $\mathbf{b}' \equiv E_i^{(n)}\mathbf{b}\;(mod\;E_i^{n+1}\U_{\nu}(\n_+))$ as wanted.
\end{proof}

\vspace{.1in}

Using Lemma~\ref{L:LuKa2} it is a simple matter to get (\ref{E:LuKa5}). Indeed, let $n \in \N$, $\mathbb{P} \in \mathcal{P}_{i;0}$ and set $b=\theta'(\mathbb{P})$. Thus $b =\mathbf{b}_{\mathbb{P}}\;mod\;\nu \mathcal{L}$. Denote by $\mathbf{b}'$ the unique element of $\mathbf{B}$ satisfying $\tilde{e}_i^n(b) = \mathbf{b}'\;mod\;\nu \mathcal{L}$. Then by Lemma~\ref{L:LuKa2} ii) we have
$$E_i^{(n)} \mathbf{b}_{\mathbb{P}} \equiv \mathbf{b}'\;(mod\;E_i^{n+1}\U_{\nu}(\n_+))$$
while by i)
$$E_i^{(n)}\mathbf{b}_{\mathbb{P}} \equiv \mathbf{b}_{(e_i^\#)^n(\mathbb{P})} \;(mod\; E_i^{n+1}\U_{\nu}(\n_+))$$
hence $\mathbf{b}'=\mathbf{b}_{(e_i^\#)^n(\mathbb{P})}$ and $\theta'((e_i^\#)^n(\mathbb{P}))=\tilde{e}_i^n(b)=\tilde{e}_i^n(\theta'(\mathbb{P}))$ as wanted.

\vspace{.1in}

Now that (\ref{E:LuKa5}) is proved, (\ref{E:LuKa1}) and (\ref{E:LuKa2}) follow. Since (\ref{E:LuKa1}) holds for all $i\in I$ and since the operators $\tilde{f}_i, f_i^\#$ are (quasi)inverses to $\tilde{e}_i$ and $e_i^\#$ we finally deduce that $\theta$ is a graph isomorphism. Proposition~\ref{P:LusKas} is proved.
\end{proof}

\vspace{.2in}

\addtocounter{theo}{1}
\noindent \textbf{Remark \thetheo .} Although Proposition~\ref{P:LusKas} shows that it is possible to find the crystal graph structure already inside the category of perverse sheaves without having to resort to any cotangent geometry, there is a price to pay. Namely the definition of the Lusztig operators $e_i^{\#}$ and $f_i^\#$ requires one to consider \textit{all} possible (re)orientations of the quiver $\vec{Q}$ at once and use the Fourier-Deligne transforms to identity the various categories $\mathcal{Q}_{\vec{Q}}$ .

\newpage

\centerline{\large{\textbf{Lecture~5.}}}
\addcontentsline{toc}{section}{\tocsection {}{}{Lecture~5.}}

\setcounter{section}{5}
\setcounter{theo}{0}
\setcounter{equation}{0}

\vspace{.2in}

In this final Lecture, we elaborate on what could be the correct analog of Lusztig's Hall category in the context of coherent sheaves on smooth projective curves, and describe the partial results obtained in that direction in \cite{SInvent}, \cite{Scano}. As explained in \cite[Lecture~4]{Trieste}, the categories $Coh(X)$ of coherent sheaves on smooth projective curves bear a strong ressemblance to the categories $Rep_{k}\;\vec{Q}$ of representations of quivers. It was shown there that the Hall algebra $\H_X$ (more precisely the spherical Hall algebra $\mathbf{C}_X$) of a curve $X$ admits a presentation very similar to that of quantum loop groups (in the so-called Drinfeld presentation).

\vspace{.1in}

With this analogy in mind, we consider convolution products of the form
$$L_{\a_1, \ldots, \a_r}=\mathbbm{1}_{\a_1} \star \cdots \star \mathbbm{1}_{\a_r}$$
where $\a_i \in K_0(X)$ is a class of rank at most one, and where $\mathbbm{1}_{\a_r}=\qlb_{\underline{Coh}^{\a}_X}[dim\;\underline{Coh}^{\a}_X]$ is the constant sheaf over the moduli stack $\underline{Coh}^{\a}_X$ parametrizing coherent sheaves over $X$ of class $\a$. Note that there are no simple objects in $Coh(X)$ other than the simple skyscraper sheaves. The construction of the convolution product is due to Laumon \cite{Laumon}. We define the Hall category $\mathcal{Q}_X$ as the (additive) category generated by the simple summands of the semisimple complexes $L_{\a_1, \ldots, \a_r}$, and denote by $\mathcal{P}_X$ the collection of all these simple summands. At this point, one could define a graded Grothendieck ring $\mathcal{K}_X$ directly as in Lecture~3. However, contrary to the case of quivers, the ensuing algebra is different from the spherical subalgebra $\mathbf{C}_X$. This is explained by the fact that the Frobenius eigenvalues of the 
objects of $\mathcal{P}_X$ are not all powers of $q^{1/2}$, but involve as well the Frobenius eigenvalues of the curve $X$ itself. This has to be expected since the category $\mathcal{Q}_X$ (and the composition algebra $\mathbf{C}_X$) depend of course on the particular choice of the curve $X$ within the moduli space of all curves (of a given genus). In other words, any sensible definition of a graded Grothendieck ring $\mathcal{K}_X$ of $\mathcal{Q}_X$ should take into account the finer weight structure of the simple perverse sheaves in $\mathcal{P}_X$.

\vspace{.1in}

This Lecture runs as follows. After a brief description of the moduli stacks $\underline{Coh}^{\a}_X$ and of Laumon's convolution product we define the Hall category $\mathcal{Q}_X$ (Sections~5.1., 5.2.). This category has been determined in several low genus cases (when the genus of $X$ is at  most one--including the case of weighted projective lines)~: we explain these results in Section~5.3. 

All the objects appearing here have been known and studied for some time because of the role they play in the so-called geometric Langlands program (for $GL(r)$). In particular, the complexes $L_{\a_1, \ldots, \a_r}$ are direct summands of Laumon's \textit{Eisenstein sheaves}, and
the direct analog of Lusztig's nilpotent variety $\underline{\Lambda}_{\vec{Q}}$ is the so-called \textit{global nilpotent cone} $\underline{\Lambda}_X$. We briefly explain this point of view in Sections~5.4. and 5.5. To finish we state a few general conjectures on the structure of $\mathcal{Q}_X$ and its precise position in the geometric Langlands program.

\vspace{.3in}

\centerline{\textbf{5.1. Moduli stacks of coherent sheaves on curves.}}
\addcontentsline{toc}{subsection}{\tocsubsection {}{}{\; 5.1. Moduli stacks of coherent sheaves on curves.}}

\vspace{.15in}

Let us fix a connected, smooth projective curve $X_0$ defined over a finite field $\mathbb{F}_q$, and let $X=X_0 \otimes_{\mathbb{F}_q} k$ be its extension of scalars to $k=\fqb$. We denote by $Coh(X)$ the abelian category of coherent sheaves over $X$ and by $K(X)$ its Grothendieck group. The Euler form
$$\langle \mathcal{F},\mathcal{G}\rangle =dim\;Hom(\mathcal{F},\mathcal{G})-dim\;Ext^1(\mathcal{F},\mathcal{G})$$
can be computed by means of the Riemann-Roch formula~: if $X$ is of genus $g$ then
\begin{equation}\label{E:5Euler}
\langle \mathcal{F},\mathcal{G}\rangle=(1-g)rank(\mathcal{F})rank(\mathcal{G}) + rank(\mathcal{F})deg(\mathcal{G})-rank(\mathcal{G})deg(\mathcal{F}).
\end{equation}
We will only use the numerical Grothendieck group $K_0(X)= \Z^2$ in which the class of a sheaf is $[\mathcal{F}]= (rank(\mathcal{F}),deg(\mathcal{F}))$. For $\a \in K_0(X)$ we consider the moduli stack $\underline{Coh}^{\a}_X$ associated to the functor $(Aff/k) \to Groupoids$ given by
$$T \mapsto \left\{ 
\begin{array}{c}
T-\text{flat\;coherent\;sheaves\;}\mathcal{F}\text{\;on\;}T \times X \\ \text{s.t.}\;[\mathcal{F}_{|t}]=\alpha\;\text{for\;all\;closed\;points\;}t \in T
\end{array}\right\}.$$
It is known that $\underline{Coh}^{\a}_X$ is a smooth stack of dimension $-\langle \a, \a \rangle$. Although it is not itself a global quotient, it may be presented as an inductive limit of open substacks which are global quotients as follows. For $\mathcal{L}$ a line bundle over $X$, let $Quot_{\mathcal{L}}^{\a}$ be the projective scheme representing the functor $(Aff/k) \to Sets$ given by
$$T \mapsto \left\{ 
\begin{array}{c}
\phi_T~:\mathcal{L}^{\oplus \langle \mathcal{L},\a\rangle}\boxtimes \mathcal{O}_T \tto \mathcal{F}\;\text{s.t.}\\ \;[\mathcal{F}_{|t}]=\alpha\;\text{for\;all\;closed\;points\;}t \in T
\end{array}\right\}\huge{/}\sim$$
where two maps $\phi,\phi'$ are considered to be equivalent if $Ker\;\phi=Ker\;\phi'$. Let $Q_{\mathcal{L}}^{\a}$ be the open subscheme of $Quot^{\a}_{\mathcal{L}}$ represented by the subfunctor
$$T \mapsto \left\{ 
\begin{array}{c}
\phi_T~:\mathcal{L}^{\oplus \langle \mathcal{L},\a\rangle}\boxtimes \mathcal{O}_T \tto \mathcal{F}\;\text{s.t.}\\ \;[\mathcal{F}_{|t}]=\alpha,\;
\phi_*~:k^{\langle \mathcal{L},\a\rangle} \stackrel{\sim}{\to} Hom(\mathcal{L},\mathcal{F}_{t})\\
\text{for\;all\;closed\;points\;}t \in T
\end{array}\right\}\huge{/}\sim.$$
Note that if $\phi_T$ belongs to $Q^{\a}_{\mathcal{L}}$ then necessarily $\text{Ext}^1(\mathcal{L},\mathcal{F}_{|t})=0$ for all $t \in T$. In addition, $Q^{\a}_{\mathcal{L}}$ is a smooth scheme. The group $G_{\a,\mathcal{L}}:=Aut(\mathcal{L}^{\oplus \langle \mathcal{L},\a\rangle})=GL(\langle \mathcal{L},\a\rangle,k)$ acts on $Q^{\a}_{\mathcal{L}}$ and we can form the quotient stack
$$\underline{Coh}^{\a}_{X,\mathcal{L}}=Q^{\a}_{\mathcal{L}}/G_{\a,\mathcal{L}}.$$
It can be shown that, for any $\mathcal{L}$, $\underline{Coh}^{\a}_{X,\mathcal{L}}$ is canonically identified with an open substack of $\underline{Coh}_{X}^{\a}$ and that these open substacks $\underline{Coh}^{\a}_{X,\mathcal{L}}$ form an open cover
$$\underline{Coh}_{X}^{\a}=\bigcup_{\mathcal{L}} \underline{Coh}^{\a}_{X,\mathcal{L}}=\underset{\longrightarrow}{\text{Lim}}\; \underline{Coh}^{\a}_{X,\mathcal{L}}.$$
Finally, there is another natural open substack $\underline{Bun}^{\a}_X$ of $\underline{Coh}^{\a}_{X}$ which parametrizes vector bundles rather then coherent sheaves. Unless $\a$ is of rank one, $\underline{Bun}^{\a}_{X}$ is also not a global quotient.

\vspace{.1in}

Let us examine a little closer $\underline{Coh}^{\a}_{X}$ in some special cases. Let us first assume that $\a=(0,1)$ so that $\underline{Coh}^{\a}_X$ classifies simple torsion sheaves $\mathcal{O}_x$ of degree one. Since these are in bijection $\mathcal{O}_x \leftrightarrow x$ with closed point of $X$, we have
$$\underline{Coh}^{(0,1)}_{X} \simeq X/G_m$$
(here the multiplicative group $G_m$ acts trivially), and thus $dim\;\underline{Coh}^{(0,1)}_{X}=0$. More generally, the stack $\underline{Coh}^{(0,d)}_X$ classifies torsion sheaves of degree $d$, and contains a dense open substack $\underline{U}^{(0,l)}_X$ classifiying torsion sheaves of the form $\mathcal{O}_{x_1} \oplus \cdots \oplus \mathcal{O}_{x_d}$ where $x_1, \ldots, x_d$ are \textit{distinct} closed points. We have
$$\underline{U}^{(0,l)}_X=\big(S^dX \backslash \Delta\big)/G_m^d$$
where $\Delta=\{(x_i)\;|\; x_k=x_l\;\text{for\;some\;}k \neq l\}$ is the large diagonal, and where again the multiplicative group $G_m^d$ acts trivially. As before, $dim\;\underline{Coh}^{(0,d)}_{X}=dim\;\underline{U}^{(0,l)}_X=0$. Observe that the stacks $\underline{Coh}^{(0,d)}_X$ are global quotients.

Next, let us consider the stacks of coherent sheaves of rank one. The open substack $\underline{Bun}^{(1,d)}_X$ parametrizes line bundles of degree $d$, and hence
$$\underline{Bun}^{(1,d)}_X \simeq Pic_X(d)/G_m.$$
As a consequence, $dim\;\underline{Coh}^{(1,d)}_X=dim\;\underline{Bun}^{(1,d)}_X=g-1$.
Any coherent sheaf of rank one $\mathcal{F}$ decomposes a direct sum $\mathcal{F} \simeq \mathcal{L} \oplus \mathcal{T}$ for some line bundle $\mathcal{L}$ and some torsion sheaf $\mathcal{T}$. We may stratify $\underline{Coh}_X^{(1,d)}$ according to the degree $l$ of $\mathcal{T}$
\begin{equation}\label{E:strat0}
\underline{Coh}^{(1,d)}_X=\bigsqcup_{l \geq 0} \underline{\mathcal{V}}_{(1,d-l),(0,l)}
\end{equation}
and for any $l \geq 0$ there is an affine fibration
$$\pi_l~:\underline{Bun}^{(1,d-l)}_X \times \underline{Coh}^{(0,l)}_X \to \underline{\mathcal{V}}_{(1,d-l),(0,l)}l$$ of rank $l$ (this fibration comes from the natural fibration $Aut(\mathcal{L} \oplus \mathcal{T}) \to Aut(\mathcal{L}) \times Aut(\mathcal{T})$). As we see, there is a rather precise description of the stacks $\underline{Coh}^{\a}_X$ for $\a$ a class of rank at most one.

\vspace{.2in}

Let $D^b(\underline{Coh}^{\a}_X)$ stand for the category of constructible sheaves on $\underline{Coh}^{\a}_X$. The correct formalism of $l$-adic sheaves on stacks has recently been fully developped in \cite{Olsson-Laszlo}, but it is not necessary for us to appeal to this general theory~: indeed our stacks $\underline{Coh}^{\a}_X$ admit open covers by global quotients $\underline{Coh}^{\a}_{X,\mathcal{L}}$, and thus an object $\mathbb{P}$ of $D^b(\underline{Coh}^{\a}_X)$ may be  described as a collection $(\mathbb{P}_{\mathcal{L}})_{\mathcal{L}}$ of objects in $D^b(\underline{Coh}^{\a}_{X,\mathcal{L}})$ for all $\mathcal{L}$ satisfying some suitable compatibility conditions\footnote{that is as a collection of $G_{\a,\mathcal{L}}$-equivariant complexes over $Q^{\a}_{\mathcal{L}}$ for all $\mathcal{L}$ satisfying some suitable compatibility conditions.}. The category of semisimple complexes $D^b(\underline{Coh}^{\a}_X)^{ss}$ admits a very similar description. 

\vspace{.2in}

\addtocounter{theo}{1}
\noindent \textbf{Example \thetheo .} The simplest curve is $X=\mathbb{P}^1$. One can show that for any $\a=(r,d)$ 
$$\underline{Coh}_{\mathbb{P}^1}^{(r,d)}=\bigcup_n \underline{Coh}^{(r,d)}_{\mathbb{P}^1,\mathcal{O}(n)}$$
$$Q^{(r,d)}_{\mathcal{O}(n)}=\big\{ \phi~: \mathcal{O}(n)^{\oplus (r+d-nr)} \tto \mathcal{F}\;|\; Ker\;\phi \simeq \mathcal{O}(n-1)^{\oplus (d-nr)}\big\},$$
$$G_{(r,d),\mathcal{O}(n)}=GL(r+d-nr,k).$$
For instance, 
$$\underline{Coh}^{(1,1)}_{\mathbb{P}^1,\mathcal{O}(1)}=Q^{(1,1)}_{\mathcal{O}(1)}/G_{(1,1),\mathcal{O}(1)}=\{pt\}/G_m$$
(the only point corresponds to the sheaf $\mathcal{O}(1)$ itself), while
$$\underline{Coh}^{(1,1)}_{\mathbb{P}^1,\mathcal{O}}=Q^{(1,1)}_{\mathcal{O}}/G_{(1,1),\mathcal{O}}=\big\{ \mathcal{O}(-1) \hookrightarrow \mathcal{O}^{\oplus 2}\big\}/GL(2,k) \simeq \mathbb{P}^3(k)/GL(2,k)$$
where the action of $GL(2,k)$ on $\mathbb{P}^3(k)$ is as follows~:
$$\begin{pmatrix} a & b\\ c& d \end{pmatrix} \cdot (x:y:z:w)=(ax+cz:ay+cw:bx+dz:by+dw).$$
It is easy to see that there is one dense orbit (corresponding to $\mathcal{O}(1)$) and a $\mathbb{P}^1$-family of orbits of dimension one (corresponding to the sheaves of the form $\mathcal{O} \oplus \mathcal{O}_x, x \in \mathbb{P}^1$).

\vspace{.1in}

Since any vector bundle on $\mathbb{P}^1$ splits as a direct sum of line bundles $\mathcal{O}(n)$, there is a natural stratification
\begin{equation}\label{E:strat1}
\underline{Bun}^{(r,d)}=\bigsqcup_{\underline{n}}\underline{\mathcal{V}}_{\underline{n}}
\end{equation}
where $\underline{n}$ runs among all classes of vector bundles of class $(r,d)$, i.e. among all $r$-tuples of integers $(n_1, \ldots, n_r)$ satisfying $n_1+ \cdots + n_r=d$ and $\underline{\mathcal{V}}_{\underline{n}}$ is the stack classifiying vector bundles isomorphic to $\mathcal{O}(n_1) \oplus \cdots \oplus \mathcal{O}(n_r)$; and likewise
\begin{equation}\label{E:strat2}
\underline{Coh}^{(r,d)}_X=\bigsqcup_{\underline{n},l} \underline{\mathcal{V}}_{\underline{n},l}
\end{equation}
where $(\underline{n},l)$ runs among all $r$-tuples of integers $(n_1, \ldots, n_r)$ satisfying $n_1+ \cdots + n_r=d-l$ where now $\underline{\mathcal{V}}_{\underline{n},l}$ is the stack classifiying vector bundles isomorphic to $\mathcal{O}(n_1) \oplus \cdots \oplus \mathcal{O}(n_r) \oplus \mathcal{T}$ for some torsion sheaf $\mathcal{T}$ of degree $l$. As in the rank one case, there is a natural affine fibration 
$\underline{\mathcal{V}}_{\underline{n}} \times \underline{Coh}^{(0,l)}_{\mathbb{P}^1} \tto \underline{\mathcal{V}}_{\underline{n},l}$.

Note that the stratifications (\ref{E:strat1}) and (\ref{E:strat2}) are \textit{locally finite}, that is, for any fixed line bundle $\mathcal{O}(n)$, there are only finitely many stratas intersecting $\underline{Bun}^{(r,d)}_{\mathbb{P}^1, \mathcal{O}(n)}$ or $\underline{Coh}^{(r,d)}_{\mathbb{P}^1, \mathcal{O}(n)}$. In particular, from this one can easily deduce that any simple perverse sheaf on $\underline{Bun}^{(r,d)}_{\mathbb{P}^1}$ is isomorphic to $IC(\underline{\mathcal{V}}_{\underline{n}})$ for some $\underline{n}$ (compare with Section~2.3. and finite type quivers). As we will see, the case of $\underline{Coh}^{(r,d)}_{\mathbb{P}^1}$ is slightly more complex since the spaces $\underline{Coh}^{(0,l)}_{\mathbb{P}^1}$ do carry some nontrivial perverse sheaves.\endexample

\vspace{.3in}

\centerline{\textbf{5.2. Convolution functors and the Hall category.}}
\addcontentsline{toc}{subsection}{\tocsubsection {}{}{\; 5.2. Convolution functors and the Hall category.}}

\vspace{.15in}

As in the case of quivers, there is a natural convolution diagram
\begin{equation}\label{E:convolcoh}
\xymatrix{ & \underline{\mathcal{E}}^{\a,\beta} \ar[dl]_-{p_1} \ar[dr]^-{p_2} &\\ 
\underline{Coh}^{\a}_X \times \underline{Coh}^{\beta}_X & & \underline{Coh}^{\a+\beta}_X}
\end{equation}
where $\underline{\mathcal{E}}^{\a,\beta}$ is the stack associated to the functor $(Aff/k) \to Groupoids$ defined by
$$T \mapsto \left\{ 
\begin{array}{c} \text{pairs\;of\;}
T-\text{flat\;coherent\;sheaves\;}\mathcal{F}\supset \mathcal{G}\text{\;on\;}T \times X \\ \text{s.t.}\;[\mathcal{F}_{|t}]=\a+\beta, [\mathcal{G}_{|t}]=\beta\;\text{for\;all\;closed\;points\;}t \in T
\end{array}\right\}$$
and where the morphisms $p_1, p_2$ are given by $p_1(\mathcal{F} \supset \mathcal{G}) = (\mathcal{F}/\mathcal{G}, \mathcal{G})$ and $p_2(\mathcal{F} \supset \mathcal{G})=\mathcal{G}$. We will spare the reader the translation of (\ref{E:convolcoh}) in terms of the quot schemes (see however \cite{SquotcanIweb} ). It is important to note that the morphism $p_2$ is proper and representable~: the fiber over a point $\mathcal{F}$ is isomorphic to the projective {scheme} $Quot_{\mathcal{F}}^{\beta}$ parametrizing all subsheaves of $\mathcal{F}$ of class $\beta$. The morphism $p_1$ is smooth. This allows us to define convolution functors of induction and restriction as before
\begin{equation}\label{E:killbill}
\begin{split}
\underline{m}~: D^b(\underline{Coh}^{\alpha}_{X} \times \underline{Coh}^{\beta}_{X}) &\to  D^b(\underline{Coh}^{\alpha+\beta}_{X} )\\
\mathbb{P} &\mapsto p_{2!} p_1^* (\mathbb{P})[dim\;p_1],
\end{split}
\end{equation}
and 
\begin{equation}
\begin{split}
\underline{\Delta}~: D^b(\underline{Coh}^{\alpha+\beta}_{X})  &\to  D^b(\underline{Coh}^{\alpha}_{X} \times \underline{Coh}^{\beta}_{X}) \\
\mathbb{P} &\mapsto p_{1!} p_2^* (\mathbb{P})[dim\;p_2].
\end{split}
\end{equation}
These have properties very similar to those of their cousins in the framework of quivers~: for instance, $\underline{m}$ commutes with Verdier duality and preserves the subcategory of semisimple complexes; both functors are associative in the appropriate sense. We will sometimes write $\mathbb{P} \star \mathbb{Q}$ for $\underline{m}(\mathbb{P} \boxtimes \mathbb{Q})$.

\vspace{.2in}

We are now in position to give the definition of the Hall category $\mathcal{Q}_X$. For $\a \in K_0(X) =\Z^2$, let us set
$\mathbbm{1}_{\a}=\qlb_{\underline{Coh}^{\a}_X}[dim\;  \underline{Coh}^{\a}_X]$. A \textit{Lusztig sheaf} is an induction product of the form
$$L_{\a_1, \ldots, \a_r} = \mathbbm{1}_{\a_1} \star \cdots \star \mathbbm{1}_{\a_r}.$$
By the Decomposition theorem, $L_{\a_1, \ldots, \a_r}$ is a semisimple complex. We let $\mathcal{P}_X=\bigsqcup_{\a} \mathcal{P}^{\a}$ stand for the set of all simple perverse sheaves appearing in some Lusztig sheaf $L_{\a_1, \ldots, \a_r}$ where for all $\a_i=(r_i,d_i)$ we have $r_i \leq 1$; we denote by $\mathcal{Q}_X=\bigsqcup_{\a} \mathcal{Q}^{\alpha}$ the additive category generated by the objects of $\mathcal{P}_X$ and their shifts. 
In other words, we take the constant sheaves over the moduli stacks in ranks zero or one as the building blocks for our category, and see what these generate under the induction product.

\vspace{.1in}

The following can be proved just like Proposition~\ref{P:closed}~:

\begin{prop}The category $\mathcal{Q}_X$ is preserved by the functors $\underline{m}$ and $\underline{\Delta}$.
\end{prop}

\vspace{.3in}

\centerline{\textbf{5.3. Examples, curves of low genera.}}
\addcontentsline{toc}{subsection}{\tocsubsection {}{}{\; 5.3. Examples, curves of low genera.}}

\vspace{.15in}

In this Section, we briefly collect  some descriptions of the objects of $\mathcal{P}_X$ in several examples. 

\vspace{.2in}

\addtocounter{theo}{1}
\noindent \textbf{Example \thetheo .} We begin with a partial result, valid for any curve $X$. Recall that the stack $\underline{Coh}^{(0,l)}_X$ contains a dense open substack $\underline{U}_X^{(0,l)}$ which parametrizes torsion sheaves $\mathcal{O}_{x_1} \oplus \cdots \oplus \mathcal{O}_{x_l}$ where $x_1, \ldots, x_d$ are distinct points of $X$. By construction, there is a morphism $\pi_1( \underline{U}_X^{(0,l)})\to BX_l$, the braid group of $X$ of rank $l$, and hence also a morphism
$\pi_1( \underline{U}_X^{(0,l)})\to \mathfrak{S}_l$. An irreducible representation $\chi$ of $\mathfrak{S}_l$
gives rise in this way to an irreducible local system $\mathcal{L}_{\chi}$ on $\underline{U}_X^{(0,l)}$ (as in Section~2.5).

\vspace{.1in}

The following result may be interpreted as a global version of the Springer construction (and is a direct consequence of that construction, see \cite{Laumon} or Section~2.4.)~:

\begin{prop}\label{P:zzx} For any $X$ and any $l \geq 1$ we have
$$\mathcal{P}^{(0,l)}=\big\{ IC(\underline{U}_X^{(0,l)},\mathcal{L}_{\chi})\;|\; \chi \in Irr\;\mathfrak{S}_l\big\}.$$
\end{prop}
\endexample

\vspace{.2in}

\addtocounter{theo}{1}
\noindent \textbf{Example \thetheo .} We still consider an arbitrary curve $X$, and we describe this time the objects of $\mathcal{P}^{\a}$ for $\a$ of rank one. Recall the stratification (\ref{E:strat0})
$\underline{Coh}^{(1,d)}_X=\bigsqcup_{l} \underline{\mathcal{V}}_{(1,d-l),(0,l)}$ and the affine fibration $\pi_l~:\underline{Bun}_X^{(1,d-l)} \times \underline{Coh}^{(0,l)}_X \to \underline{\mathcal{V}}_{(1,d-l),(0,l)}$. Set $\underline{\mathcal{V}}^0_{(1,d-l),(0,l)}=\pi_l (\underline{Bun}_X^{(1,d-l)} \times \underline{U}^{(0,l)}_X)$. This open substack of $\underline{\mathcal{V}}_{(1,d-l),(0,l)}$ carries, for every representation $\chi$ of $\mathfrak{S}_l$, a unique local system $\widetilde{\mathcal{L}}_{\chi}$ such that $\pi_l^*(\widetilde{\mathcal{L}}_{\chi})=\qlb \boxtimes \mathcal{L}_{\chi}$.

\vspace{.1in}

\begin{prop}\label{P:xzs} For any $X$ and any $d \in \Z$ we have
$$\mathcal{P}^{(1,d)}=\big\{ IC(\underline{\mathcal{V}}^0_{(1,d-l),(0,l)},\widetilde{\mathcal{L}}_{\chi})\;|\; \chi \in Irr\;\mathfrak{S}_l, l \geq 0 \big\}.$$
\end{prop}

\begin{proof} The proof is easy so we sketch it briefly. The support of a simple perverse sheaf $\mathbb{P}$ over $\underline{Coh}_X^{(1,d)}$ has a dense intersection with precisely one of the stratas $\underline{\mathcal{V}}_{(1,d-l),(0,l)}$; we will say that such a $\mathbb{P}$ is \textit{generically supported} on that strata. Let us first consider the case of a $\mathbb{P} \in \mathcal{P}^{(1,d)}$ generically supported on $\underline{Bun}_X^{(1,d)}$. By definition, such a $\mathbb{P}$ appears in a Lusztig sheaf $L_{\a_1, \ldots, \a_r}$ with $\sum \a_i=(1,d)$. It is easy to see that $(L_{\a_1, \ldots, \a_r})_{|\underline{Bun}_X^{(1,d)}} \neq 0$ if and only if $rk(\a_1)=\cdots =rank(\a_{r-1})$ and $rk(\a_r)=1$. Thus $\mathbb{P}$ appears in some product $\mathbb{R} \star \mathbbm{1}_{\a_r}$ with $\mathbb{R}$ belonging to $\mathcal{P}^{(0,l)}$ for some $l$. Because any object of $\mathcal{P}^{(0,l,)}$ itself appears in $\mathbbm{1}_{(0,1)}^l$, it is actually enough to consider products of the form $\mathbbm{1}_{(0,1)}^l \star \mathbbm{1}_{\a}$. When $l=1$, we have by construction
\begin{equation*}
\begin{split}
\mathbbm{1}_{(0,1)} \star \mathbbm{1}_{\a}=& p_{2!}p_1^*(\qlb_{\underline{Coh}^{(0,1)}_X}\boxtimes \qlb_{\underline{Coh}^{\a}_X} \qlb_{X})[dim\;  \underline{Coh}^{\a}_X+dim\;p_1]\\
=&p_{2!}(\qlb_{\underline{\mathcal{E}}^{(0,1),\a}})[dim\;\underline{\mathcal{E}}^{(0,1),\a}].
\end{split}
\end{equation*}
The fiber of $p_2$ over a vector bundle $\mathcal{F}$ is equal to $Quot^{(0,1)}_{\mathcal{F}}$. But  any nonzero morphism to a simple torsion sheaf $\mathcal{O}_x, x \in X$ is surjective, and we have $Hom(\mathcal{F},\mathcal{O}_x)=1$ for all $x$. It follows that
$$p_{2!}(\qlb_{\underline{\mathcal{E}}^{(0,1),\a}})[dim\;\underline{\mathcal{E}}^{(0,1),\a}]_{|\underline{Bun}^{\a+(0,1)}_X} \simeq \big( \mathbbm{1}_{\a+(0,1)} \otimes H^*(X,\qlb)[1]\big)_{|\underline{Bun}_X^{\a+(0,1)}}.$$
In particular, $(\mathbbm{1}_{(0,1)} \star \mathbbm{1}_{\a})_{|\underline{Bun}^{\a+(0,1)}_X}$ is again a constant sheaf. By induction, we see that the same is true of any product $\mathbbm{1}_{(0,1)}^l \star \mathbbm{1}_{\a}$. The proves Proposition~\ref{P:xzs} for objects $\mathbb{P}$ generically supported on $\underline{Bun}_X^{(1,d)}$. If $\mathbb{P}$ is generically supported on some lower strata 
$\underline{\mathcal{V}}_{(1,d-l),(0,l)}$ then we apply the restriction functor $\underline{\Delta}_{(1,d-l),(0,l)}$ to $\mathbb{P}$. This yields an object of $\mathcal{Q}^{(1,d-l)}\times \mathcal{Q}^{(0,l)}$ which is generically supported on $\underline{Bun}^{(1,d-l)}_X \times \underline{Coh}^{(0,l)}_X$. Using the first part of the proof together with Proposition~\ref{P:zzx}, and the setup of Remark~2.13, one arrives at the conclusion that $\mathbb{P}$ is of the desired form.
\end{proof}
\endexample

\vspace{.2in}

\addtocounter{theo}{1}
\noindent \textbf{Example \thetheo .} Let us now specialize to $X=\mathbb{P}^1$. This case was fully treated by Laumon in \cite{Laumon}. The relevant stratification data here is
$$\underline{Coh}_{\mathbb{P}^1}^{(r,d)}=\bigsqcup_{\underline{n},l} \underline{\mathcal{V}}_{\underline{n},l},$$
$$\pi_{\underline{n},l}~: \underline{\mathcal{V}}_{\underline{n}} \times \underline{Coh}^{(0,l)}_{\mathbb{P}^1} \to \underline{\mathcal{V}}_{\underline{n},l}.$$
Again, to any $\underline{n},l$ and any representation $\chi$ of $\mathfrak{S}_{l}$ is associated a local system $\widetilde{\mathcal{L}}_{\chi}$ over $\underline{\mathcal{V}}^0_{\underline{n},l}:=\pi_{\underline{n},l}(\underline{\mathcal{V}}_{\underline{n}} \times \underline{U}^{(0,l)}_{\mathbb{P}^1} )$.

\vspace{.1in}

\begin{theo}[Laumon] Let $X=\mathbb{P}^1$. For any $(r,d)$ we have
$$\mathcal{P}^{(r,d)}=\big\{ IC(\underline{\mathcal{V}}^0_{\underline{n},l}, \widetilde{\mathcal{L}}_{\chi})\;|\; \chi \in Irr\;\mathfrak{S}_l\}$$
where $l \geq 0$ and $\underline{n}=(n_1 \leq \cdots \leq n_s)$ runs among all tuples such that
$\sum_i n_i+l=d$.
\end{theo}

The above theorem can be proved in much the same way as Proposition~\ref{P:xzs}. Observe that it states in particular that, after restriction to the open susbtack $\underline{Bun}_X$, we obtain \textit{all} possible simple perverse sheaves. In this sense, $\mathbb{P}^1$ behaves very much like a quiver of finite type (compare with Section~2.3.).
\endexample

\vspace{.2in}

\addtocounter{theo}{1}
\noindent \textbf{Example \thetheo .} To conclude the list of examples we consider the case of an elliptic curve $X$. Let us briefly recall the structure of the category $Coh(X)$ (see \cite{Atiyah} or \cite{Scano} for a more detailed treatment). Define the \textit{slope} of a sheaf $\mathcal{F} \in Coh(X)$ as
$$\mu(\mathcal{F})=\frac{deg(\mathcal{F})}{rank(\mathcal{F})} \in \mathbb{Q} \cup \{\infty\}.$$
We say that $\mathcal{F}$ is semistable of slope $\nu$ if $\mu(\mathcal{F})=\nu$ and if $\mu(\mathcal{G}) \leq \nu$ for any subsheaf $\mathcal{G} \subset \mathcal{F}$. The full subcategory $Coh^{(\nu)}(X)$ consisting of semsitable sheaves of slope $\nu$ is abelian and artinian. For instance, $Coh^{(\infty)}(X) = Tor(X)$ is the subcategory of torsion sheaves over $X$. Atiyah's proved the following fundamental theorem~:

\vspace{.1in}

\begin{theo}[Atiyah] The following holds~:
\begin{enumerate}
\item[i)] there are (canonical) equivalences of abelian categories $Coh^{(\nu)}(X) \simeq Coh^{(\nu')}(X)$ for any $\nu,\nu' \in \mathbb{Q} \cup \{\infty\}$,
\item[ii)] any $\mathcal{F} \in Coh(X)$ admits a decomposition
\begin{equation}\label{E:plkio}
\mathcal{F} \simeq \mathcal{F}_{\nu_1} \oplus \cdots \oplus \mathcal{F}_{\nu_l}
\end{equation}
where $\mathcal{F}_{\nu_i} \in Coh^{(\nu_i)}(X)$ and $\nu_1 < \cdots < \nu_l$.
\end{enumerate}
\end{theo}
 
 The decomposition (\ref{E:plkio}) is \textit{not} canonical, but the isomorphism classes of the factors $\mathcal{F}_{\nu_i}$ are uniquely determined. In addition, the subsheaves $\mathcal{F}_{\nu_l}, \mathcal{F}_{\nu_{l-1}} \oplus \mathcal{F}_{\nu_l}, \cdots$ are canonical\footnote{The decomposition (\ref{E:plkio}) refines the standard decomposition of a coherent sheaf over a curve as a direct sum of a vector bundle and a torsion sheaf.}. Atiyah's Theorem bears the following geometric consequences. Let us call a sequence $\underline{\a}=(\a_1, \ldots, \a_l)$ of classes in $K_0(X)$ satisfying $\sum_i \a_i=\a$ and $\mu(\a_1)< \cdots < \mu(\a_l)$ a \textit{Harder-Narasimhan-type} (or simply \textit{HN-type}) of weight $\a$. For such a sequence, let $\underline{\mathcal{V}}_{\underline{\a}}$ denote the stack parametrizing sheaves
$\mathcal{F}= \mathcal{F}_{\nu_1} \oplus \cdots \oplus \mathcal{F}_{\nu_l}$ with $\mu(\mathcal{F}_i)=\a_i$ for $i=1, \ldots, l$. Then
 
\vspace{.1in}

\begin{cor}  There exists a stratification
$$\underline{Coh}^{\a}_{X} = \bigsqcup_{\underline{\a}} \underline{\mathcal{V}}_{\underline{\a}}$$
indexed by all HN-types $\underline{\a}$ of weight $\a$.
\end{cor}

There is as usual an affine fibration
$$\pi_{\underline{\a}}~: \underline{\mathcal{V}}_{\a_1} \times \cdots \underline{\mathcal{V}}_{\a_l} \to \underline{\mathcal{V}}_{\underline{\a}}.$$
Using the equivalence $Coh^{(\mu(\a_i))}(X) \simeq Coh^{(\infty)}(X)$, one can show that $\underline{\mathcal{V}}_{\a_i} \simeq \underline{Coh}_X^{(0,d(\a_i))}$, where $d((r_i,d_i))=g.c.d(r_i,d_i)$. To put it in words, we have that the whole stack $\underline{Coh}_X$ can be cut into pieces, each of which looks like a cartesian product of the simpler stacks $\underline{Coh}_X^{(0,d)}$. For a fixed class $\a_i$ let $\underline{\mathcal{V}}_{\a_i}^0$ denote the open substack of $\underline{\mathcal{V}}_{\a_i}^0$ corresponding to $\underline{U}_X^{(0,d(\a_i))}$, and if $\chi$ is a representation of $\mathfrak{S}_{d(\a_i)}$ let $\mathcal{L}_{\chi}$ be the associated local system over $\underline{\mathcal{V}}_{\a_i}^0$. If $\a=(\a_1, \ldots, \a_l)$ is an HN type and if $\chi_1, \ldots, \chi_l$ are representations of $\mathfrak{S}_{d(\a_1)}, \ldots, \mathfrak{S}_{d(\a_l)}$ respectively then we write $\widetilde{\mathcal{L}}_{\chi_1, \ldots, \chi_l}$ for the unique local system over $\underline{\mathcal{V}}^0_{\underline{\a}}:=\pi_{\a}(\underline{\mathcal{V}}_{\a_1}^0 \times \cdots \underline{\mathcal{V}}_{\a_l}^0)$ such that $\pi_{\a}^*(\widetilde{\mathcal{L}}_{\chi_1, \ldots, \chi_l}) =\mathcal{L}_{\chi_1} \boxtimes \cdots \boxtimes \mathcal{L}_{\chi_l}$. 

\vspace{.1in}

The following result is proved in \cite{Scano}~:

\vspace{.1in}

\begin{theo}\label{T:Pelliptic} Let $X$ be an elliptic curve. Then for any $\a \in K_0(X)$ we have
$$\mathcal{P}^{\a}=\big\{ IC(\underline{\mathcal{V}}^0_{\underline{\a}},\widetilde{\mathcal{L}}_{\chi_1, \ldots, \chi_l})\;|\; \chi_1 \in Irr\;\mathfrak{S}_{d(\a_1)}, \ldots, \chi_l \in Irr\;\mathfrak{S}_{d(\a_l)}\big\}$$
where $\underline{\a}=(\a_1, \ldots, \a_l)$ runs among all HN types of weight $\a$.
\end{theo}
\endexample

\vspace{.2in}

\addtocounter{theo}{1}
\noindent \textbf{Remark \thetheo .} The examples~5.7. and 5.9. may be generalized to the case of weighted projective lines $\mathbb{X}_{p,\lambda}$ of genus $g<1$ and $g=1$ respectively (see \cite{SInvent} ). Namely, when $g<1$ (i.e. when $\mathbb{X}_{p,\lambda}$ is of finite type) all perverse sheaves on $\underline{Bun}_{\mathbb{X}_{p,\lambda}}$ do belong to $\mathcal{P}_{\mathbb{X}_{p,\lambda}}$, and if $g=1$ (i.e. when $\mathbb{X}_{p,\lambda}$ is of tame, or tubular type) then the objects of $\mathcal{P}_{\mathbb{X}_{p,\lambda}}$ are essentially parametrized by pairs consisting of an HN-type $\underline{\a}$ and a tuple of irreducible representations of some symmetric groups as in Theorem~\ref{T:Pelliptic}.

\vspace{.2in}

\addtocounter{theo}{1}
\noindent \textbf{Remark \thetheo .} In all the examples we have been able to compute, the elements of $\mathcal{P}_X$ are self-dual perverse sheaves. 

\vspace{.3in}

\centerline{\textbf{5.4. Higgs bundles and the global nilpotent cone.}}
\addcontentsline{toc}{subsection}{\tocsubsection {}{}{\; 5.4. Higgs bundles and the global nilpotent cone.}}

\vspace{.15in}

Motivated by the analogy between curves and quivers and the content of Lecture~4, it is natural to expect the existence of a ``Lagrangian'' version of the Hall category $\mathcal{Q}_X$, in which the role of the simple perverse sheaves from $\mathcal{P}_X$ is played by certain Lagrangian subvarieties
in the cotangent stack $T^*\underline{Coh}_X$. It turns out that this cotangent stack is a well-known and well-studied object in geometry, and that there are natural candidates for the Lagrangian subvarieties.
In this Section we assume that we are working over the field of complex numbers $\mathbb{C}$.

\vspace{.2in}

Let $\Omega_X$ denote the canonical bundle of $X$. A \textit{Higgs sheaf} over $X$ consists of a pair $(\mathcal{F}, \varphi)$ where $\mathcal{F} \in Coh(X)$ and where $\varphi$ --the so-called \textit{Higgs field}-- is an element of the vector space $Hom(\mathcal{F}, \mathcal{F} \otimes \Omega_X)$. The notion of (iso)morphism between Higgs sheaves is the obvious one. For $\a \in K_0(X)$ one defines as in Section~5.1. a stack $\underline{Higgs}^{\a}_X$ classifying Higgs sheaves over $X$ of class $\a$.

\vspace{.1in}

The stack $\underline{Higgs}^{\a}_X$ can be identified  with the cotangent stack\footnote{here and after, we make the same abuse of notions as in Lecture~4~: we only consider the underived cotangent stack.} of $\underline{Coh}^{\a}_X$. Let us try to justify this briefly. We have $T^*\underline{Coh}^{\a}_X = \bigcup\;  T^*\underline{Coh}^{\a}_{X,\mathcal{L}}= \underset{\longrightarrow}{\text{Lim}}\; T^*\underline{Coh}^{\a}_{X,\mathcal{L}}$ where $\mathcal{L}$ runs among all line bundles over $X$. Since $\underline{Coh}^{\a}_{X,\mathcal{L}}=Q^{\a}_{\mathcal{L}}/G_{\a,\mathcal{L}}$ is a global quotient, its cotangent stack is obtained by symplectic reduction. The tangent space to $Q^{\a}_{\mathcal{L}}$ at a point $\phi~:\mathcal{L}^{\oplus \langle \mathcal{L},\a\rangle} \tto \mathcal{F}$ is canonically identified with $Hom(Ker\;\phi, \mathcal{F})$. Hence, $T^*Q^{\a}_{\mathcal{L}}$ represents the functor classifying pairs $(\phi~:\mathcal{L}^{\oplus \langle \mathcal{L},\a\rangle} \tto \mathcal{F}, \theta)$ where $\theta \in Hom(Ker\;\phi,\mathcal{F})^*$.
There is a long exact sequence
\begin{align*}
0 &\longrightarrow Hom(\mathcal{F},\mathcal{F}) \longrightarrow Hom(\mathcal{L}^{\oplus \langle \mathcal{L},\a\rangle}, \mathcal{F}) \longrightarrow Hom(Ker\;\phi, \mathcal{F}) \longrightarrow \\
 &\longrightarrow Ext(\mathcal{F}, \mathcal{F}) \longrightarrow Ext(\mathcal{L}^{\oplus \langle \mathcal{L},\a\rangle}, \mathcal{F})=0.
 \end{align*}
Dualizing, we get in particular
\begin{equation}\label{E:momentmydear}
0 \longrightarrow Ext(\mathcal{F},\mathcal{F})^* \longrightarrow  Hom(Ker\;\phi, \mathcal{F})^* \longrightarrow Hom(\mathcal{L}^{\oplus \langle \mathcal{L},\a\rangle}, \mathcal{F})^* =\mathfrak{g}_{\a,\mathcal{L}}^*.
\end{equation}
It can be shown that the moment map $\mu~: T^*_{\phi}Q^{\a}_{\mathcal{L}} \to \mathfrak{g}_{\a,\mathcal{L}}^*$ at the point $\phi$ is given by the connecting map in 
(\ref{E:momentmydear}). Therefore its kernel is $Ext(\mathcal{F},\mathcal{F})^*$, which is itself canonically identified with $Hom(\mathcal{F},\mathcal{F}\otimes \Omega_X)$ via Serre duality. We have shown that
$T^* \underline{Coh}^{\a}_{X,\mathcal{L}}=\mu^{-1}(0)/G_{\a,\mathcal{L}}=\big\{(\phi~:\mathcal{L}^{\oplus \langle \mathcal{L},\a\rangle} \tto \mathcal{F}, \varphi)\;|\; \varphi \in Hom(\mathcal{F},\mathcal{F}\otimes \Omega_X)\big\}/G_{\a,\mathcal{L}}$. The limit (or union) of all these quotient stacks as $\mathcal{L}$ varies is equal to $\underline{Higgs}^{\a}_X$.

\vspace{.2in}

Following Hitchin (see \cite{Hitchin}), consider for each $\a$ the substack $\underline{\Lambda}^{\a}_X \subset \underline{Higgs}^{\a}_X$ classifying Higgs sheaves $(\mathcal{F},\varphi)$ of class $\a$ for which $\varphi$ is \textit{nilpotent} (i.e. for which a sufficiently high composition $\varphi ^l~: \mathcal{F} \to \mathcal{F}\otimes \Omega_X^{\otimes l}$ vanishes). The union of all these substacks as $\a$ varies is called the \textit{global nilpotent cone} of $X$.

\vspace{.1in}

\begin{theo}[Ginzburg, \cite{Ginzburg}] For any $\a$, the substack $\underline{\Lambda}^{\a}_X$ is Lagrangian.
\end{theo}

\vspace{.2in}

\addtocounter{theo}{1}
\noindent \textbf{Example \thetheo .} Assume that $X=\mathbb{P}^1$. Then $\Omega_X = \mathcal{O}(-2)$. If $\mathcal{F}$ is a vector bundle then any $\varphi$ is automatically nilpotent. This is again reminiscent of the behavior of a finite type quiver (compare with Proposition~\ref{P:finiteLag}). In particular the whole conormal bundle $T^*_{\underline{\mathcal{V}}_{\underline{n}}}\underline{Coh}_{X}$ belongs to $\underline{\Lambda}_X$ for any $\underline{n}=(n_1 \leq \cdots \leq n_r)$, and in fact
$$\underline{\Lambda}_X \cap T^*\underline{Bun}_X=\bigcup_{\underline{n}} T^*_{\underline{\mathcal{V}}_{\underline{n}}}\underline{Coh}_{X}$$
is the (locally finite) union of conormal bundles to the various stratas $\underline{\mathcal{V}}_{\underline{n}}$. This describes the irreducible components of $\underline{\Lambda}_X$ which intersect $T^*\underline{Bun}_X$.

More generally, let $\underline{n}=(n_1 \leq \cdots \leq n_r)$ and let $\lambda=(\lambda_1 \geq \cdots \lambda_s)$ be any partition. Let $\underline{Z}_{\underline{n},\lambda}$ be the substack of $\underline{Higgs}_X$ classifying Higgs sheaves $(\mathcal{F},\varphi)$ where 
$$\mathcal{F}\simeq \mathcal{O}(n_1) \oplus \cdots \oplus \mathcal{O}(n_r) \oplus \mathcal{O}_{x_1}^{(\lambda_1)} \oplus \cdots \oplus \mathcal{O}_{x_s}^{(\lambda_s)}$$
with $x_1, \ldots, x_s$ distinct and $\varphi \in Hom(\mathcal{F}, \mathcal{F}(-2))$ nilpotent. 

\vspace{.1in}

\begin{prop}[Pouchin, \cite{Pouchin}] Assume $X=\mathbb{P}^1$. Then the irreducible components of $\underline{\Lambda}_X$ are given by
the closures $\overline{\underline{Z}_{\underline{n},\lambda}}$ 
as $\underline{n}$ varies among all vector bundle types and $\lambda$ varies among all partitions.
\end{prop}
\endexample

\vspace{.2in}

Ginzburg proved in \cite{Ginzburg} that for any $\mathbb{P} \in \mathcal{P}_X$ the singular support
$SS(\mathbb{P})$ is contained in $\underline{\Lambda}_X$. In view of the analogy with quivers and Theorem~\ref{T:KSSSup}, it is natural to make the following

\begin{conj} There exists a partial order $\prec$ on $\mathcal{P}_X$ and a bijection $\mathcal{P}_X \simeq Irr\;\underline{\Lambda}_X, \mathbb{P} \leftrightarrow \Lambda_{\mathbb{P}}$ such that
$$\Lambda_{\mathbb{P}} \subset SS(\mathbb{P}) \subset \Lambda_{\mathbb{P}} \cup \bigcup_{\substack{\mathbb{Q} \in \mathcal{P}_X\\ \mathbb{Q} \prec \mathbb{P}}} \Lambda_{\mathbb{Q}}.$$
\end{conj}

It can be checked (by direct inspection) that this conjecture holds when $X=\mathbb{P}^1$, or when $X$ is an elliptic curve.

\vspace{.1in}

At this point, a natural question arises : is it possible to equip the set $Irr\;\underline{\Lambda}_X$ with an extra structure similar to that of a crystal ? In other words, is it possible to relate the various irreducible components of $\underline{\Lambda}_X$ by well-chosen correspondences, and to obtain in this way a rich and rigid combinatorial structure\footnote{which might even be relevant to some representation theory of the associated Hall algebra ?}? There are some notable differences with the quiver case~: for one thing, the most natural correspondences are now parametrized by the infinite set of line bundles rather then by the finite set of vertices of a quiver, hence we cannot expect a to obtain a crystal in the usual sense. This question has been addressed by Pouchin in the case of $\mathbb{P}^1$ (see \cite{Pouchin}).

\vspace{.3in}

\centerline{\textbf{5.5. Conjectures and relation to the geometric Langlands program.}}
\addcontentsline{toc}{subsection}{\tocsubsection {}{}{\; 5.5. Conjectures and relation to the geometric Langlands program.}}

\vspace{.15in}

In this final Section, we discuss some conjectures and some heuristics. 

\vspace{.1in}

Let's begin with the conjectures. Recall that $X=X_0 \otimes \fqb$ where $X_0$ is a smooth projective curve defined over the finite field $\mathbb{F}_q$. Thus there is a geometric Frobenius automorphism $F$ acting on $X$, and this action extends to the stacks $\underline{Coh}_X, \underline{Bun}_X,\ldots$. For $\a \in K_0(X)$ the constant sheaves $\mathbbm{1}_{\a}$ have an obvious Weil structure. Hence any Lusztig sheaf $L_{\a_1, \ldots, \a_r}$ also has a Weil structure, and by Deligne's theorem $L_{\a_1, \ldots, \a_r}$ is pure of weight zero for this Weil structure (since it is obtained as a proper pushforward of a pure complex).

\vspace{.1in}

\begin{conj}\label{C:one} Any $\mathbb{P} \in \mathcal{P}_X$ has a (canonical) Weil structure for which it is pure of weight zero.
\end{conj}

This does not automatically follow from the fact that Lusztig sheaves have this property~: the Frobenius automorphism $F$ could in principle permute the simple factors of any given Lusztig sheaf $L_{\a_1, \ldots, \a_r}$.

\vspace{.1in}

Granted Conjecture~\ref{C:one} holds, we can construct just as in Section~3.6. a $\C$-algebra and coalgebra $\mathfrak{U}_{X}$ as 
$$\mathfrak{U}_{X}=\bigoplus_{\gamma \in K_0(X)} \mathfrak{U}^{\gamma}, \qquad 
\mathfrak{U}^{\gamma}=\bigoplus_{\mathbb{P} \in \mathcal{P}^{\gamma}} \C \bo_{\mathbb{P}}\;\;;$$
with (co)multiplication given by
$$\bo_{\mathbb{P}'} \cdot \bo_{\mathbb{P}''}=\sum_{\mathbb{P}} Tr(M_{\mathbb{P}', \mathbb{P}''}^{\mathbb{P}}) \bo_{\mathbb{P}}$$
$$\Delta(\bo_{\mathbb{P}})=\sum_{\mathbb{P}',\mathbb{P}''} Tr(N^{\mathbb{P}', \mathbb{P}''}_{\mathbb{P}}) \bo_{\mathbb{P}'} \otimes \bo_{\mathbb{P}''}.$$
Here $M_{\mathbb{P}', \mathbb{P}''}^{\mathbb{P}}$ and $N^{\mathbb{P}', \mathbb{P}''}_{\mathbb{P}}$ are the structure complexes defined by
$$M_{\mathbb{P}', \mathbb{P}''}^{\mathbb{P}}= Hom(\mathbb{P} ,\mathbb{P}' \star \mathbb{P}''), \qquad
N^{\mathbb{P}', \mathbb{P}''}_{\mathbb{P}}= Hom(\mathbb{P}' \otimes \mathbb{P}'', \underline{\Delta}(\mathbb{P})).$$

\vspace{.1in}

Note that the product and coproduct are a priori given by infinite sums, hence $\mathfrak{U}_X$ is only a topological (co)algebra. We may also directly introduce a completion $\widehat{\mathfrak{U}}_X$ of $\mathfrak{U}_X$, defined as
$$\widehat{\mathfrak{U}}_{X}=\bigoplus_{\gamma \in K_0(X)} \widehat{\mathfrak{U}}^{\gamma}, \qquad 
\widehat{\mathfrak{U}}^{\gamma}=\prod_{\mathbb{P} \in \mathcal{P}^{\gamma}} \C \bo_{\mathbb{P}}\;\;;$$
(i.e., we authorize infinite linear combinations of homogeneous elements of the same degree). One can check that  $\widehat{\mathfrak{U}}_{X}$ is still an algebra and a (topological) coalgebra.

\vspace{.1in}

Let $\mathbf{H}_X$ be the Hall algebra of $Coh(X_0)$, and let $\mathbf{C}_X$ be its spherical subalgebra. We briefly recall their definition here and refer to Section~3.6, and \cite[Lecture~4]{Trieste} for more. 
Let us write $Coh^{\gamma}(X_0)(\fq)$ for the set of (isomorphism classes of) coherent sheaves over $X_0$ of class $\gamma$, and let $\C
[Coh^{\gamma}(X_0)(\fq)]$ be the set of $\C$-valued functions with finite support. Then $$\mathbf{H}_X =\bigoplus_{\gamma} \C[Coh^{\gamma}(X_0)(\fq)]$$
as a vector space, equipped with a convolution (co)product (defined by (\ref{E:Hall11}), (\ref{E:Hall22})).
Denote by $\mathbf{1}_{\a}$, $\mathbf{1}^{vec}_{\a}$ the constant functions on $Coh^{\a}(X_0)(\fq)$ and $Bun^{\a}(X_0)(\fq)$ respectively. Then $\mathbf{C}_X \subset \mathbf{H}_X$ is by definition the subalgebra generated by $\{\mathbf{1}_{(0,d)}\;|\; d \geq 1\}$ and $\{\mathbf{1}^{vec}_{(1,n)}\;|\; n \in \Z\}$. Dropping the assumption of finite support in the definition of $\mathbf{H}_X$ we obtain a completed algebra $\widehat{\mathbf{H}}_X$; we write $\widehat{\mathbf{C}}_X$ for the closure of $\mathbf{C}_X$ in $\widehat{\mathbf{H}}_X$.

\vspace{.1in}

As in the case of quivers, there is natural $\C$-linear trace map 
\begin{align*}
tr: \widehat{\mathfrak{U}}_X &\to \widehat{\mathbf{H}}_X\\
\mathfrak{U}^{\gamma} \ni \bo_{\mathbb{P}} &\mapsto \underset{\longrightarrow}{\text{Lim}} \;\nu^{dim\; G_{\gamma,\mathcal{L}}} Tr(\mathbb{P}_{|\underline{Coh}^{\gamma}_{X,\mathcal{L}}}).
\end{align*}
This map is a morphism of algebras and coalgebras.

\vspace{.1in}

\begin{conj}\label{C:conj2} The trace map is an isomorphism of $\widehat{\mathfrak{U}}_X$ onto $\widehat{\mathbf{C}}_X$.
\end{conj}

\vspace{.1in}

Note that it is essential that we compare \textit{completed} algebras on both sides~: for instance the constant sheaf $\mathbbm{1}_{\a}$ is sent, by the trace map, to $\nu^{\langle \a, \a \rangle} \mathbf{1}_{\a}$ which belongs to $\mathbf{H}_X$ only when $\a$ is a torsion class. Conjecture~\ref{C:conj2} has now been proved in genus zero (\cite{Laumon}), one (\cite{Scano}), and for weighted projective lines of genus at most one (\cite{SInvent}). This way one obtains ``canonical bases'' for several interesting quantum groups, like quantum affine and toroidal algebras, or the spherical Cherednik algebra of $GL(\infty)$ (see \cite[Lecture 4.]{Trieste} ). A few explicit examples of such canonical basis elements are computed for $X=\mathbb{P}^1$ in \cite[Section~12]{SInvent}.

\vspace{.2in}

The obvious next question is~: is it possible to define a graded Grothendieck group $\mathcal{K}_X$ of $\mathcal{Q}_X$, which would be an ``integral form'' of $\mathfrak{U}_X$, and to which it would specialize ? In view of Conjecture~\ref{C:conj2}, this should yield integral forms of several interesting quantum groups along with some canonical bases for them. Although one could in principle define $\mathcal{K}_X$ as in Section~3.1., this would not give the desired result unless $X=\mathbb{P}^1$ (or is a weighted projective line) since the ensuing algebra would only depend on one parameter whereas the Hall algebra $\mathbf{H}_X$ (and the spherical subalgebra $\mathbf{C}_X$) depend, in addition to the choice of the finite field $\fq$, on the choice of the particular curve $X$. More precisely (see \cite[Lecture~4]{Trieste}), $\mathbf{C}_X$ admits a presentation depending on the Frobenius eigenvalues in $H^1(X,\qlb)$. 

This motivates the following conjecture. Let $\{\xi_1, \ldots, \xi_{2g}\}$ be the set of Frobenius eigenvalues in $H^1(X,\qlb)$, where $g$ is the genus of $X$; we may arrange these in such a way that $\xi_{2k} \xi_{2k-1}=q$ for $k=1, \ldots, g$.
For any $i \in \Z$, set
$$\mathcal{S}_i=\{\xi_1^{n_1} \cdots \xi_{2g}^{n_{2g}}\;|\;\forall\;i,\; n_i \geq 0 \; ; n_1+\cdots + n_{2g}=i\}.$$

\vspace{.1in}

\begin{conj}\label{C:conj3} The following holds~:
\begin{enumerate}
\item[i)] For any $\mathbb{P},\mathbb{P}', \mathbb{P}'' \in \mathcal{P}_X$, all the Frobenius eigenvalues in $H^i(M_{\mathbb{P}', \mathbb{P}''}^{\mathbb{P}})$ and $H^i(N_{\mathbb{P}}^{\mathbb{P}',\mathbb{P}''})$ belong to $\mathcal{S}_i$.
\item[ii)] For any $\gamma \in K_0(X)$, for any $\mathbb{P} \in \mathcal{P}^{\gamma}$, for any $x^0 \in \underline{Coh}^\gamma(X_0)$ and any $i \in \Z$, all the Frobenius eigenvalues in $H^i(\mathbb{P}_{|x^0})$ belong to $\mathcal{S}_i$.
\end{enumerate}
\end{conj}

\vspace{.1in}

Let us assume that Conjecture~\ref{C:conj3} holds.  Let us introduce the torus
$$T^g_{a}=\{(\zeta_1, \ldots, \zeta_{2g}) \in (\C^*)^{2g}\;|\; \zeta_1\zeta_2=\zeta_{3}\zeta_4=\cdots=\zeta_{2g-1}\zeta_{2g}\} \simeq (\C^*)^{g+1}.$$
For any curve $X$ of genus $g$, we have $(\xi_1, \ldots, \xi_{2g}) \in T^g_a$.
We will call $X$ \textit{generic} if $\chi\big((\xi_1,\ldots, \xi_{2g})\big) \neq 1$ for any character $\chi$ of $T^g_a$. If $X$ is generic and $u \in \C^*$ is any monomial in $\xi_1, \ldots, \xi_{2g}$ then there exists a unique character $\chi_u$ of $T^g_a$ satisfying $u=\chi_u\big( ( \xi_1, \ldots, \xi_{2g})\big)$. This allows us to define, for any $\mathbb{P}, \mathbb{P}',\mathbb{P}''$, some ``formal traces'' 
$$ftr\big( M_{\mathbb{P}', \mathbb{P}''}^{\mathbb{P}}\big)=\sum_{\lambda} \chi_{\lambda}, \qquad ftr\big(N_{\mathbb{P}}^{\mathbb{P}',\mathbb{P}''}\big)=\sum_{\lambda} \chi_{\lambda}$$
taking values in the representation ring $Rep(T^g_a)$, where $\lambda$ runs over all Frobenius eigenvalues of $M_{\mathbb{P}', \mathbb{P}''}^{\mathbb{P}}$ and $N_{\mathbb{P}}^{\mathbb{P}',\mathbb{P}''}$ respectively.
Using these, we may define a graded Grothendieck group $\mathcal{K}_X$ over $Rep(T^g_{a})$ as follows~:
$$\mathcal{K}_X=\bigoplus_{\mathbb{P} \in \mathcal{P}_X} Rep(T^g_{a}) \mathbf{b}_{\mathbb{P}},$$
$$\mathbf{b}_{\mathbb{P}'} \mathbf{b}_{\mathbb{P}''} =\sum_{\mathbb{P}} ftr\big( M_{\mathbb{P}', \mathbb{P}''}^{\mathbb{P}}\big) \mathbf{b}_{\mathbb{P}},$$
$$\Delta(\mathbf{b}_{\mathbb{P}})=\sum_{\mathbb{P}',\mathbb{P}''} ftr\big(N_{\mathbb{P}}^{\mathbb{P}',\mathbb{P}''}\big)\mathbf{b}_{\mathbb{P}'}\otimes \mathbf{b}_{\mathbb{P}''}.$$
Then $\mathcal{K}_X$ is a (topological) algebra and coalgebra, and moreover, $(\mathcal{K}_X)_{|\zeta_i =-\xi_i} \simeq \mathfrak{U}_X$. The formal analogue of Conjecture~\ref{C:conj2} is

\vspace{.1in}

\begin{conj}\label{C:conj4} The algebra $\mathcal{K}_X$ is topologically generated by $\{\mathbf{b}_{\mathbbm{1}_{\a}}\;|\; rank(\a) \leq 1\}$.\footnote{As in the case of quivers, this conjecture implies that $\mathcal{K}_X$ is a (topological) bialgebra--see Corollary~\ref{C:bialg}.}
\end{conj}

\vspace{.1in}

We also expect the following~:

\vspace{.1in}

\begin{conj}\label{C:conj5} For any two generic curves $X,X'$ we have $\mathcal{K}_X \simeq \mathcal{K}_{X'}$.
\end{conj}

\vspace{.1in}

Granted that all of the above conjectures hold (!), we obtain in this way some kind of ``universal spherical Hall algebra'' (or ``universal quantum group'')  $\mathcal{K}_{g}$ defined over $Rep(T^g_a)$ and which only depends on the genus $g$. Conjectures~\ref{C:conj3}--\ref{C:conj5} have been shown to hold for $g=0$ (by Laumon, \cite{Laumon}) or $g=1$ (see \cite{Scano}). 

\vspace{.1in}

What about nongeneric curves ? One can still expect that there exists a natural grading by the character group $\boldsymbol{\chi}(T^a_g)$ on the Hall category $\mathcal{Q}_{X}$, even though it is not directly readable from the weights. Another possibility is to define a Hall category varying \textit{over} the moduli stack $\mathcal{M}_g$ of curves of genus $g$. Better yet, one could try to build a Hall category directly out of the motive of the universal curve over $\mathcal{M}_g$, whose genuine Grothendieck group should be $\mathcal{K}_{g}$. 

\vspace{.1in}

In the same spirit, we propose~:

\vspace{.1in}

\begin{conj}\label{C:conj6} Let $k=\C$. There exists a natural grading by $\boldsymbol{\chi}(T^a_g)$ on the Hall category $\mathcal{Q}_X$, and the associated graded Grothendieck group satisfies $\mathcal{K}_X \simeq \mathcal{K}_{g}$.
\end{conj}

\vspace{.2in}

We have abused enough of the reader's gullibility. Let's conclude these notes with some heuristics related to the geometric Langlands program for $GL(r)$. This lies at the origins of the theory of Hall algebras for curves (it is the principal motivation for \cite{Kapranov}, \cite{Laumon}). For more on the geometric Langlands conjecture, we refer to \cite{Frenkel}. 

\vspace{.1in}

Let us fix a smooth projective curve $X$ over $\mathbb{C}$. Define the $i$th Hecke correspondence (of rank $r$) in this context as the diagram
$$\xymatrix{ & \mathcal{H}ecke_i \ar[dl]_-{\overline{p}_1} \ar[dr]^-{\overline{p}_2} &\\ \underline{Bun}^r_X & & X \times \underline{Bun}^r_X}$$
where we have set $\underline{Bun}^r_X=\bigsqcup_d \underline{Bun}^{(r,d)}_X$, where $\mathcal{H}ecke_i$ is the stack classifying 
$$\left\{ (\mathcal{F} \supset \mathcal{G})\; |\; \begin{array}{c} \mathcal{F}, \mathcal{G}\; \text{vector\;bundles\;on\;X},\; rank(\mathcal{F})=rank(\mathcal{G})=r\\ \mathcal{F}/\mathcal{G} \simeq \mathcal{O}_x^{\oplus i}\;\text{for\;some\;}x \in X\end{array}\right\}$$
and where $\overline{p}_1(\mathcal{F} \supset \mathcal{G})=\mathcal{G}$ and $\overline{p}_2(\mathcal{F} \supset \mathcal{G})=(supp(\mathcal{F}/\mathcal{G}), \mathcal{F})$. We define the $i$th Hecke operator as
\begin{equation*}
\begin{split}
Hecke_i~: D^b(\underline{Bun}^r_X) \to& D^b(X \times \underline{Bun}^r_X)\\
\mathbb{P} \mapsto& \overline{p}_{2!}\overline{p}_1^*(\mathbb{P})[dim\;\overline{p}_1].
\end{split}
\end{equation*}
Note that for any point $i_x: x \hookrightarrow X$ the operator $(i_x^* \times Id) \circ Hecke_i$ is simply the restriction to the open substack $\underline{Bun}^r_X \subset \underline{Coh}^r_X$ of the convolution operator $\mathbb{P} \mapsto \mathbbm{1}_{\mathcal{O}_x^{\oplus i}} \star \mathbb{P}$ in $D^b(\underline{Coh}^r_X)$. Here $\mathbbm{1}_{\mathcal{O}_x^{\oplus i}} \in D^b(\underline{Coh}^0_X)$ is the constant (perverse) sheaf supported on the closed substack paramatrizing sheaves isomorphic to $\mathcal{O}_x^{\oplus n}$. More general Hecke operators may be obtained by replacing $\mathbbm{1}_{\mathcal{O}_x^{\oplus i}}$ with other complexes supported at the point $x$.

\vspace{.1in}

Fix a local system $E$ of rank $r$ over $X$. A constructible complex $\mathbb{P} \in D^b(\underline{Bun}^r_X)$ is said to be a \textit{Hecke eigensheaf with eigenvalue $E$} if for any $i$ we have $Hecke_i(\mathbb{P}) \simeq \Lambda^iE \boxtimes \mathbb{P}$. The $\C$-version of the Langlands conjecture for $GL(r)$ proved in \cite{FGV}, \cite{G} states that

\vspace{.1in}

\begin{theo}[Frenkel, Gaitsgory, Vilonen]\label{T:geolang} For any irreducible local system $E$ of rank $r$ on $X$ there exists an irreducible perverse sheaf $Aut_E$ over $\underline{Bun}^r_X$ which is a Hecke eigensheaf with eigenvalue $E$.
\end{theo}

For example, the constant perverse sheaf $\mathbbm{1}_{\underline{Bun}^1_X}:=\bigoplus_{d} \C_{\underline{Bun}^{(1,d)}_X}[dim\; \underline{Bun}^{(1,d)}_X]$ satisfies $Hecke_1(\mathbbm{1}_{\underline{Bun}^1_X}) \simeq \C_X \boxtimes \mathbbm{1}_{\underline{Bun}^1_X}$ hence it is a Hecke eigensheaf with eigenvalue given by the trivial local system $\C_X$ over $X$.

\vspace{.1in}

The perverse sheaf $Aut_E$ constructed in Theorem~\ref{T:geolang} is moreover expected to be unique. Thus there is in principle a correspondence
\begin{equation}\label{E:corresp1}
\left\{ \begin{array}{c} \text{Irreducible\;local\;systems}\\\text{over\;}X\;\text{of\;rank}\;r\end{array}\right\} \rightarrow \left\{ \begin{array}{c} \text{cuspidal\;Hecke\;eigensheaves}\\\text{over\;}\underline{Bun}^r_X\;\end{array}\right\} 
\end{equation}

One may wonder about what happens when one considers \textit{reducible} local systems, like the trivial local system $\C^r_X$ of rank $r >1$ ? It is easy to prove that if $\mathbb{P}, \mathbb{P}'$ are Hecke eigensheaves with eigenvalues $E$, $E'$ then the induction product\footnote{we state everything for perverse sheaves or complexes over $\underline{Bun}^r_X$ and tend to forget about $\underline{Coh}^r_X$; nevertheless, note that the definition (\ref{E:killbill}) of the induction functor $*$ uses the stacks $\underline{Coh}^r_X$ in a crucial way.} $\mathbb{P} \star \mathbb{P}'$ is a Hecke eigensheaf of eigenvalue $E \oplus E'$. There are however no reasons for $\mathbb{P} \star \mathbb{P}'$ to be either perverse, irreducible, or  for the Hecke eigensheaf to be unique. For instance, to the trivial local system $\C^r_X$ is associated in this way the so-called \textit{Eisenstein sheaf} (see \cite{Laumon})
$$\mathbbm{1}_{\underline{Bun}^1_X} \star \mathbbm{1}_{\underline{Bun}^1_X} \cdots \star \mathbbm{1}_{\underline{Bun}^1_X},$$
which is in general not irreducible~: indeed, its simple direct summands are given \textit{by contruction} by the elements\footnote{more precisely, by the elements of $\mathcal{P}_X^{r}$ whose support intersects $\underline{Bun}^r_X$.} of $\mathcal{P}_X^r=\bigsqcup_d \mathcal{P}_X^{(r,d)}$. Therefore one may think of the Hall category $\mathcal{Q}_X$ as corresponding, in some weak sense, to the collection of trivial local systems $\C^r_X$ over $X$ (for $r \geq 1$).

\vspace{.1in}

It is possible to get a slightly better heuristic if one considers Beilinson and Drinfeld's broad generalization of (\ref{E:corresp1}), according to which there should be an equivalence of derived categories
\begin{equation}\label{E:corresp2}
D(\mathcal{O}_{\underline{Loc}^r_X}) \simeq D(\mathcal{D}_{\underline{Bun}^r_X})
\end{equation}
where $\underline{Loc}^r_X$ is the stack classifying local systems of rank $r$ over $X$, and where $\mathcal{O}, \mathcal{D}$ stand for the categories of $\mathcal{O}$-modules and $\mathcal{D}$-modules respectively. Presumably, (\ref{E:corresp1}) is obtained by restricting (\ref{E:corresp2}) to the skyscraper sheaves at points of $\underline{Loc}^r_X$ corresponding to irreducible local systems. The formal neighborhood $\widehat{\C}^r_X$ of $\C^r_X$ in $\underline{Loc}^r_X$ may be described as follows\footnote{the case of $X = \mathbb{P}^1$ is a little special (see \cite{Lafforgue}) so we assume that $X$ is of genus one or more.} (see \cite[Section~2]{BD}). Let $\underline{Loc}^{r,triv}_X$ be the stack parametrizing local systems on $X$ with a trivialization at a fixed point $x_0 \in X$. Then $\underline{Loc}^r_X=\underline{Loc}^{r,triv}_X/GL(r,\C)$. The tangent complex of $\underline{Loc}^{r,triv}_X$ at $\C^r_X$ is quasi-isomorphic to the cohomology of $X$ with values in $\g=\mathfrak{gl}(r,\C)$
$$\mathbb{T}_{\C^r_X} \underline{Loc}^{r,triv}_X \simeq H^*(X,\g)$$
and $\mathbb{T}_{\C^r_X} \underline{Loc}^{r,triv}_X[1] \simeq H^{*+1}(X,\g)$ has a canonical dg Lie algebra structure given by $(c \otimes g, c' \otimes g') \mapsto c \cdot c' \otimes [g,g']$. The formal neighborhood of $\C^r_X$ in $\underline{Loc}^{r,triv}_X$ is isomorphic to the formal neighborhood around $0$ of the kernel of the Maurer-Cartan equation $[u,u]=0$ in $H^0( \mathbb{T}_{\C^r_X} \underline{Loc}^{r,triv}_X[1]) \simeq H^{1}(X,\g)$. Therefore, $\widehat{\C}^r_X $ is isomorphic to the formal neighborhood of $0$ in
$$\mathcal{C}^r_X:= \{ u \in H^1(X) \otimes \mathfrak{g}\;|\; [u,u]=0\}/GL(r,\C).$$
In particular, if $g$ is the genus of $X$ then $\widehat{\C}^r_X$ is the formal neighborhood of $0$ in the quotient stack
$$\mathcal{C}^r_g=\big\{ (x_1, \ldots, x_{2g})\in \mathfrak{gl}(r,\C)\;|\;\sum_{i=1}^g [x_{2i-1},x_{2i}]=0\big\}/GL(r,\C).$$
The torus 
$$T^g_s=\{(\eta_1, \ldots, \eta_{2g}) \in (\C^*)^{2g}\;|\; \eta_1\eta_2=\eta_{3}\eta_4=\cdots=\eta_{2g-1}\eta_{2g}\}$$
acts on $\mathcal{C}^r_g$ by multiplication. To give an example, if $g=1$ then 
$$\mathcal{C}^r_g=\{(x,y) \in \mathfrak{gl}(r,\C)\;|\; [x,y]=0\}/GL(r,\C)$$ is the commuting variety and $T^g_s=(\C^*)^2$ acts on $\mathcal{C}^r_g$ by $(\eta_1, \eta_2) \cdot (x,y)=(\eta_1 x, \eta_2 y)$.

\vspace{.1in}

We expect a strong relationship between the category $Coh_{T^g_s}(\mathcal{C}^r_g)$ equipped with its natural $\boldsymbol{\chi}(T^g_s)$-grading on the one hand and the category $\mathcal{Q}_X^r=\bigsqcup_{d} \mathcal{Q}^{(r,d)}_X$ equipped\footnote{actually, instead of $\mathcal{Q}_X^r$ --which is semisimple-- it is probably better to consider the $Ext$-algebra of the set of simple perverse sheaves in $\mathcal{P}_X^r$.} with the (hypothetical) $\boldsymbol{\chi}(T^g_a)$-grading coming from Conjecture~\ref{C:conj6} on the other hand. Note that $T^g_s=T^g_a$ but the two tori play rather different roles~: $T^g_s$ is a group of infinitesimal symmetries around the trivial local system in $\underline{Loc}^r_X$ while $T^g_a$ accounts for the moduli space of curves of genus $g$.

\vspace{.1in}

A first step towards understanding this relationship is to pass to the Grothendieck groups. It is possible to define a convolution product of equivariant K-theory groups 
$$m_{r,r'}~:K^{T^g_s}(\mathcal{C}^r_{g}) \otimes K^{T^g_s}(\mathcal{C}^{r'}_g) \to K^{T^g_s}(\mathcal{C}_g^{r+r'})$$
and hence to define an associative $Rep(T^g_s)$-algebra $K^{T^g_s}(\mathcal{C}_g)=\bigoplus_{r \geq 0} K^{T^g_s}(\mathcal{C}^r_{g})$ (see \cite{SV2}, \cite{SV3}). 
Here $Rep(T^g_s) \simeq \Z[ \eta_1^{\pm 1}, \ldots, \eta_{2g}^{\pm 1}]/ \langle \eta_{2i-1}\eta_{2i}=\eta_{2j-1}\eta_{2j}\;\forall\; i, j \rangle$ is the representation ring of $T^g_s$.

\vspace{.1in}

One would hope that the algebra $K^{T^g_s}(\mathcal{C}_g)$ is strongly related to the Grothendieck groups of the Hall categories $\mathcal{Q}_X$ for $X$ of genus $g$, and hence to the ``universal spherical Hall algebra'' $\mathcal{K}_g$. As shown in \cite{SV3}, this is essentially true although there are some subtle issues related to completions. The spherical Hall algebra $\mathbf{C}_X$ has a generic form $\mathbf{C}_g$ defined over $Rep(T^g_a)$. Modulo Conjectures~\ref{C:conj4} and \ref{C:conj5}, $\mathbf{C}_g$ may be embedded into a completion $\widehat{\mathcal{K}}_g$ of $\mathcal{K}_g$ as a dense subalgebra. Morally, $\mathbf{C}_g$ corresponds to those elements of $\widehat{\mathcal{K}}_g$ which are supported on a finite number of HN strata of $\underline{Bun}_X$.

\vspace{.1in}

\begin{theo}[S.-Vasserot]\label{T:SV3} Let $\mathbf{K}_g \subset K^{T^g_s}(\mathcal{C}_g)$ be the $Rep(T^g_s)$-subalgebra generated by $K^{T^g_s}(\mathcal{C}^1_g)$. There exists an algebra isomorphism
\begin{equation}\label{E:SV3}
\Phi~: \mathbf{C}_g \otimes k_a \stackrel{\sim}{\to} \mathbf{K}_g \otimes k_s
\end{equation}
where $k_a=Frac(Rep(T^g_a))=Frac(Rep(T^g_s))=k_s$.
\end{theo}

\vspace{.1in}

The proof of the above theorem relies on an algebraic description of both sides of (\ref{E:SV3}) as some kind of shuffle algebras. When $g=0$, Theorem~\ref{T:SV3} relates two different constructions
of the (Drinfeld) positive half of the quantum group $\mathbf{U}_{\nu}(\widehat{\mathfrak{sl}}_2)$~: one involving the Nakajima quiver variety associated to the quiver $A_1$ and the other in terms of the Hall algebra of $\mathbb{P}^1$. Similarly, when $g=1$ Theorem~\ref{T:SV3} relates two realizations of the positive half of the spherical Cherednik algebra $\mathbf{S}\ddot{\mathbf{H}}_{\infty}$ of type $GL(\infty)$~: one involving the commuting variety and Hilbert schemes of points on $\C^2$ and the other in terms of the Hall algebra of an elliptic curve (see \cite{SV}).

\vspace{.1in}

To finish, let us illustrate the subtle issue of extending (\ref{E:SV3}) to $\widehat{\mathcal{K}}_g$. Fix $r>1$ and let $\mathbf{1}_r$ be the constant function on $\underline{Bun}_X^r$. This element does not belong to $\mathbf{C}_X$ although it belongs to $\mathcal{K}_X$ --it corresponds to the class of the constant perverse sheaf $\mathbbm{1}_r$ on $\underline{Bun}_X^r$. We may, however approximate it by a sequence of elements $\mathbf{1}_{r,\geq n} \in \mathbf{C}_X$, where $\mathbf{1}_{r, \geq n}$ is the characteristic function of the set of bundles of rank $r$ which are generated by $Pic^d(X)$ for $d \geq n$. The elements $\mathbf{1}_{r,\geq n}$ may be lifted to the generic form $\mathbf{C}_g$.

\vspace{.1in}

\begin{prop}[S.-Vasserot] The sequence $\{\Phi(\mathbf{1}_{r,\geq n})\}$ converges to zero as $n$ tends to $-\infty$.\end{prop}

\vspace{.1in}

This seems to suggest that the constant sheaf $\mathbbm{1}_r$ is mapped, under the Langlands correspondence, to an \textit{acyclic} unbounded complex of
coherent sheaves on (the formal neighborhood $\widehat{\mathbb{C}}_X^r$ of the trivial local system in) $\underline{Loc}^r_X$. This is in accordance with the results in \cite{Lafforgue} when $X=\mathbb{P}^1$.

\newpage

\centerline{\large{\textbf{Windows.}}}
\addcontentsline{toc}{section}{Windows}

\setcounter{section}{6}
\setcounter{theo}{0}
\setcounter{equation}{0}

\vspace{.2in}

We wrap up this survey by providing the interested reader with a few pointers towards further developments of the theory and towards some potential research directions.
Needless to say, this section claims no kind of exhaustivity whatsoever, but rather reflects the author's personal taste.

\vspace{.2in}

\noindent
\textit{Motivic Hall algebras.} Passing from the ``naive'' Hall algebra $\mathbf{H}_\mathcal{A}$ of some abelian finitary category $\mathcal{A}$ to the more
geometric Hall category $\mathcal{Q}_X$ essentially amounts to two things~: \\
i) replacing the (discrete) set of isomorphism classes $\mathcal{M}_{\mathcal{A}}$ of objects in
$\mathcal{A}$ by the stack $\underline{\mathcal{M}}_{\mathcal{A}}$ parametrizing objects in $\mathcal{A}$; and ii) replacing functions on $\mathcal{M}_{\mathcal{A}}$ by 
constructible sheaves on $\underline{\mathcal{M}}_{\mathcal{A}}$. In terms of the product and coproduct, we replace a point count, typically the number $\#Z(\mathbb{F}_q)$ of $\mathbb{F}_q$-rational points of some algebraic variety $Z$ parametrizing extensions in the category $\mathcal{A}$, by the whole cohomology $H^*_c(Z \otimes \overline{\mathbb{F}_q})$.

One can try to go even further, and to keep not just the cohomology $H^*_c(Z\otimes \overline{\mathbb{F}_q})$ but the actual variety $Z$ itself, or its \textit{motive}. In other words, one can try to endow a suitable category of motives $M \to \underline{\mathcal{M}}_{\mathcal{A}}$ over $\underline{\mathcal{M}}_{\mathcal{A}}$ with some induction and restriction functors. Such a construction has been worked out by To\"en (unpublished) and Joyce \cite{Joyce2}. This was also later generalized to the $A_{\infty}$ setting by Kontsevich-Soibelman \cite{KoSoi}. The motivation for the papers \cite{Joyce2}, \cite{KoSoi} is the study of the wall-crossing phenomenon in
Donaldson-Thomas invariants on Calabi-Yau manifolds.

\vspace{.2in}

\noindent
\textit{Canonical bases in modules.}  Fix a quiver $\vec{Q}$ associated to a Kac-Moody algebra $\mathfrak{g}$. By Theorem~\ref{T:Rep}, the canonical basis
$\mathbf{B}$ of $\U_\nu(\n_+)$ is compatible with all lowest weight integrable representations $L_{\lambda}$ of $\U_q(\mathfrak{g})$ in the sense that $\{ \mathbf{b} \cdot v_{\lambda}\;|\; \mathbf{b} \cdot v_{\lambda} \neq 0\}$ forms a basis of $L_{\lambda}$. Here $v_{\lambda}$ is the lowest weight vector of $L_{\lambda}$.
It would be interesting to understand an analogue of this result for other Hall algebras equipped with canonical bases, such as the Hall algebras of smooth projective curves, or weighted projective lines. In the simplest case of $\mathbb{P}^1$, the canonical basis $\mathbf{B}$ of $\U_\nu^+(\widehat{\mathfrak{sl}}_2)$ constructed in \cite{SInvent} has been checked to be compatible with some integrable lowest weight representations; one would also expect it to be compatible with certain \textit{finite-dimensional} modules of $\U_q(\widehat{sl}_2)$ since $\mathbf{B}$ is adapted to the Drinfeld realization of $\U_\nu(\widehat{sl}_2)$ (see \cite[Lecture~4]{Trieste}). It might also be interesting to try to develop an algebraic theory of canonical bases for quantum loop algebras adapted to the Drinfeld realization, taking the above example of $\widehat{sl}_2$ as a model.

\vspace{.2in}

\noindent
\textit{Derived equivalences and canonical bases.} As explained in \cite[Lecture~5]{Trieste}, the Hall algebras of two hereditary finitary categories which are derived equivalent are very closely related. More precisely, if $\mathcal{A}$ and $\mathcal{B}$ are as above then a derived equivalence $F: D^b(\mathcal{A}) \to D^b(\mathcal{B})$ induces (under some mild finiteness conditions, see \textit{loc. cit.}) an isomorphism of \textit{Drinfeld doubles} $\mathbf{D}F:~\mathbf{DH}_{\mathcal{A}} \simeq \mathbf{DH}_{\mathcal{B}}$. It seems natural to try to lift the isomorphism $\mathbf{D}F$ to the geometric level; i.e. to try relate in some way the relevant categories of constructible sheaves over the moduli spaces $\underline{\mathcal{M}}_{\mathcal{A}}$ 
and $\underline{\mathcal{M}}_{\mathcal{B}}$ (or possibly some categories of coherent sheaves over the cotangent bundles $T^*\underline{\mathcal{M}}_{\mathcal{A}}$, $T^*\underline{\mathcal{M}}_{\mathcal{B}}$).
Of course, one has to account for the fact that the geometric constructions explained in Lectures~2 and ~4 relate to Hall algebras $\mathbf{H}_{\mathcal{A}}$ rather then to Drinfeld doubles $\mathbf{DH}_{\A}$. For instance, if $\A=Coh(\mathbb{P}^1)$ and $\mathcal{B}=Rep\;\vec{Q}$ where $\vec{Q}$ is the Kronecker quiver then $\mathbf{DC}_{\mathcal{A}} \simeq \U_\nu(\widehat{\mathfrak{sl}}_2) \simeq \mathbf{DC}_{\mathcal{B}}$ but the composition algebras $\mathbf{C}_{\mathcal{A}}$ and $\mathbf{C}_{\mathcal{B}}$ correspond to different choices of a Borel subalgebra in $\widehat{\mathfrak{sl}}_2$ (see \cite[Lecture~5]{Trieste} for details). Hence it is not directly possible, for exemple, to compare the canonical basis of $\mathbf{C}_{\mathcal{A}}$ (coming from the geometry of the moduli space of vector bundles over $\mathbb{P}^1$) with that of $\mathbf{C}_{\mathcal{B}}$ (coming from the geometry of the moduli space of representations of $\vec{Q}$). However, it might be possible to compare the images of these two canonical bases in some \textit{representation} of $\U_\nu(\widehat{\mathfrak{sl}}_2)$. Some simple cases are computed in \cite[Section~11]{SInvent}, where the two bases are shown to agree in some lowest weight integrable representations.

\vspace{.2in}

\noindent
\textit{Khovanov-Lauda-Rouquier algebras and categorification.} For a given quiver $\vec{Q}$,  the category $\mathcal{Q}_{\vec{Q}}$ is by definition semisimple (its objects are by construction semisimple complexes). However the simple perverse sheaves $\mathbb{P} \in \mathcal{P}_{\vec{Q}}$ of course have some nontrivial extensions spaces between themselves in the derived category $D^b(\underline{\mathcal{M}}_{\vec{Q}})$. One way to encode these extension spaces is to consider the
associative Ext-algebra $\mathbf{R}'=\bigoplus_{\mathbb{P}, \mathbb{P}'} Ext^*(\mathbb{P},\mathbb{P}')$, where $\mathbb{P}, \mathbb{P}'$ vary in $\mathcal{P}_{\vec{Q}}$. Here the algebra structure comes from the Yoneda product on the Ext groups in $D^b(\underline{\mathcal{M}}_{\vec{Q}})$. Observe that $\mathbf{R}'$ splits as a direct sum of associative algebras $\mathbf{R}'=\bigoplus_{\gamma \in \N^I} \mathbf{R}'_{\gamma}$ according to the dimension vector for $\vec{Q}$.
It turns out that it is better to consider a slightly different algebra. Recall that to any sequence $\mathbf{i}=(i_1, \ldots, i_l)$ of vertices of $\vec{Q}$ is associated the Lusztig sheaf $L_{\mathbf{i}}:=L_{\epsilon_{i_1}, \ldots, \epsilon_{i_l}}$ (a semisimple complex on $\underline{\mathcal{M}}_{\vec{Q}}$) and that the simple perverse sheaves of $\mathcal{P}_{\vec{Q}}$ are precisely the simple factors of the $L_{\mathbf{i}}$s as the sequence $\mathbf{i}$ varies. Put $\mathbf{R}=\bigoplus_{\mathbf{i},\mathbf{i}'} Ext^*(L_{\mathbf{i}}, L_{\mathbf{i'}})$. There is again a decomposition $\mathbf{R}=\bigoplus_{\gamma \in \N^I} \mathbf{R}_{\gamma}$. Khovanov-Lauda (see \cite{KL}) and independently (and in a wider context) Rouquier (see \cite{Rouquier}) conjectured an explicit, combinatorial presentation for the algebra $\mathbf{R}$, which was later verified by Varagnolo-Vasserot (see \cite{VV2}). The algebra $\mathbf{R}$ and its presentation play a fundamental role in Rouquier's theory of $2$-categorifications of quantum groups (see \textit{loc. cit}). By some standard yoga, the assignement $\mathbb{P} \mapsto \bigoplus_{\mathbf{i}} Ext^*(L_{\mathbf{i}},\mathbb{P})$ induces an isomorphism at the level of the Grothendieck groups $\U_{\nu}(\n_+) \simeq \mathcal{K}_{\vec{Q}} \stackrel{\sim}{\to} K_0(\mathbf{R}\text{-Modgr})$, where $\mathbf{R}$-Modgr is the category of finitely generated graded $\mathbf{R}$-modules. Moreover, this map sends the canonical basis to the classes of indecomposable projective objects. This may be viewed as a realization of the project of `categorification' of $\U_{\nu}(\n_+)$.

\vspace{.2in}

\noindent
\textit{Zheng's geometric construction of integrable representations.} Let $\vec{Q}$ be a quiver as in Lecture~1. The category $\mathcal{Q}_{\vec{Q}}$, along with its set of simple objects $\mathcal{P}_{\vec{Q}}$ and induction and restriction functors $\underline{m}, \underline{\Delta}$, yields a geometric lift of the quantum enveloping algebra $\U_\nu(\n_+)$ with a canonical basis $\mathbf{B}$. It has long been sought to obtain such a geometric lift not only of $\U_q(\n_+)$ itself but also of all the integrable highest weight representations $L_{\lambda}$ of $\U_\nu(\mathfrak{g})$, for $\lambda$ a dominant integral weight of $\mathfrak{g}$. In the recent work \cite{Zheng}, Hao Zheng manages to produce such a lift
by suitably ``truncating'' the category $\mathcal{Q}_{\vec{Q}}$, keeping only objects satisfying certain stability conditions (depending on $\lambda$). The induction and restriction functors, as well as the set of simple objects $\mathcal{P}_{\vec{Q}}$ get ``truncated'' in a similar way. The notion of stability used is directly inspired by Nakajima's theory of quiver varieties (see below). Zheng's construction requires one to consider all orientations of the quiver $\vec{Q}$ at once and to use the Fourier-Deligne transform; once again it is natural to expect that the best formulation would use some kind of Fukaya or Floer category in the cotangent bundle to the moduli spaces of (stable) representations of $\vec{Q}$. 

By a similar method, Zheng is also able to realize geometrically arbitrary tensor products $L_{\lambda_1} \otimes \cdots \otimes L_{\lambda_l}$. This provides a geometric construction of the canonical basis for such tensor products and has some interesting corollaries (such as the positivity of the structure constants in this canonical basis, etc.). 
This construction in terms of perverse sheaves may also be used to construct an explicit `categorification' of all the irreducible modules $L_{\lambda}$ as well as of their tensor products, in the same sense as in the previous paragraph (see \cite{Rouquier}).

\vspace{.1in}

It would be very nice to generalize Zheng's construction to the setting of Hall algebras of curves. This would provide, for instance, a nice class of representations of quantum toroidal algebras (when the curve is a weighted projective line of genus one, see \cite{Trieste}, Lecture~4).

\vspace{.2in}

\noindent
\textit{The semicanonical basis and constructible functions on the nilpotent variety.} Fix again a quiver $\vec{Q}$ associated to a Kac-Moody algebra 
$\mathfrak{g}$. Besides Lusztig's geometric canonical basis $\mathbf{B}$
of $\U_\nu(\n_+)$ (constructed via perverse sheaves on $\underline{\mathcal{M}}_{\vec{Q}}$), and Kashiwara-Saito's crystal graph $\mathcal{B}(\infty)$ (realized in the set of irreducible components of the Lagrangian variety $\underline{\Lambda}_{\vec{Q}} \subset T^* \underline{\mathcal{M}}_{\vec{Q}}$), there is a third construction, due to Lusztig, of a basis $B$ which is somewhat intermediate called the \textit{semicanonical basis} (\cite{LUsemi}). This is a basis of the \textit{classical} enveloping algebra $\U(\n_+)$, and it comes from a realization of $\U(\n_+)$ as a convolution algebra of constructible \textit{functions} on $\underline{\Lambda}_{\vec{Q}}$. Although it differs in general from the specialization of $\mathbf{B}$ at $\nu=1$ the basis $B$ also has a number of remarkable properties, such as compatibility with all lowest weight integrable representations. In addition, it seems to be intimately related to the theory of cluster algebras and 2-Calabi-Yau categories (see \cite{CGL}, \cite{CGL2},...).

In \cite{Pouchin}, Pouchin constructed and studied an analogue of the semicanonical in the case of weighted projective curves $\mathbb{X}_{p,\lambda}$. In that setting, elements of the semicanonical basis are certain constructible functions on the global nilpotent cone $\underline{\Lambda}_{\mathbb{X}_{p,\lambda}}$. The case of arbitrary smooth projective curves remains to be done. It should be interesting to relate these semicanonical bases have with some representation theory or combinatorics (as in \cite{CGL}). 

\vspace{.2in}

\noindent
\textit{Nakajima quiver varieties.} In the mid 90s H. Nakajima introduced a family of smooth algebraic varieties $\mathfrak{M}_{\mathbf{v},\mathbf{w}}$ attached to an arbitrary quiver $\vec{Q}$ and to a pair of dimension vectors $\mathbf{v}, \mathbf{w}$ of $\vec{Q}$. These \textit{quiver varieties} may (roughly) be thought of as moduli spaces of representations of the double quiver $\overline{Q}$ over $\mathbb{C}$, of dimension $\mathbf{v}$, which satisfy the moment map equation $\mu(\underline{x})=0$ as in Section~4.2., and which are stable with respect to some stability condition depending on $\mathbf{w}$. Thus the variety $\mathfrak{M}_{\mathbf{w}}:=\bigsqcup_{\mathbf{v}} \underline{\mathcal{M}}_{\mathbf{v},\mathbf{w}}$ is a kind of smooth approximation of the whole cotangent stack $T^*\underline{\mathcal{M}}_{\vec{Q}}$. The auxiliary parameter $\mathbf{w}$ may be thought of as an integral  dominant (or antidominant) weight $\Lambda_\mathbf{w}:=\sum_i w_i \Lambda_i$ for $\g$.
 Just as $T^*\underline{\mathcal{M}}_{\vec{Q}}$ has the Lagrangian subvariety $\underline{\Lambda}_{\vec{Q}}$, $\mathfrak{M}_{\mathbf{w}}$ comes with a remarkable Lagrangian subvariety $\mathfrak{L}_{\mathbf{w}}$, defined by means of certain nilpotency conditions. Note that $\mathfrak{M}_{\mathbf{w}}$ and $\mathfrak{L}_{\mathbf{w}}$ are genuine algebraic varieties whereas $T^*\underline{\mathcal{M}}_{\vec{Q}}$ and $\underline{\Lambda}_{\vec{Q}}$ are only stacks.

Nakajima's quiver varieties are expected to play a fundamental role in the problem of geometric realization or categorification of integrable lowest (or highest) weight modules $L_{\Lambda_{\mathbf{w}}}$ for $\U_\nu(\g)$, similar to the one which $T^*\underline{\mathcal{M}}_{\vec{Q}}$ is expected to play for $\U_{\nu}(\n_+)$ (see Lecture~4). In particular, one would ideally like to realize the modules $L_{\Lambda_{\mathbf{w}}}$ as the Grothendieck group of a suitable category of coherent sheaves on $\mathfrak{M}_{\mathbf{w}}$ supported on $\mathfrak{L}_{\mathbf{w}}$ or as a subcategory of the Fukaya category of $\mathfrak{M}_{\mathbf{w}}$.

Although --to the author's knowledge-- such a program has not been carried through yet, there exist a profusion of important results relating the quiver varieties $\mathfrak{M}_{\mathbf{w}}$ to the modules $L_{\Lambda_{\mathbf{w}}}$. For instance, the module $L_{\Lambda_\mathbf{w}}$ for the classical enveloping algebra $\U(\g)$ may be realized as the top Borel-Moore homology group $H^{top}(\mathfrak{L}_{\mathbf{w}})$, with the action of the Chevalley generators $e_i, f_i, h_i, i \in I$ of $\g$ given by the convolution with some natural Hecke correspondences (see \cite{Nak1}). Taking the fundamental classes $[C]$ of all the irreducible components $C$ of $\mathfrak{L}_{\mathbf{w}}$ yields a natural basis $B_{\mathbf{w}}$ of $L_{\Lambda_\mathbf{w}}$; it is conjectured to coincide with the projection of the semicanonical basis $B$ onto $L_{\Lambda_{\mathbf{w}}}$.
  
We briefly state a number of other results: Saito equipped the set $Irr\;\mathfrak{L}_{\mathbf{w}}$ of irreducible components of the Lagrangian quiver variety $\mathfrak{L}_{\mathbf{w}}$ with an $I$-colored graph structure, and showed that it is isomorphic to $\mathcal{B}(\Lambda_{\mathbf{w}})$ (see \cite{Saito}); Nakajima realized the classical $L_{\Lambda_{\mathbf{w}}}$ in terms of constructible functions on the same Lagrangian quiver variety $\mathfrak{L}_{\mathbf{w}}$ (see \cite{Nak2}); Finally, he also constructed an action of the quantum loop algebra $\U_{\nu}(L\g)$ on the equivariant K-theory group $K^{G_{\mathbf{w}} \times \C^*}(\mathfrak{L}_{\mathbf{w}})$ and identified the corresponding module (see \cite{Nak3}). This last result has some important consequences for the study of finite dimensional representations of quantum affine algebras. We refer the interested reader to the surveys \cite{Naksurvey}, \cite{Snaksurvey} and the references therein for more.

\vspace{.1in}

The theory of Nakajima quiver varieties has, at the moment, no analogue when the quiver gets replaced by a smooth projective curve. The author can't help thinking that there is a lot yet to be discovered there.

\newpage

\centerline{\textbf{Acknowledgements}}

\vspace{.1in}

I would like to thank all the organisers of the summer school ``Geometric representation theory'' for giving me the opportunity to give these lectures, and the participants for attending them. I am especially grateful to Michel Brion for encouraging me to write up these notes, to M. Anel for making me (slightly) less ignorant about stacks and to J.-F. Dat, D. Juteau, E. Vasserot and G. Williamson for some helpful comments.

\vspace{.2in}

\small{}

\vspace{.2in}

\noindent
Olivier Schiffmann,\\
Institut Math\'ematique de Jussieu, CNRS,\\
175 rue du Chevaleret, 75013 Paris \\
FRANCE,\\
email:\;\texttt{olive@math.jussieu.fr}
\end{document}